%% file: thesis.tex
\documentclass[a4paper,11 pt,twoside,reqno]{amsbook}

\include{basis}

\begin{document}

\addtolength{\abovedisplayskip}{0em plus .3em minus .15em}
\addtolength{\belowdisplayskip}{0em plus .3em minus .15em}
\addtolength{\abovedisplayshortskip}{0em plus .2em minus .1em}
\addtolength{\belowdisplayshortskip}{0em plus .2em minus .1em}



\include{cover}

\begin{singlespace}
\include{abstract}
\include{acknowledgements}
\end{singlespace}

\tableofcontents

\include{introduction}

\include{summary}


\include{terminology}

\include{internal}

\include{sticky}

\include{MS1}

\include{embeddings}

\include{filtered}

\include{conf_emb}

\include{homotopycat}

\include{invariant}


\begin{singlespace}
\include{bibliography}

\end{singlespace}

\end{document}

%% file: basis.tex
\usepackage{fixltx2e,fix-cm}

\usepackage{ifpdf}

\usepackage[T1]{fontenc}
\usepackage{lmodern}
\usepackage[
   protrusion={true,nocompatibility},
   \ifpdf expansion={true,nocompatibility},\else\fi
   spacing=true,
   kerning=true,
   babel]
{microtype}

\usepackage{textcomp,ccicons}

\usepackage{amsfonts,amsmath}
\usepackage{extarrows,nicefrac}
\usepackage{stackrel}

\usepackage{mathtools}
\mathtoolsset{mathic=true,centercolon}

\usepackage{amssymb,stmaryrd}

\usepackage{indentfirst}
\usepackage{setspace}
\allowdisplaybreaks[1]

\makeatletter
\def\@textbottom{\vskip \z@ \@plus 3.5em \@minus 1.5em}\let\@texttop\relax
\makeatother

\renewenvironment*{thebibliography}[1]{%
   \begin{oldthebibliography}{#1}%
      \setlength{\itemsep}{.5em}%
}%
{%
   \end{oldthebibliography}%
}

\usepackage{enumitem}

\newenvironment*{enum}[1][\enumlabel]%
{\begin{enumerate}[label=#1,topsep=0pt,leftmargin=*]}%
{\end{enumerate}}

\newcommand{\enumlabel}{\,-}

\usepackage[silent]{diagrams}
\diagramstyle[size=2.5em,nohug,heads=LaTeX,labelstyle=\scriptstyle,dpi=1270,PS]
\newarrowhead{s>}\shortrightarrow\shortleftarrow\shortdownarrow\shortuparrow
\newcommand{\dhook}{\rotatebox[origin=b]{-90}{\ensuremath{\mkern-1.6mu\lhook\mkern-1.1mu}}}
\newcommand{\uhook}{\rotatebox[origin=b]{90}{\ensuremath{\mkern-1.35mu\lhook\mkern-1mu}}}
\newarrowtail{sC}\lhook\rhook{\dhook\mkern.5mu}{\mkern.5mu\uhook}
\newcommand{\downmapstochar}{\rotatebox[origin=b]{-90}{\ensuremath{\mkern-1.6mu\mapstochar\mkern2mu}}}
\newcommand{\upmapstochar}{\rotatebox[origin=b]{90}{\ensuremath{\mkern-1.5mu\mapstochar\mkern2mu}}}
\newarrowtail{s|}{\mapstochar\mkern3mu}{\mkern3mu\mapsfromchar}{\downmapstochar\mkern2.5mu}{\mkern2.5mu\upmapstochar}
\newarrow{To}----{->}
\newarrow{Mapsto}{s|}---{->}
\newarrow{Into}{sC}---{->}
\newarrow{Cofto}{Y}---{->}  
\newarrow{Equal}=====
\newarrow{Dashto}{}{dash}{}{dash}{s>}

\usepackage{amsthm}

\newtheoremstyle{mystyle}
{.75em plus .25em minus .1em} 
{.8em plus .3em minus .15em}
{}{}{}{}{\newline}
{\thmnumber{#2. }{\scshape\thmname{#1}}\thmnote{ -- {\sffamily #3}}}

\theoremstyle{mystyle}
\newtheorem{theorem}{Theorem}[section]
\newtheorem{lemma}[theorem]{Lemma}
\newtheorem{proposition}[theorem]{Proposition}
\newtheorem{corollary}[theorem]{Corollary}
\newtheorem{definition}[theorem]{Definition}
\newtheorem{conjecture}[theorem]{Conjecture}
\newtheorem{construction}[theorem]{Construction}
\newtheorem{notation}[theorem]{Notation}

\newtheorem{convention}[theorem]{Convention}
\newtheorem{remark}[theorem]{Observation}
\newtheorem{examples}[theorem]{Examples}
\newtheorem{example}[theorem]{Example}
\newtheorem{describe}[theorem]{Description}

\makeatletter
\renewenvironment*{proof}[1][\proofname]{\par
  \pushQED{\qed}%
  \normalfont \topsep6\p@\@plus6\p@\relax
  \trivlist
  \item[\hskip\labelsep
        \scshape
    #1\@addpunct{:}]\mbox{}\par\nopagebreak[4]\ignorespaces
}{%
  \ignorespaces\par\nopagebreak[4]\popQED\endtrivlist\@endpefalse
}
\makeatother

\numberwithin{equation}{section}

\usepackage{hyperref}
\hypersetup{
   unicode=true,
   breaklinks=true,
   linkbordercolor={0 0 1},
   pdfborderstyle={/S/D/D[3 2]/W 1},
   colorlinks=true,
   linkcolor=blue,
   citecolor=green,
   urlcolor=magenta,
}

\usepackage{smartref}

\addtoreflist{chapter}

\newcommand*\defrefnospace[2]{%
   \expandafter\newcommand\expandafter{\csname #1#2refnospace\endcsname}[2]{%
      \hyperref[##2]{##1\csname #2reftagform\endcsname{\csname #1refaux\endcsname{##2}}}%
   }%
}

\def\defrefspace#1#2{%
   \expandafter\newcommand\expandafter{\csname #1#2refnospace\endcsname}[2]{%
      \hyperref[##2]{##1\csname #2reftagform\endcsname{\csname #1refaux\endcsname{##2}}}%
   }%
}

\makeatletter
\newcommand\ExecuteMapList[2]{%
   \@for\@ii:= #2\do{\expandafter #1\@ii}%
}
\makeatother

\ExecuteMapList{\defrefnospace}{{short}{},{short}{eq},{full}{},{full}{eq},{s}{},{s}{eq}}

\newcommand{\shortref}[2][]{\shortrefnospace{#1}{#2}}

\newcommand{\sref}[2][]{\srefnospace{#1}{#2}}

\newcommand{\seqref}[2][]{\seqrefnospace{#1}{#2}}

\DeclareMathAlphabet{\mathpzc}{OT1}{pzc}{m}{it}

\newcommand*{\ie}{i.e.\ }

\makeatletter
\newcommand*\etc{etc\@ifnextchar.{}{.\@}}
\makeatother

\newcommand{\comment}[1]{}

\newcommand{\vcent}[2][]{\vcenter{\hbox{#1\ensuremath{#2}}}}

\newcommand{\newchoice}[6]{\mathchoice{#2{#3#1}}{#2{#4#1}}{#2{#5#1}}{#2{#6#1}}}
\newcommand{\low}[2]{\lower#1\hbox{\ensuremath{#2}}}

\newcommand{\textsymbol}[2][\sizenormal]{\newchoice{#2}{\mbox}#1} 

\comment{
\newcommand{\sizenormal}{\noexpand\normalsize\noexpand\normalsize\noexpand\scriptsize\noexpand\tiny}
\newcommand{\sizelarge}{\noexpand\large\noexpand\large\noexpand\small\noexpand\scriptsize}
\newcommand{\sizeLarge}{\noexpand\Large\noexpand\Large\noexpand\normalsize\noexpand\footnotesize}
\newcommand{\sizeLARGE}{\noexpand\LARGE\noexpand\LARGE\noexpand\large\noexpand\normalsize}
\newcommand{\sizehuge}{\noexpand\huge\noexpand\huge\noexpand\Large\noexpand\large}
\newcommand{\sizeHuge}{\noexpand\Huge\noexpand\Huge\noexpand\LARGE\noexpand\Large}
}

\newcommand{\sizenormal}{\unexpanded\normalsize\normalsize\scriptsize\tiny}
\newcommand{\sizelarge}{\unexpanded\large\large\small\scriptsize}
\newcommand{\sizeLarge}{\unexpanded\Large\Large\normalsize\footnotesize}
\newcommand{\sizeLARGE}{\unexpanded\LARGE\LARGE\large\normalsize}
\newcommand{\sizehuge}{\unexpanded\huge\huge\Large\large}
\newcommand{\sizeHuge}{\unexpanded\Huge\Huge\LARGE\Large}

\newcommand{\rotc}[2]{{\rotatebox[origin=c]{#1}{#2}}}

\newcommand{\func}[6][\rTo]{
\begin{diagram}[h=1.2em,w=.2em]
{#2}&\,:\,&{#3}&#1&{#4}\\
&&{#5}&\rMapsto&{#6}
\end{diagram}}
\newcommand{\sfunc}[5][\rTo]{
\begin{diagram}[h=1.2em,w=.2em]
{#2}&#1&{#3}\\
{#4}&\rMapsto&{#5}
\end{diagram}}

\DeclareMathOperator*{\colimit}{colim}
\newcommand{\colim}{\mathop{\colimit}}
\DeclareMathOperator*{\hocolimit}{hocolim}
\newcommand{\hocolim}{\mathop{\hocolimit}}

\DeclareMathOperator*{\tensorproduct}{\otimes}
\newcommand{\tensor}{\mathop{\tensorproduct}}
\DeclareMathOperator*{\Tensorproduct}{\bigotimes}
\newcommand{\Tensor}{\mathop{\Tensorproduct}}

\DeclareMathOperator*{\identity}{id}
\newcommand*{\id}{\identity\nolimits}
\DeclareMathOperator*{\source}{dom}
\DeclareMathOperator*{\target}{codom}
\DeclareMathOperator*{\mor}{\mathtt{mor}}
\DeclareMathOperator*{\ob}{\mathtt{ob}}

\DeclareMathOperator*{\Map}{Map}

\DeclareMathOperator*{\LKE}{LKE}

\newcommand{\Cat}{{\mathrm{Cat}}}
\newcommand{\CAT}{{\mathrm{CAT}}}
\newcommand{\bigCAT}{{\mathbf{CAT}}}
\newcommand{\Ord}{{\mathrm{Ord}}}
\newcommand{\FinSet}{{\mathrm{FinSet}}}
\newcommand{\Set}{{\mathrm{Set}}}
\newcommand{\SET}{{\mathrm{SET}}}
\newcommand{\bigSET}{{\mathbf{SET}}}
\newcommand{\Top}{{\mathrm{Top}}}
\newcommand{\TOP}{{\mathrm{TOP}}}
\newcommand{\spectra}{{\mathrm{Sp}}}

\newcommand{\SymmMonCat}{{\mathrm{SMCat}}}
\newcommand{\SymmMonCAT}{{\mathrm{SMCAT}}}

\newcommand{\NN}{{\mathbb N}}
\newcommand{\ZZ}{{\mathbb Z}}

\newcommand{\RR}{{\mathbb R}}

\newcommand{\To}{\longrightarrow}
\renewcommand{\Mapsto}{\longmapsto}
\newcommand{\into}{\hookrightarrow}
\newcommand{\longhookrightarrow}{\lhook\joinrel\relbar\joinrel\rightarrow}
\newcommand{\Into}{\longhookrightarrow}

\newcommand{\defeq}{\coloneqq}

\newcommand{\xto}{\xrightarrow}
\newcommand{\xTo}{\xlongrightarrow}

\newcommand{\suchthat}{\mathrel{}\vcentcolon\mathrel{}}
\mathchardef\dash="2D

\newcommand{\op}[1]{{{#1}^\mathrm{op}}}
\newcommand{\set}[1]{{\left\{ {#1} \right\} }}

\DeclareMathOperator{\disc}{{\texttt{disc}}}

\DeclareMathOperator{\indisc}{{\texttt{indisc}}}
\DeclareMathOperator{\composition}{comp}
\DeclareMathOperator{\BarConst}{Bar}

\DeclareMathOperator{\im}{im}
\DeclareMathOperator{\Diff}{Diff}
\DeclareMathOperator{\interior}{int}

\DeclareMathOperator{\hopb}{{ho\,pb}}
\DeclareMathOperator{\hofibre}{{ho\,fib}}
\DeclareMathOperator{\Nerve}{Nerve}
\DeclareMathOperator{\poMap}{\overrightarrow{\mathrm{Map}}}

\newcommand{\hoLKE}{ho\hspace{-.1em}\LKE}
\newcommand{\setuniverse}{\mathfrak{S}}
\newcommand{\Groth}{{\mathbf{Groth}}}
\newcommand{\EE}{\mathcal{E}}
\newcommand{\StrictInt}{{\mathrm{StrictInt}}}
\newcommand{\LeftInt}{{\mathrm{LeftInt}}}
\newcommand{\RightInt}{{\mathrm{RightInt}}}
\newcommand{\Int}{{\mathrm{Int}}}
\newcommand{\reverse}{\text{rev}}
\newcommand{\dual}{\text{dual}}
\newcommand{\sSet}{s\Set}

\newcommand{\Group}{{\mathrm{Grp}}}
\newcommand{\overGL}[1]{{\Group_{\smash{\nicefrac{}{n}}}}}
\newcommand{\T}{{\mathbf{T}}}

\renewcommand{\mod}{{\mathrm{mod}}}
\newcommand{\alg}{{\mathrm{alg}}}

\newcommand{\PROP}{{\mathrm{PROP}}}

\newcommand{\internal}{{\mathcal{I}}}
\newcommand{\Yoneda}{{\mathrm{Yon}}}
\newcommand{\onearrow}{{({\scriptstyle 0\to 1})}}
\newcommand{\proj}{{\mathrm{proj}}}
\newcommand{\inclusion}{{\mathrm{incl}}}
\newcommand{\concat}{{cc}}
\newcommand{\discrete}{{\delta}}
\newcommand{\textbig}{{\mathrm{big}}}

\newcommand{\oriented}{{\mathrm{or}}}
\newcommand{\derived}{{\ensuremath{\scriptscriptstyle\mathsf{L}}}}

\newcommand{\twocell}[5][0pt]{{\overset{\raisebox{{#3}}{$\scriptstyle {#2}$}}{\raisebox{#1}{\rotatebox[origin=c]{#5}{\clap{\Large{\ensuremath{#4}}}}}}}}

\newcommand{\bigast}{\vcent[\Large]{\ast}}

\newcommand{\clop}[2]{{\left[#1,#2\right[\mbox{}\hspace{-.8pt}}}
\newcommand{\opcl}[2]{{\hspace{-.8pt}\mbox{}\left]#1,#2\right]}}
\newcommand{\opop}[2]{{\hspace{-.8pt}\mbox{}\left]#1,#2\right[\mbox{}\hspace{-.8pt}}}
\newcommand{\clcl}[2]{{\left[#1,#2\right]}}

\newcommand{\Conf}{{\mathsf{Conf}}}
\newcommand{\D}{{\mathbf{D}}}
\newcommand{\E}{{\mathsf{E}}}
\newcommand{\Total}{{\mathsf{T}}}
\newcommand{\Emb}{{\mathrm{Emb}}}
\newcommand{\REmb}{{\mathsf{I\hspace{-1.3pt}Emb}}}  
\newcommand{\RDer}{{\mathsf{I\hspace{-1.3pt}D}}}  
\newcommand{\Princ}{{\mathsf{P}}}
\newcommand{\Frame}{{\mathsf{Fr}}}
\newcommand{\M}{{\mathbb{M}}}
\newcommand{\B}[1]{{\mathfrak{B}{#1}}}
\newcommand{\free}[2]{{{#1}\langle{#2}\rangle}}
\newcommand{\quot}[2]{{\nicefrac{{#2}}{{#1}}}}
\newcommand{\evaluation}[1]{{\mathbf{ev}_{#1}}}
\newcommand{\restrict}[2]{{{#1}|_{#2}}}

\newcommand{\HomCat}[3][]{{[{#2},{#3}]_{#1}}}
\newcommand{\HomCatlarge}[3][]{{\textsymbol[\large\large\small\scriptsize]{[}{#2},{#3}\textsymbol[\large\large\small\scriptsize]{]}_{#1}}}
\newcommand{\HomCatLarge}[3][]{{\textsymbol[\Large\Large\normalsize\footnotesize]{[}{#2},{#3}\textsymbol[\Large\Large\normalsize\footnotesize]{]}_{#1}}}
\newcommand{\HomCatLARGE}[3][]{{\textsymbol[\LARGE\LARGE\large\normalsize]{[}{#2},{#3}\textsymbol[\LARGE\LARGE\large\normalsize]{]}_{#1}}}

\newcommand{\cartesian}{\text{cart}}
\newcommand{\isocart}{\text{qc}}

\newcommand{\poset}{{poset}}
\newcommand{\potop}{{potop}}
\newcommand{\filtTop}{\mathfrak{F}\Top}
\newcommand{\popathcat}{\overrightarrow{path}}

\newcommand{\rtriangle}{\triangle}  
\newcommand{\squares}[2]{\square(#1\vert#2)}
\newcommand{\squaresbig}[2]{\square\big(#1\big\vert#2\big)}

\newcommand{\rtriangles}[2]{\rtriangle(#1\vert#2)}
\newcommand{\rtrianglesbig}[2]{\rtriangle\big(#1\big\vert#2\big)}
\newcommand{\rtrianglesBig}[2]{\rtriangle\Big(#1\Big\vert#2\Big)}
\newcommand{\popath}{\overrightarrow{H}}
\newcommand{\popathstrong}{\popath_\mathrm{s}}
\newcommand{\holink}{\mathrm{holink}}
\newcommand{\reparam}{rprm}

\newcommand{\pathcat}{{path}}
\newcommand{\stickypathcat}{{st}}
\newcommand{\stcat}{{\stickypathcat\dash\pathcat}}
\newcommand{\disccat}[1]{{{#1}^{\smash{\discrete}}}}
\newcommand{\disccatsub}[2]{{{#1}^{\smash{\discrete}}_{\hphantom{\discrete}\hspace{-0.18em}{#2}}}}
\newcommand{\bimod}[2]{{{#1}\text{-}\mathrm{bimod}\text{-}{#2}}}
\newcommand{\realization}[1]{{\left\vert #1 \right\vert}}

\newcommand{\undoperad}[1]{{\mathbf{\Sigma}\hspace{.5pt}{#1}}}	
\newcommand{\catop}[1]{{{\Xi}\hspace{.5pt}{#1}}}	

\newcommand{\Ass}{{\mathbf{Ass}}}
\newcommand{\Comm}{{\mathbf{Comm}}}

%% file: cover.tex
\thispagestyle{empty}

\newcommand*\mycopyright{{\fontfamily{pag}\fontseries{b}\selectfont\textcopyright}}
\newcommand*\mycopyleft{\reflectbox{\mycopyright}}

\begin{center}
{\Large\begin{doublespace}
\textbf{From manifolds to invariants of $E_n$-algebras}\\
{\large by}\\
Ricardo Andrade\\
\end{doublespace}}

\vspace{4em}

{\large
\raisebox{-.08ex}{\mycopyleft}~Ricardo Andrade, MMXI.\\
This work is licensed under a Creative Commons\\
Attribution-ShareAlike 3.0 License.\\[1.2ex]
{\huge\ccLogo\,\ccAttribution\,\ccShareAlike}\\
To view a summary of this license, visit\\
\url{http://creativecommons.org/licenses/by-sa/3.0/}}
\end{center}

\cleardoublepage


%% file: abstract.tex
\chapter*{Abstract}

This thesis is the first step in an investigation of an interesting class of invariants of $E_n$-algebras which generalize topological Hochschild homology. The main goal of this thesis is to simply give a definition of those invariants.

We define PROPs $\E^G_n$, for $G$ a structure group sitting over $GL(n,\RR)$. Given a manifold with a (tangential) $G$-structure, we define functors
\[ \E^G_n[M]:\op{\big(\E^G_n\big)}\To\Top \]
constructed out of spaces of {\em $G$-augmented embeddings} of disjoint unions of euclidean spaces into $M$. These spaces are modifications to the usual spaces of embeddings of manifolds.

Taking $G=1$, $\E^1_n$ is equivalent to the $n$-little discs PROP, and $\E^1_n[M]$ is defined for any parallelized $n$-dimensional manifold $M$.

The invariant we define for a $\E^G_n$-algebra $A$ is morally defined by a derived coend
\[ \T^G(A;M)\defeq\E^G_n[M]\tensor^\derived_{\E^G_n} A \]
for any $n$-manifold $M$ with a $G$-structure.

The case $\T^1(A;S^1)$ recovers the topological Hochschild homology of an associative ring spectrum $A$.

These invariants also appear in the work of Jacob Lurie and Paolo Salvatore, where they are involved in a sort of non-abelian Poincar\'e duality.

\thispagestyle{empty}

%% file: acknowledgements.tex
\chapter*{Acknowledgements}

This is a personal note, attempting to recall some of the people who I feel have significantly enriched my life as a graduate student at MIT, academic and otherwise. It is marked by my current perspective, with all its prejudices and faulty memory.

I must start with my advisor, Haynes Miller. I want to thank him not just for the vast mathematical knowledge he shared with me. He also showed me unending patience, and I learned about the value and difficulty of finding a subject I care about. In my biased opinion, I cannot think of a better choice for my Ph.D. advisor.

On the academic side, I cannot possibly begin to comprehend the great network of people, events, and discussions which have directed my current ideas and research. This is compounded by my ephemeral memory, and my tendency to keep academic discussions to a minimum. Still, I would like to mention, in no particular order, Mark Behrens, Tyler Lawson, Mike Hill, Teena Gerhardt, Josh Nichols-Barrer, Gustavo Granja, Vigleik Angeltveit, John Francis, Matthew Gelvin, Jacob Lurie, Ang\'elica Osorno, Jennifer French, Sam Isaacson, and Nick Rozenblyum. Many more are certainly missing, and the details of our interactions are lost.

I would also like to thank Mark Behrens and Clark Barwick for taking part in my thesis committee.

On a more personal note, some in the group of graduate students at MIT have been like a family away from home to me. A few were there at the beginning and left shortly. Teena Gerhardt, Mike Hill, Josh Nichols-Barrer, Vigleik Angeltveit, and Max Lipyanskiy made me feel welcome from the start, even if I came off as annoying or derisive at times. Others have been there with me for most of my stay at MIT. My friendships with Matthew Gelvin, Ana Rita Pires, Ang\'elica Osorno, Amanda Redlich, Olga Stroilova, Craig Desjardins, and Jennifer French are priceless.

I could go on thanking more people, or analyzing the discussions which influenced me most\ldots\ Instead, I will simply state my appreciation of my parents, my brother, and the beautiful complexity of the universe.

%% file: introduction.tex



\chapter*{Introduction}\label{chapter:introduction}

This introduction will describe a bit of the short history leading to the research presented in this text. In a nutshell, the material stems from an investigation of a sufficiently natural diagrammatic interpretation of topological Hochschild homology ($THH$). This naturality shows the way towards generalizations of $THH$ for $E_n$-algebras.

The author's gateway into this problem was an apparent lack of naturality of the usual definition of topological Hochschild homology, as the geometric realization of the cyclic bar construction of an associative ring spectrum. A telling sign is that the indexing category $\op{\Delta}$ for the cyclic bar construction does not reflect the full rotational symmetry of $S^1$. Consequently, it is insufficient to recover the action of $S^1$ on $THH$.

In order to repair this state of affairs, the author conceived of the category $\EE$, here called {\em Elmendorf's category}. This category is essentially a combinatorial description of the spaces of configurations in $S^1$. Topological Hochschild homology can be recovered as a homotopy colimit along $\EE$. This draws an analogy with the well-known result that $THH$ of a commutative ring spectrum is given by tensoring with $S^1$.

A more natural, yet equivalent, amalgamation of the spaces of configurations of $S^1$ is given by the topologically enriched category $\M(S^1)$ of {\em sticky configurations in $S^1$}. The advantage of $\M(S^1)$ is that it generalizes promptly to a category $\M(X)$ for any space $X$. The significance of $\M(X)$ is most apparent for the case of a manifold, where it carries homotopical information about the tangent space of the manifold, and about embeddings into other manifolds.

Restricting then to the case of a $n$-manifold $M$, the analogies between $\M(M)$ and $E_n$-operads are very strong. This raises the question of whether one can obtain invariants of $E_n$-algebras as homotopy colimits along $\M(M)$, just as $THH$ is a homotopy colimit along $\M(S^1)\simeq\EE$.

To answer that question, we reformulate the little discs operads in terms of modifications to the spaces of embeddings of manifolds. We designate these modifications by {\em \(G\)-augmented embedding spaces}, where $G$ is a structure group. The augmented embedding spaces give rise to PROPs $\E^G_n$ and, for each appropriate $n$-manifold $M$, a right module
\[ \E^G_n[M]:\op{\big(\E^G_n\big)}\To\Top \]
over $\E^G_n$. In the case $G=1$, $E^1_n$ is equivalent to the little $n$-discs PROP, and the right modules $\E^1_n[M]$ are defined for any parallelized $n$-manifold $M$.

The category $\M(M)$ reappears as the Grothendieck construction of the functor $\E^G_n[M]$. The hypothesized invariant of $\E_n$-algebras can now be phrased simply as a derived enriched colimit
\[ \E^1_n[M]\tensor_{\E^1_n}^\derived\underline{A} \]
for any $\E^1_n$-algebra $\underline{A}$, and any parallelized $n$-manifold $M$. In case $M=S^1$, we recover topological Hochschild homology.

\section*{Bibliographic references and influences}

The category $\EE$ has appeared repeatedly in the literature in different guises, e.g.\ \cite{Elmendorf} and \cite{Bokstedt-Hsiang-Madsen}; see also \cite{Dwyer-Kan} for a very similar category studied even earlier.

The operations on $E_n$-algebras which we describe have already appeared in the literature. The first construction known to the author was given by Paolo Salvatore in \cite{Salvatore}, using the Fulton-MacPherson operads. More recently, Jacob Lurie has defined topological chiral homology, as explained in \cite{Lurie1} and \cite{Lurie2}.

The work of Lurie was influential in the present research. This happened at a time when the author had defined the categories $\M(M)$, and was analyzing their relationship with $E_n$-operads, with the goal of constructing generalizations of $THH$ for $E_n$-algebras. Immediately, the author happened upon \cite{Lurie2}, where topological chiral homology is briefly described. From that exposition, the precise form of the connection between $\M(M)$ and $E_n$-operads --- via augmented embedding spaces --- became obvious, and the remaining pieces of the research for this thesis finally fell into place.

On a more historical context, these ideas closely follow earlier work of Graeme Segal and Dusa McDuff (among others too numerous to name) on spaces of labeled configurations (see \cite{Segal} and \cite{McDuff}). Additionally, the diagrammatic approach taken here is most reminiscing of the characterization of the infinite symmetric product of a space detailed in \cite{Kuhn} (called there the {\em McCord model}): this model can actually be seen as a sort of limiting case $n\to\infty$ of our framework.


%% file: summary.tex



\chapter*{Summary}\label{chapter:summary}

\shortref[Chapter ]{chapter:terminology_basic} establishes some basic terminology and concepts which will be used throughout the text. It discusses quite disparate subjects and is meant only for reference.

Chapter \ref{chapter:internal_categories} gives some basic theory of internal categories, with the dual aim of relating them to enriched categories, and of defining the Grothendieck construction in a sufficiently general context.

Chapter \ref{chapter:sticky_conf} associates to each space $X$ the topologically enriched category of sticky configurations on $X$, $\M(X)$, together with an equivariant analogue. Chapter \ref{chapter:sticky_conf_S^1} analyzes the example $\M(S^1)$, which is weakly equivalent to Elmendorf's category $\EE$. It finishes by showing that topological Hochschild homology is a homotopy colimit along $\EE$.

Chapter \ref{chapter:embeddings_manifolds} defines the concept of $G$-structure on a manifold. Then the $G$-augmented embedding spaces of manifolds are constructed as modifications of the usual embedding spaces of manifolds. These are used to define a PROP $\E^G_n$, together with a right module over it for each $n$-manifold with a $G$-structure.

Chapter \ref{chapter:stratified_spaces} describes convenient concepts of stratified spaces, together with some basic results. Moreover, it recovers $\M(X)$ from spaces of filtered paths on stratified spaces. This analysis comes in handy in the next chapter \ref{chapter:sticky<->embeddings} where it is shown that the category $\M(M)$ --- for $M$ a $n$-manifold with a $G$-structure --- is essentially the Grothendieck construction of the corresponding right module over the PROP $\E^G_n$. More precisely, a zig-zag of weak equivalences is given between the category $\M(M)$ and a Grothendieck construction of the functor $\E^G_n[M]$, which we denote by $\disccat{\Total^G_n[M]}$.

Chapter \ref{chapter:homotopical_properties_enriched_categories} is another technical chapter describing the concepts of homotopy colimits necessary in the final chapter. In particular, it is stated, without proof, how the homotopy colimits along Grothendieck constructions can be computed as derived enriched colimits over the base category.

The last chapter \ref{chapter:invariants_En-algebras} defines the invariant $\T^G(A;M)$ of a $\E^G_n$-algebra $A$, for each $n$-manifold $M$ with a $G$-structure. A proof is given that $\T^1(-;S^1)$ for the category of spectra is equivalent to topological Hochschild homology.


%% file: terminology.tex



\chapter{Basic terminology}\label{chapter:terminology_basic}

\section*{Introduction}

In this very disconnected chapter, we introduce some basic notation, terminology, and definitions which we will use in the remainder of the text.

\section*{Summary}

Section \sref{section:terminology_categories_sets_categories} describes some categories of sets, and categories of categories designed to deal with issues of size. It also mentions the category of finite sets and categories of finite ordinals. Section \sref{section:terminology_categories} makes some important abstract remarks on categories, $2$-categories, functors, and natural transformations.

Section \sref{section:terminology_topological_spaces} introduces the category of topological spaces and the category of weak Hausdorff compactly generated topological spaces.

Section \sref{section:terminology_principal_bundles} settles some language concerning principal bundles and principal spaces.

Section \sref{section:terminology_homotopy_theory} makes a few comments regarding the homotopy theory of spaces, with a focus on the notion of homotopy equivalence. Section \sref{section:terminology_Moore_path_space} defines the space of Moore paths on a topological space and gives some important maps based on this space.

Section \sref{section:terminology_enriched_categories} discusses basic concepts of enriched category theory. Section \sref{section:terminology_properties_enriched_functors} establishes some terminology regarding properties of enriched functors.

Section \sref{section:model_categories} explains basic notions regarding enriched model categories and monoidal model categories.

The last three section deal with PROPs and operads. Section \sref{section:terminology_PROPs} defines the notions of PROP, and algebra for a PROP. Section \sref{section:terminology_PROPs_operads} relates PROPs to operads: to each PROP it associates an operad, and for each operad it constructs a PROP called the associated category of operators. Furthermore, it relates the notions of algebras for PROPs and operads. Finally, section \sref{section:terminology_examples_PROPs} gives examples of PROPs in $\Set$ and $\Top$: the associative and commutative PROPs, and the little discs PROPs.

\section{Categories of sets and categories}\label{section:terminology_categories_sets_categories}

\begin{convention}[sets]
To avoid problems relating to size of sets, we will assume the existence of several universes of sets.\\
More precisely, we will assume the existence of three categories of sets
\[ \Set\Into\SET\Into\bigSET \]
such that
\begin{enum}
\item all three are closed under taking subsets and elements;
\item $\Set$ is $\Set$-bicomplete, i.e.\ it has all products and coproducts indexed by sets in $\Set$;
\item $\SET$ is $\SET$-bicomplete;
\item $\bigSET$ is $\bigSET$-bicomplete;
\item $\ob(\Set)$ is in $\SET$;
\item $\ob(\SET)$ is in $\bigSET$;
\end{enum}
These categories can be constructed assuming the existence of three inaccessible cardinals.\\
We will refer to the elements of $\Set$ as {\em small sets}, and a set will be by default a small set. The elements of $\SET$ will be called {\em large sets}.
\end{convention}

\begin{notation}[categories]\label{notation:categories_of_categories}
We will need many categories of categories.\\
Given two categories of sets, $\mathfrak{S}$ and $\mathfrak{T}$, we will have a corresponding $2$-category of categories, $\mathfrak{T}\dash\Cat_{\mathfrak{S}}$, whose objects are categories $C$ such that $\ob C\in\mathfrak{S}$, and $C(x,y)\in\mathfrak{T}$ for any $x,y\in\ob C$.\\
We will use a few useful abbreviations:
\begin{align*}
\Cat&\defeq\Set\dash\Cat_\Set\\
\CAT&\defeq\Set\dash\Cat_\SET\\
\bigCAT&\defeq\SET\dash\Cat_\bigSET
\end{align*}
A category in $\Cat$ is a {\em small category}. A category in $\CAT$ is a {\em locally small large category}. Without further mention, a category will be, by default, in $\bigCAT$, except if it is constructed not to be. Often, it will actually not matter where exactly the category is.
\end{notation}

We will also need a few explicit smaller categories, such as the categories of ordinals and of finite sets.

\begin{notation}[finite sets]
The category of (small) {\em finite sets} will be denoted $\FinSet$. We consider the inclusion
\[ \sfunc[\rInto]{\NN}{\ob\FinSet}{n}{\set{1,\ldots,n}} \]
and we will, for brevity, denote the set $\set{1,\ldots,n}$ (for $n\in\NN$) simply by $n$.
\end{notation}

\begin{notation}[categories of ordinals]\label{notation:Ord_Delta}
The category of (small) {\em finite ordinals} and order preserving functions will be denoted $\Ord$.\\
For each $n\in\NN$, the ordinal corresponding to the set $\set{1,\ldots,n}$ with the order induced from $\NN$ is denoted simply by $n$.\\
We will denote by $\Delta$ the full subcategory of $\Ord$ generated by the ordinals $n$ for $n\in\NN\setminus\set{0}$.
\end{notation}

\begin{remark}
The category $\Ord$ is monoidal, with the monoidal product given by
\[ +:\Ord\times\Ord\To\Ord \]
where for any finite ordinals $x$ and $y$, the ordinal $x+y$ is the disjoint union of $x$ and $y$ with the unique total order which recovers the total orders on $x$ and $y$, and such that any element of $x$ is less than any element of $y$.
\end{remark}

\section{Categories}\label{section:terminology_categories}

As we mentioned, $\mathfrak{S}\dash\Cat_{\mathfrak{T}}$, $\Cat$, $\CAT$, \ldots, are all $2$-categories. We leave some remarks about the $2$-categories which will appear.

\begin{convention}[$2$-categories, $2$-functors]
For us, a $2$-category is always a strict $2$-category. Furthermore, all functors and natural transformations between $2$-categories will be {\em strict}, unless specifically stated otherwise.\\
Any ordinary category will be viewed as a $2$-category with only identity $2$-morphisms.
\end{convention}

\begin{convention}[lax natural transformation]
One of the rare instances that we will need of a non-strict natural transformation will be that of a {\em lax natural transformation} $\alpha$ between strict $2$-functors, as in definition 7.5.2 of \cite{Borceux}. This will appear in chapter \ref{chapter:sticky_conf}.\\
If all the $2$-morphisms in the lax naturality squares for $\alpha$ are isomorphisms then we call $\alpha$ a {\em pseudo-natural transformation} (again following definition 7.5.2 of \cite{Borceux}). This will only appear in proposition \sref{proposition:adjunction_operads_PROPs} below.
\end{convention}

\begin{notation}[categories of functors]
Given any $2$-categories $A$ and $B$, we will denote the $2$-category of strict functors, strict natural transformations, and strict modifications (see chapter 7 of \cite{Borceux}) from $A$ to $B$ by $\HomCat{A}{B}$.\\
If $A$, $B$ are ordinary categories, then $\HomCat{A}{B}$ is just the usual category of functors and natural transformations from $A$ to $B$, $\bigCAT(A,B)$. We will, nevertheless, prioritize the notation $\HomCat{A}{B}$.
\end{notation}

\begin{notation}[composition of natural transformations]
For any $2$-categories $A$, $B$, and $C$, the composition functor will be denoted
\[ -\circ-:\HomCat{B}{C}\times\HomCat{A}{B}\To\HomCat{A}{C} \]
In particular, given a natural transformation $\alpha$ from $A$ to $B$, and a natural transformation $\beta$ from $B$ to $C$, we will denote their {\em horizontal composition} by $\beta\circ\alpha$.\\
The composition of $1$-morphisms in the category $\HomCat{A}{B}$ will be denoted differently: given functors $F,G,H:A\to B$, and natural transformations \begin{align*}
\alpha&:F\to G\\
\beta&:G\to H
\end{align*}
their {\em vertical composition} will be abbreviated $\beta\cdot\alpha$.\\
This notation is the one used in the book \cite{MacLane}.
\end{notation}

\begin{notation}[opposite of category]
Given a category $A$, $\op{A}$ will denote the {\em opposite category} of $A$.\\
We will also consider the opposite $\op{A}$ of a $2$-category $A$. In this case $\op{A}$ only reverses the $1$-morphisms, and not the $2$-morphisms.
\end{notation}

\begin{convention}[limits and colimits in categories]
We will follow the common convention of writing limits (respectively, colimits) in a category $C$ as if there were a preassigned limit (respectively, colimit) for each diagram in $C$ which does have a limit (respectively, colimit).\\
For that purpose, the reader may assume that each category $C$ comes equipped with an assignment of a limit (respectively, colimit) to each diagram in $C$ which has a limit (respectively, colimit).
\end{convention}

\begin{notation}[terminal object]
Given a category $C$ with a terminal object, we will often denote the terminal object by $1$. For example, $1\in\Cat$ denotes a category with one object and one morphism, and $1\in\Top$ denotes a topological space with a single element.\\
This convention will not be followed when $1$ already has an assigned meaning, such as in the categories $\op{\FinSet}$ and $\op{\Ord}$.
\end{notation}

\begin{notation}[category associated to a monoid]
Given an associative monoid $A$ in the cartesian category $\Set$, we will denote by $\B{A}$ the category associated with $A$: $\ob(\B{A})=1$, $\B{A}(1,1)=A$, and the composition in $\B{A}$ comes from the binary operation on $A$.
\end{notation}

\section{Categories of topological spaces}\label{section:terminology_topological_spaces}

\begin{notation}[topological spaces]
We will denote by $\Top$ the category of topological spaces and continuous maps. This will be our default category of spaces.\\
We will also occasionally need the category $\TOP$ of large topological spaces and continuous maps.
\end{notation}

\begin{convention}[sets as topological spaces]
The canonical inclusion functor
\[ \Set\Into\Top \]
will be used to consider all small sets canonically as discrete topological spaces. In particular, any ordinary locally small category will be considered as a $\Top$-category whenever necessary.
\end{convention}

\begin{notation}
$\Top$ is not cartesian closed, but for any topological spaces $X$ and $Y$, we will consider the space $\Map(X,Y)$ of continuous maps $X\to Y$ with the {\em compact-open topology}.
\end{notation}

Since $\Top$ is not cartesian closed, it is useful to introduce a category of spaces which is.

\begin{definition}[compactly generated space]
A topological space $X$ is said to be {\em compactly generated} if it has the initial topology induced by all maps $C\to X$ with $C$ compact Hausdorff.
\end{definition}

\begin{notation}[category of compactly generated spaces]\label{notation:compactly_generated_spaces}
We will denote by $k\Top$ the category of compactly generated spaces, and continuous maps. This category is cartesian closed.\\
The inclusion $k\Top\hookrightarrow\Top$ has a right adjoint
\[ \kappa:\Top\To k\Top \]
called the {\em \(k\)-ification functor}. We remark that $\kappa$ preserves small coproducts.
\end{notation}

\section{Principal bundles}\label{section:terminology_principal_bundles}

We will follow the nomenclature of \cite{Husemoller}, which we briefly describe below. We fix a topological group $G$.

\begin{definition}[principal bundle]
By a {\em principal \(G\)-bundle} we will mean a right $G$-space $X$ and a map \[ p:X\To B \] of topological spaces such that
\begin{enum}
\item $p$ induces a homeomorphism $\quot{G}{X}\cong B$;
\item there exists a map
\[ transl:X\underset{Y}{\times}X\To G \]
such that
\[ x\cdot transl(x,y)=y \]
for any $x,y\in X$ such that $p(x)=p(y)$.
\end{enum}
\end{definition}

\begin{definition}[map of principal bundles]
Let $p:X\to B$, and $p':X'\to B'$ be principal $G$-bundles.\\
A {\em map of principal \(G\)-bundles} $(f,g):p\to p'$ is a $G$-equivariant map $f:X\to X'$ of right $G$-spaces, together with a map $g:B\to B'$ such that
\begin{diagram}[midshaft,h=2em]
X&\rTo{f}&X'\\
\dTo{p}&&\dTo{p'}\\
B&\rTo{g}&B'
\end{diagram}
commutes.
\end{definition}

\begin{proposition}
Let $p:X\to B$, and $p':X'\to B'$ be principal $G$-bundles.\\
Given a map of principal $G$-bundles, $(f,g):p\to p'$, the diagram
\begin{diagram}[midshaft,h=2em]
X&\rTo{f}&X'\\
\dTo{p}&&\dTo{p'}\\
B&\rTo{g}&B'
\end{diagram}
is a cartesian square.
\end{proposition}

\begin{definition}[locally trivial principal bundle]
We say that a principal $G$-bundle $p:X\to B$ is {\em locally trivial} if each point of $B$ has a neighborhood $U$ in $B$ such that there exists a $G$-equivariant map $p^{-1}(U)\to G$.
\end{definition}

\begin{notation}[principal $G$-space]
We call a right $G$-space $X$ a {\em principal \(G\)-space} if the map $X\to\quot{G}{X}$ gives a principal $G$-bundle. We say $X$ is a {\em locally trivial} principal $G$-space if $X\to\quot{G}{X}$ gives a locally trivial principal $G$-bundle.
\end{notation}

\begin{remark}[principal left $G$-space]
We will occasionally refer to (locally trivial) principal left $G$-spaces. This will mean a left $G$-space (i.e.\ a right $\op{G}$-space) which is a (locally trivial) principal $\op{G}$-space.
\end{remark}

\section{Homotopy theory of topological spaces}\label{section:terminology_homotopy_theory}

We will be mostly interested in homotopy equivalences of topological spaces. Therefore, we will mostly deal with the category $\Top$ of topological spaces equipped with the Str{\o}m model structure from \cite{Strom}:
\begin{enum}
\item the weak equivalences are homotopy equivalences of spaces;
\item the cofibrations are closed maps with the homotopy extension property;
\item the fibrations are Hurewicz fibrations.
\end{enum}

The Str{\o}m model structure on $\Top$ is right proper and naturally {\em framed} (see section 16.6 of \cite{Hirschhorn}), in a way that recovers the usual homotopy pullbacks and pushouts, and the usual homotopy limits and colimits in $\Top$ (see chapter 19 of \cite{Hirschhorn}). While the framing is not strictly necessary, and all necessary results can be derived directly, it is nevertheless useful.

We will also make use of the analogous Str{\o}m model structure on $k\Top$, which is the only model structure on $k\Top$ which we consider. $k\Top$ with the Str{\o}m model structure is actually a simplicial model category.

In keeping with our focus on homotopy equivalences, we will say that a commutative square in $\Top$
\begin{diagram}[h=2em]
A&\rTo&B\\
\dTo&&\dTo\\
C&\rTo&D
\end{diagram}
is {\em homotopy cartesian} (or a {\em homotopy pullback square}) if the natural map from $A$ to the homotopy pullback of $C\to D\leftarrow B$ is a homotopy equivalence.

Despite our focus on homotopy equivalences, when talking about topological spaces, ``weak equivalence'' always has the usual meaning (isomorphisms on homotopy groups). We will always refer explicitly to homotopy equivalences of topological spaces as such.

\section{Moore path space}\label{section:terminology_Moore_path_space}

\begin{definition}[Moore path space]\label{definition:Moore_path_space}
Let $X$ be a topological space. Recall that $\Map(\clop{0}{+\infty},X)$ is endowed with the compact-open topology.\\
The subspace of $\Map(\clop{0}{+\infty},X)\times\clop{0}{+\infty}$ corresponding to its subset
\[ \set{(\alpha,\tau)\in\Map(\clop{0}{+\infty},X)\times\clop{0}{+\infty}\suchthat\restrict{\alpha}{\clop{\tau}{+\infty}} \text{ is constant}} \]
will be called the {\em Moore path space} of $X$, $H(X)$.
\end{definition}

\begin{remark}[functoriality of Moore path space]
The Moore path space above extends to a functor \[ H:\Top\To\Top \]
\end{remark}

\begin{definition}[maps on Moore path space]\label{definition:maps_Moore_path_space}
Let $X$ be a topological space.\\
We define the following continuous maps on the Moore path space of $X$:
\begin{enum}
\item the {\em length} map, $l$:
\[ l:H(X)\Into\Map(\clop{0}{+\infty},X)\times\clop{0}{+\infty}\,\xTo{\proj}\clop{0}{+\infty} \]
\item the {\em source} map, $s$:
{\small
\[ s:H(X)\Into\Map(\clop{0}{+\infty},X)\times\clop{0}{+\infty}\,\xTo{\proj}\Map(\clop{0}{+\infty},X)\xTo{\evaluation{0}}X \]}
\item the {\em target} map, $t$:
\[ t:H(X)\Into\Map(\clop{0}{+\infty},X)\times\clop{0}{+\infty}\,\xTo{\evaluation{}}X \]
\item the {\em inclusion} of $X$ is the unique map $i:X\to H(X)$ such that
\begin{align*}
s\circ i&=\id_X\\
l\circ i&=0
\end{align*}
\end{enum}
\end{definition}

\begin{proposition}
Let $X$ be a topological space.\\
The map \[ (s,t):H(X)\To X\times X \] is a Hurewicz fibration.
\end{proposition}

\begin{definition}[concatenation of Moore paths]\label{definition:concat_Moore_paths}
Let $X$ be a topological space.\\
Let $P$ be determined by the pullback square
\begin{diagram}[midshaft,height=2.1em]
P&\rTo{p_1}&H(X)\\
\dTo{p_2}&&\dTo{t}\\
H(X)&\rTo{s}&X
\end{diagram}
The {\em concatenation} map $cc:P\to H(X)$ is now characterized by:
\[ l\circ cc=l\circ p_1+l\circ p_2 \]
and, for $x\in P$
\begin{align*}
\restrict{cc(x)}{\clcl{0}{l\circ p_1(x)}}&=p_1(x)\\
\restrict{cc(x)}{\clop{l\circ p_1(x)}{+\infty}}&=p_2(x)(?-l\circ p_1(x))
\end{align*}
\end{definition}

\begin{definition}[reparametrization of Moore paths]\label{definition:reparametrization_Moore_paths}
Let $X$ be a topological space.\\
Define the {\em reparametrization map}
\[ \reparam:H(X)\To\Map(I,X) \]
by
\[ \reparam(\gamma,\tau)\defeq \gamma(\tau\cdot -) \]
for $(\gamma,\tau)\in H(X)$
\end{definition}

\begin{proposition}\label{proposition:reparametrization_Moore_paths_equiv}
Let $X$ be a topological space.\\
The map \[ \reparam:H(X)\To\Map(I,X) \] is a homotopy equivalence over $X\times X$. Here, $H(X)$ maps by $(s,t)$ to $X\times X$; $\Map(I,X)$ maps by $(\evaluation{0},\evaluation{1})$ to $X\times X$.
\end{proposition}

\section{Enriched categories}\label{section:terminology_enriched_categories}

For material on enriched category theory, we refer the reader to the book \cite{Kelly}. We leave here some notation regarding enriched categories. $V$ will denote a symmetric monoidal category with unit $I$.

\begin{notation}[$2$-category of $V$-categories]\label{notation:categories_V-categories}
If $\mathfrak{S}$ is a category of sets, $V\dash\Cat_\mathfrak{S}$ is the $2$-category whose:
\begin{enum}
\item objects are $V$-categories whose set of objects lies in $\mathfrak{S}$;
\item $1$-morphisms are the $V$-functors;
\item $2$-morphisms are the $V$-natural transformations.
\end{enum}
For convenience, we make the following abbreviations
\begin{align*}
V\dash\Cat&\defeq V\dash\Cat_\Set\\
V\dash\CAT&\defeq V\dash\Cat_\SET
\end{align*}
for the $2$-categories of {\em small \(V\)-categories} and {\em large \(V\)-categories}, respectively.\\
A $V$-category will typically be in $V\dash\CAT$, by default.
\end{notation}

\begin{notation}[$V$-category of functors]
Assume $V$ is symmetric monoidal closed.\\
If $A$ and $B$ are $V$-categories then $\HomCat[V]{A}{B}$ denotes the enriched $V$-category of functors from $A$ to $B$, whose objects are the $V$-functors from $A$ to $B$.
\end{notation}

\begin{notation}[change of enriching category]
Assume $F:V\to W$ is a lax symmetric monoidal functor.\\
For any $V$-category $C$, there is an induced $W$-category $F(C)$ such that \[ \ob\big(F(C)\big)=\ob C \] and \[ F(C)(x,y)=F\big(C(x,y)\big) \] for any $x,y\in\ob C$.\\
This extends to functors
\begin{align*}
F&:V\dash\Cat\To W\dash\Cat\\
F&:V\dash\CAT\To W\dash\CAT\\
F&:V\dash\Cat_{\mathfrak{S}}\To W\dash\Cat_{\mathfrak{S}}
\end{align*}
\end{notation}

\begin{notation}[underlying category]
For the special case of the lax symmetric monoidal functor
\[ V(I,-):V\To\SET \]
we denote the corresponding functor from $V$-categories to $\SET$-categories (i.e.\ ordinary categories) by
\[ (-)_0:V\dash\CAT\To\bigCAT \]
In particular, given a $V$-category $C$, we denote its underlying category by $C_0$. Moreover, the underlying functor of a $V$-functor $F:C\to D$ is denoted
\[ F_0:C_0\To D_0 \] This is in conformance with the notation in \cite{Kelly}.
\end{notation}

We finish this section with a few remarks on symmetric monoidal $V$-categories.

\begin{notation}[symmetric monoidal $V$-category]
We will occasionally need the concept of a symmetric monoidal $V$-category. This is exactly parallel to the notion of symmetric monoidal category, only all structure functors and natural transformations are now required to be $V$-functors and $V$-natural transformations.\\
We will denote by $V\dash\SymmMonCat$ the $2$-category of
\begin{enum}
\item small symmetric monoidal $V$-categories,
\item strong symmetric monoidal $V$-functors, and
\item symmetric monoidal $V$-natural transformations.
\end{enum}
The analogous $2$-category of large symmetric monoidal $V$-categories will be abbreviated $V\dash\SymmMonCAT$.
\end{notation}

\begin{proposition}[change of enriching category]\label{proposition:change_enriching_cat_smcat}
Assume $F:V\to W$ is a lax symmetric monoidal functor.\\
There exists a natural functor
\[ F:V\dash\SymmMonCAT\To W\dash\SymmMonCAT \]
such that
\begin{diagram}[midshaft,h=2.1em]
V\dash\SymmMonCAT&\rTo{\ F\ }&W\dash\SymmMonCAT\\
\dInto&&\dInto\\
V\dash\CAT&\rTo{F}&W\dash\CAT
\end{diagram}
commutes.
\end{proposition}

\begin{remark}
The above functor restricts to a functor
\[ F:V\dash\SymmMonCat\To W\dash\SymmMonCat \]
\end{remark}

\section{Properties of enriched functors}\label{section:terminology_properties_enriched_functors}

We will now turn to properties of $V$-functors, where $V$ is again a symmetric monoidal category.

\begin{definition}[essentially surjective $V$-functor]
Let $F:A\to B$ be a $V$-functor between $V$-categories.\\
$F$ is said to be {\em essentially surjective} if the underlying functor
\[ F_0:A_0\To B_0 \]
is essentially surjective.
\end{definition}

\begin{definition}[local isomorphism]\label{definition:local_isomorphism}
Let $F:A\to B$ be a $V$-functor.\\
We say $F$ is a {\em local isomorphism} if for all $x,y\in\ob A$
\[ F:A(x,y)\To B(Fx,Fy) \]
is an isomorphism in $V$.
\end{definition}

This last definition gives way to the nomenclature ``local \{name of property\}'' which we now introduce.

\begin{notation}\label{notation:local_nameproperty}
Let $nameP$ be the name of a property of morphisms in $V$.\\
We say a $V$-functor $F:A\to B$ is {\em locally \(nameP\)} (or {\em a local \(nameP\)}) if for every $x,y\in\ob A$, the morphism
\[ F:A(x,y)\To B(Fx,Fy) \]
verifies the property $nameP$.
\end{notation}

We give a few examples of this notation involving the cartesian category $\Top$.

\begin{examples}
If we take $nameP$ to be ``isomorphism'', then we recover the notion of local isomorphism from definition \sref{definition:local_isomorphism}.\\
Other relevant examples are given by ``weak equivalence'' and ``homotopy equivalence'' in $\Top$. From these we get the notion of {\em local weak equivalences} and {\em local homotopy equivalences} of $\Top$-categories.
\end{examples}

This example can be used to define the notion of weak equivalence of $\Top$-categories.

\begin{definition}[weak equivalence of $\Top$-categories]\label{definition:weak_equivalence_Top-cat}
Let $F:A\to B$ be a $\Top$-functor between $\Top$-categories.\\
We say $F$ is a {\em weak equivalence} if $F$ is a local homotopy equivalence and
\[ \pi_0 F:\pi_0 A\To\pi_0 B \]
is an essentially surjective functor (between ordinary categories).
\end{definition}

\begin{remark}
Note that we mean a local {\em homotopy equivalence} when we refer to a weak equivalence of $\Top$-categories.\\
This is in accordance with our focus on homotopy equivalences, and the Str{\o}m model structure in $\Top$, even if conflicting with our convention to refer explicitly to homotopy equivalences of spaces as such.\\
We use this terminology for simplicity, since it is the only case which will appear.
\end{remark}

\section{Model categories}\label{section:model_categories}

We will mostly follow the book \cite{Hirschhorn} on matters relating to model categories. In this section we fix some terminology regarding model categories, and discuss some concepts not appearing in \cite{Hirschhorn}, namely enriched model categories and monoidal model categories.

\begin{notation}
By a model category, we will mean exactly the notion explained in \cite{Hirschhorn}. In particular, a model category is a bicomplete category verifying the classical Quillen axioms for a closed model category, in which the factorizations can be chosen functorially.\\
Nevertheless, we will occasionally utilize redundant expressions concerning (co)completeness, such as ``cocomplete model category'', or ``bicomplete model category''. This is merely for emphasis of a necessary property.
\end{notation}

\begin{examples}
The main examples of model structures for us are the Kan model structure on $\sSet$, and the Str{\o}m model structures on $\Top$ and $k\Top$. All of these will always be implicit when dealing with these categories.
\end{examples}

\begin{definition}[left Quillen bifunctor]\label{definition:left_Quillen_bifunctor}
Let $V$, $W$, $X$ be model categories, and
\[ F:V\times W\To X \]
a functor. We say $F$ is a {\em left Quillen bifunctor} (or that it verifies the {\em pushout-product axiom}) if for any cofibrations $f:x\to y$ in $V$ and $g:x'\to y'$ in $W$, the canonical map
\[ f\,\underset{F}{\square}\,g:F(y,x')\ \underset{\mathclap{F(x,x')}}{\amalg}F(x,y')\ \To F(y,y') \]
is a cofibration in $X$, which is a weak equivalence if either $f$ or $g$ is a weak equivalence.
\end{definition}

\begin{definition}[symmetric monoidal model category]
Let $V$ be a model category, and a symmetric monoidal category (with monoidal product $\tensor$).\\
We will say $V$ is a {\em symmetric monoidal model category} if
\[ \tensor:V\times V\To V \]
verifies the pushout-product axiom (i.e.\ is a left Quillen bifunctor).
\end{definition}

\begin{remark}
It is usual to assume some condition on the unit of the monoidal structure on $V$, for example that it be cofibrant. We will always explicitly state as much by saying, for example, that $V$ is a symmetric monoidal model category with cofibrant unit.
\end{remark}

\begin{notation}
If the monoidal structure on $V$ is cartesian, we will say that $V$ is a cartesian model category.\\
If the monoidal structure on $V$ is closed, we will say that $V$ is a symmetric monoidal closed model category.
\end{notation}

\begin{examples}
All of $\sSet$ (with the Kan model structure), $\Top$, and $k\Top$ (the last two with the Str{\o}m model structures) are cartesian model categories with cofibrant unit. $\sSet$ and $k\Top$ are cartesian closed model categories.
\end{examples}

\begin{definition}[$V$-model category]\label{definition:V-model_category}
Let $V$ be a bicomplete symmetric monoidal closed model category.\\
A bicomplete $V$-category $C$ with a model structure on $C_0$ is called a {\em $V$-model category} if the functor
\[ \tensor:V\times C_0\To C_0 \]
corresponding to tensoring an object of $V$ and an object of $C$ (which is defined since $C$ is cocomplete as a $V$-category), is a left Quillen bifunctor.
\end{definition}

\begin{notation}[simplicial model category]
In the special case that $V=\sSet$ (with the Kan model structure), we call $C$ a {\em simplicial model category}.
\end{notation}

\begin{examples}
$V$ is naturally a $V$-model category.\\
$k\Top$ (with the Str{\o}m model structure) is a simplicial model category.
\end{examples}

\begin{definition}[symmetric monoidal $V$-model category]\label{definition:symmetric_monoidal_V-model_category}
Let $V$ be a bicomplete symmetric monoidal closed model category.\\
A symmetric monoidal $V$-category $C$ with a model structure on $C_0$ is called a {\em symmetric monoidal $V$-model category} if $C$ is a $V$-model category and the monoidal product on $C$ gives a left Quillen bifunctor
\[ \tensor:C_0\times C_0\To C_0 \]
\end{definition}

\begin{notation}[symmetric monoidal simplicial model category]
In case $V=\sSet$, we say that $C$ is a {\em symmetric monoidal simplicial model category}.
\end{notation}

\begin{notation}
As above, we will say that $C$ is a cartesian $V$-model category if the monoidal structure on $C$ is cartesian.\\
We will also say that $C$ is a symmetric monoidal closed $V$-model category if the monoidal $V$-category $C$ is closed (as a $V$-category).
\end{notation}

\begin{examples}
$V$ is a symmetric monoidal closed $V$-model category.\\
$k\Top$ is a symmetric monoidal closed simplicial model category.
\end{examples}

\section{PROPs}\label{section:terminology_PROPs}

Fix a symmetric monoidal category $V$.

\begin{definition}[PROP]
A {\em \(V\)-PROP} is a pair $(\mathsf{P},a)$ where $\mathsf{P}$ is a symmetric monoidal $V$-category (with monoidal product given by $\tensor$), and $a\in\ob\mathsf{P}$ is such that any object of $\mathsf{P}$ is isomorphic to $a^{\tensor n}$ for some $n\in\NN$.
\end{definition}

\begin{notation}[generator of a PROP]
The distinguished object $a\in\ob\mathsf{P}$ is called the {\em generator} of the $V$-PROP $(\mathsf{P},a)$.\\
For convenience, we will often confuse the $V$-PROP with its underlying symmetric monoidal $V$-category, leaving the generator implicit.
\end{notation}

\begin{notation}
If the category $V$ is clear from context, we will often omit it, and simply call the above a PROP.
\end{notation}

\begin{definition}[category of PROPs]
The {\em category of $V$-PROPs}, $V\dash\PROP$, is the $2$-category determined by
\begin{enum}
\item the objects of $V\dash\PROP$ are the $V$-PROPs;
\item the $1$-morphisms from a $V$-PROP $(\mathsf{P},a)$ to a $V$-PROP $(\mathsf{Q},b)$ are the symmetric monoidal $V$-functors \[ F:\mathsf{P}\To\mathsf{Q} \] such that $F(a)=b$;
\item given two $1$-morphisms \[ F,G:(\mathsf{P},a)\To(\mathsf{Q},b) \] $V\dash\PROP(F,G)$ is the set of symmetric monoidal $V$-natural transformations $\alpha:F\to G$ such that $\alpha_a=\id_b$.
\end{enum}
\end{definition}

\begin{proposition}
The $2$-category $V\dash\PROP$ is equivalent to a $1$-category. Equivalently, given two $V$-PROPs $(\mathsf{P},a)$ and $(\mathsf{Q},b)$, the category $V\dash\PROP\big((\mathsf{P},a),(\mathsf{Q},b)\big)$ is equivalent to a set.
\end{proposition}

The following result states that given a $V$-PROP $(\mathsf{P},a)$, and a lax symmetric monoidal functor $F:V\to W$, we get a $W$-PROP \[ F(\mathsf{P},a)=(F\mathsf{P},a) \]
It follows from proposition \sref{proposition:change_enriching_cat_smcat}.

\begin{proposition}[change of enriching category]
Let $F:V\to W$  be a lax symmetric monoidal functor.\\
There exists a natural functor
\[ F:V\dash\PROP\To W\dash\PROP \]
such that
\begin{diagram}[midshaft,h=2.1em]
V\dash\PROP&\rTo{\ F\ }&W\dash\PROP\\
\dTo{\proj}&&\dTo{\proj}\\
V\dash\SymmMonCAT&\rTo{\ F\ }&W\dash\SymmMonCAT
\end{diagram}
commutes.
\end{proposition}

\begin{definition}[algebra for PROP]
Let $C$ be a symmetric monoidal $V$-category.\\
Given a $V$-PROP $(\mathsf{P},a)$, the {\em category of \((\mathsf{P},a)\)-algebras in \(C\)} is the category
\[ (\mathsf{P},a)\dash\alg(C)\defeq V\dash\SymmMonCAT(\mathsf{P},C) \]
An object of $(\mathsf{P},a)\dash\alg(C)$ is called a {\em \((\mathsf{P},a)\)-algebra in \(C\)} (or an {\em algebra over \((\mathsf{P},a)\) in \(C\)}).
\end{definition}

\begin{definition}[right module for PROP]
Let $V$ be a symmetric monoidal closed category.\\
Given a $V$-PROP $(\mathsf{P},a)$, a {\em right module over \((\mathsf{P},a)\)} is a $V$-functor
\[ \op{\mathsf{P}}\To V \]
\end{definition}

\section{Operads and categories of operators}\label{section:terminology_PROPs_operads}

Let $V$ be a symmetric monoidal category.

\begin{remark}[operads from PROPs]
Every $V$-PROP $(\mathsf{P},a)$ has an {\em underlying \(V\)-operad}, whose underlying symmetric sequence in $V$ is $\big(\mathsf{P}(a^{\tensor n},a)\big)_{n\in\NN}$.\\
The actions of the symmetric groups and the structure maps for the operad come from the composition in $\mathsf{P}$.\\
This construction gives a functor
\[ \undoperad{}:V\dash\PROP\To\text{operad}(V) \]
form the category of $V$-PROPs to the category of operads in $V$.
\end{remark}

To continue the comparison of PROPs and operads, call $V$ a {\em good} symmetric monoidal category if $V$ has all finite coproducts, and the monoidal product on $V$
\[ \tensor:V\times V\To V \]
is such that the functor $x\tensor -: V\to V$ preserves finite coproducts for any $x\in\ob V$.

\begin{proposition}\label{proposition:adjunction_operads_PROPs}
If $V$ is a good symmetric monoidal category, the functor
\[ \undoperad{}:V\dash\PROP\To\text{operad}(V) \]
has a (bicategorical\comment{meaning we get an equivalence of morphism categories instead of a bijection}) left adjoint 
\[ \catop{}:\text{operad}(V)\To V\dash\PROP \]
Moreover, the counit of this adjunction (which is a pseudo-natural transformation)
\[ \catop{}\circ\undoperad{}\To\id_{V\dash\PROP} \]
is an isomorphism.
\end{proposition}

\begin{notation}[category of operators]\label{notation:category_operators}
Given an operad $P$, the corresponding $V$-PROP $\catop{P}$ is called the {\em category of operators associated with \(P\)}.\\
A $V$-PROP is called a {\em category of operators} in $V$ if it is equivalent (in the $2$-category $V\dash\PROP$) to the category of operators associated to some $V$-operad.\\
The above result implies that a category of operators can be essentially recovered (up to equivalence of PROPs) from its underlying operad.
\end{notation}

\begin{remark}
Given a $V$-operad $P=\big(P(n)\big)_{n\in\NN}$, the category of operators associated with $P$ has
\[ \ob(\catop{P})=\NN \]
with generator the object $1$, and monoidal structure given on objects by addition on $\NN$.\\
Furthermore, for any $k,l\in\NN$, we have
\[ (\catop{P})(k,l)=\hspace{-.8em}\coprod_{{f\in\FinSet(k,l)}}\Tensor_{i\in l}P\big(f^{-1}(\set{i})\big) \]
\end{remark}

The next result says that taking algebras for an operad $P$ is equivalent to taking algebras over the category of operators associated with $P$.

\begin{proposition}
Let $V$ be a good symmetric monoidal category, and $C$ a symmetric monoidal $V$-category.\\
For any $V$-operad $P$, there exists an equivalence of categories
\[ \catop{P}\dash\alg(C)\To P\dash\alg(C) \]
natural in $P$ and $C$. Here, $P\dash\alg(C)$ denotes the category of algebras in $C$ over the operad $P$.
\end{proposition}

\begin{remark}
There is a similar equivalence between right modules over an operad, $P$, and right modules over its associated category of operators, $\catop{P}$.
\end{remark}

One of the advantages with operads is that we can push-forward algebras along maps of operads. We state the consequence for categories of operators.

\begin{proposition}
Let $V$ be a symmetric monoidal closed category, and $C$ a symmetric monoidal $V$-category.\\
Let $(\mathsf{P},a)$ and $(\mathsf{Q},b)$ be categories of operators in $V$, and $f:(\mathsf{P},a)\to (\mathsf{Q},b)$ a morphism of $V$-PROPs.\\
The functor
\[ \SymmMonCAT(f,C):(\mathsf{Q},b)\dash\alg(C)\To(\mathsf{P},a)\dash\alg(C) \]
has a left adjoint
\[ f_\ast:(\mathsf{P},a)\dash\alg(C)\To(\mathsf{Q},b)\dash\alg(C) \]
\end{proposition}

\section{Examples of PROPs in $\Set$ and $\Top$}\label{section:terminology_examples_PROPs}

We will now give a few important examples of $\Top$-PROPs (where we consider $\Top$ as a cartesian category). All our examples are actually categories of operators in $\Top$, and can therefore be essentially recovered for their underlying operads.

\begin{example}[commutative PROP]
The {\em commutative PROP}, $\Comm$, is the $\Set$-PROP given by the cocartesian category $\FinSet$ (i.e.\ the symmetric monoidal structure is given by disjoint union), with the generator $1$.\\
Given a symmetric monoidal category $C$, there is an equivalence of categories
\[ u:\Comm\dash\alg(C)\To\text{CommMon}(C) \]
between the category of $\Comm$-algebras in $C$ and the category of commutative monoids in $C$.\\
The equivalence $u$ takes a $\Comm$-algebra $F:\FinSet\to C$ to the commutative monoid $F(1)$ in $C$, which is called the {\em underlying commutative monoid of \(F\)}.
\end{example}

\begin{construction}[symmetric monoidal category $\Ord\Sigma$]\label{construction:OrdSigma}
The category $\Ord\Sigma$ is defined by
\begin{enum}
\item the objects of $\Ord\Sigma$ are the (small) finite sets;
\item given two finite sets $A$ and $B$, a morphism $A\to B$ in $\Ord\Sigma$ is a function $f:A\to B$ together with a total order on $f^{-1}(i)$ for each $i\in B$.
\end{enum}
We leave the composition in $\Ord\Sigma$ as an exercise to the reader.\\
The category $\Ord\Sigma$ fits canonically in a commutative diagram
\begin{diagram}[h=2em]
\Ord&\rTo&\Ord\Sigma\\
&\rdTo_{\mathllap{\proj}}&\dTo_{\proj}\\
&&\FinSet
\end{diagram}
There exists a unique symmetric monoidal structure on $\Ord\Sigma$ such that the functor \[ \Ord\To\Ord\Sigma \] is strict monoidal and the functor \[ \Ord\Sigma\To\FinSet \] is strict symmetric monoidal (the symmetric monoidal structure on $\FinSet$ is given by disjoint union).\comment{If we don't demand strictness of the monoidal functors, then we would have an essentially unique symmetric monoidal structure.}\\
Additionally, the monoidal functor $\Ord\to\Ord\Sigma$ induces an equivalence between $\Ord\Sigma$ and the free symmetric monoidal category on the monoidal category $\Ord$.
\end{construction}

\begin{example}[associative PROP]\label{example:associative_PROP}
The {\em associative PROP}, $\Ass$, is the $\Set$-PROP $(\Ord\Sigma,1)$.\\
Given a symmetric monoidal category $C$, there is an equivalence of categories
\[ u:\Ass\dash\alg(C)\To\text{AssMon}(C) \]
between the category of $\Ass$-algebras in $C$ and the category of associative monoids in $C$.\\
The equivalence $u$ takes a $\Ass$-algebra $F:\Ord\Sigma\to C$ to the associative monoid $F(1)$ in $C$, which is called the {\em underlying associative monoid of \(F\)}.
\end{example}

\begin{example}[little discs PROPs]
We define the {\em little $n$-discs \(\Top\)-PROP}, $\D_n$, to be the $\Top$-category whose objects are
\[ \ob(\D_n)\defeq\big\{(D^n)^{\amalg k}\suchthat k\in\NN\big\} \]
and such that for $k,l\in\NN$, the space of morphisms
\[ \D_n\big((D^n)^{\amalg k},(D^n)^{\amalg l}\big) \]
is the subspace of $\Map\big((D^n)^{\smash{\amalg k}},(D^n)^{\smash{\amalg l}}\big)$ consisting of the maps 
\[ f:(D^n)^{\amalg k}\To(D^n)^{\amalg l} \]
such that
\begin{enum}
\item $f$ restricted to $(\interior D^n)^{\smash{\amalg k}}$ is injective;
\item the restriction of $f$ to each disc in the disjoint union $(D^n)^{\amalg k}$ is the composition of a translation with multiplication by a positive real number.
\end{enum}
Composition in $\D_n$ is given by composition of maps.\\
The symmetric monoidal structure on $\D_n$ is given by disjoint union, and the generator of $\D_n$ is the object $D^n$.\\
It is straightforward to check that $\D_n$ is isomorphic to the category of operators associated with the usual little $n$-discs operad.
\end{example}

\begin{example}[framed little discs PROPs]
The {\em framed little $n$-discs} PROP, $\D_n^{O(n)}$ (respectively, $\D_n^{SO(n)}$), are defined similarly to $\D_n$. The only difference is that we require the restriction of the maps
\[ f:(D^n)^{\amalg k}\To(D^n)^{\amalg l} \]
to each disc in the domain to be the composition of (I) a translation, (II) multiplication by a positive real number, and (III) an element of $O(n)$ (respectively, an element of $SO(n)$).
Again, these PROPs are isomorphic to the categories of operators of the usual framed little discs operads.
\end{example}


%% file: internal.tex



\chapter{Internal categories}\label{chapter:internal_categories}

\section*{Introduction}

The purpose of this chapter is to cover the constructions on topological categories which will be required later in the text. Accordingly, this chapter discusses internal categories and their relation to enriched categories. Most importantly, we define the Grothendieck construction for internal presheaves of categories, and apply it to the case of topological categories.

\section*{Summary}

The present chapter is quite long, and is meant mostly as reference for later chapters. On that note, the most relevant sections are the last two, which deal with Grothendieck constructions and topological categories.

A quick general reference on internal categories is chapter 8 of the book \cite{Borceux}, although it does not contain all the material necessary for our applications.

Section \sref{section:internal_categories} details the basic concepts of internal category, internal functor, and internal natural transformation in a category with pullbacks, $V$. Section \sref{section:categories_internal_categories} defines the $2$-category of internal categories in $V$, $\Cat(V)$. Section \sref{section:examples_internal_categories} gives examples of internal categories, including ordinary categories, and the path category of a space $X$, $\pathcat(X)$.

Section \sref{section:coproducts_finitely_complete_category} gives some useful definitions relating to coproducts in finitely complete categories. These conditions are then used in section \sref{section:internal_external_categories} to compare enriched categories and internal categories. In particular, section \sref{section:internal_external_categories} defines, under appropriate conditions, an internal category $\internal A$ associated to a category $A$ enriched over a cartesian closed category. Inversely, it also associates to each internal category $A$ an enriched category $\disccat{A}$.

Section \sref{section:internal_presheaves} dwells on the concept of $V$-valued functor (or presheaf) on an internal category in $V$. The following section \sref{section:internal_enriched_presheaves} resumes the comparison between enriched and internal concepts, now focusing on presheaves.

Section \sref{section:internal_presheaves_categories} renews the discussion of internal presheaves to define the concept of internal $\Cat(V)$-valued functors. This is the necessary background for section \sref{section:Grothendieck_construction}, where the Grothendieck construction, $\Groth(F)$, of an internal functor $F:\op{A}\to\Cat(V)$ is described: $\Groth(F)$ is again an internal category in $V$.

Finally, the last two sections apply the concepts introduced in this chapter to the case of internal categories in $\Top$. Section \sref{section:variation_Groth_Top} gives a variation of the Grothendieck construction for functors from an ordinary category to $\Top\dash\Cat$. Section \sref{section:homotopical_properties_Groth} briefly discusses the Grothendieck construction for topological categories from a homotopical perspective.

\section{Internal categories}\label{section:internal_categories}

\begin{definition}[internal category]\label{definition:cat_objects}
Let $V$ be a category with pullbacks.\\
An {\em internal category} (or {\em category object}) in $V$ is given by
\begin{enum}
\item an object $\ob A$ in $V$ called the {\em object of objects} of $A$;
\item an object $\mor A$ in $V$ called the {\em object of morphisms} of $A$;
\item a morphism $s:\mor A\to\ob A$ in $V$ called the {\em source};
\item a morphism $t:\mor A\to\ob A$ in $V$ called the {\em target};
\item a morphism $i:\ob A\to\mor A$ in $V$ called the {\em identity};
\item a morphism
\[ c:\lim\left(\mor A\xTo{t}\ob A\xlongleftarrow{s}\mor A\right)\To\mor A \]
called the {\em composition}.
\end{enum}
These data are required to verify
\begin{align*}
s\circ i=\id_{\ob A}\\
t\circ i=\id_{\ob A}
\end{align*}
and to make the three diagrams
\begin{diagram}[midshaft]
\mor A&\lTo{\ \proj_L\;}&\lim\big(\overbrace{\mor A}^{L}\xto{t}\ob A\xleftarrow{s}\overbrace{\mor A}^R\big)&\rTo{\ \proj_R\ }&\mor A\\
\dTo{s}&&\dTo{c}&&\dTo{t}\\
\ob A&\lTo{s}&\mor A&\rTo{t}&\ob A\\
\end{diagram}
\begin{diagram}[midshaft,objectstyle=\scriptstyle,labelstyle=\scriptscriptstyle]
\lim{\displaystyle(}\mor A\xto{t}\ob A\xleftarrow{s}\mor A\xto{t}\ob A\xleftarrow{s}\mor A{\displaystyle)}&\rTo_{\;c\underset{\!\ob A\!}{\times}\mor A\ }&\lim{\displaystyle(}\mor A\xto{t}\ob A\xleftarrow{s}\mor A{\displaystyle)}\\
\dTo{\mor A\underset{\!\ob A\!}{\times}c}&&\dTo{c}\\
\lim{\displaystyle(}\mor A\xto[t]{}\ob A\xleftarrow[s]{}\mor A{\displaystyle)}&\rTo_{c}&\mor A
\end{diagram}
\begin{diagram}[midshaft,balance,objectstyle=\scriptstyle,labelstyle=\scriptscriptstyle]
&&\hspace{-2em}\lim{\displaystyle(}\mor A\xto{t}\ob A\xleftarrow{s}\mor A{\displaystyle)}\hspace{-2em}&&\\
&\ruTo{\mathllap{i\underset{\!\ob A\!}{\times}\mor A}}&\dTo{c}&\luTo{\mathrlap{\mor A\underset{\!\ob A\!}{\times}i}}&\\
\lim{\displaystyle(}\ob A\xto[\id]{}\ob A\xleftarrow[s]{}\mor A{\displaystyle)}&\rTo{\simeq}&\mor A&\lTo{\simeq}&\lim{\displaystyle(}\mor A\xto[t]{}\ob A\xleftarrow[\id]{}\ob A{\displaystyle)}
\end{diagram}
all commute.
\end{definition}

\begin{remark}[opposite internal category]\label{remark:opposite_internal_category}
Given an internal category $A$ in $V$, reversing the roles of the source and target morphisms, $s$ and $t$, gives a new internal category in $V$, $\op{A}$, called the {\em opposite of \(A\)}. It has the same objects and morphisms as $A$.
\end{remark}

\begin{definition}[internal functor]\label{definition:internal_functor}
Let $V$ be a category with pullbacks. Let $A$, $B$ be internal categories in $V$.\\
An {\em internal functor} $F:A\to B$ is given by a pair of morphisms in $V$
\begin{align*}
\ob F&:\ob A\To\ob B\\
\mor F&:\mor A\To\mor B
\end{align*}
such that the three diagrams
\begin{diagram}[h=2.1em]
\ob A&\rTo{\ob F}&\ob B\\
\dTo{i}&&\dTo{i}\\
\mor A&\rTo{\mor F}&\mor B
\end{diagram}
\begin{diagram}[midshaft,h=2.1em]
\ob A&\lTo{s}&\mor A&\rTo{t}&\ob A\\
\dTo{\ob F}&&\dTo{\mor F}&&\dTo{\ob F}\\
\ob B&\lTo{s}&\mor B&\rTo{t}&\ob B
\end{diagram}
\begin{diagram}
\lim\big(\mor A\xrightarrow{t}\ob A\xleftarrow{s}\mor A\big)&\rTo_{\ \mor F\underset{\!\ob F\!}{\times}\mor F\ }&\lim\big(\mor B\xrightarrow{t}\ob B\xleftarrow{s}\mor B\big)\\
\dTo{c}&&\dTo{c}\\
\mor A&\rTo{\mor F}&\mor B
\end{diagram}
all commute.
\end{definition}

\begin{definition}[internal natural transformation]\label{definition:internal_natural_transformation}
Let $V$ be a category with pullbacks.\\
Let $A$, $B$ be internal categories in $V$, and $F,G:A\to B$ internal functors.\\
An internal natural transformation in $V$, $\alpha:F\to G$ is a morphism
\[ \alpha:\ob A\To\ob B \]
such that
\begin{align*}
s\circ\alpha&=\ob F\\
t\circ\alpha&=\ob G
\end{align*}
and the following diagram commutes
\begin{diagram}[midshaft]
\mor A&\rTo{(\alpha\circ s,\mor G)}&\mor B\underset{\!\ob B\!}{_t\times_s}\mor B\\
\dTo{(\mor F,\alpha\circ t)}&&\dTo{c}\\
\mor B\underset{\!\ob B\!}{_t\times_s}\mor B&\rTo{c}&\mor B\\
\end{diagram}
\end{definition}

\section{Categories of internal categories}\label{section:categories_internal_categories}

\begin{proposition}[2-category of internal categories]\label{proposition:Cat(V)}
Let $V$ be a category with pullbacks.\\
The internal categories, internal functors, and internal natural transformations in $V$ form the objects, $1$-morphisms, and $2$-morphisms, respectively, of a $2$-category $\Cat(V)$.\\
Furthermore, $\Cat(V)$ has pullbacks.
\end{proposition}

\begin{remark}
We leave it to the reader unfamiliarized with internal categories to define the several compositions in $\Cat(V)$, and to check that these indeed give a $2$-category.
\end{remark}

\comment{
The {\em category of internal categories in \(V\)}, $\Cat(V)$, is defined by:
\begin{enum}
\item the objects of $\Cat(V)$ are the internal categories in $V$;
\item given internal categories $A$, $B$ in $V$, $\ob\big(\Cat(V)(A,B)\big)$ is the set of internal functors $A\to B$;
\item given internal categories $A$, $B$ in $V$, and $F,G\in\Cat(V)(A,B)$, $\Cat(V)(F,G)$ is the set of internal natural transformations $F\to G$;
\item given internal functors between internal categories in $V$
\[ A\xTo{F}B\xTo{G}C \]
the composite is defined by
\begin{align*}
\mor(G\circ F)&\defeq\mor G\circ\mor F\\
\ob(G\circ F)&\defeq\ob G\circ\ob F
\end{align*}
\end{enum}
}

\begin{remark}
The association of $\ob A$, $\mor A$ to an internal category $A$ in $V$ extends to functors
\begin{align*}
\ob&:\Cat(V)\To V\\
\mor&:\Cat(V)\To V
\end{align*}
Furthermore, the source, target, and identity of internal categories give rise to natural transformations
\begin{align*}
s&:\mor\To\ob\\
t&:\mor\To\ob\\
i&:\ob\To\mor
\end{align*}
\end{remark}

Given that the previous definitions were very long, we give a short alternative characterization of the underlying $1$-category of $\Cat(V)$ (\ie forget all the 2-cells in $\Cat(V)$).

\begin{proposition}\label{proposition:internal_categories_simplicial_objects}
For each category with pullbacks $V$, there is an equivalence of categories (natural in $V$)
\[ \HomCat{\op{\Delta}}{V}^{pb}\xTo{\ \sim\ }\Cat(V)^{\leq 1} \]
Here, $\HomCat{\op{\Delta}}{V}^{pb}$ denotes the full subcategory of $\HomCat{\op{\Delta}}{V}$ generated by the functors which preserve all pullbacks that exist in $\op{\Delta}$. Moreover, $\Cat(V)^{\leq 1}$ denotes the underlying $1$-category of $\Cat(V)$.
\end{proposition}

\begin{remark}
The functor above associates to a simplicial object in $V$, $X$, which preserves all pullbacks in $\op{\Delta}$, an internal category $A$ such that
\begin{align*}
\ob A&\defeq X(1)\\
\mor A&\defeq X(2)
\end{align*}
\end{remark}

\begin{definition}[nerve functor]\label{definition:nerve_functor}
Any specified inverse to the equivalence in proposition \sref{proposition:internal_categories_simplicial_objects} is called the {\em nerve functor} for $V$:
\[ \Nerve:\Cat(V)^{\leq 1}\To\HomCat{\op{\Delta}}{V}^{pb} \]
\end{definition}

\begin{proposition}[transfer of internal categories]\label{proposition:transfer_cat}
Let $V$, $W$ be categories with pullbacks. Let $F:V\to W$ be a functor which preserves all pullbacks.\\
There exists a natural induced functor
\[ \Cat(F):\Cat(V)\To\Cat(W) \]
which verifies
\begin{align*}
\ob\circ\Cat(F)&=F\circ\ob\\
\mor\circ\Cat(F)&=F\circ\mor
\end{align*}
Moreover, $\Cat(F)$ preserves all pullbacks.
\end{proposition}

\begin{remark}
The functor $\Cat(F)$ can be described quite simply on the level of the simplicial objects.\\
The functor
\[ \HomCat{\op{\Delta}}{F}:\HomCat{\op{\Delta}}{V}\To\HomCat{\op{\Delta}}{W} \]
restricts to a functor
\[ \HomCat{\op{\Delta}}{F}:\HomCat{\op{\Delta}}{V}^{pb}\To\HomCat{\op{\Delta}}{W}^{pb} \]
and the diagram (where each vertical functor is induced by the equivalence from proposition \sref{proposition:internal_categories_simplicial_objects})
\begin{diagram}[midshaft,h=2.3em]
\HomCat{\op{\Delta}}{V}^{pb}&\rTo{\HomCat{\op{\Delta}}{F}}&\HomCat{\op{\Delta}}{W}^{pb}\\
\dTo&&\dTo\\
\Cat(V)&\rTo{\Cat(F)}&\Cat(W)
\end{diagram}
commutes.
\end{remark}

\begin{example}[ordinary categories from internal categories]\label{example:Cat(V)->Cat(SET)}
Let $V$ be a $\SET$-category with pullbacks.\\
Given $x\in V$, the functor
\[ V(x,-):V\To\SET \]
preserves pullbacks, and so we get a functor
\[ \Cat\big(V(x,-)\big):\Cat(V)\To\Cat(\Set) \]
\end{example}

\section{Examples of internal categories}\label{section:examples_internal_categories}

\begin{example}[usual categories]\label{example:internal_set_categories}
A category object in $\Set$ is the same thing as an ordinary small category. A category object in $\SET$ is the same thing as a large category (large set of objects and large sets of morphisms).
We thus obtain equivalences of categories
\begin{align*}
\Cat(\Set)&\xhookrightarrow{\ \simeq\ }\Cat\\
\Cat(\SET)&\xhookrightarrow{\ \simeq\ }\SET\dash\Cat_{\SET}
\end{align*}
(recall notation from \sref{notation:categories_of_categories}).
\end{example}

\begin{example}[discrete category]\label{example:discrete_cat}
Let $V$ be a category with pullbacks.\\
For any object $x$ of $V$, the corresponding constant simplicial object in $V$ gives an internal category in $V$, which is called the {\em discrete category} on $x$, $\disc(x)$.\\
To be more concrete, this category verifies
\[ \ob\big(\disc(x)\big)=\mor\big(\disc(x)\big)=x \]
\[ s=t=i=\id_x \]
which also determines the composition.\\
We thus get a functor
\[ \disc:V\To\Cat(V) \]
which actually preserves pullbacks.
\end{example}

\begin{example}[indiscrete category]
Let $V$ be a category with all finite limits.\\
Given $x\in V$, the {\em indiscrete category} on $x$, $\indisc(x)$ is the category determined by
\begin{align*}
\ob\big(\indisc(x)\big)&\defeq x\\
\mor\big(\indisc(x)\big)&\defeq x\times x
\end{align*}
\begin{align*}
s&\defeq\proj_1\\
t&\defeq\proj_2\\
i&\defeq\text{diag}_x
\end{align*}
which uniquely determines the composition morphism.\\
We thus get a functor
\[ \indisc:V\To\Cat(V) \]
\end{example}

\begin{example}[path category]\label{example:path_category}
Given a topological space $X$, we define the {\em path category of \(X\)}, $\pathcat(X)$, to be an internal category in $\Top$ whose objects and morphisms are
\begin{align*}
\ob\big(\pathcat(X)\big)&\defeq X\\
\mor\big(\pathcat(X)\big)&\defeq H(X)
\end{align*}
where $H(X)$ is the Moore path space of $X$ from section \sref{section:terminology_Moore_path_space}. The source, target, and identity structure morphisms for $\pathcat(X)$ are given by the maps
\begin{align*}
s&:H(X)\To X\\
t&:H(X)\To X\\
i&:X\To H(X)
\end{align*}
defined in \sref{definition:maps_Moore_path_space}. The composition in $\pathcat(X)$ is given by concatenation of Moore paths (see \sref{definition:concat_Moore_paths}).\\
This construction extends to a functor
\[ \pathcat:\Top\To\Cat(\Top) \]
\end{example}

\section{Coproducts in finitely complete categories}\label{section:coproducts_finitely_complete_category}

In this section, we fix a full subcategory, $\setuniverse$, of $\SET$ closed under taking subsets, and finite limits in $\SET$ (\ie given any diagram $F:D\to\setuniverse$ indexed by a finite category $D$, there exists an object in $\setuniverse$ which is the limit of $F$ in $\SET$). Equivalently, $\setuniverse$ is closed under taking subsets, and finite products in $\SET$.

We will discuss some notions which will be useful in the following section to compare internal and enriched categories.

\begin{definition}
Let $V$ be a category.\\
Define the category $\setuniverse/\!/V$ by
\begin{enum}
\item the objects of $\setuniverse/\!/V$ are pairs $(S,F)$ where $S$ is a set in $\setuniverse$ and $F:S\to V$ is a functor;
\item a morphism $(S,F)\to(S',F')$ in $\setuniverse/\!/V$ is a pair $(f,\alpha)$ where $f\in\setuniverse(S,S')$, and
\[ \alpha:F\To F'\circ f \]
is a natural transformation;
\item given morphisms
\[ (S,F)\xTo{(f,\alpha)}(S',F')\xTo{(g,\beta)}(S''.F'') \]
in $\setuniverse/\!/V$, their composite is $\big(g\circ f,(\beta\circ f)\cdot\alpha\big)$.
\end{enum}
\end{definition}

\begin{construction}
Let $V$ be a category with all coproducts indexed by sets in $\setuniverse$.\\
There exists a functor
\[ \amalg:\setuniverse/\!/V\To V \]
given on objects by
\[ \amalg(S,F)\defeq\colim_S F=\coprod_S F \]
The functor $\amalg$ is defined on morphisms via the functoriality of colimits.
\end{construction}

\begin{definition}[category with disjoint coproducts]\label{definition:disjoint_coproducts}
Let $V$ be a finitely complete category which has all coproducts indexed by sets in $\setuniverse$.\\
$V$ is said to have {\em disjoint \(\setuniverse\)-coproducts} if the functor
\[ \amalg:\setuniverse/\!/V\To V \]
preserves all finite limits.
\end{definition}

\begin{remark}
$V$ has disjoint $\setuniverse$-coproducts if \[ \amalg:\setuniverse/\!/V\To V \] preserves pullbacks. This condition can be restated as: given
\begin{enum}
\item a diagram $A\xto{\smash{f}}C\xleftarrow{\smash{g}}B$ in $\setuniverse$,
\item functors $F_A:A\to V$, $F_B:B\to V$, and $F_C:C\to V$,
\item natural transformations $\alpha:F_A\to F_C\circ f$, $\beta:F_B\to F_C\circ g$,
\end{enum}
the limit of the induced diagram
\[ \coprod_A F_A\xTo{\alpha}\coprod_C F_C\xlongleftarrow{\beta}\coprod_B F_B \]
is given by the coproduct
\[ \coprod_{\hspace{-.5em}(a,b)\in A\underset{C}{\times}B\hspace{-.5em}}\lim\Big(F_A(a)\xto{\alpha_a}F_C(f(a))=F_C(g(b))\xleftarrow{\beta_b} F_B(b)\Big) \]
\end{remark}

\begin{notation}\label{notation:small_disjoint_coproducts}
If $\setuniverse=\FinSet$, we say $V$ has {\em disjoint finite coproducts}. If $\setuniverse=\Set$, we say $V$ has {\em disjoint small coproducts}. If $\setuniverse=\SET$, we say $V$ has {\em disjoint large coproducts}.
\end{notation}

\begin{examples}
The categories $\Set$, $\Top$, and $k\Top$ have disjoint small (and finite) coproducts. The categories $\SET$ and $\TOP$ have disjoint large coproducts.
\end{examples}

\begin{definition}[category with strongly/totally disjoint coproducts]\label{definition:totally_disjoint_coproducts}
Let $V$ be a finitely complete category with disjoint $\setuniverse$-coproducts.\\
$V$ is said to have {\em strongly disjoint \(\setuniverse\)-coproducts} (respectively, {{\em totally disjoint \(\setuniverse\)-coproducts}}) if
\begin{enum}
\item for each set $S$ in $\setuniverse$,
\item and each functor $F:S\to V$,
\end{enum}
the natural functor, obtained by taking the coproduct along $S$,
\[ \amalg_S:\HomCat{S}{V}/F\To V/(\amalg_S F) \]
is full and faithful (respectively, an equivalence of categories).
\end{definition}

\begin{notation}\label{notation:small_totally_disjoint_coproducts}
As before, we use the terminology ``small'' and ``large'' to refer to the cases of $\Set$ and $\SET$.\\
In particular, if $\setuniverse=\Set$, we say $V$ has {\em strongly disjoint small coproducts} (or {\em totally disjoint small coproducts}).
\end{notation}

\begin{examples}
The categories $\Set$, $\Top$, and $k\Top$ have totally disjoint $\Set$-coprod\-ucts. The categories $\SET$ and $\TOP$ have totally disjoint $\SET$-coprod\-ucts.
\end{examples}

\begin{definition}[connected object]\label{definition:connected_object}
Let $V$ be a finitely complete category which has all coproducts indexed by sets in $\setuniverse$.\\
An object $x$ of $V$ is said to be {\em connected over \(\setuniverse\)} if for each set $S$ in $\setuniverse$, the natural function
\[ S\To V\big(x,1^{\amalg S}\big) \]
is a bijection.
\end{definition}

\begin{examples}
The object $1$ in $\Set$ (or $\Top$, or $k\Top$) is connected over $\Set$. The object $1$ in $\SET$ (or $\TOP$) is connected over $\SET$.
\end{examples}

\section{Relation between internal and enriched categories}\label{section:internal_external_categories}

With the definitions of the preceding section, we are ready to tackle the passage from internal categories to enriched categories and vice-versa. We fix again, throughout this section, a full subcategory, $\setuniverse$, of $\SET$ closed under taking subsets, and finite limits in $\SET$.

\begin{remark}
Given a category with finite products, we view it as a cartesian monoidal category $(V,\times,1)$, where $1\in V$ is the terminal object.
\end{remark}

\begin{definition}[enriched categories from internal categories]
Assume $V$ is a finitely complete category.\\
Given an internal category $A$ in $V$, we define the associated {\em discretized} $(V,\times,1)$-category, $\disccat{A}$:
\begin{enum}
\item the objects are $\ob\big(\disccat{A}\big)\defeq V(1,\ob A)$;
\item given $x,y:1\to\ob A$, we define $\disccat{A}(x,y)$ to be the pullback of
\[ 1\xTo{(x,y)}\ob A\times\ob A\xlongleftarrow{(s,t)}\mor A \]
\item the composition in $\disccat{A}$ (for $x,y,z\in\ob\disccat{A}$)
\[ \composition:\disccat{A}(x,y)\times\disccat{A}(y,z)\To\disccat{A}(x,z) \]
is the unique morphism for which the diagram
\begin{diagram}[h=2.3em]
\disccat{A}(x,y)\times\disccat{A}(y,z)&\rTo{\ \composition\ }&\disccat{A}(x,z)\\
\dTo{\proj\times\proj}&&\dTo{\proj}\\
\mor A\underset{\!\ob A\!}{_t\times_s}\mor A&\rTo{c}&\mor A
\end{diagram}
commutes.
\end{enum}
\end{definition}

\begin{remark}\label{remark:disccat_internal_categories}
Assume now $V$ is a finitely complete $\setuniverse$-category.\\
The above construction extends to a functor (recall the notation $V\dash\Cat_{\setuniverse}$ from \sref{notation:categories_V-categories})
\[ \disccat{(-)}:\Cat(V)\To V\dash\Cat_{\setuniverse} \]
where $V$ is viewed as a cartesian monoidal category.
\end{remark}

\begin{example}[topological spaces]
Given an internal category, $A$, in $\Top$, the set of objects of the $\Top$-category $\disccat{A}$ is the underlying set of $\ob A$.
\end{example}

\begin{definition}[internal categories from enriched categories]\label{definition:enriched_cat->internal_cat}
Let $V$ be a finitely complete category with disjoint $\setuniverse$-coproducts.\\
Given a $(V,\times,1)$-category, $A$, whose set of objects is in $\setuniverse$, we define the {\em internalization of \(A\)}, $\internal A$, to be the internal category in $V$ determined by
\begin{enum}
\item the object of objects is \[ \ob(\internal A)\defeq\coprod_{\ob A} 1=1^{\amalg(\ob A)} \]
\item the object of morphisms is \[ \mor(\internal A)\defeq\ \ \coprod_{\mathclap{(x,y)\in\ob A\times\ob A}}\ \ A(x,y)=\ \coprod_{\mathclap{x,y\in\ob A}}\ A(x,y) \]
\item the source map makes the following diagram commute for each $x,y\in\ob A$:
\begin{diagram}[midshaft,h=2.3em]
A(x,y)&\rTo&1\\
\dInto_{\inclusion}&&\dInto_{\inclusion_x}\\
\mor(\internal A)&\rTo{\ s\ }&\coprod_{\ob A} 1
\end{diagram}
\item the target map makes the following diagram commute for each $x,y\in\ob A$:
\begin{diagram}[midshaft,h=2.3em]
A(x,y)&\rTo&1\\
\dInto_{\inclusion}&&\dInto_{\inclusion_y}\\
\mor(\internal A)&\rTo{\ t\ }&\coprod_{\ob A} 1
\end{diagram}
\item the identity map makes the following diagram commute for each $x\in\ob A$:
\begin{diagram}[midshaft]
1&\rInto{\inclusion_x}&\coprod_{\ob A}1\\
\dTo_{\id_x}&&\dTo_{i}\\
A(x,x)&\rInto{\quad\inclusion\ \ }&\mor(\internal A)
\end{diagram}
\item composition in $\internal A$ makes the following diagram commute
\begin{diagram}[midshaft]
\coprod_{\mathclap{x,y,z\in\ob A}}\ A(x,y)\times A(x,z)&\rTo{\ \composition\ }&\ \coprod_{\mathclap{x,z\in\ob A}}\ A(x,z)\\
\dTo[rightshortfall=-.5em]{\rotc{90}{$\scriptstyle\simeq$}}&&\dEqual[rightshortfall=0.1em]\\
\lim\big({\scriptstyle\mor(\internal A)\xto{\smash{t}}\ob(\internal A)\xleftarrow{\smash{s}}\mor(\internal A)}\big)&\rTo{c}&\mor(\internal A)
\end{diagram}
where the left vertical map is the natural map from the coproduct (which is an isomorphism since $V$ has disjoint $\setuniverse$-coproducts), and the top map is canonically obtained from composition in $A$.
\end{enum}
\end{definition}

\begin{remark}\label{remark:functoriality_internal_cat}
Let $V$ be a finitely complete category with disjoint $\setuniverse$-coproducts.\\
The above construction extends to a functor (recall the notation $V\dash\Cat_{\setuniverse}$ from \sref{notation:categories_V-categories})
\[ \internal:V\dash\Cat_{\setuniverse}\To\Cat(V) \]
where $V$ is viewed as a cartesian monoidal category.
\end{remark}

\begin{proposition}[correspondence between enriched and internal categories]\label{proposition:disccat_internal=id}
Let $V$ be a finitely complete $\setuniverse$-category with disjoint $\setuniverse$-coproducts.\\
There is a canonical natural transformation $\Gamma^V$ from the identity functor on $V\dash\Cat_{\setuniverse}$ to the composition
\[ V\dash\Cat_{\setuniverse}\xTo{\internal}\Cat(V)\xTo{\disccat{(-)}}V\dash\Cat_{\setuniverse} \]
If the object $1$ of $V$ is connected over $\setuniverse$, then $\Gamma^V$ is a natural isomorphism.
\end{proposition}

\begin{proposition}[correspondence between enriched and internal categories]
Let $V$ be a finitely complete category with strongly disjoint $\setuniverse$-coproducts.\\
If the object $1$ of $V$ is connected over $\setuniverse$ then the functor
\[ \internal:V\dash\Cat_\setuniverse\To\Cat(V) \]
is a local isomorphism of $2$-categories (\ie induces an isomorphism of categories
\[ \internal:V\dash\Cat_\setuniverse(A, B)\To\Cat(V)(\internal A,\internal B) \]
for all $A$, $B$ in $V\dash\Cat_\setuniverse$).
\end{proposition}

We finish with a simple case in which the transfer of internal categories is compatible with the transfer of enriched categories.

\begin{proposition}\label{proposition:transfer_internal_enriched_categories}
Let $V$, $W$ be finitely complete $\setuniverse$-categories, and $F:V\to W$ a functor which preserves all finite limits.\\
Assume furthermore that $F$ induces a bijection
\[ F:V(1,x)\xTo{\simeq}W(1,Fx) \]
for any object $x$ in $V$.\\
Then there is a canonical natural isomorphism which makes the following diagram commute
\begin{diagram}[midshaft,h=2.3em]
\Cat(V)&\rTo{\disccat{(-)}}&V\dash\Cat_\setuniverse\\
\dTo{\Cat(F)}&&\dTo{F}\\
\Cat(W)&\rTo{\disccat{(-)}}&W\dash\Cat_\setuniverse
\end{diagram}
In particular, for each internal category $A$ in $V$, there exists a canonical isomorphism
\[ \disccat{\big(\Cat(F)(A)\big)}=F\big(\disccat{A}\big) \]
\end{proposition}

\section{Internal presheaves}\label{section:internal_presheaves}

We have defined in \sref{definition:internal_functor} and \sref{definition:internal_natural_transformation} the notion of internal functor and internal natural transformation in $V$. These gives us, for any internal categories $A$ and $B$, a category of internal functors $\Cat(V)(A,B)$ (see proposition \sref{proposition:Cat(V)}). One might try to extract from this a canonical notion of presheaf on an internal category, analogous to the notion of $\Set$-valued functors on a small category. Unfortunately, there is (in general) no canonical internal category to take as the target for internal functors. In this section we will define the notion of $V$-valued functors on an internal category in $V$ which will play the desired role of presheaves on an internal category.

\begin{definition}[internal $V$-valued functor]\label{definition:internal_V-valued_functor}
Let $V$ be a category with pullbacks, and $A$ an internal category in $V$.\\
An {\em internal \(V\)-valued functor on \(A\)}, $F:A\to V$, is a triple $F=(P,p_0,p_1)$ where
\begin{enum}
\item $P$ is an object of $V$;
\item $p_0:P\to\ob A$ is a morphism in $V$;
\item $p_1$ is a morphism in $V$
\[ p_1:\lim\big(P\xto{p_0}\ob A\xleftarrow{s}\mor A\big)\To P \]
\end{enum}
These data are required to make the three diagrams
\begin{diagram}[midshaft,h-2.3em,balance]
\lim\big(P\xto{p_0}\ob A\xleftarrow{s}\mor A\big)&\rTo{p_1}&P\\
\dTo{\proj}&&\dTo{p_0}\\
\mor A&\rTo{t}&\ob A
\end{diagram}
\begin{diagram}[midshaft,hug]
\lim\big(P\xto{p_0}\ob A\xleftarrow{\id}\ob A\big)\\
\dTo{\id_P\underset{\!\ob A\!}{\times}i}&\rdTo{\simeq}\\
\lim\big(P\xto{\smash{p_0}}\ob A\xleftarrow{\smash{s}}\mor A\big)&\rTo_{p_1}&P
\end{diagram}
\begin{diagram}[midshaft,objectstyle=\scriptstyle,labelstyle=\scriptscriptstyle]
\lim\big(P\xto{p_0}\ob A\xleftarrow{s}\mor A\xto{t}\ob A\xleftarrow{s}\mor A\big)&\rTo_{\ p_1\underset{\!\ob A\!}{\times}\mor A\ }&\lim\big(P\xto{p_0}\ob A\xleftarrow{s}\mor A\big)\\
\dTo{P\underset{\!\ob A\!}{\times}c}&&\dTo{p_1}\\
\lim\big(P\xto{\smash{p_0}}\ob A\xleftarrow{\smash{s}}\mor A\big)&\rTo_{p_1}&P
\end{diagram}
commute.
\end{definition}

\begin{definition}[internal $V$-valued natural transformation]
Let $V$ be a category with pullbacks, $A$ an internal category in $V$.\\
Additionally, assume $F,G:A\to V$ are internal $V$-valued functors, with $F=(P,p_0,p_1)$ and $G=(Q,q_0,q_1)$.\\
An {\em internal \(V\)-valued natural transformation on \(A\)}, $\alpha:F\to G$ is a morphism $\alpha:P\to Q$ in $V$ such that
\[ q_0\circ\alpha=p_0 \]
and the diagram
\begin{diagram}[midshaft]
\lim\big(P\xto{p_0}\ob A\xleftarrow{s}\mor A\big)&\rTo{p_1}&P\\
\dTo{\alpha\underset{\!\ob A\!}{\times}\mor A}&&\dTo{\alpha}\\
\lim\big(Q\xto{q_0}\ob A\xleftarrow{s}\mor A\big)&\rTo{q_1}&Q
\end{diagram}
commutes.
\end{definition}

\begin{example}[Yoneda presheaf]\label{example:internal_Yoneda_presheaf}
Let $V$ be a finitely complete category, and $A$ an internal category in $V$.\\
Given a morphism $x:1\to\ob A$ (``an object of $A$''), we define the {\em Yoneda presheaf of \(x\)} \[ \Yoneda_A(x):\op{A}\To V \] to be the triple $(P,p_0,p_1)$, where $P$ is the pullback of
\[ 1\xTo{x}\ob A\xlongleftarrow{t}\mor A \]
and $p_0$ is the restriction of $s$ to $P$, and $p_0$ is induced from the composition, $c$, in $A$.\\
This construction extends to a functor (recall example \sref{example:Cat(V)->Cat(SET)})
\[ \Yoneda_A:\Cat\big(V(1,-)\big)(A)\To\Cat(V)(\op{A},V) \]
\end{example}

\begin{proposition}[category of internal $V$-valued functors]
Let $V$ be a category with pullbacks, and $A$ an internal category in $V$.\\
The internal $V$-valued functors on $A$, and the internal $V$-valued natural transformations on $A$ form the objects and the morphisms, respectively, of a category $\Cat(V)(A,V)$.
\end{proposition}

\begin{remark}
The composition of internal natural transformations is just given by composing the corresponding morphisms in $V$.
\end{remark}

\begin{notation}
We call the category $\Cat(V)(A,V)$ the {\em category of internal \(V\)-valued functors on \(A\)}. We denote it by $\Cat(V)(A,V)$ in analogy with the category of internal functors $\Cat(V)(A,B)$ between internal categories $A$, $B$ in $V$, despite $V$ not being an internal category in $V$.
\end{notation}

\begin{example}[internal $\Set$-valued functors]
Under the identification of an internal category in $\Set$ with an ordinary small category (example \sref{example:internal_set_categories}), we obtain an isomorphism
\[ \Cat(\Set)(A,\Set)\simeq\HomCat{A}{\Set} \]
for any category $A$ internal to $\Set$.
\end{example}

\begin{construction}
Let $V$ be a category with pullbacks.\\
Assume $A$, $B$ are internal categories in $V$, and $F:A\to B$ is an internal functor.
Then we get an induced functor
\[ \Cat(V)(F,V):\Cat(V)(B,V)\To\Cat(V)(A,V) \]
which we now describe on objects. Suppose $(P,p_0,p_1)$ is an object of $\Cat(V)(B,V)$. Then the internal $V$-valued functor
\[ (Q,q_0,q_1)=\big(\Cat(V)(F,V)\big)(P,p_0,p_1) \]
on $A$ is determined by
\begin{enum}
\item $Q$ is the pullback of
\[ \ob A\xTo{\ob F}\ob B\xlongleftarrow{p_0}P \]
and $q_0:Q\to\ob A$ is the canonical projection;
\item the morphism $q_1$ makes the diagram
\begin{diagram}[midshaft]
\mor A&\rTo{t}&\ob A\\
\uTo{\proj}&&\uTo_{q_0}\\
\lim\big(Q\xto{\smash{q_0}}\ob A\xleftarrow{\smash{s}}\mor A\big)&\rTo{q_1}&Q\\
\dTo{\proj\underset{\!\ob F\!}{\times}\mor F}&&\dTo_{\proj}\\
\lim\big(P\xto{\smash{p_0}}\ob B\xleftarrow{\smash{s}}\mor B\big)&\rTo{p_1}&P
\end{diagram}
commute.
\end{enum}
\end{construction}

This construction is the basis for the next proposition.

\begin{proposition}[functoriality of internal presheaves]\label{proposition:functor_Cat(V)(-,V)}
Let $V$ be a category (in $\bigCAT$) with pullbacks.\\
There is a functor
\[ \Cat(V)(-,V):\op{\Cat(V)}\To\bigCAT \]
which associates to an internal category in $V$, $A$, the category $\Cat(V)(A,V)$ of internal $V$-valued functors on $A$.
\end{proposition}

\begin{remark}
We leave it to the reader to supply the remaining ingredients for the functor declared in the above proposition.
\end{remark}

\begin{notation}\label{notation:composition_internal_V-valued_functor}
Given an internal functor $f:A\to B$, we suggestively denote the functor
\[ \Cat(V)(f,V):\Cat(V)(B,V)\To\Cat(A,V) \]
by
\[ \func{-\circ f}{\Cat(V)(B,V)}{\Cat(V)(A,V)}{F}{F\circ f} \]
\end{notation}

There is one extra piece of functoriality for internal presheaves, induced by functors $V\to W$ which preserve pullbacks.

\begin{proposition}\label{proposition:transfer_internal_presheaves}
Let $V$, $W$ be categories with pullbacks, and $F:V\to W$ a functor which preserves all pullbacks.\\
For each category object $A$ in $V$, there exists a functor
\[ \Cat(F):\Cat(V)(A,V)\To\Cat(W)\big(\Cat(F)(A),W\big) \]
which associates to an internal $V$-valued functor $(P,p_0,p_1)$ on $A$, the $W$-valued functor on $\Cat(F)(A)$ (see proposition \sref{proposition:transfer_cat})
\[ \Cat(F)(P,p_0,p_1)\defeq(F(P),F(p_0),F(p_1)) \]
\end{proposition}

\begin{remark}
The functors in the proposition are actually natural in $A$, defining a natural transformation
\[ \Cat(F):\Cat(V)(-,V)\To\Cat(W)\big(\Cat(F)(-),W\big) \]
between functors $\op{\Cat(V)}\to\bigCAT$
\end{remark}

\section{Relation between external presheaves and internal presheaves}\label{section:internal_enriched_presheaves}

We fix, throughout this section, a full subcategory, $\setuniverse$, of $\SET$ closed under taking subsets, and finite limits in $\SET$.

\begin{definition}[enriched presheaves from internal presheaves]
Assume $V$ is a finitely complete, cartesian closed category, with internal morphism objects given by \[\hom_V(-,-):\op{V}\times V\To V \]
Let $A$ be an internal category in $V$.\\
Given an internal $V$-valued functor on $A$, $F=(P,p_0,p_1)$, we define the {\em discretized} $(V,\times,1)$-functor
\[ \disccat{F}:\disccat{A}\To V \]
\begin{enum}
\item given $x\in\ob\big(\disccat{A}\big)=V(1,\ob A)$, let
\[ \disccat{F}(x)\defeq\lim\big(1\xto{x}\ob A\xleftarrow{p_0}P\big) \]
\item for any $x,y\in\ob\big(\disccat{A}\big)$, the map
\[ \disccat{F}:\disccat{A}(x,y)\To\hom_V\big(\disccat{F}(x),\disccat{F}(y)\big) \]
is adjoint to the unique morphism
\[ \disccat{F}:\disccat{F}(x)\times\disccat{A}(x,y)\To\disccat{F}(x) \]
for which the diagram
\begin{diagram}[midshaft,h=2.3em]
\disccat{F}(x)\times\disccat{A}(x,y)&\rTo{\disccat{F}}&\disccat{F}(x)\\
\dTo_{\proj\times\proj}&&\dTo_{\proj}\\
\lim\big(P\xto{\smash{p_0}}\ob A\xleftarrow{\smash{s}}\mor A\big)&\rTo{p_1}&P
\end{diagram}
commutes.
\end{enum}
\end{definition}

\begin{remark}
Assume now $V$ is a finitely complete, cartesian closed $\setuniverse$-category, and $A$ is an internal category in $V$.\\
The above construction extends to a functor (recall observation \sref{remark:disccat_internal_categories})
\[ \disccat{(-)}:\Cat(V)(A,V)\To V\dash\Cat_{\setuniverse}\big(\disccat{A},V\big) \]
where $V$ is viewed as a cartesian monoidal category.\\
Furthermore, this functor is natural in $A$, with respect to the functor from proposition \sref{proposition:functor_Cat(V)(-,V)}, $\Cat(V)(-,V)$.
\end{remark}

\begin{definition}[internal presheaves from enriched presheaves]
Let $V$ be a finitely complete, cartesian closed category with disjoint $\setuniverse$-coproducts.\\
Let $A$ be a $(V,\times,1)$-category whose set of objects is in $\setuniverse$, and \[ F:A\To V \] a $(V,\times,1)$-functor.\\
We define the {\em internalization of \(F\)}, to be the internal $V$-valued functor on $\internal A$
\[ \internal F=(P,p_0,p_1):\internal A\To V \]
\begin{enum}
\item $P$ is defined as
\[ P\defeq\ \coprod_{\mathclap{x\in\ob A}}\ F(x) \]
\item $p_0$ is the morphism (where for any $y\in V$, $!\in V(y,1)$ is the unique element)
\[ \coprod_{\mathclap{x\in\ob A}}\ !:\ \coprod_{\mathclap{x\in\ob A}}\ F(X)\To\ \coprod_{\mathclap{x\in\ob A}}\ 1=\ob(\internal A) \]
\item $p_1$ is the unique map for which
\begin{diagram}[midshaft]
\coprod_{\mathclap{x,y\in\ob A}}\ F(x)\times A(x,y)&\rTo{F}&\coprod_{\mathclap{y\in\ob A}}\ F(y)\\
\dTo[rightshortfall=-.6em]{\rotc{90}{$\scriptstyle\simeq$}}&&\dEqual\\
\lim\big(P\xto{\smash{p_0}}\ob(\internal A)\xleftarrow{s}\mor(\internal A)\big)&\rTo{p_1}&P
\end{diagram}
commutes, where the top map is determined by the functor $F$ on morphisms, and the left vertical map is the natural map from the coproduct (which is an isomorphism because $V$ has disjoint $\setuniverse$-coproducts).
\end{enum}
\end{definition}

\begin{remark}
Let $V$ be a finitely complete cartesian closed category with disjoint $\setuniverse$-coproducts.\\
The above construction extends to a functor
\[ \internal:V\dash\Cat_{\setuniverse}(A,V)\To\Cat(V)(\internal A, V) \]
where $V$ is viewed as a cartesian monoidal category.\\
Additionally, this functor is natural with respect to $A$ (recall proposition \sref{proposition:functor_Cat(V)(-,V)} and observation \sref{remark:functoriality_internal_cat}).
\end{remark}

\begin{proposition}[correspondence between enriched and internal presheaves]
Let $V$ be a finitely complete, cartesian closed $\setuniverse$-category with disjoint $\setuniverse$-coproducts.\\
If $A$ is a $V$-category such that $\ob A$ is in $\setuniverse$, then the composition
\begin{align*}
V\dash\Cat_{\setuniverse}(A,V)&\xTo{\internal}\Cat(V)(\internal A,V)\\
&\xTo{\disccat{(-)}}V\dash\Cat_{\setuniverse}\big(\disccat{(\internal A)},V\big)\\
&\xTo{V\dash\Cat_{\setuniverse}(\Gamma^V_A,V)}V\dash\Cat_{\setuniverse}(A,V)
\end{align*}
is the identity functor (where $\Gamma^V$ appearing in the last arrow is the natural transformation from proposition \sref{proposition:disccat_internal=id}).
\end{proposition}

\begin{proposition}[correspondence between enriched and internal presheaves]
Let $V$ be a finitely complete category with strongly disjoint $\setuniverse$-coproducts.\\
If $A$ is a $V$-category such that $\ob A$ is in $\setuniverse$, the functor
\[ \internal:V\dash\Cat_\setuniverse(A,V)\To\Cat(V)(A,V) \]
is full and faithful. Moreover, it is an equivalence of categories if $V$ has totally disjoint $\setuniverse$-coproducts.
\end{proposition}

\section{Internal presheaves of categories}\label{section:internal_presheaves_categories}

In section \sref{section:internal_presheaves} we discussed $V$-valued functors on an category object in $V$. However, our ultimate goal is to define Grothendieck constructions of functors with values in categories. With that in mind, we now define the concept of $\Cat(V)$-valued functors on internal categories. For that purpose, recall the definition of the discrete internal category from example \sref{example:discrete_cat}.

\begin{definition}[internal $\Cat(V)$-valued functor]\label{definition:Cat(V)-valued_functor}
Let $V$ be a category with pullbacks, and $A$ an internal category in $V$.\\
An {\em internal \(\Cat(V)\)-valued functor on \(A\)} \[ F:A\To\Cat(V) \] is a triple $F=(P,p_0,p_1)$ where
\begin{enum}
\item $P$ is an object of $\Cat(V)$;
\item $p_0:P\to\disc(\ob A)$ is a morphism in $\Cat(V)$;
\item $p_1$ is a morphism in $\Cat(V)$
\[ p_1:\lim\big(P\xto{p_0}\disc(\ob A)\xleftarrow{\disc(s)}\disc(\mor A)\big)\To P \]
\end{enum}
These data are required to make the three diagrams
\begin{diagram}[midshaft,h=2.3em]
\hspace{4em}\lim\big(P\xto{p_0}\disc(\ob A)\xleftarrow{\!\disc(s)\!}\disc(\mor A)\big)&\rTo{p_1}&P\\
\dTo{\proj}&&\dTo{p_0}\\
\disc(\mor A)&\rTo{\disc(t)}&\disc(\ob A)&\qquad\quad
\end{diagram}
\begin{diagram}[midshaft,hug]
\lim\big(\disc(P)\xto{p_0}\disc(\ob A)\xleftarrow{\id}\disc(\ob A)\big)\\
\dTo{\id_P\underset{\ \ \mathclap{\disc(\ob A)}\ \ }{\times}\disc(i)}&\rdTo{\simeq}\\
\lim\big(P\xto{\smash{p_0}}\disc(\ob A)\xleftarrow{\smash{\!\disc(s)\!}}\disc(\mor A)\big)&\rTo_{p_1}&P
\end{diagram}
\begin{diagram}[midshaft,objectstyle=\scriptstyle,labelstyle=\scriptscriptstyle,balance]
&&\mathclap{\lim\big(P\xto{p_0}\disc(\ob A)\xleftarrow{\!\!\disc(s)\!\!}\disc(\mor A)\xto{\!\!\disc(t)\!\!}\disc(\ob A)\xleftarrow{\!\!\disc(s)\!\!}\disc(\mor A)\big)}\qquad\quad\\
&\ldTo[rightshortfall=.7em,leftshortfall=.5em]{\mathllap{P\underset{\ \ \mathclap{\disc(\ob A)}\ \ }{\times}\disc(c)}}&&\rdTo[leftshortfall=.7em]{\ \ \mathrlap{p_1\underset{\ \ \mathclap{\disc(\ob A)}\ \ }{\times}\disc(\mor A)}}\\
\qquad\lim\big(P\xto{\smash{p_0}}\disc(\ob A)\xleftarrow{\smash{\!\!\disc(s)\!\!}}\disc(\mor A)\big)&&&&\lim\big(P\xto{p_0}\ob A\xleftarrow{\smash{\!\!\disc(s)\!\!}}\disc(\mor A)\big)\\
&\rdTo_{p_1}&&\ldTo_{p_1}\\
&&P
\end{diagram}
commute.
\end{definition}

This is a very long definition, so now we restate it in a much more compact form.

\begin{remark}[restatement of definition]\label{remark:restatement_definition_Cat(V)-valued_functor}
This definition is a copy of the definition \sref{definition:internal_V-valued_functor} of an internal $V$-valued functor on $A$: we have only replaced $\ob A$, $\mor A$, and $V$ by $\disc(\ob A)$, $\disc(\mor A)$, and $\Cat(V)$, respectively.\\
In other words, an internal $\Cat(V)$-valued functor on an internal category $A$ in $V$
\[ F:A\To\Cat(V) \]
is exactly the same as a $\Cat(V)$-valued functor on the category object $\disc^t A$ in $\Cat(V)$
\[ F:\disc^t A\To\Cat(V) \]
\ie an object of the already defined category \[ F\in\Cat\big(\Cat(V)\big)\big(\disc^t A,\Cat(V)\big) \]
Here, we define
\[ \disc^t\defeq\Cat(\disc):\Cat(V)\To\Cat\big(\Cat(V)\big) \]
(where $\Cat(\disc)$ is the functor from proposition \sref{proposition:transfer_cat}), or more explicitly
\begin{align*}
\ob(\disc^t A)&=\disc(\ob A)\\
\mor(\disc^t A)&=\disc(\mor A)
\end{align*}
so $\disc^t A$ is not the discrete (``constant'') category object $\disc A$ in $\Cat(V)$, but a ``transposed'' version of it (which is not discrete in general).\\
In conclusion, we have almost reduced the above definition \sref{definition:Cat(V)-valued_functor} to a particular case of the definition \sref{definition:internal_V-valued_functor} of internal presheaf. There is only one oversight in this discussion: the category $\Cat(V)$ is a $2$-category, not a $1$-category. Therefore, the definitions of internal category and internal presheaf do not immediately apply to $\Cat(V)$. One way to rectify this (which is sufficient for our needs) is to consider instead the underlying $1$-category $\Cat(V)^{\leq 1}$ of $\Cat(V)$ (forget the $2$-morphisms). A more satisfactory solution is to generalize the preceding definitions to the case of internal categories in $2$-categories, thus allowing for the case of $\Cat(V)$. We will adopt the first solution, but the reader should note that it is possible to remove all the superscripts ``${\leq 1}$'' in what follows by generalizing our definitions to the case of $2$-categories.
\end{remark}

\begin{remark}[transposition in $\Cat\big(\Cat(V)\big)$]
The definition of $\disc^t$ in the previous remark reflects a fundamental symmetry in $\Cat(\Cat(V))$ which, informally, switches the two symbols ``$\Cat$''.\\ More precisely, there is a {\em transposition functor} (which is an isomorphism of categories)
\[ (-)^t:\Cat\big(\Cat(V)^{\leq 1}\big)^{\leq 1}\To\Cat\big(\Cat(V)^{\leq 1}\big)^{\leq 1} \]
that generalizes the transposition of double categories: small double categories are the same as objects of $\Cat(\Cat(\Set))$.\\
This transposition functor verifies a commutative diagram
\begin{diagram}
\Cat(V)^{\leq 1}&\rTo{\ \disc\ }&\Cat\big(\Cat(V)^{\leq 1}\big)^{\leq 1}\\
&\rdTo_{\mathllap{\disc^t}}&\dTo_{(-)^t}\\
&&\Cat\big(\Cat(V)^{\leq 1}\big)^{\leq 1}
\end{diagram}
\end{remark}

\begin{definition}[$\ob$ and $\mor$ of $\Cat(V)$-valued presheaf]\label{definition:ob_mor_Cat(V)-calued_functor}
Assume $V$ is a category with pullbacks. Let \[ F=(P,p_0,p_1):A\To\Cat(V) \] be a $\Cat(V)$-valued functor on the category object $A$ in $V$.\\
We define the internal $V$-valued functors on $A$
\begin{align*}
\ob F&\defeq(\ob P,\ob p_0,\ob p_1)\\
\mor F&\defeq(\mor P,\mor p_0,\mor p_1)
\end{align*}
\end{definition}

Using the restatement \sref{remark:restatement_definition_Cat(V)-valued_functor} of the definition, the discussion in section \sref{section:internal_presheaves} applies immediately to internal $\Cat(V)$-valued functors on category objects in $V$.

\begin{definition}[internal $\Cat(V)$-valued natural transformation]
Let $V$ be a category with pullbacks, $A$ an internal category in $V$.\\
Additionally, assume $F,G:A\to\Cat(V)$ are internal $\Cat(V)$-valued functors.\\
An {\em internal \(\Cat(V)\)-valued natural transformation on \(A\)} \[ \alpha:F\to G \] is defined to be an internal natural transformation between the $\Cat(V)$-valued functors $F,G:\disc^t A\to\Cat(V)$ (on the internal category $\disc^t A$ in $\Cat(V)^{\leq 1}$).
\end{definition}

\begin{definition}[category of internal $\Cat(V)$-valued functors]
Let $V$ be a category with pullbacks, and $A$ an internal category in $V$.\\
The {\em category of internal \(\Cat(V)\)-valued functors on \(A\)}, $\Cat(V)(A,\Cat(V))$, is defined to be
\[ \Cat(V)\big(A,\Cat(V)\big)\defeq\Cat\big(\Cat(V)^{\leq 1}\big)\big(\disc^t A,\Cat(V)^{\leq 1}\big) \]
\end{definition}

\begin{remark}
One can easily extend the category $\Cat(V)(A,\Cat(V))$ to a $2$-category by adding natural transformations between internal categories over $\disc(\ob A)$.\\
This comes essentially for free if we generalize our definitions to allow for internal categories in $2$-categories. We could then drop the superscripts ``$\leq 1$'' in the above definition.
\end{remark}

\begin{proposition}[functoriality of internal presheaves of categories]\label{proposition:functor_Cat(V)(-,V)}
Let $V$ be a category (in $\bigCAT$) with pullbacks.\\
There is a functor
\[ \Cat(V)\big(-,\Cat(V)\big):\op{\Cat(V)}\To\bigCAT \]
which associates to an internal category in $V$, $A$, the category of internal $\Cat(V)$-valued functors on $A$, $\Cat(V)\big(A,\Cat(V)\big)$.
\end{proposition}

\begin{notation}\label{notation:composition_internal_Cat(V)-valued_functor}
Analogously to notation \sref{notation:composition_internal_V-valued_functor}, given an internal functor $f:A\to B$, we suggestively denote the functor $\Cat(V)\big(f,\Cat(V)\big)$ by
\[ \func{-\circ f}{\Cat(V)\big(B,\Cat(V)\big)}{\Cat(V)\big(A,\Cat(V)\big)}{F}{F\circ f} \]
\end{notation}

\begin{example}[discrete $\Cat(V)$-valued presheaves]
Assume $V$ is a category with pullbacks, and $A$ is a category object in $V$.\\
Given a $V$-valued internal functor on $A$ \[ F:A\To V \] with $F=(P,p_0,p_1)$, there is an associated {\em discrete \(\Cat(V)\)-valued internal functor on $A$}
\[ \disc(F):A\To\Cat(V) \]
with $\disc(F)\defeq(\disc(P),\disc(p_0),\disc(p_1))$.\\
This extends to a functor
\[ \disc:\Cat(V)(A,V)\To\Cat(V)\big(A,\Cat(V)\big) \]
which is natural in $A$.
\end{example}

\begin{example}[usual functors into $\Cat$]
Recall from example \sref{example:internal_set_categories} that an internal category in $\Set$ is the same as a small category.\\
An internal $\Cat(\Set)$-valued functor on a small category $A$ is then the same as an ordinary functor $A\to\Cat$.
\end{example}

We finish this section by stating a result on transfer of internal category-valued functors across functors $F:V\to W$ which preserve pullbacks. It is a consequence of propositions \sref{proposition:transfer_internal_presheaves} and \sref{proposition:transfer_cat}.

\begin{proposition}\label{proposition:transfer_internal_presheaves_categories}
Let $V$, $W$ be categories with pullbacks, and $F:V\to W$ a functor which preserves all pullbacks.\\
For each category object $A$ in $V$, there exists a functor (recall proposition \sref{proposition:transfer_cat})
\[ \Cat(F):\Cat(V)\big(A,\Cat(V)\big)\To\Cat(W)\big(\Cat(F)(A),\Cat(W)\big) \]
which associates to an internal $\Cat(V)$-valued functor $(P,p_0,p_1)$ on $A$, the $\Cat(W)$-valued functor on $\Cat(F)(A)$
\[ \Cat(F)(P,p_0,p_1)\defeq\big(\Cat(F)(P),\Cat(F)(p_0),\Cat(F)(p_1)\big) \]
\end{proposition}

\begin{remark}
The functors in the proposition are actually natural in $A$, defining a natural transformation
\[ \Cat(F):\Cat(V)\big(-,\Cat(V)\big)\To\Cat(W)\big(\Cat(F)(-),\Cat(W)\big) \]
between functors $\op{\Cat(V)}\to\bigCAT$
\end{remark}

\section{Grothendieck construction}\label{section:Grothendieck_construction}

The last section introduced $\Cat(V)$-valued internal functors. With these, we can define a sufficiently general Grothendieck construction for our purposes. Recall the definition of the opposite of an internal category from observation \sref{remark:opposite_internal_category}.

\begin{definition}[Grothendieck construction]
Assume $V$ is a category with pullbacks, $A$ is a category object in $V$, and \[ F=(P,p_0,p_1):\op{A}\To\Cat(V) \] is an internal $\Cat(V)$-valued functor on $\op{A}$.\\
The {\em Grothendieck construction of \(F\)}, $\Groth(F)$, is the internal category in $V$ defined by:
\begin{enum}
\item the objects are $\ob\big(\Groth(F)\big)\defeq\ob P$
\item the object of morphisms $\mor\big(\Groth(F)\big)$ is the limit of
\[ \mor P\xTo{t}\ob P\xlongleftarrow{\ob p_1}\lim\big(\mor A\xto{\;t\;}\ob A\xleftarrow{\ob p_0}\ob P\big) \]
\item the source for $\Groth(F)$ \[ s:\mor\big(\Groth(F)\big)\To\ob\big(\Groth(F)\big) \] is given by the composition
\[ \mor\big(\Groth(F)\big)\xTo{\proj}\mor P\xTo{s}\ob P \]
\item the target for $\Groth(F)$ \[ t:\mor\big(\Groth(F)\big)\To\ob\big(\Groth(F)\big) \] is given by the composition
\[ \mor\big(\Groth(F)\big)\xTo{\proj}\lim\big(\mor A\xto{\;t\;}\ob A\xleftarrow{\ob p_0}\ob P\big)\xTo{\proj}\ob P \]
\item the identity for $\Groth(F)$ \[ i:\ob P=\ob\big(\Groth(F)\big)\To\mor\big(\Groth(F)\big) \] is the unique morphism such that
\begin{diagram}
\lim\big({\scriptstyle\ob A\xto{\;t\;}\ob A\xleftarrow{\ob p_0}\ob P}\big)&\rTo_{\ i\underset{\!\ob A\!}{\times}\mor A\ }&\lim\big({\scriptstyle\mor A\xto{\;t\;}\ob A\xleftarrow{\ob p_0}\ob P}\big)\\
\uTo_{\rotc{-90}{$\scriptstyle\simeq$}}&&\uTo_{\proj}\\
\ob P&\rTo{i}&\mor\big(\Groth(F)\big)\\
&\rdTo_{i}&\dTo_{\proj}\\
&&\mor P
\end{diagram}
commutes.
\item the composition for $\Groth(F)$
\[ c:\lim\big({\scriptstyle\mor(\Groth(F))\xto{t}\ob P\xleftarrow{s}\mor(\Groth(F))}\big)\To\mor\big(\Groth(F)\big) \]
is the unique morphism such that
\begin{diagram}[midshaft]
\lim\!\big({\scriptstyle\mor(\Groth(F))\xto{\;t\;}\ob P\xleftarrow{\;s\;}\mor(\Groth(F))}\big)&\rTo{\ \ c\ \ }&\mor\big(\Groth(F)\big)\\
\dTo{\id\underset{\!\ob P\!}{\times}\proj}&&\dTo{\proj}\\
\lim\big({\scriptstyle\mor(\Groth(F))\xto{\smash{\;t\;}}\ob P\xleftarrow{\smash{\;s\;}}\mor P}\big)&&\mor P\\
\raisebox{.7em}{\huge{\rotc{90}{$\simeq$}}}&&\uTo[lowershortfall=0em]{c}\\
\lim\!\left(
\begin{diagram}[h=2em,objectstyle=\scriptstyle,labelstyle=\scriptscriptstyle]
&&\hspace{-3em}\lim\big(\mor A\xto{t}\ob A\xleftarrow{\!\!\mor p_0\!\!}\mor P\big)\\
&&\dTo{s\circ(\mor p_1)}\\
\mor P&\rTo_{t}&\ob P
\end{diagram}
\right)&\rTo_{\id\underset{\id}{\times}\mor p_1}&\lim\!\left(
\begin{diagram}[w=1.2em,h=1.5em,objectstyle=\scriptstyle,labelstyle=\scriptscriptstyle]
&&\mor P\\
&&\dTo{s}\\
\mor P&\rTo{t}&\ob P
\end{diagram}
\right)
\end{diagram}
and
\begin{diagram}[midshaft]
\lim\!\big({\scriptstyle\mor(\Groth(F))\xto{\;t\;}\ob P\xleftarrow{\;s\;}\mor(\Groth(F))}\big)&\rTo{c}&\mor\big(\Groth(F)\big)\\
\dTo{\proj}&&\dTo{\proj}\\
\lim\!\big(\smash{\scriptstyle\mor A\xto{t}\ob A\xleftarrow{s}\mor A\xto{t}\ob A\xleftarrow{\ob p_0}\ob P}\big)&\rTo_{c\underset{\!\ob A\!}{\times}\ob P}&\lim\!\big(\smash{\scriptstyle\mor A\xto{t}\ob A\xleftarrow{\ob p_0}\ob P}\big)
\end{diagram}
both commute.
\end{enum}
\end{definition}

\begin{remark}
We have defined the Grothendieck construction of $\Cat(V)$ valued functors on $\op{A}$. We chose $\op{A}$ because that is the case which will appear in our applications.
\end{remark}

We will leave it to the reader to verify the claims in the following propositions.

\begin{proposition}\label{proposition:projection_Groth_construction}
Assume $V$ is a category with pullbacks, $A$ is a category object in $V$, and $F:\op{A}\to\Cat(V)$ is an internal $\Cat(V)$-valued functor on $\op{A}$.\\
Then there is a natural internal functor
\[ \pi:\Groth(F)\To A \]
\end{proposition}

\begin{proposition}
Let $V$ be a category with pullbacks, and $A$ a category object in $V$.\\
There is a natural functor
\[ \Groth:\Cat(V)\big(\op{A},\Cat(V)\big)\To\Cat(V)/A \]
which associates to each internal $\Cat(V)$-valued functor $F:\op{A}\to\Cat(V)$, the morphism
\[ \pi:\Groth(F)\To A \]
\end{proposition}

\begin{notation}
We will, for simplicity, also denote by $\Groth$ the composition
\[ \Cat(V)\big(\op{A},\Cat(V)\big)\xTo{\Groth}\Cat(V)/A\xTo{\proj}\Cat(V) \]
\end{notation}

\begin{proposition}[naturality of Grothendieck construction on base category]\label{proposition:Groth(functorA->B)}
Let $V$ be a category with pullbacks.\\
If $f:A\to B$ is an internal functor between internal categories in $V$, there is a canonical natural transformation
\begin{diagram}
\Cat(V)\big(\op{A},\Cat(V)\big)&\lTo{\ -\circ\op{f}\ }&\Cat(V)\big(\op{B},Cat(V)\big)\\
\dTo{\Groth}&\twocell[-.3em]{\scriptscriptstyle\Groth(f)}{.25em}{\Longrightarrow}{0}&\dTo{\Groth}\\
\Cat(V)/A&\rTo_{f\circ-}&\Cat(V)/B
\end{diagram}
Moreover, these natural transformations compose in the obvious way (when one places two of these diagrams side by side).
\end{proposition}

\begin{construction}\label{construction:Groth(f,alpha)}
In particular, given \[ f:A\To B \] an internal functor, we have a natural transformation
\begin{diagram}
\Cat(V)\big(\op{A},\Cat(V)\big)&\lTo{\ -\circ\op{f}\ }&\Cat(V)\big(\op{B},Cat(V)\big)\\
&\rdTo(1,2)_{\mathllap{\Groth}}\twocell[.2em]{\scriptscriptstyle\Groth(f)}{.25em}{\Longrightarrow}{-15}\ldTo(1,2)_{\mathrlap{\Groth}}\\
&\Cat(V)
\end{diagram}
So if
\begin{align*}
F&:\op{A}\To\Cat(V)\\
G&:\op{B}\To\Cat(V)
\end{align*}
are internal $\Cat(V)$-valued functors, and
\[ \alpha:F\To G\circ\op{f} \]
is an internal $\Cat(V)$-valued natural transformation, we have a canonical internal functor
\[ \Groth(f,\alpha):\Groth(F)\To\Groth(G) \]
given by the composition
\[ \Groth(F)\xTo{\Groth(\alpha)}\Groth(G\circ\op{f})\xTo{\Groth(f)}\Groth(G) \]
Moreover, the diagram
\begin{diagram}[h=2.2em]
\Groth(F)&\rTo{\ \Groth(f,\alpha)\ }&\Groth(G)\\
\dTo{\pi}&&\dTo{\pi}\\
A&\rTo{f}&B
\end{diagram}
commutes.
\end{construction}

The last result of this section shows that Grothendieck constructions are compatible with transfer along functors $V\to W$ which preserve pullbacks.

\begin{proposition}\label{proposition:Groth_transfer_categories}
Let $V$, $W$ be categories with pullbacks, and $F:V\to W$ a functor which preserves all pullbacks.\\
For each internal category $A$ in $V$, the following diagram commutes up to canonical natural isomorphism
\begin{diagram}[midshaft]
\Cat(V)\big(\op{A},\Cat(V)\big)&\rTo{\Groth}&\Cat(V)/A\\
\dTo^{\Cat(F)}_{\sref{proposition:transfer_internal_presheaves_categories}}&&\dTo{\Cat(F)}\\
\Cat(W)\big(\op{\Cat(F)(A)},\Cat(W)\big)&\rTo{\ \Groth\ }&\Cat(W)/\big(\Cat(F)(A)\big)
\end{diagram}
In particular, for each internal $\Cat(V)$-valued functor $f:\op{A}\to\Cat(V)$, there is a canonical isomorphism
\[ \Cat(F)\big(\Groth(f)\big)=\Groth\big(\Cat(F)(f)\big) \]
\end{proposition}

\section{Variation on Grothendieck construction}\label{section:variation_Groth_Top}

We will now deal with the case of topological categories. First we define a few variations of the Grothendieck construction.

\begin{construction}
If $A$ is a category in $\Set\dash\CAT$, and
\[ F:\op{A}\To\Cat(\Top) \]
there is a canonical associated internal $\Cat(\TOP)$-functor
\[ \internal F:\op{\internal A}\To\Cat(\TOP) \]
where $\internal A$ is the category internal to $\TOP$ associated with the $\Top$-category $C$. This internal functor is {\em not} an instance of our previous constructions. Instead, $\internal F\defeq(P,p_0,p_1)$ is such that
\[ P\defeq\coprod_{\!a\in\ob A\!}F(a) \]
and \[ p_0:P\To\disc(\ob A) \] is the canonical projection, and $p_1$ is obtained from the functoriality of $F$.\\
Thus we get a canonical (full faithful) inclusion
\[ \HomCat{\op{A}}{\Cat(\Top)}\Into\Cat(\TOP)\big(\op{\internal A},\Cat(\TOP)\big) \]
which is natural in $A$.\\
Applying the Grothendieck construction, we obtain
{\small
\[ \HomCat{\op{A}}{\Cat(\Top)}\Into\Cat(\TOP)\big(\op{\internal A},\Cat(\TOP)\big)\xTo{\Groth}\Cat(\TOP)/(\internal A) \]
}which we call
\[ \Groth:\HomCat{\op{A}}{\Cat(\Top)}\To\Cat(\TOP)/(\internal A) \]
\end{construction}

\begin{construction}[variation of Grothendieck construction]\label{construction:variation_Groth_construction}
We will only apply the previous construction to functors
\[ F:\op{A}\To\Top\dash\Cat \]
thus we define a new Grothendieck construction
\[ \Groth:\HomCat{\op{A}}{\Top\dash\Cat}\To\Top\dash\CAT/A \]
by the composite
\begin{align*}
\HomCat{\op{A}}{\Top\dash\Cat}&\xTo{\HomCat{\op{A}}{\internal}}\HomCat{\op{A}}{\Cat(\Top)}\\
&\xTo{\Groth}\Cat(\TOP)/(\internal A)\\
&\xTo{\disccat{(-)}}\TOP\dash\CAT/\disccat{(\internal A)}\\
&\xTo[\sref{proposition:disccat_internal=id}]{(\Gamma_A)^{-1}\circ-}\TOP\dash\CAT/A
\end{align*}
which can be seen to factor through $\Top\dash\Cat/A$.
\end{construction}

\begin{notation}
As before, we will also denote by $\Groth$ the functor
\[ \HomCat{\op{A}}{\Top\dash\Cat}\xTo{\Groth}\Top\dash\CAT/A\xTo{\proj}\Top\dash\CAT \]
\end{notation}

We leave here a description of the categories obtained through this construction.

\begin{proposition}[description of the Grothendieck construction]\label{proposition:description_Groth_discrete}
Let $A$ be a category in $\Set\dash\CAT$, and $F:\op{A}\to\Top\dash\Cat$ a functor.\\
Then the Grothendieck construction of $F$ verifies:
\begin{enum}
\item the set of objects of $\Groth(F)$ is
\[ \ob\big(\Groth(F)\big)=\coprod_{x\in\ob A}\ob\big(F(x)\big) \]
\item given $x,y\in\ob A$, $a\in\ob F(x)$, and $b\in\ob F(y)$, we have
\[ \Groth(F)(a,b)=\ \coprod_{\mathclap{f\in A(x,y)}}\ F(x)\big(a,F(f)(b)\big) \]
\end{enum}
\end{proposition}

\begin{remark}
This description is natural with respect to $A$ and $F$.
\end{remark}

\begin{construction}[naturality of variation of Grothendieck construction]\label{construction:Groth(f,alpha)_variation}
Given a functor $f:A\to B$, we have a natural transformation
\begin{diagram}[midshaft]
\HomCat{\op{A}}{\Top\dash\Cat}&\lTo{\ -\circ\op{f}\ }&\HomCat{\op{B}}{\Top\dash\Cat}\\
\dTo{\Groth}&\twocell[-.3em]{\scriptscriptstyle\Groth(f)}{.25em}{\Longrightarrow}{0}&\dTo_{\Groth}\\
\Top\dash\CAT/A&\rTo_{f\circ-}&\Top\dash\CAT/B
\end{diagram}
obtained from proposition \sref{proposition:Groth(functorA->B)}.\\
Consequently, given functors
\begin{align*}
F&:\op{A}\To\Top\dash\Cat\\
G&:\op{B}\To\Top\dash\Cat
\end{align*}
and a natural transformation
\[ \alpha:F\To G\circ\op{f} \]
we have an induced $\Top$-functor
\[ \Groth(f,\alpha):\Groth(F)\To\Groth(G) \]
given by the composition
\[ \Groth(F)\xTo{\Groth(\alpha)}\Groth(G\circ\op{f})\xTo{\Groth(f)}\Groth(G) \]
This is analogous to construction \sref{construction:Groth(f,alpha)}, and indeed can be recovered from it.
\end{construction}

\section{Homotopical properties of Grothendieck construction}\label{section:homotopical_properties_Groth}

We now change direction and turn to internal presheaves of categories on native category objects in $\Top$. Our main example of $\Cat(\Top)$-valued internal presheaves are obtained by taking the path category of a $\Top$-valued internal presheaf, which we now proceed to describe. It involves defining a fibrewise version of the path category from example \sref{example:path_category}.

\begin{definition}[path category presheaf]\label{definition:path_category_presheaf}
Let $A$ be a category internal to $\Top$, and $F=(P,p_0,p_1):A\To\Top$ an internal $\Top$-valued functor.\\
We define the {\em path category of \(F\)} to be the $\Cat(\Top)$-valued internal functor
\[ \pathcat\circ F:A\To\Cat(\Top) \]
to be the triple $(Q,q_0,q_1)$ determined by:
\begin{enum}
\item the objects of $Q$ are $\ob Q\defeq P$;
\item $\mor Q$ is the subspace of $H(P)$ defined by
\[ \mor Q\defeq\set{(\gamma,\tau)\in H(P):p_0\circ\gamma\text{ is constant}} \]
\item the source $s:\mor Q\to\ob Q$ is the composition
\[ \mor Q\Into H(P)\xTo{s}P \]
\item the target $t:\mor Q\to\ob Q$ is the composition
\[ \mor Q\Into H(P)\xTo{t}P \]
\item the identity $i:\ob Q\to\mor Q$ makes the square
\begin{diagram}[midshaft,h=2.1em]
\ob Q&\rTo{i}&\mor Q\\
\dEqual&&\dInto{\inclusion}\\
P&\rTo{i}&H(P)
\end{diagram}
commute.
\item the composition is given by concatenation of paths (denoted $\concat$), that is
\begin{diagram}
\lim\big(\mor Q\xto{t}\ob Q\xleftarrow{s}\mor Q\big)&\rTo{c}&\mor Q\\
\dInto{\inclusion\underset{P}{\times}\inclusion}&&\dInto{\inclusion}\\
\lim\big(H(P)\xto{t}P\xleftarrow{s}H(Q)\big)&\rTo{\concat}& H(P)\\
\end{diagram}
commutes.
\item the map $\ob q_0$ is just $p_0$;
\item the map $\mor q_0$ is the composition
\[ \mor Q\Into H(P)\xTo{s} P\xTo{p_0}\ob A \]
\item the map $\ob q_1$ is just $p_1$;
\item the map $\mor q_1$ is defined by (where the pullback is the appropriate one)
\[ \func{\mor q_1}{\mor Q\underset{\!\ob A\!}{\times}\mor A}{\mor Q}{\big((\gamma,\tau),f\big)}{\big(p_1(\gamma(-),f),\tau\big)} \]
\end{enum}
\end{definition}

\comment{This can be generalized to fibrewise functors.}

\begin{proposition}[functoriality of path category presheaf]
Let $A$ be a category object in $\Top$.\\
There is a functor
\[ \pathcat\circ -:\Cat(\Top)(A,\Top)\To\Cat(\Top)\big(A,\Cat(\Top)\big) \]
which associates to a $\Top$-valued internal functor $F:A\to\Top$ the path category of $F$, $\pathcat\circ F$.
\end{proposition}

\begin{proposition}[naturality of path category presheaf]\label{proposition:naturality_internal_path_category_presheaf}
Let $f:A\to B$ be a morphism in $\Cat(\Top)$.
For each internal $\Top$-valued functor
\[ F:B\To\Top \]
we have a natural isomorphism
\[ (\pathcat\circ F)\circ f=\pathcat\circ(F\circ f) \]
\end{proposition}

\begin{remark}
As a consequence of proposition \sref{proposition:functor_Cat(V)(-,V)}, we conclude that $\pathcat\circ -$ extends to a natural transformation
\[ \pathcat\circ -:\Cat(\Top)(-,\Top)\To\Cat(\Top)\big(-,\Cat(\Top)\big) \]
between functors $\op{\Cat(\Top)}\to\Set\dash\CAT$.
\end{remark}

We will finish this chapter with a few homotopical properties of the categories obtained by taking the Grothendieck construction of a $\Cat(\Top)$-valued presheaf.

\begin{definition}[fibrant internal category in $\Top$]\label{definition:fibrant_internal_category}
Let $A$ be an internal category in $\Top$.\\
We say $A$ is {\em fibrant} if the map
\[ (s,t):\mor A\To\ob A\times\ob A \]
is a Hurewicz fibration.
\end{definition}

The relevance of this fibrancy condition is explained by the next result.

\begin{proposition}
Let $F:A\to B$ be a morphism in $\Cat(\Top)$.\\
If $A$ and $B$ are fibrant, and the square
\begin{diagram}[h=2.3em]
\mor A&\rTo{\mor F}&\mor B\\
\dTo{(s,t)}&&\dTo{(s,t)}\\
\ob A\times\ob A&\rTo{\ \ob F\times\ob F\ }&\ob B\times\ob B
\end{diagram}
is homotopy cartesian, then the $\Top$-functor
\[ \disccat{F}:\disccat{A}\To\disccat{B} \]
is a local homotopy equivalence.
\end{proposition}

We can now give simple conditions for Grothendieck constructions to be fibrant internal categories in $\Top$.

\begin{proposition}[fibrancy of Grothendieck construction]
Let $A$ be a category object in $\Top$, and \[ F=(P,p_0,p_1):\op{A}\To\Cat(\Top) \] an internal $\Cat(\Top)$-valued functor.\\
Then $\Groth(F)$ is fibrant if the maps
\begin{align*}
(s,t)&:\mor P\To\ob P\times\ob P\\
t&:\mor A\To\ob A
\end{align*}
are Hurewicz fibrations.
\end{proposition}

\begin{corollary}\label{corollary:fibrant_Groth(path)}
Let $A$ be a small $\Top$-category, and $F:\op{\internal A}\To\Top$ an internal $\Top$-valued functor.\\
The category $\Groth(\pathcat\circ F)$ is fibrant.
\end{corollary}

Now we give a description of the morphism spaces in $\disccat{\Groth(\pathcat\circ F)}$.

\begin{definition}[value of internal $\Top$-valued functor at object]\label{definition:value_internal_Top_presheaf}
Let $A$ be an internal category in $\Top$, and \[ F=(P,p_0,p_1):A\To\Top \] an internal $\Top$-valued functor.\\
If $x\in\ob A$, we define the {\em value of \(F\) at \(x\)}, $F(x)$, to be the pullback of
\[ 1\xTo{x}\ob A\xTo{p_0}P \]
\end{definition}

\begin{construction}
The map $p_1$ induces maps
\[ F:\disccat{A}(x,y)\times F(x)\To F(y) \]
which we denote simply by $F$ to analogize with the case of external functors.
\end{construction}

\begin{proposition}\label{proposition:morphism_space_Groth(path)}
Let $A$ be an internal category in $\Top$, and $F:\op{A}\To\Top$ a $\Top$-valued functor.\\
Let $x,y\in\ob A$, $a\in F(x)$, and $b\in F(y)$.\\
The topological space $\disccat{\big(\Groth(\pathcat\circ F)\big)}(a,b)$ is the limit of
\[ \disccat{A}(x,y)\xTo{F(-,b)}F(x)\xlongleftarrow{t}H(F(x))\xTo{s}F(x)\xlongleftarrow{a}1 \]
In particular, there is a canonical homotopy equivalence
\[ \disccat{\big(\Groth(\pathcat\circ F)\big)}(a,b)\xTo{\ \sim\ }\hofibre_{a\!}\big(F(-,b):\disccat{A}(x,y)\To F(x)\big) \]
induced by reparametrization of Moore paths (see \sref{definition:reparametrization_Moore_paths} and \sref{proposition:reparametrization_Moore_paths_equiv}).
\end{proposition}

The next results give conditions under which a functor between two Grothendieck constructions induces a local equivalence on the discretized categories.

\begin{proposition}\label{proposition:local_homotopy_equiv_Groth(path)}
Let $f:A\to B$ be a morphism in $\Cat(\Top)$. Let
\begin{align*}
F&:\op{A}\To\Top\\
G&:\op{B}\To\Top
\end{align*}
be internal $\Top$-valued functors, and $\alpha:F\to G\circ\op{f}$ be an internal natural transformation.\\
The functor (see construction \sref{construction:Groth(f,alpha)})
\[ \disccat{\Groth(f,\pathcat\circ\alpha)}:\disccat{\Groth(\pathcat\circ F)}\To\disccat{\Groth(\pathcat\circ G)} \]
is a local homotopy equivalence if for all $x,y\in\ob A$ and $a\in F(y)$, the square
\begin{diagram}[midshaft]
\disccat{A}(x,y)&\rTo{\ \disccat{f}\ }&\disccat{B}\big(\disccat{f}x,\disccat{f}y\big)\\
\dTo{F(-,a)}&&\dTo{G(-,\alpha a)}\\
F(x)&\rTo{\alpha}&G\big(\disccat{f}x\big)
\end{diagram}
is homotopy cartesian.
\end{proposition}
\begin{proof}[Sketch of proof]
This result follows from the natural homotopy equivalence in proposition \sref{proposition:morphism_space_Groth(path)}.
\end{proof}

\begin{proposition}
Let $f:A\to B$ be a morphism in $\Cat(\Top)$. Also, let
\begin{align*}
F=(P,p_0,p_1)&:\op{A}\To\Cat(\Top)\\
G=(Q,q_0,q_1)&:\op{B}\To\Cat(\Top)
\end{align*}
be internal $\Cat(\Top)$-valued functors, and $\alpha:F\to G\circ\op{f}$ be an internal natural transformation.\\
Assume that the maps
\begin{align*}
t&:\mor P\To\ob P\\
t&:\mor Q\To\ob Q
\end{align*}
are Hurewicz fibrations and homotopy equivalences.\\
The functor (see construction \sref{construction:Groth(f,alpha)})
\[ \disccat{\Groth(f,\alpha)}:\disccat{\Groth(F)}\To\disccat{\Groth(G)} \]
is a local homotopy equivalence if $\Groth(F)$, $\Groth(G)$ are fibrant and the square
\begin{diagram}[midshaft]
\lim\big(\mor A\xto{t}\ob A\xleftarrow{\ob p_0}\ob F\big)&\rTo{\ (\ob p_1,\proj)\ }&\ob P\times\ob P\\
\dTo{\mor f\underset{\!\ob f\!}{\times}\ob\alpha}&&\dTo{\ob\alpha\times\ob\alpha}\\
\lim\big(\mor B\xto{\smash{t}}\ob B\xleftarrow{\smash{\ob q_0}}\ob Q\big)&\rTo{\ (\ob q_1,\proj)\ }&\ob Q\times\ob Q
\end{diagram}
is homotopy cartesian.
\end{proposition}

\begin{corollary}\label{corollary:local_homotopy_equiv_id_equiv}
Let $f:A\to B$ be a morphism in $\Top\dash\Cat$. Also, let
\begin{align*}
F=(P,p_0,p_1)&:\op{\internal A}\To\Cat(\Top)\\
G=(Q,q_0,q_1)&:\op{\internal B}\To\Cat(\Top)
\end{align*}
be internal $\Cat(\Top)$-valued functors, and $\alpha:F\to G\circ\op{(\internal f)}$ be an internal natural transformation.\\
Assume that the maps
\begin{align*}
s,t&:\mor P\To\ob P\\
s,t&:\mor Q\To\ob Q
\end{align*}
are Hurewicz fibrations and homotopy equivalences (note that $s$ is a homotopy equivalence if and only if $t$ is).\\
The functor (see construction \sref{construction:Groth(f,alpha)})
\[ \disccat{\Groth(\internal f,\alpha)}:\disccat{\Groth(F)}\To\disccat{\Groth(G)} \]
is a local homotopy equivalence if the square (recall definition \sref{definition:ob_mor_Cat(V)-calued_functor})
\begin{diagram}[midshaft]
A(x,y)\times(\ob F)(y)&\rTo{\ (\ob F,\proj)\ }&(\ob F)(x)\times(\ob F)(y)\\
\dTo{f\times\ob\alpha}&&\dTo{\ob\alpha\times\ob\alpha}\\
B(fx,fy)\times(\ob G)(fy)&\rTo{\ (\ob G,\proj)\ }&(\ob G)(fx)\times(\ob G)(fy)
\end{diagram}
is homotopy cartesian for all $x,y\in\ob A$.
\end{corollary}


%% file: sticky.tex



\chapter{Categories of sticky configurations}\label{chapter:sticky_conf}

\section*{Introduction}

This chapter introduces the first interesting construction in this text. To each space $X$, we associate a topological category $\M(X)$. The objects of $\M(X)$ are finite subsets of $X$. The morphisms of $\M(X)$ are ``sticky homotopies'', so called because they are homotopies in which any two points stick together when they collide.

The construction $\M(X)$ gives a very concrete model for categories which parametrize algebraic structures like $E_n$-algebras, as we will see later in chapter \ref{chapter:sticky<->embeddings}. Also, it allows us to recover topological Hochschild homology in the case of $X=S^1$, as we will see in the next chapter \ref{chapter:sticky_conf_S^1}. Putting these two observations together is the motivation for chapter \ref{chapter:invariants_En-algebras} where we define an invariant of $E_n$-algebras which generalizes topological Hochschild homology, and is related to $\M(X)$.

\section*{Summary}

The first three sections in this chapter lay out a formalism for constructing spaces and categories of sticky homotopies, as mentioned in the introduction. Section \sref{section:sticky_homotopies} defines the notion of sticky homotopy for a functor $C\to\Top$, relative to a subcategory of $C$. These sticky homotopies form a space for each object of $C$. Section \sref{section:functoriality_sticky_homotopies} analyzes the functoriality of the spaces of sticky homotopies. Section \sref{section:categories_sticky_homotopies} assembles topological categories whose morphisms are sticky homotopies, giving a functor $C\to\Cat(\Top)$.

Section \sref{section:sticky_configurations} uses the categories of sticky homotopies constructed in section \sref{section:categories_sticky_homotopies} to define the topologically enriched category $\M(X)$ of sticky configurations in a space $X$.

At this point, the discussion turns to defining an equivariant analogue of $\M(X)$. Section \sref{section:G-objects} describes some basic concepts on $G$-equivariant objects, while section \sref{section:G_Set_Top} deals specifically with $G$-sets and $G$-spaces. This is put to use in section \sref{section:G-equivariant_sticky_configurations} where the category $\M_G(X)$ of $G$-equivariant sticky configurations in a $G$-space $X$ is defined (using again the formalism of sticky homotopies).

The last two sections deal with comparing the categories of equivariant and non-equivariant sticky configurations. Section \sref{section:M_G->M} defines a functor \[ \rho_X:\M_G(X)\To\M(\quot{G}{X}) \] Section \sref{section:sticky_conf_covering_spaces} proves that $\rho_X$ is an essentially surjective local isomorphism if $G$ acts freely on $X$ and $X\to\quot{G}{X}$ is a covering space.

\section{Sticky homotopies}\label{section:sticky_homotopies}

The path category of a space $X$ (example \sref{example:path_category}) is an interesting homotopical replacement of the discrete category on $X$, $\disc(X)$ (example \sref{example:discrete_cat}). It is constructed out of the space of Moore paths of $X$, $H(X)$ (definition \sref{definition:Moore_path_space}). However, we will need a more refined notion of path or homotopy in $X$, which we introduce in this section.

\begin{definition}
Define $\CAT^{(2)}$ to be the full subcategory of the arrow category of $\CAT$
\[ \text{arrow}(\CAT)=\HomCat{\onearrow}{\CAT} \]
generated by the arrows which are inclusions of subcategories.\\
An object of $\CAT^{(2)}$ is called a {\em category pair}.
\end{definition}

\begin{remark}
$\onearrow$ denotes the category with two objects, $0$ and $1$, and a unique non-identity arrow, $0\to 1$.
\end{remark}

\begin{notation}
We will denote the functor which takes an arrow and returns the source of the arrow, $\evaluation{0}:\CAT^{(2)}\to\CAT$, simply by
\[ (-)_0:\CAT^{(2)}\To\CAT \]
Similarly, we will denote $\evaluation{1}:\CAT^{(2)}\to\CAT$ by
\[ (-)_1:\CAT^{(2)}\To\CAT \]
In particular, given an object $C$ of $\CAT^{(2)}$, $C_0$ is a subcategory of $C_1$.\\
Additionally, given a morphism $G:C\to D$ in $\CAT^{(2)}$, $G_1:C_1\to D_1$ is a functor which takes the subcategory $C_0$ of $C_1$ into $D_0$. $G$ is fully determined by $G_1$.
\end{notation}

\begin{definition}[sticky homotopy]\label{definition:sticky}
Let $C$ be an object of $\CAT^{(2)}$, $x$ an object of $C_1$, and $F:C_1\to\Top$ a functor.\\
An element $(\alpha,\tau)\in H(F(x))$ is a {\em \(C\)-sticky homotopy for \(F\) at \(x\)} if for any morphism $f:y\to x$ in the subcategory $C_0$, the image of $h:P\to\clop{0}{+\infty}$ --- as in the pullback square
\begin{diagram}[midshaft,height=2.2em]
P&\rTo{h}&\clop{0}{+\infty}\\
\dTo&&\dTo_{\alpha}\\
F(y)&\rTo{F(f)}&F(x)
\end{diagram}
--- is an interval which is empty or contains $\tau$.\\
The subspace of $H\circ F(x)$ corresponding to the $C$-sticky homotopies, denoted $SH_C(F)(x)$, will be called the {\em space of \(C\)-sticky homotopies} for $F$ at $x$.
\end{definition}

\begin{remark}[concatenation of sticky homotopies]\label{remark:concat_sticky}
Note that the concatenation (definition \sref{definition:concat_Moore_paths}) of $C$-sticky homotopies for $F$ is a $C$-sticky homotopy for $F$.
\end{remark}

We now summarize the behavior of sticky homotopies with respect to functors between the base categories.

\begin{proposition}\label{proposition:sticky_changebase}
Let $G:C\to D$ be a morphism in $\CAT^{(2)}$, $F:D_1\to\Top$ a functor, and $x$ an object of $C$.\\
Recall that $SH_D(F)(G_1 x)$ and $SH_C(F\circ G_1 )(x)$ are both subspaces of $H\circ F\circ G_1(x)$. We have the inclusion
\[ SH_D(F)(G_1 x)\subset SH_C(F\circ G_1)(x) \]
between those subspaces of $H\circ F\circ G(x)$.
\end{proposition}

\begin{proposition}\label{proposition:sticky_changebase_iso}
Let $G:C\to D$ be a morphism in $\CAT^{(2)}$, and $x$ an object of $C_1$.\\
Assume that for any morphism $f$ in $D_0$ with $\target f=G_1 x$, there exists an isomorphism $a$ in $D_1$, and an arrow $b$ in $C_0$ such that
\begin{align*}
\target b&=x\\
(G_1 b)\circ a&=f
\end{align*}
Then, for any functor $F:D_1\to\Top$, the subspaces $SH_D(F)(G_1 x)$ and $SH_C(F\circ G_1)(x)$ of $H\circ F\circ G_1(x)$ are equal:
\[ SH_D(F)(G_1 x)=SH_C(F\circ G_1)(x) \]
\end{proposition}

\section{Functoriality of sticky homotopies}\label{section:functoriality_sticky_homotopies}

\begin{definition}[cartesian natural transformation]
Assume $C$, $D$ are categories with pullbacks, and $F,G:C\to D$ are functors.\\
A natural transformation
\[ \alpha:F\To G \]
is said to be {\em cartesian} if the square
\begin{diagram}[h=2.1em]
F(x)&\rTo{F(f)}&F(y)\\
\dTo{\alpha_x}&&\dTo{\alpha_y}\\
G(X)&\rTo{\ G(f)\ }&G(y)
\end{diagram}
is cartesian for each morphism $f:x\to y$ in $C$.
\end{definition}

\begin{definition}
Define $\CAT^{(2)}_\cartesian$ to be the sub-2-category of $\CAT^{(2)}$ whose
\begin{enum}
\item objects are $C\in\ob\big(\CAT^{(2)}\big)$ such that $C_1$ has pullbacks and, for any pullback diagram in $C_1$
\begin{diagram}[h=2.1em]
a&\rTo&a'\\
\dTo{f}&&\dTo{f'}\\
b&\rTo&b'
\end{diagram}
if $f'$ is in $C_0$ then $f$ is also in $C_0$.
\item 1-morphisms are the $1$-morphisms $G$ of $\CAT^{(2)}$ such that $G_1$ preserves all pullbacks.
\item 2-morphisms are the $2$-morphisms $\alpha$ of $\CAT^{(2)}$ such that $\alpha_0$ is a cartesian natural transformation.
\end{enum}
\end{definition}

\begin{example}\label{example:cat2_cart_cat_mono}
If $C$ is a category with pullbacks in which monomorphisms are stable under pullback (along any arrow), then the inclusion of the subcategory of monomorphisms
\[ \text{mono}(C)\Into C \]
gives an element of $\CAT^{(2)}_\cartesian$.\\
Later, we will give two cases of this example in the form of the opposites of the categories of finite sets, and finitely generated free $G$-sets.
\end{example}

\begin{example}[$\Top$ as an object of $\CAT^{(2)}_\cartesian$]
$\id_\Top:\Top\to\Top$ is an element of $\CAT^{(2)}_\cartesian$. We will call it simply $\Top$.\\
Observe that a morphism $F:C\to\Top$ in $\CAT^{(2)}_\cartesian$ is determined by any pullback preserving functor $F_1:C_1\to\Top$.
\end{example}

\begin{proposition}\label{proposition:sticky_homotopies_functor_pushout}
Let $C$ be an object of $\CAT^{(2)}_\cartesian$, and $F:C_1\to\Top$ a pullback preserving functor.\\
There is a functor
\[ SH_C(F):C_1\To\Top \]
which is given on objects by the space of sticky homotopies for $F$ from definition \sref{definition:sticky}.\\
Moreover, there is a natural transformation
\[ SH_C(F)\To H\circ F \]
which at each object $x$ of $C$ is the inclusion
\[ SH_C(F)(x)\Into H\circ F(X) \]
\end{proposition}

\begin{proposition}
There is a functor
\[ SH_C:\CAT^{(2)}_\cartesian(C,\Top)\To\HomCat{C_1}{\Top} \]
which associates to each morphism $F:C\to\Top$ in $\CAT^{(2)}_\cartesian$ the functor $SH_C(F_1)$ (as given in the previous proposition).
\end{proposition}

\begin{construction}
Thanks to proposition \sref{proposition:sticky_changebase}, we can extend this to a lax natural transformation
\[ SH:\CAT^{(2)}_\cartesian(-,\Top)\To\HomCatlarge{(-)_1}{\Top} \]
between functors $\op{\big(\CAT^{(2)}_\cartesian\big)}\to\SET\dash\CAT$.
\end{construction}

We will need slightly more functoriality from $SH$ later on, so we introduce it here.

\begin{definition}[quasi-cartesian natural transformation]
Let $C$, $D$ be categories, $d$ an object of $D$, and $F,G:C\to D$ functors.\\
We say a natural transformation $\alpha:F\to G$ is {\em quasi-cartesian with respect to \(d\)} if for every morphism $f:x\to y$ in $C$, and every commutative diagram
\begin{diagram}
d&\rTo&G(x)\\
\dTo&&\dTo{G(f)}\\
F(y)&\rTo{\alpha_y}&G(y)
\end{diagram}
there exist morphisms $d\to F(x)$ in $D$, and $\sigma:y\to y$ in $C$ such that
\begin{diagram}[midshaft]
&&d&\rTo&F(y)&\rTo{\alpha_y}&G(y)\\
&\ldTo&\dTo&&\dTo{F(\sigma)}&\ruTo_{\alpha_y}\\
G(x)&\lTo_{\alpha_x}&F(x)&\rTo_{F(f)}&F(y)
\end{diagram}
commutes.
\end{definition}

\begin{definition}
Let $C$ be a category pair in $\CAT^{(2)}_\cartesian$.\\
We define $\CAT^{(2)}_\isocart(C,\Top)$ to be the subcategory $\CAT^{(2)}(C,\Top)$ whose
\begin{enum}
\item objects are $F\in\CAT^{(2)}(C,\Top)$ such that $F_1$ preserves all pullbacks;
\item morphisms are the morphisms $\alpha$ in $\CAT^{(2)}(C,\Top)$ such that $\alpha_0$ is quasi-cartesian with respect to $1\in\Top$.
\end{enum}
\end{definition}

\begin{remark}
The category $\CAT^{(2)}_\cartesian(C,\Top)$ is a subcategory of $\CAT^{(2)}_\isocart(C,\Top)$ which possesses the same objects.
\end{remark}

\begin{remark}
The category $\CAT^{(2)}_\isocart(C,\Top)$ is not functorial in $C$ in $\CAT^{(2)}_\cartesian$. It is only functorial on full functors, for example.
\end{remark}

\begin{proposition}
Let $C$ be a category pair in $\CAT^{(2)}_\cartesian$.\\
There is a functor
\[ SH_C:\CAT^{(2)}_\isocart(C,\Top)\To\HomCat{C_1}{\Top} \]
for which the diagram
\begin{diagram}[midshaft]
\CAT^{(2)}_\cartesian(C,\Top)&\rTo{\ SH_C\ }&\HomCat{C_1}{\Top}\\
\dInto[lowershortfall=.2em]{\inclusion}&\ruTo_{\mathrlap{SH_C}}\\
\CAT^{(2)}_\isocart(C,\Top)
\end{diagram}
commutes.
\end{proposition}

\section{Categories of sticky homotopies}\label{section:categories_sticky_homotopies}

As suggested by remark \sref{remark:concat_sticky}, and stated in the following proposition, sticky homotopies form the morphisms of a functorial subcategory of $\pathcat\circ F$ for appropriate functors $F$ with values in $\Top$.

\begin{proposition}[internal category of sticky homotopies]\label{proposition:sticky_path-cat}
Let $C$ be an object of $\CAT^{(2)}_\cartesian$, and $F:C\to\Top$ a morphism in $\CAT^{(2)}_\cartesian$.\\
There is a unique functor \[ \stickypathcat_C F:C_1\To\Cat(\Top) \] and a unique natural transformation ($\pathcat$ is defined in example \sref{example:path_category}) \[ \sigma:\stickypathcat_C F\To\pathcat\circ F_1 \] such that the following conditions hold:
\begin{enum}
\item $\ob\circ\sigma=\id_{F_1}$;
\item for any object $x$ of $C_1$, $\mor\circ\sigma_x$ is the inclusion $SH_C(F)(x)\into H\circ F_1(x)$.
\end{enum}
\end{proposition}

\begin{notation}
We call $\stickypathcat_C F$ the {\em functorial category of sticky homotopies for $F$}.
\end{notation}

The functoriality of $\stickypathcat_C$ stated in the next result follows from the functoriality of $SH$ analyzed in the previous section.

\begin{proposition}[functoriality of $\stickypathcat_C$]\label{proposition:sticky_path-cat_functor}
Let $C$ be an object of $\CAT^{(2)}_\cartesian$.\\
There is a functor
\[ \stickypathcat_C:\CAT^{(2)}_\isocart(C,\Top)\To\HomCat{C_1}{\Cat(\Top)} \]
which is given on objects by the functorial category of sticky homotopies.
\end{proposition}

\begin{notation}
The restriction of $\stickypathcat_C$ to $\CAT^{(2)}_\cartesian(C,\Top)$ will also be designated by $\stickypathcat_C$.
\end{notation}

We can extract greater naturality for $\stickypathcat_C$ --- this time on the base category $C$ --- from proposition \sref{proposition:sticky_changebase}.

\begin{proposition}[lax naturality of $\stickypathcat$ on the base category]\label{proposition:stickypathcat_laxsquare}
Let $G:C\to D$ be a morphism in $\CAT^{(2)}_\cartesian$.\\
There is a canonical natural transformation $\vartheta_G$
\begin{equation}\label{equation:lax_nat_stickypathcat}
\begin{diagram}
\CAT^{(2)}_\cartesian(D,\Top)&\rTo{\ \stickypathcat_D\ }&\HomCat{D_1}{\Cat(\Top)}\\
\dTo{\CAT^{(2)}_\cartesian(G,\Top)}&\twocell{\vartheta_G\hspace{1em}}{.28em}{\Longleftarrow}{40}&\dTo_{\HomCat{G_1}{\Cat(\Top)}}\\
\CAT^{(2)}_\cartesian(C,\Top)&\rTo{\ \stickypathcat_C\ }&\HomCat{C_1}{\Cat(\Top)}
\end{diagram}
\end{equation}
\end{proposition}

\begin{remark}
More concretely, for each $F:D\to\Top$ in $\CAT^{(2)}_\cartesian$, $\vartheta_G$ gives a natural transformation
\[ (\vartheta_G)_F:\stickypathcat_D F\circ G_1\To\stickypathcat_C(F\circ G) \]
\end{remark}

\begin{remark}
The natural transformations
\begin{diagram}
\CAT^{(2)}_\cartesian(D,\Top)&\rTo{\ \stickypathcat_D\ }&\HomCat{D_1}{\Cat(\Top)}\\
\dTo{\CAT^{(2)}_\cartesian(G,\Top)}&\twocell{\vartheta_G\hspace{1em}}{.28em}{\Longleftarrow}{40}&\dTo_{\HomCat{G_1}{\Cat(\Top)}}\\
\CAT^{(2)}_\cartesian(C,\Top)&\rTo{\ \stickypathcat_C\ }&\HomCat{C_1}{\Cat(\Top)}
\end{diagram}
compose in the obvious manner, when one stacks two of these diagrams on top of each other.\\
In other words, they endow the family of functors $\stickypathcat_\bullet$ with the structure of a lax natural transformation
\[ \stickypathcat_\bullet:\CAT^{(2)}_\cartesian(-,\Top)\To\HomCatlarge{(-)_1}{\Cat(\Top)} \]
between functors $\op{\big(\CAT^{(2)}_\cartesian\big)}\to\SET\dash\CAT$.
\end{remark}

The following proposition is now a consequence of \sref{proposition:sticky_changebase_iso}. For conciseness, we first give a definition derived from proposition \sref{proposition:sticky_changebase_iso}.

\begin{definition}[iso-full morphism of category pairs]\label{definition:iso-full}
We say a morphism $G:C\to D$ in $\CAT^{(2)}$ is {\em iso-full} if for any object $x$ in $C_1$, and any morphism $f$ in $D_0$ with $\target f=G_1 x$, there exists an isomorphism $a$ in $D_1$, and an arrow $b$ in $C_0$ such that
\begin{align*}
\target b&=x\\
(G_1 b)\circ a&=f
\end{align*}
\end{definition}

\begin{proposition}\label{proposition:stickypathcat_laxsquare_iso}
Let $G:C\to D$ be an iso-full morphism in $\CAT^{(2)}_\cartesian$.\\
The natural transformation $\vartheta_G$ in diagram \seqref{equation:lax_nat_stickypathcat} is the identity natural transformation.\\
In particular, the diagram
\begin{diagram}
\CAT^{(2)}_\cartesian(D,\Top)&\rTo{\ \stickypathcat_D\ }&\HomCat{D_1}{\Cat(\Top)}\\
\dTo{\CAT^{(2)}_\cartesian(G,\Top)}&&\dTo_{\HomCat{G_1}{\Cat(\Top)}}\\
\CAT^{(2)}_\cartesian(C,\Top)&\rTo{\ \stickypathcat_C\ }&\HomCat{C_1}{\Cat(\Top)}
\end{diagram}
commutes. Consequently, for each $F\in\CAT^{(2)}_\cartesian(D,\Top)$
\[ \stickypathcat_D F\circ G_1=\stickypathcat_C(F\circ G) \]
\end{proposition}

\begin{definition}[enriched category of sticky homotopies]\label{definition:discrete_stickypathcat}
Let $C$ be an object of $\CAT^{(2)}_\cartesian$.\\
For convenience, we let $\disccatsub{\stickypathcat}{C}$ abbreviate the composition
\[ \CAT^{(2)}_\cartesian(C,\Top)\xTo{\stickypathcat_C}\HomCat{C_1}{\Cat(\Top)}\xTo{\HomCatlarge{C_1}{\disccat{(-)}}}\HomCat{C_1}{\Top\dash\Cat} \]
\end{definition}

\begin{remark}
We will make frequent use of this notation: adding a superscript $\discrete$ to the name of a functor or natural transformation, $f$, in $\HomCat{A}{F(\Cat(\Top))}$ --- thus obtaining $\disccat{f}$ --- will indicate
\[ \disccat{f}\defeq F\big(\disccat{(-)}\big)\circ f \]
This is done for convenience, since the expression of $F$ (commonly of the form $\HomCat{B}{-}$) could add some cumbersome overhead to the notation.\\
As an example, the lax naturality square for $\disccatsub{\stickypathcat}{\bullet}$ at a morphism $G:C\to D$ is given by the natural transformation
\[ \disccat{\vartheta_G}=\HomCatLarge{C_1}{\disccat{(-)}}\circ\vartheta_G \]
\end{remark}

\section{Category of sticky configurations}\label{section:sticky_configurations}

In the present section, we introduce one of the central constructions in this text.

\begin{definition}
Define $\Top_{inj}$ to be the subcategory of $\Top$ whose morphisms are all continuous injective maps.
\end{definition}

\begin{notation}[$\op{\FinSet}$ as an object of $\CAT^{(2)}_\cartesian$]
The opposite of the inclusion of the subcategory of epimorphisms of $\FinSet$
\[ \op{\text{epi}(\FinSet)}\Into\op{\FinSet} \]
is a particular case of example \sref{example:cat2_cart_cat_mono}, and therefore is an object of $\CAT^{(2)}_\cartesian$. For ease of notation, we will denote it by $\op{\FinSet}$.
\end{notation}

\begin{construction}\label{definition:stickypathcat_Map}
Let $\widehat{\Map}$ denote the functor
\[ \widehat{\Map}:\Top\xTo{\Map}\HomCat{\op{\Top}}{\Top}\xTo{\HomCat{\op{\inclusion}}{\Top}}\HomCat{\op{\FinSet}}{\Top} \]
Then there is a unique functor
\[ \overline{\Map}:\Top_{inj}\To\CAT^{(2)}_\cartesian(\op{\FinSet},\Top) \]
for which
\begin{diagram}[midshaft]
\Top_{inj}&\rTo{\ \overline{\Map}\ }&\CAT^{(2)}_\cartesian(\op{\FinSet},\Top)\\
\dInto{\inclusion}&&\dTo{(-)_1}\\
\Top&\rTo{\widehat{\Map}}&\HomCat{\op{\FinSet}}{\Top}
\end{diagram}
commutes.
With this, we can define the functor \[ \stcat:\Top\To\HomCat{\op{\FinSet}}{\Cat(\Top)} \] as the composition
\[ \Top_{inj}\xTo{\overline{\Map}}\CAT^{(2)}_\cartesian(\op{\FinSet},\Top)\xTo[\sref{proposition:sticky_path-cat_functor}]{\stickypathcat_{\op{\FinSet}}}\HomCat{\op{\FinSet}}{\Cat(\Top)} \]
\end{construction}

\begin{definition}[category of sticky finite sets]\label{definition:sticky_sets}
We let $\M^\textbig$ denote the composition
\[ \Top\xTo{\disccat{\stcat}}\HomCat{\op{\FinSet}}{\Top\dash\Cat}\xTo[\sref{construction:variation_Groth_construction}]{\Groth}\Top\dash\CAT \]
\end{definition}

This category is too big: we want to consider only its objects which correspond to injective maps from finite sets into $X$, or {\em configurations} in $X$. Thus we will restrict to an appropriate full subcategory of $\M^\textbig(X)$.

First, note that proposition \sref{proposition:description_Groth_discrete} identifies the class of objects of $\M^\textbig(X)$ as
\[ \ob\big(\M^{\textbig}(X)\big)=\;\,\coprod_{\mathclap{z\in\FinSet}}\;\,\Top(z,X) \]

\begin{definition}[category of sticky configurations]\label{definition:sticky_configurations}
Let $X$ be a topological space.\\
The {\em category of sticky configurations in \(X\)}, $\M(X)$, is the full $\Top$-subcate\-go\-ry of $\M^\textbig(X)$ (definition \sref{definition:sticky_sets}) generated by all injective maps from finite sets to $X$.
\end{definition}

\begin{remark}\label{remark:functoriality_M(X)}
This subcategory of $\M^\textbig$ inherits the functoriality: there is a functor
\[ \M:\Top_{inj}\To\Top\dash\CAT \]
given on objects by the previous definition.
\end{remark}

\section{Generalities on $G$-equivariance}\label{section:G-objects}

Having defined the category of sticky configurations, $\M(X)$, we will similarly introduce an equivariant version of it. For that purpose, this section revises some basic facts on equivariant objects. Assume for the remainder of this section that $G$ is a monoid in the cartesian category $\Set$. Recall that $\B{G}$ denotes a category with one object and morphisms given by $G$.

\begin{definition}[category of $G$-objects]
Let $C$ be a category.\\
$G\dash C$ denotes the category $\HomCat{\B{G}}{C}$ of {\em \(G\)-objects in \(C\)}.
\end{definition}

\begin{remark}
There is an isomorphism $1\dash C\cong C$ natural in $C$.
\end{remark}

\begin{definition}[functors on $G$-objects]\label{definition:functors_G-objects}
Let $C$ be a category.\\
The {\em forgetful} functor
\[ G\dash C=\HomCat{\B{G}}{C}\xTo{(1\to\B{G})^\ast}\HomCat{1}{C}=C \]
is called $u:G\dash C\to C$.\\
The {\em trivial \(G\)-object} functor
\[ C=\HomCat{1}{C}\xTo{(\B{G}\to 1)^\ast}\HomCat{\B{G}}{C}=G\dash C \]
is called $k:C\to G\dash C$.
\end{definition}

\begin{remark}\label{remark:uk=id}
Note that $u\circ k=\id_C$.
\end{remark}

\begin{proposition}[free $G$-object]\label{proposition:free_G-object}
Let $C$ be a cocomplete category.\\
The functor $u:G\dash C\to C$ has a left adjoint, which will be denoted \[ \free{G}{-}:C\To G\dash C \]
\comment{The unit of the adjunction will be denoted
\begin{equation}\label{equation:unit_freeG}
\eta:\id_C\To u\circ\free{G}{-}
\end{equation}}
\end{proposition}

\begin{proposition}[quotient $G$-object]\label{proposition:quotient_G-object}
Let $C$ be a cocomplete category.\\
The functor $k:C\to G\dash C$ has a left adjoint, denoted \[ \quot{G}{-}:G\dash C\To C \]
The counit of the adjunction is the identity natural transformation.
\end{proposition}

\section{$G$-objects in $\Set$ and $\Top$}\label{section:G_Set_Top}

In this section, let $G$ abbreviate a monoid in $\Set$.

\begin{remark}
All the functors defined in the preceding section \sref{section:G-objects} commute appropriately with the inclusions $\Set\to\Top$ and $G\dash\Set\to G\dash\Top$.
\end{remark}

\begin{remark}
A monoid $G$ in the cartesian category $\Set$ passes to a monoid $G$ in the cartesian category $\Top$. We can therefore consider left modules over these monoids.\\
The categories of $G$-objects in $\Set$ and $\Top$ admit canonical isomorphisms with the categories of left $G$-modules in $\Set$ and $\Top$:
\begin{align*}
G\dash\Set&\cong G\dash\mod(\Set)\\
G\dash\Top&\cong G\dash\mod(\Top)
\end{align*}
\end{remark}

\begin{definition}[free finitely generated $G$-sets]\label{definition:free_G-sets}
The full subcategory of $G\dash\Set$ generated by the essential image of
\[ \FinSet\To\Set\xTo{\free{G}{-}}G\dash\Set \]
is abbreviated $\FinSet_G$.
\end{definition}

\begin{remark}
Note that $\FinSet_1$ is isomorphic to $\FinSet$.
\end{remark}

\begin{construction}[$\op{\FinSet_G}$ as an object of $\CAT^{(2)}_\cartesian$]\label{construction:FinSet_G_CAT^2_cart}
The opposite of the inclusion of the subcategory of epimorphisms of $\FinSet_G$
\[ \op{\text{epi}(\FinSet_G)}\Into\op{\FinSet_G} \]
gives an object of $\CAT^{(2)}_\cartesian$ (an instance of example \sref{example:cat2_cart_cat_mono}), which is denoted simply by $\op{\FinSet_G}$.\\
The functor
\[ \op{(\quot{G}{-})}:\op{\FinSet_G}\To\op{\FinSet} \]
determines a morphism $\op{(\quot{G}{-})}\in\CAT^{(2)}_\cartesian\big(\op{\FinSet_G},\op{\FinSet}\big)$.\\
If $G$ is a group, then the functor
\[ \op{\free{G}{-}}:\op{\FinSet}\To\op{\FinSet_G} \]
determines a morphism $\op{\free{G}{-}}\in\CAT^{(2)}_\cartesian\big(\op{\FinSet},\op{\FinSet_G}\big)$.\\
\end{construction}

\comment{We could consider instead the category of $G$-sets whose quotient by $G$ is finite, with the subcategory of epimorphisms $f:x\to y$ which induce a bijection $f:aG\to f(a)G$ for every $a\in x$.\\
Using this category instead, the equivalence $\M_G(X)\to\M(\quot{G}{X})$ is always true, without needing the action of $G$ on $X$ to be free; we only need $X\to\quot{G}{X}$ to be a covering space. If the action of $G$ on $X$ is free, this new category $\M_G(X)$ just happens to be the same as the one with $\FInSet_G$ above.}

\begin{proposition}\label{proposition:iso-full_free_quot}
The morphism in $\CAT^{(2)}_\cartesian$
\[ \op{(\quot{G}{-})}:\op{\FinSet_G}\To\op{\FinSet} \]
is iso-full (see definition \sref{definition:iso-full}).\\
If $G$ is the monoid  underlying a group, then the morphism
\[ \op{\free{G}{-}}:\op{\FinSet}\To\op{\FinSet_G} \]
is iso-full.
\end{proposition}

\begin{proposition}[equivariant mapping space]\label{proposition:Map_G}
There is a unique functor
\[ \Map^G:\op{G\dash\Top}\times G\dash\Top\To\Top \]
and a unique natural transformation (recall $u:G\dash\Top\to\Top$ from definition \sref{definition:functors_G-objects})
\[ j:\Map^G\To\Map\circ(\op{u}\times u) \]
such that for any objects $X$, $Y$ of $G\dash\Top$, $j_{(X,Y)}$ is the inclusion of the subspace of $G$-equivariant maps in $\Map(uX,uY)$.
\end{proposition}

\begin{remark}
The isomorphism $1\dash\Top\cong\Top$ carries $\Map^1$ to $\Map$.
\end{remark}

\section{$G$-equivariant sticky configurations}\label{section:G-equivariant_sticky_configurations}

Throughout this section, let $G$ denote a monoid in the cartesian category $\Set$. We introduce a category of $G$-equivariant sticky configurations, similar to the construction in section \sref{section:sticky_configurations}.

\begin{definition}
Let $G\dash\Top_{inj}$ denote the subcategory of $G\dash\Top$ whose morphisms are all injective $G$-equivariant maps.
\end{definition}

\begin{construction}\label{definition:stickypathcat_MapG}
Define $\widehat{\Map}_G$ to be the functor
\[ G\dash\Top\xTo[\sref{proposition:Map_G}]{\Map^G}\HomCat{\op{G\dash\Top}}{\Top}\xTo{\HomCat{\op{\inclusion}}{\Top}}\HomCat{\op{\FinSet_G}}{\Top} \]
Then there is a unique functor
\[ \overline{\Map}_G:G\dash\Top_{inj}\To\CAT^{(2)}_\cartesian(\op{\FinSet_G},\Top) \]
for which
\begin{diagram}[midshaft]
G\dash\Top_{inj}&\rTo{\ \overline{\Map}_G\ }&\CAT^{(2)}_\cartesian(\op{\FinSet_G},\Top)\\
\dInto{\inclusion}&&\dTo{(-)_1}\\
G\dash\Top&\rTo{\widehat{\Map}_G}&\HomCat{\op{\FinSet_G}}{\Top}
\end{diagram}
commutes.
\end{construction}

\begin{definition}\label{definition:stcat_G}
The composition
{\small \[ G\dash\Top_{inj}\xTo{\overline{\Map}_G}\CAT^{(2)}_\cartesian(\op{\FinSet_G},\Top)\xTo[\sref{proposition:sticky_path-cat}]{\stickypathcat_{\op{\FinSet_G}}}\HomCat{\op{\FinSet_G}}{\Cat(\Top)} \]
}is abbreviated $\stcat_G$.
\end{definition}

\begin{definition}[$G$-equivariant sticky sets]\label{definition:M^big_G}
We let $\M^{\textbig}_G$ denote the composite functor
\[ G\dash\Top\xTo{\disccat{\stcat_G}}\HomCat{\op{\FinSet_G}}{\Top\dash\Cat}\xTo[\sref{construction:variation_Groth_construction}]{\Groth}\Top\dash\CAT \]
Note that the Grothendieck construction provides a cocone \[ \pi:\M^\textbig_G\to\FinSet_G \]
\end{definition}

Proposition \sref{proposition:description_Groth_discrete} determines the set of objects of $\M_G^\textbig(X)$ to be
\[ \ob\big(\M_G^{\textbig}(X)\big)=\;\,\coprod_{\mathclap{z\in\FinSet_G}}\;\,G\dash\Top(z,X) \]

\begin{definition}[$G$-equivariant sticky configurations]\label{definition:M_G}
Let $X$ be an object of $G\dash\Top$.\\
The {\em category of \(G\)-equivariant sticky configurations} in $X$, $\M_G(X)$, is the full $\Top$-subcategory of $\M^\textbig_G(X)$ generated by the injective ($G$-equivariant) maps into $X$.
\end{definition}

\begin{remark}
Note that the objects defined in this section for the case $G=1$ are naturally isomorphic to the corresponding objects defined in section \sref{section:sticky_configurations}, after taking into account the isomorphisms $\FinSet_1\cong\FinSet$ and $1\dash\Top\cong\Top$.
\end{remark}

\section{From $\M_G$ to $\M$}\label{section:M_G->M}

Throughout this section, we let $G$ be a group in $\Set$.

\begin{construction}
Let $X$ be an object of $G\dash\Top$.\\
For each $Y$ in $\FinSet_G$, define the map
\[ \Map^G(Y,X)\To\Map\big(\quot{G}{Y},\quot{G}{X}\big) \]
which takes $f:Y\to X$ to the map induced by $f$ on the quotients by $G$.\comment{This is not continuous for $Y$ a general element $G$-space.}\\
These maps, for $Y$ in $\FinSet_G$, assemble into a natural transformation
\[ \widehat{\theta}_X:\widehat{\Map}_G(X)\To\widehat{\Map}(\quot{G}{X})\circ\op{(\quot{G}{-})} \]
The restriction of $\widehat{\theta}_X$ to the category $\op{\text{epi}(\FinSet_G)}$ is quasi-cartesian with respect to $1\in\Top$.
\end{construction}

\begin{construction}
The preceding construction defines the components of a natural transformation
\[ \widehat{\theta}:\widehat{\Map}_G\To\widehat{\Map}(\quot{G}{X})\circ\op{(\quot{G}{-})} \]
from $\widehat{\Map}_G$ to the composition
\[ G\dash\Top\xTo{\quot{G}{-}}\Top\xTo{\widehat{\Map}}\HomCat{\op{\FinSet}}{\Top}\xTo{\HomCat{\op{(\quot{G}{-})}}{\Top}}\HomCat{\op{\FinSet_G}}{\Top} \]
$\widehat{\theta}$ determines a unique natural transformation $\overline{\theta}$ between functors of type
\[ G\dash\Top_{inj}\To\CAT^{(2)}_\isocart(\op{\FinSet_G},\Top) \]
such that \[ \evaluation{1}\circ\overline{\theta}=\restrict{\widehat{\theta}}{G\dash\Top_{inj}} \]
The source of $\overline{\theta}$ is
{\small \[ \overline{\Map}_G:G\dash\Top_{inj}\xTo{\overline{\Map}_G}\CAT^{(2)}_\cartesian(\op{\FinSet_G},\Top)\Into\CAT^{(2)}_\isocart(\op{\FinSet_G},\Top) \]
}The target of $\overline{\theta}$ is
\begin{align*}
G\dash\Top_{inj}&\xTo{\quot{G}{-}}\Top_{inj}\\
&\xTo{\overline{\Map}}\CAT^{(2)}_\cartesian(\op{\FinSet},\Top)\\
&\xTo{\CAT^{(2)}_\cartesian(\op{(\quot{G}{-})},\Top)}\CAT^{(2)}_\cartesian(\op{\FinSet_G},\Top)\\
&\Into\CAT^{(2)}_\isocart(\op{\FinSet_G},\Top)
\end{align*}
\end{construction}

\begin{definition}
Define $\stickypathcat\theta$ to be the natural transformation
\[ \stickypathcat\theta\defeq\stickypathcat_{\op{\FinSet_G}}\circ\overline{\theta} \]
\end{definition}

\begin{proposition}
$\stickypathcat\theta$ is a natural transformation from $\stcat_G$ to the composition
\begin{align*}
G\dash\Top_{inj}&\xTo{\quot{G}{-}}\Top_{inj}\\
&\xTo{\stcat}\HomCat{\op{\FinSet}}{\Cat(\Top)}\\
&\xTo{\HomCat{\op{(\quot{G}{-})}}{\Cat(\Top)}}\HomCat{\op{\FinSet_G}}{\Cat(\Top)}
\end{align*}
\end{proposition}
\begin{proof}
Since the source of $\overline{\theta}$ is $\overline{\Map}_G$, the source of $\stickypathcat\theta=\stickypathcat_{\op{\FinSet_G}}\circ\overline{\theta}$ is
\[ \stickypathcat_{\op{\FinSet_G}}\circ\overline{\Map}_G=\stcat_G \]

We know from proposition \sref{proposition:iso-full_free_quot} that the morphism in $\CAT^{(2)}_\cartesian$
\[ \op{(\quot{G}{-})}:\op{\FinSet_G}\To\op{\FinSet} \]
is iso-full. Consequently, by virtue of proposition \sref{proposition:stickypathcat_laxsquare_iso}, the following diagram commutes
\begin{diagram}
\CAT^{(2)}_\cartesian(\op{\FinSet},\Top)&\rTo{\ \CAT^{(2)}_\cartesian(\op{(\quot{G}{-})},\Top)}&\CAT^{(2)\ }_\cartesian(\op{\FinSet_G},\Top)\\
\dTo_{\stickypathcat_{\op{\FinSet}}}&&\dTo_{\stickypathcat_{\op{\FinSet_G}}}\\
\HomCat{\op{\FinSet}}{\Cat(\Top)}&\rTo{\HomCat{\op{(\quot{G}{-})}}{\Cat(\Top)}}&\HomCat{\op{\FinSet_G}}{\Cat(\Top)}\\
\end{diagram}
This diagram and the knowledge that the target of $\overline{\theta}$ is
\[ \CAT^{(2)}_\cartesian\big(\op{(\quot{G}{-})},\Top\big)\circ\overline{\Map}\circ(\quot{G}{-}) \]
imply that the target of $\stickypathcat\theta=\stickypathcat_{\FinSet_G}\circ\overline{\theta}$ is
\begin{align*}
G\dash\Top&\xTo{\quot{G}{-}}\Top\\
&\xTo{\overline{\Map}}\HomCat{\op{\FinSet}}{\Top}\\
&\xTo{\stickypathcat_{\FinSet}}\HomCat{\op{\FinSet}}{\Cat(\Top)}\\
&\xTo{\HomCat{\op{(\quot{G}{-})}}{\Top}}\HomCat{\op{\FinSet_G}}{\Cat(\Top)}
\end{align*}
\end{proof}

\begin{construction}
In view of the previous proposition, the natural transformation $\stickypathcat\theta$ induces, for each $X$ in $G\dash\Top$, a functor
\[ \Groth\big(\quot{G}{-},\disccatsub{\stickypathcat\theta}{X}\big):\M_G^{\textbig}(X)\To\M^\textbig(\quot{G}{X}) \]
by construction \sref{construction:Groth(f,alpha)_variation}. These functors are the components of a natural transformation
\[ \rho:\M_G^{\textbig}\To\M^\textbig\circ(\quot{G}{-}) \]
between functors $G\dash\Top_{inj}\To\Top\dash\CAT$.
\end{construction}

\begin{proposition}
Let $X$ be an object of $G\dash\Top$.\\
The $\Top$-functor 
\[ \rho_X:\M_G^\textbig(X)\To\M^\textbig(\quot{G}{X}) \]
restricts to a $\Top$-functor
\[ \rho_X:\M_G(X)\To\M(\quot{G}{X}) \]
(to which we give the same name).
\end{proposition}

\begin{proposition}\label{proposition:rho_ess_surj}
Let $X$ be an object of $G\dash\Top$.\\
The $\Top$-functor
\[ \rho_X:\M_G^\textbig(X)\To\M^\textbig(\quot{G}{X}) \]
is essentially surjective. It restricts to an essentially surjective $\Top$-functor
\[ \rho_X:\M_G(X)\To\M(\quot{G}{X}) \]
if the action of $G$ on $uX$ is free.
\end{proposition}
\begin{proof}
According to proposition \sref{proposition:description_Groth_discrete}, the set of objects of $\M_G^\textbig(X)$ is
\[ \ob\big(\M_G^\textbig(X)\big)=\ \coprod_{\mathclap{x\in\FinSet_G}}\ G\dash\Top(x,X) \]
and the set of objects of $\M^\textbig(\quot{G}{X})$ is
\[ \ob\big(\M^\textbig(\quot{G}{X})\big)=\ \coprod_{\mathclap{x\in\FinSet}}\ \Top\big(x,\quot{G}{X}\big) \]
In addition, the diagram
\begin{diagram}[midshaft]
G\dash\Top(y,X)&\rInto{\inclusion_y}&\coprod_{\mathclap{x\in\FinSet_G}}\ G\dash\Top(x,X)\\
\dTo_{\quot{G}{-}}&&\dTo[uppershortfall=-.6em,lowershortfall=.15em]_{\rho_X}\\
\Top\big(\quot{G}{y},\quot{G}{X}\big)&\rInto{\ \inclusion_{\quot{G}{y}}\ }&\coprod_{\mathclap{x\in\FinSet}}\ \Top(x,X)
\end{diagram}
commutes for every $y$ in $\FinSet_G$.

Given an object of $\M^\textbig(\quot{G}{X})$ \[ f:x\To\quot{G}{X} \] choose a free $G$-set $y$ such that $x\simeq\quot{G}{y}$ in $\FinSet$. Then $f\simeq f'$ in $\M^\textbig(\quot{G}{X})$, where \[ f':\quot{G}{y}\To\quot{G}{X} \]
Since $y$ is free, there exists a $G$-equivariant map \[ \overline{f}:y\To X \] such that the map induced by $\overline{f}$ on the quotients is $f':\quot{G}{y}\to\quot{G}{X}$. Therefore (by the above commutative square) \[ \rho_X(\overline{f})=f'\simeq f \]
We conclude that \[ \rho_X:\M_G^\textbig(X)\To\M^\textbig(\quot{G}{X}) \] is essentially surjective.

Now assume that the action of $G$ on $uX$ is free. If $f:x\to\quot{G}{X}$ is injective, then the $G$-map $\overline{f}:y\to X$ chosen above is also injective. We thus conclude that the functor
\[ \rho_X:\M_G(X)\To\M(\quot{G}{X}) \]
is essentially surjective.
\end{proof}

\section{Sticky configurations and covering spaces}\label{section:sticky_conf_covering_spaces}

Throughout this section, we fix a group $G$ in $\Set$, and an object $X$ of $G\dash\Top$.

\begin{notation}\label{notation:H(Y;a,b)}
Let $Y$ be a topological space, and $a,b\in Y$.\\
We let $H(Y;a,b)$ denote the subspace of $H(Y)$ given by
\[ H(Y;a,b)\defeq\set{x\in H(y)\suchthat s(x)=a\,,\,t(x)=b} \]
\end{notation}

\begin{construction}
Let $x$, $y$ be objects of $\FinSet_G$, and consider $G$-equivariant maps
\begin{align*}
f&:x\To X\\
g&:y\To X
\end{align*}
Let $\quot{G}{f}=\rho_X f$ and $\quot{G}{g}=\rho_X g$ be the maps
\begin{align*}
\quot{G}{f}&:\quot{G}{x}\To\quot{G}{X}\\
\quot{G}{g}&:\quot{G}{y}\To\quot{G}{X}
\end{align*}
induced by $f$ and $g$ on the quotients.
We know from proposition \sref{proposition:description_Groth_discrete} and the definition of $\stickypathcat_\bullet$ that
\begin{align*}
\M^\textbig_G(X)(f,g)&=\quad\coprod_{\mathclap{h\in\FinSet_G(x,y)}}\quad\disccat{\Big(\stickypathcat_{\op{\FinSet_G}}\big(\overline{\Map}_G(X)\big)(x)\Big)}(f,g\circ h)\\
&\subset\quad\coprod_{\mathclap{h\in\FinSet_G(x,y)}}\quad H\big(\Map^G(x,X);f,g\circ h\big)
\end{align*}
and
\begin{align*}
\M^\textbig(\quot{G}{X})(\quot{G}{f},\quot{G}{g})&=\quad\coprod_{\mathclap{h\in\FinSet(\quot{G}{x},\quot{G}{y})}}\quad\disccat{\Big(\stickypathcat_{\op{\FinSet}}\big(\overline{\Map}(\quot{G}{X})\big)(\quot{G}{x})\Big)}\big(\quot{G}{f},(\quot{G}{g}){\circ}h\big)\\
&\subset\quad\coprod_{\mathclap{h\in\Set(\quot{G}{x},\quot{G}{y})}}\ \ H\big(\Map(\quot{G}{x},\quot{G}{X});\quot{G}{f},(\quot{G}{g}){\circ}h\big)
\end{align*}
We thus get a commutative square
\begin{equation}\label{diagram:aux_comm_diag_M_H}
\begin{diagram}[midshaft]
\M^\textbig_G(X)(f,g)&\rInto{\ \inclusion\ }&\ \;\coprod_{\mathclap{h\in\FinSet_G(x,y)}}\ \,H\big(\Map^G(x,X);f,g\circ h\big)\\
\dTo_{\rho_X}&&\dTo[lowershortfall=.2em,uppershortfall=-.6em]_{q}\\
\M^\textbig(\quot{G}{X})(\quot{G}{f},\quot{G}{g})&\rInto{\ \ \,\inclusion\ \ }&\ \;\coprod_{\mathclap{h\in\Set(\quot{G}{x},\quot{G}{y})}}\ \,H\big(\Map(\quot{G}{x},\quot{G}{X});\quot{G}{f},(\quot{G}{g}){\circ}h\big)
\end{diagram}
\end{equation}
where $q$ makes the square
{\footnotesize\begin{diagram}[midshaft,balance]
H\big(\Map^G(x,X);f,g\circ j\big)&\rInto{\ \inclusion_j\ }&\ \;\coprod\limits_{\mathclap{h\in\FinSet_G(x,y)}}\ \,H\big(\Map^G(x,X);f,g\circ h\big)\\
\dTo_{H(\proj)}&&\dTo[lowershortfall=.2em,uppershortfall=-.25em]_{q}\\
H\big(\Map(\quot{G}{x},\quot{G}{X});\quot{G}{f},(\quot{G}{g}){\circ}(\quot{G}{j})\big)&\rInto{\ \ \inclusion_{\quot{G}{j}}\ }&\ \;\coprod\limits_{\mathclap{h\in\Set(\quot{G}{x},\quot{G}{y})}}\ \,H\big(\Map(\quot{G}{x},\quot{G}{X});\quot{G}{f},(\quot{G}{g}){\circ}h\big)
\end{diagram}
}commute for all $j\in\FinSet_G(x,y)$.
\end{construction}

\begin{lemma}\label{lemma:aux_pullback_square_q_rho}
Assume $X$ is a principal left $G$-space.\\
Let $x$, $y$ be objects of $\FinSet_G$, $f\in G\dash\Top(x,X)$, and $g\in G\dash\Top(y,X)$.\\
The square diagram \seqref{diagram:aux_comm_diag_M_H} is a pullback square in $\Top$.
\end{lemma}
\begin{proof}[Sketch of proof]
Since the horizontal maps in diagram \seqref{diagram:aux_comm_diag_M_H} are inclusions of subspaces, it is enough to show that \seqref{diagram:aux_comm_diag_M_H} gives a pullback square in $\Set$. Since the top horizontal map in diagram \seqref{diagram:aux_comm_diag_M_H} is injective, our task is reduced to proving that the induced map from $\M^\textbig_G(X)(f,g)$ to the pullback of
\begin{equation}\label{diagram:aux_pullback_diag_M_H}
\begin{diagram}[midshaft]
&&\ \;\coprod_{\mathclap{h\in\FinSet_G(x,y)}}\ \,H\big(\Map^G(x,X);f,g\circ h\big)\\
&&\dTo[lowershortfall=.2em,uppershortfall=-.6em]_{q}\\
\M^\textbig(\quot{G}{X})(\quot{G}{f},\quot{G}{g})&\rInto{\ \ \,\inclusion\ \ }&\ \;\coprod_{\mathclap{h\in\Set(\quot{G}{x},\quot{G}{y})}}\ \,H\big(\Map(\quot{G}{x},\quot{G}{X});\quot{G}{f},(\quot{G}{g}){\circ}h\big)
\end{diagram}
\end{equation}
is surjective.

Let then
\begin{align*}
h&\in\FinSet_G(x,y)\\
(\overline{\gamma},\tau)&\in H\big(\Map^G(x,X);f,g\circ h\big)
\end{align*}
be such that the Moore path induced from $(\overline{\gamma},\tau)$ on the quotients
\[ (\gamma,\tau)\in H\big(\Map(\quot{G}{x},\quot{G}{X});\quot{G}{f},\quot{G}{(g\circ h)}\big) \]
is a sticky homotopy for $\overline{\Map}(\quot{G}{X})$ at $\quot{G}{x}$ (see definition \sref{definition:sticky}). This data determines an (arbitrary) element $\mathfrak{z}$ of the pullback of diagram \seqref{diagram:aux_pullback_diag_M_H}. We will prove that $(\overline{\gamma},\tau)$ is a sticky homotopy for $\overline{\Map}_G(X)$ at $x$. That implies \[ (\overline{\gamma},\tau)\in\M^\textbig_G(X)(f,g) \] (via the inclusion that is the top map in \seqref{diagram:aux_comm_diag_M_H}), and this element must map to $\mathfrak{z}$ in the pullback of \seqref{diagram:aux_pullback_diag_M_H}. This proves the required surjectivity.

To check that $(\overline{\gamma},\tau)$ is a sticky homotopy for $\overline{\Map}_G(X)$ at $x$, consider any epimorphism $v:x\to z$ in $\FinSet_G$, and the pullback square
\begin{diagram}[midshaft,height=2.2em]
P&\rTo{\overline{h}}&\clop{0}{+\infty}\\
\dTo&&\dTo_{\overline{\gamma}}\\
\Map^G(z,X)&\rTo{\ \Map^G(v,X)\ }&\Map^G(x,X)
\end{diagram}
We must show the image of $\overline{h}$ is an interval which is empty or contains $\tau$. Let us analyze the commutative cube
\begin{diagram}[midshaft,hug,h=2.3em]
P&&\rTo{\overline{h}}&&\clop{0}{+\infty}\\
\dTo[lowershortfall=1em]&\rdTo&&&\vLine&\rdEqual\\
&&Q&\rTo{\hspace{-2.5em}h}&\HonV&&\clop{0}{+\infty}\\
&&&&\dTo_{\overline{\gamma}}\\
\mathclap{\Map^G(z,X)}\hLine[leftshortfall=3em]&&\VonH&\rTo{\Map^G(v,X)}&\Map^G(x,X)&&\dTo{\gamma}\\
&\rdTo[uppershortfall=1.1em]{\mathrlap{\proj}\ }&\dTo&&&\rdTo{\mathrlap{\proj}\ }\\
&&\Map(\quot{G}{z},\quot{G}{X})&&\rTo{\Map(\quot{G}{v},\quot{G}{X})}&&\Map(\quot{G}{x},\quot{G}{X})
\end{diagram}
whose front face is defined to be a pullback square. Since $(\gamma,\tau)$ is a sticky homotopy for $\overline{\Map}(\quot{G}{X})$ at $\quot{G}{x}$, the image of $h$ is an interval $J$ which is empty or contains $\tau$. In fact, $h:Q\to\clop{0}{+\infty}$ is a homeomorphism onto $J$. We will finish by proving that if $P$ is not empty then $\im\overline{h}=J$.

Considering the adjoint maps
\begin{align*}
\overline{\gamma}&:x\times\clop{0}{+\infty}\To uX\\
\gamma&:\quot{G}{x}\times\clop{0}{+\infty}\To\quot{G}{X}
\end{align*}
we are required to show that $\restrict{\overline{\gamma}}{x\times J}$ factors through $z\times J$. Instead, we know that $\restrict{\gamma}{\quot{G}{x}\times J}$ factors through $\quot{G}{z}\times J$. In particular, there exists a commutative diagram
\begin{diagram}[midshaft,h=2.1em]
\quot{G}{x}\times J&\rInto{\qquad}&\quot{G}{x}\times\clop{0}{+\infty}&\lTo{\ \proj\times\clop{0}{+\infty}\ }&x\times\clop{0}{+\infty}&\\
\dTo{\quot{G}{v}\times J}&&\dTo{\gamma}&&\dTo{\overline{\gamma}}\\
\quot{G}{z}\times J&\rTo{\gamma'}&\quot{G}{X}&\lTo{\proj}&uX
\end{diagram}
Let $\sigma:z\to x$ (in $\FinSet_G$) be a section of $v$
\[ v\circ \sigma=\id_z \]
which exists because $v$ is an epimorphism of free $G$-sets. The commutative diagram above implies that
\[ \proj\circ\restrict{\overline{\gamma}}{x\times J}=\proj\circ\restrict{\overline{\gamma}}{x\times J}\circ\big((\sigma\circ v)\times J\big) \]
where $\proj:uX\to\quot{G}{X}$ is the projection. Since $X$ is a principal left $G$-space, we conclude there is a continuous map
\[ f:x\times J\To G \]
which verifies a commutative diagram
\begin{diagram}[midshaft,h=2.1em,balance]
x\times J&\rTo{\ (f,\overline{\gamma})\ }&G\times uX\\
\dTo{(\sigma\circ v)\times J}&&\dTo{\mu}\\
x\times J&\rTo{\overline{\gamma}}&uX
\end{diagram}
where $\mu:G\times uX\to uX$ is the action of $G$ on $uX$.
In equation form
\[ \mu(f,\restrict{\overline{\gamma}}{x\times J})=\restrict{\overline{\gamma}}{x\times J}\circ\big((\sigma\circ v)\times J\big) \]

Finally, if $P$ is not empty, we know that for some $a\in J$, $\overline{\gamma}(a,-)$ factors through $v:x\to z$. Consequently
\[ \overline{\gamma}(a,-)=\overline{\gamma}\big(a,\sigma\circ v(-)\big) \]
which implies that
\[ f(a,-)=e \]
(where $e$ is the unit of $G$) in view of $G$ acting freely on $X$. Since $f$ is continuous and $J$ is connected, we conclude that $f=e$. Therefore $\restrict{\overline{\gamma}}{x\times J}$ factors through $z\times J$, and so $J\subset\im\overline{h}$.

In summary, if $P$ is not empty then $\im\overline{h}=J$.
\end{proof}

\begin{lemma}\label{lemma:proj_Map^G->Map_quotG_cover}
Assume the projection
\[ \proj:uX\To\quot{G}{X} \]
is a covering space.\\
For any object of $\FinSet_G$, $x$, the canonical map
\[ \proj:\Map^G(x,X)\To\Map(\quot{G}{x},\quot{G}{X}) \]
is a covering space.
\end{lemma}
\begin{proof}
A section $\sigma:\quot{G}{x}\to ux$ of $\proj:ux\to\quot{G}{x}$ induces a map
\[ \Map^G(x,X)\Into\Map(ux,uX)\xTo{\Map(\sigma,uX)}\Map(\quot{G}{x},uX) \]
which is a homeomorphism because $x$ is a free $G$-set. This homeomorphism sits in a commutative diagram
\begin{diagram}
\Map^G(x,X)&\rTo{\ \cong\ }&\Map(\quot{G}{x},uX)\\
&\rdTo_{\mathllap{\proj}}&\dTo_{\Map(\quot{G}{x},\proj)}\\
&&\Map(\quot{G}{x},\quot{G}{X})
\end{diagram}
Since $\quot{G}{x}$ is finite, the vertical map on the right is a covering space. Consequently
\[ \proj:\Map^G(x,X)\To\Map(\quot{G}{x},\quot{G}{X}) \]
is a covering space.
\end{proof}

We state the following proposition without proof.

\begin{proposition}\label{proposition:cover_open_H}
If $p:A\to B$ is a covering space and $a,b\in A$, then
\[ H(p):H(A;a,b)\To H(B;pa,pb) \]
is an open map.
\end{proposition}

The next result follows immediately from lemma \sref{lemma:proj_Map^G->Map_quotG_cover} and proposition \sref{proposition:cover_open_H}.

\begin{corollary}\label{lemma:H_quotG_homeo_image}
Assume the projection
\[ \proj:uX\To\quot{G}{X} \]
is a covering space.\\
Let $x$ be an object of $\FinSet_G$, and $f,g\in G\dash\Top(x,X)$.\\
The canonical map
\[ H(\proj):H\big(\Map^G(x,X);f,g\big)\To H\big(\Map(\quot{G}{x},\quot{G}{X});\quot{G}{f},\quot{G}{g}\big) \]
is an open map.
\end{corollary}

\begin{lemma}\label{lemma:aux_q_homeo}
Assume the projection
\[ \proj:uX\To\quot{G}{X} \]
is a covering space.\\
Let $x$, $y$ be objects of $\FinSet_G$, $f\in G\dash\Top(x,X)$, and $g\in G\dash\Top(y,X)$.\\
The map (from diagram \seqref{diagram:aux_comm_diag_M_H})
{\small
\[ q:\!\qquad\coprod\limits_{\mathclap{h\in\FinSet_G(x,y)}}\ \,H\big(\Map^G(x,X);f,g{\circ}h\big)\xTo{\ \quad}\ \ \coprod\limits_{\mathclap{h\in\Set(\quot{G}{x},\quot{G}{y})}}\ \,H\big(\Map(\quot{G}{x},\quot{G}{X});\quot{G}{f},(\quot{G}{g}){\circ}h\big) \]
}is a surjective open map.\\
Additionally, if the action of $G$ on $uX$ is free, the map $q$ is a homeomorphism.
\end{lemma}
\begin{proof}[Sketch of proof]
We conclude from the preceding corollary that $q$ is open. The surjectivity will follow from the existence of lifts of paths across the cover (lemma \sref{lemma:proj_Map^G->Map_quotG_cover})
\[ \proj:\Map^G(x,X)\To\Map(\quot{G}{x},\quot{G}{X}) \]
Given $h\in\Set(\quot{G}{x},\quot{G}{y})$, and
\[ (\gamma,\tau)\in H\big(\Map(\quot{G}{x},\quot{G}{X});\quot{G}{f},(\quot{G}{g}){\circ}h\big) \]
there exists a unique lift
\[ (\overline{\gamma},\tau)\in H\big(\Map^G(x,X)\big) \]
such that
\begin{align*}
\proj\circ\overline{\gamma}&=\gamma\\
\overline{\gamma}(0)&=f
\end{align*}
(since $\proj:\Map^G(x,X)\to\Map(\quot{G}{x},\quot{G}{X})$ is a covering space). Consequently
\[ \proj\circ\overline{\gamma}(\tau)=\gamma(\tau)=t(\gamma)=(\quot{G}{g})\circ h \]
One can now find $\overline{h}\in\FinSet_G(x,y)$ such that
\begin{align*}
\quot{G}{\overline{h}}&=h\\
\overline{\gamma}(\tau)&=g\circ\overline{h}
\end{align*}
because $x$ is a free $G$-set. If the action of $G$ on $uX$ is free, then there is a unique such $\overline{h}$.

We conclude that the image by $q$ of the point
\[ (\overline{\gamma},\tau)\in H\big(\Map^G(x,X);f,g\circ\overline{h}\big)\xhookrightarrow{\ \,\inclusion_{\overline{h}}\ \,}\ \,\coprod\limits_{\mathclap{h\in\FinSet_G(x,y)}}\ \,H\big(\Map^G(x,X);f,g\circ\overline{h}\big) \]
is the (arbitrarily chosen) point
{\footnotesize \[ (\gamma,\tau)\in H\big(\Map(\quot{G}{x},\quot{G}{X});\quot{G}{f},(\quot{G}{g}){\circ}h\big)\xhookrightarrow{\ \inclusion_h\ }\;\ \coprod\limits_{\mathclap{h\in\Set(\quot{G}{x},\quot{G}{y})}}\ \,H\big(\Map(\quot{G}{x},\quot{G}{X});\quot{G}{f},(\quot{G}{g}){\circ}h\big) \]
}Hence, $q$ is surjective.

If the action of $G$ on $uX$ is free, the uniqueness of the lift $(\overline{\gamma},\tau)$ and of $\overline{h}\in\FinSet_G(x,y)$ guarantee that
\[ q^{-1}\big(\set{(\gamma,\tau)}\big)=\set{(\overline{\gamma},\tau)} \]
(where the points $(\gamma,\tau)\in\target q$, and $(\overline{\gamma},\tau)\in\source q$ are as above). In conclusion, $q$ is injective, and therefore a homeomorphism.
\end{proof}

\begin{proposition}\label{proposition:M_G(X)_equiv_M(quotG_X)}
Let $G$ be a group in $\Set$. Let $X$ be a locally trivial principal left $G$-space.\\
The $\Top$-functor
\[ \rho_X:\M_G^\textbig(X)\To\M^\textbig(\quot{G}{X}) \]
is an essentially surjective local isomorphism. Furthermore, it restricts to an essentially surjective local isomorphism
\[ \rho_X:\M_G(X)\To\M(\quot{G}{X}) \]
\end{proposition}
\begin{proof}
The essential surjectiveness follows from proposition \sref{proposition:rho_ess_surj}. The local isomorphism property is a consequence of lemmas \sref{lemma:aux_pullback_square_q_rho} and \sref{lemma:aux_q_homeo}. For any $x$, $y$ in $\FinSet_G$, and any $f\in G\dash\Top(x,X)$, $g\in G\dash\Top(y,X)$ the square
\begin{diagram}[midshaft]
\M^\textbig_G(X)(f,g)&\rInto{\ \inclusion\ }&\ \;\coprod_{\mathclap{h\in\FinSet_G(x,y)}}\ \,H\big(\Map^G(x,X);f,g\circ h\big)\\
\dTo_{\rho_X}&&\dTo[lowershortfall=.2em,uppershortfall=-.6em]_{q}\\
\M^\textbig(\quot{G}{X})(\quot{G}{f},\quot{G}{g})&\rInto{\ \ \,\inclusion\ \ }&\ \;\coprod_{\mathclap{h\in\Set(\quot{G}{x},\quot{G}{y})}}\ \,H\big(\Map(\quot{G}{x},\quot{G}{X});\quot{G}{f},(\quot{G}{g}){\circ}h\big)
\end{diagram}
is cartesian by lemma \sref{lemma:aux_pullback_square_q_rho}. Lemma \sref{lemma:aux_q_homeo} states that the vertical map on the right, $q$, is a homeomorphism.
\end{proof}


%% file: MS1.tex



\chapter{Sticky configurations in $S^1$}\label{chapter:sticky_conf_S^1}

\section*{Introduction}

This chapter analyzes the simplest case of interest of the construction $\M(X)$ from chapter \ref{chapter:sticky_conf}, namely $\M(S^1)$. We study a few properties of $\M(S^1)$, and give different weakly equivalent models for it. The ultimate goal of this chapter is to show how to recover topological Hochschild homology from $\M(S^1)$.

\section*{Summary}

This is a short chapter dedicated to an analysis of the $\Top$-category of sticky configurations in $S^1$, $\M(S^1)$.

Section \sref{section:homotopical_discreteness} sets up some basic results on $\Top$-categories whose morphism spaces are homotopically discrete. In section \sref{section:ZZsticky_conf_RR}, we state that $\M(S^1)$ is equivalent to the $\Top$-category $\M_\ZZ(\RR)$ of $\ZZ$-equivariant sticky configurations in $\RR$. The results of section \sref{section:homotopical_discreteness} are used to prove that $\M_\ZZ(\RR)$ --- and consequently, $\M(S^1)$ --- is weakly equivalent to a $\Set$-category. One specific such $\Set$-category, $\EE$, is given in section \sref{section:cat_Elmendorf}: it is essentially the category introduced by Elmendorf in \cite{Elmendorf}.

Section \sref{section:M(S^1)_relation_associative_PROP} gives a functor from Elmendorf's category $\EE$ to the associative PROP. Section \sref{section:M(S^1)_relation_opDelta} constructs a homotopy cofinal functor from $\op{\Delta}$ to $\EE$. Finally, section \sref{section:cyclic_bar_const} uses the previous two sections to recover the cyclic bar construction and topological Hochschild homology via the category $\EE$.

\section{Homotopical discreteness}\label{section:homotopical_discreteness}

\begin{definition}[homotopically discrete space]
We say a topological space $X$ is {\em homotopically discrete} if it is homotopy equivalent to a discrete space (i.e.\ a set).
\end{definition}

\begin{proposition}
A topological space $X$ is homotopically discrete if and only if the canonical function
\[ \proj:X\To\pi_0(X) \]
is continuous and a homotopy equivalence.
\end{proposition}

\begin{definition}[locally homotopically discrete $\Top$-category]
Let $C$ be a $\Top$-category.\\
We say $C$ is {\em locally homotopically discrete} if for any objects $x$, $y$ of $C$, the space $C(x,y)$ is homotopically discrete.
\end{definition}

\begin{proposition}
Let $C$ be a locally homotopically discrete $\Top$-category.\\
There exists a canonical $\Top$-functor
\[ \proj:C\To\pi_0(C) \]
which is surjective on objects and a local homotopy equivalence.
\end{proposition}

The notion of local homotopical discreteness interacts appropriately with the Grothendieck construction from \sref{construction:variation_Groth_construction}.

\begin{proposition}\label{proposition:Groth(homotopical_discrete)=homotopical_discrete}
Let $C$ be a category in $\CAT$, and $F:\op{C}\to\Top\dash\Cat$ a functor.\\
If for every object $x$ of $C$, the $\Top$-category $F(x)$ is locally homotopically discrete, then the $\Top$-category $\Groth(F)$ is locally homotopically discrete.
\end{proposition}
\begin{proof}
Choose objects $x,y\in\ob C$, $a\in\ob(Fx)$, and $b\in\ob(Fy)$. According to proposition \sref{proposition:description_Groth_discrete}, the morphism space $\Groth(F)(a,b)$ is
\[ \Groth(F)(a,b)=\ \coprod_{\mathclap{f\in A(x,y)}}\ (Fx)\big(a,(Ff)b\big) \]
Therefore $\Groth(F)(a,b)$ is homotopically discrete, since the coproduct of homotopically discrete spaces is homotopically discrete.
\end{proof}

\section{$\ZZ$-equivariant sticky configurations in $\RR$}\label{section:ZZsticky_conf_RR}

In this section, we will prove that $\M(S^1)$ is locally homotopically discrete, by comparing it with $\ZZ$-equivariant sticky configurations in $\RR$.

\begin{definition}[$\RR$ as an object of $\ZZ\dash\Top$]
$\RR$ will be considered as an object of $\ZZ\dash\Top$ whose underlying space is $\RR$, and whose action of $\ZZ$ is given by addition:
\[ \ZZ\times\RR\Into\RR\times\RR\xTo{+}\RR \]
\end{definition}

Viewing $\RR$ as a $\ZZ$-space gives us a new presentation of $\M(S^1)$ as the category of $\ZZ$-equivariant sticky configurations on $\RR$, by applying proposition \sref{proposition:M_G(X)_equiv_M(quotG_X)} and noticing that $S^1=\quot{\ZZ}{\RR}$.

\begin{proposition}\label{proposition:rho_RR_MZ(RR)->M(S^1)}
The functor
\[ \rho_\RR:\M_\ZZ(\RR)\To\M(S^1) \]
is an essentially surjective local isomorphism.
\end{proposition}

The advantage of considering $\M_\ZZ(\RR)$ is that it is a full subcategory of the Grothendieck construction of categories which are homotopically discrete.

\begin{proposition}\label{proposition:homotopical_discrete_stcat_ZZ(RR)}
Let $x$ be an object of $\FinSet_\ZZ$.\\
For any objects $f$, $g$ of $\disccat{\big(\stcat_\ZZ(\RR)(x)\big)}$ (see definition \sref{definition:stcat_G}), the topological space $\disccat{\big(\stcat_\ZZ(\RR)(x)\big)}(f,g)$ is contractible or empty.\\
In particular, $\disccat{\big(\stcat_\ZZ(\RR)(x)\big)}$ is locally homotopically discrete.
\end{proposition}
\begin{proof}[Sketch of proof]
Let
\[ f,g\in\ob\disccat{\big(\stcat_\ZZ(\RR)(x)\big)}=\ZZ\dash\Top(x,\RR) \]
We will give a strong deformation retraction of the space
\[ X\defeq\disccat{\big(\stcat_\ZZ(\RR)(x)\big)}(f,g) \]
onto a singleton subspace, assuming that $X$ is not empty. This will finish the proof.

The space $X$ is the subspace of $H\big(\Map^\ZZ(x,\RR);f,g\big)$ consisting of the $\op{\FinSet_\ZZ}$-sticky homotopies for $\Map^\ZZ(-,\RR)$ at $x$ (see definitions in \sref{definition:sticky} and \sref{construction:FinSet_G_CAT^2_cart}). Recall (notation \sref{notation:H(Y;a,b)}) that $H\big(\Map^\ZZ(x,\RR);f,g\big)$ refers to the subspace of Moore paths starting at $f$ and ending at $g$.

Recall also the length map
\[ l:H\big(\Map^\ZZ(x,\RR);f,g\big)\To\clop{0}{+\infty} \]
and consider the subspace $l^{-1}(\set{1})$ of $H\big(\Map^\ZZ(x,\RR);f,g\big)$. Using repa\-ram\-e\-trization of paths, one can construct a strong deformation retraction of the space $H\big(\Map^\ZZ(x,\RR);f,g\big)$ onto $l^{-1}(\set{1})$, which preserves the subspace $X$. In particular, there is a strong deformation retraction of $X$ onto $X\cap l^{-1}(\set{1})$.

Now we will define a strong deformation retraction of $l^{-1}(\set{1})$. Set
\[ \lambda(\mathtt{t})=\left\{
\begin{array}{ll}
\!\mathtt{t}g+(1-\mathtt{t})f&\ \text{if }\mathtt{t}\in\clcl{0}{1}\\
\!g&\ \text{if }\mathtt{t}\in\clop{1}{+\infty}
\end{array}
\right. \]
for $\mathtt{t}\in\clop{0}{+\infty}$, and define
\[ \func{G}{l^{-1}(\set{1})\times I}{l^{-1}(\set{1})}{\big((\gamma,1),\tau\big)}{\big(\tau\lambda+(1-\tau)\gamma,1\big)} \]
$G$ is a strong deformation retraction of $l^{-1}(\set{1})$ onto the subspace $\set{(\lambda,1)}$. Moreover, $G$ preserves $X$:
\[ G\big(\big(X\cap l^{-1}(\set{1})\big)\times I\big)\subset X \]
In conclusion, if $X$ is not empty then $\set{(\lambda,1)}$ is a strong deformation retract of $X$, so $X$ is contractible.

In order to prove that $G$ preserves $X$, it is enough to note that $(\lambda,1)\in X$ if $X$ is not empty, and that
\[ (\tau\gamma+(1-\tau)\gamma',1)\in X \]
for any $(\gamma,1)\in X$, $(\gamma',1)\in X$, and $\tau\in\RR$. We leave the proof of this claim to the interested reader.
\end{proof}

\begin{corollary}\label{corollary:M(S^1)_locally_homotopically_discrete}
The $\Top$-category $\M_\ZZ(\RR)$ is locally homotopically discrete. Consequently, $\M(S^1)$ is also locally homotopically discrete.
\end{corollary}
\begin{proof}
The first statement follows from propositions \sref{proposition:homotopical_discrete_stcat_ZZ(RR)} and \sref{proposition:Groth(homotopical_discrete)=homotopical_discrete}, given that $\M_\ZZ(\RR)$ is a full $\Top$-subcategory of $\Groth\big(\disccat{\stcat_\ZZ}(\RR)\big)$ (see definition \sref{definition:M_G}). The second statement then follows from proposition \sref{proposition:rho_RR_MZ(RR)->M(S^1)}.
\end{proof}

\begin{corollary}
There are canonical $\Top$-functors
\begin{align*}
\proj&:\M_\ZZ(\RR)\To\pi_0\big(\M_\ZZ(\RR)\big)\\
\proj&:\M(S^1)\To\pi_0\big(\M(S^1)\big)
\end{align*}
which are surjective on objects and local homotopy equivalences. Additionally, the square diagram
\begin{diagram}[midshaft,h=2.3em]
\M_\ZZ(\RR)&\rTo{\ \proj\ }&\pi_0\big(\M_\ZZ(\RR)\big)\\
\dTo{\rho_\RR}&&\dTo{\pi_0(\rho_\RR)}\\
\M(S^1)&\rTo{\ \proj\ }&\pi_0\big(\M(S^1)\big)
\end{diagram}
commutes.
\end{corollary}

\section{A category of Elmendorf}\label{section:cat_Elmendorf}

\begin{construction}[category of Elmendorf]\label{construction:category_Elmendorf}
Consider the category $\ZZ\dash\poset$ of $\ZZ$-objects in the category $\poset$ of partially ordered sets.\\
For each $n\in\NN\setminus\set{0}$, let $x_n$ denote the object of $\ZZ\dash\poset=\HomCat{\B{\ZZ}}{\poset}$ determined by:
\begin{enum}
\item the underlying partially ordered set of $x_n$ is $u(x_n)=(\ZZ,\leq)$;
\item the action of $\ZZ$ is
\[ \func[\rTo{\smash{\qquad}}]{x_n}{\ZZ}{\poset\big((\ZZ,\leq),(\ZZ,\leq)\big)}{k}{(-+kn)\qquad} \]
\end{enum}
Then {\em Elmendorf's category}, $\EE$, is defined to be the full subcategory of $\ZZ\dash\poset$ generated by all objects $y$ of $\ZZ\dash\poset$ such that either
\begin{enum}
\item the underlying partially ordered set of $y$ is empty, or
\item $y$ is isomorphic in $\ZZ\dash\poset$ to $x_n$, for some $n\in\NN\setminus\set{0}$.
\end{enum}
\end{construction}

\begin{remark}
The category $\EE$ is the full subcategory of $\ZZ\dash\poset$ generated by the objects $P$ such that
\begin{enum}
\item the underlying partially ordered set, $uP$, is a total order;
\item for every $x\in uP$, the function (where $\mu$ is the action of $\ZZ$ on the set $uP$)
\[ \ZZ\Into\ZZ\times\set{x}\Into\ZZ\times uP\xTo{\mu}uP \]
is order preserving and cofinal (with respect to the orders).
\end{enum}
We can similarly characterize a full subcategory $\EE^\textbig$ of the category of $\ZZ$-objects in {\em preordered sets}. While we will prove that $\EE$ is weakly equivalent to $\M_\ZZ(\RR)$, $\EE^\textbig$ is actually weakly equivalent to $\M_\ZZ^\textbig(\RR)$.\comment{Proof: weak equivalence of presheaves on $\FinSet_\ZZ$, which follows essentially from the proof of \sref{proposition:homotopical_discrete_stcat_ZZ(RR)}.}
\end{remark}

\begin{proposition}
Recall the canonical forgetful functor
\[ \proj:\poset\To\Set \]
The functor
\[ \ZZ\dash\poset=\HomCat{\B{\ZZ}}{\poset}\xTo{\HomCat{\B{\ZZ}}{\proj}}\HomCat{\B{\ZZ}}{\Set}=\ZZ\dash\Set \]
restricts to a functor
\[ U:\EE\To\FinSet_\ZZ \]
\end{proposition}

\begin{remark}[comparison with Elmendorf]
The {\em linear category} $L$, introduced by Elmendorf in \cite{Elmendorf}, is the skeletal full subcategory of $\EE$ generated by the objects $x_n$ for $n\in\NN\setminus\set{0}$.\\
Consequently, $L$ is equivalent to the full subcategory of $\EE$ generated by the objects whose underlying partially ordered set is {\em non-empty}.
\end{remark}

We suggest the reader look at \cite{Elmendorf}, where he can obtain a wealth of information about $L$, and from there extrapolate to $\EE$. In \cite{Elmendorf} a strict unique factorization system is given on $L$, which recovers it as a crossed simplicial group (see \cite{Fiedorowicz-Loday} for material on crossed simplicial groups). Also, several familiar categories, such as $\Delta$ and $\op{\Delta}$ are embedded in $L$ --- we will describe the embedding of $\op{\Delta}$ in section \sref{section:M(S^1)_relation_opDelta}. Moreover, a duality is established which gives an isomorphism $L\simeq\op{L}$ --- this does {\em not} extend to a duality on $\EE$. Finally, it is also explained how Connes' cyclic category, $\Lambda$, is a quotient of $L$ by an action of the group $\B{\ZZ}$ in $\Cat$.

In case the reader prefers presentations of categories by generators and relations, one is also available for $L$, and is given in definition 1.5 of the article \cite{Bokstedt-Hsiang-Madsen}, where it is called $\Lambda_\infty$.

\section{Equivalence with Elmendorf's category}

\begin{construction}
We will now define the functor
\[ ZO:\pi_0\big(\M_\ZZ(\RR)\big)\To\EE \]
Consider an object of $\pi_0\big(\M_\ZZ(\RR)\big)$, that is, an injective $\ZZ$-equivariant map $f:x\to\RR$ for some $x$ in $\FinSet_\ZZ$. Since the underlying map of sets \[ uf:ux\To\RR \] is injective, it endows $ux$ with a unique total order such that $uf$ becomes order preserving. Call this total order $(ux,\leq)$.\\
We define $ZO(f)$ to be the unique object of $\EE$
\begin{enum}
\item whose underlying partially ordered set is $u\big(ZO(f)\big)=(ux,\leq)$, and
\item whose underlying $\ZZ$-set is $U\big(ZO(f)\big)=x$.
\end{enum}
Recall now that there is a natural functor (see definition \sref{definition:M^big_G})
\[ \pi:\M_\ZZ(\RR)\Into\M^\textbig_\ZZ(\RR)\xTo{\pi}\FinSet_\ZZ \]
This functor into a $\Set$-category necessarily factors through $\pi_0\big(\M_\ZZ(\RR)\big)$:
\[ \pi:\pi_0\big(\M_\ZZ(\RR)\big)\To\FinSet_\ZZ \]
Now let $x$, $y$ be in $\FinSet_\ZZ$, and
\begin{align*}
f&:x\To\RR\\
g&:y\To\RR
\end{align*}
be injective $\ZZ$-equivariant maps. Given $a:f\to g$ in $\pi_0\big(\M_\ZZ(\RR)\big)$, the morphism $\pi(a)\in\FinSet_\ZZ(x,y)$ gives an order preserving function relative to the orders induced from $\RR$
\[ u(\pi(a)):ux\To uy \]
We leave it to the reader to verify that the order preservation follows from the stickiness condition on the morphisms in $\M_\ZZ(\RR)$.\\
$\pi(a)$ thus determines a unique morphism $ZO(a)\in\EE\big(ZO(f),ZO(g)\big)$ such that the underlying map of $\ZZ$-sets is \[ U\big(ZO(a)\big)=\pi(a) \]
\end{construction}

\begin{proposition}\label{proposition:comm_diagram_U_ZO_pi}
The diagram
\begin{diagram}[midshaft,h=2.1em]
\pi_0\big(\M_\ZZ(\RR)\big)&\rTo{ZO}&\EE\\
&\rdTo(1,2)_{\mathllap{\pi}}\ldTo(1,2)_{\mathrlap{U}}\\
&\FinSet_\ZZ
\end{diagram}
commutes.
\end{proposition}

\begin{proposition}\label{proposition:ZO_equiv}
The functor
\[ ZO:\pi_0\big(\M_\ZZ(\RR)\big)\To\EE \]
is an equivalence of categories.
\end{proposition}
\begin{proof}
Let us prove first that $ZO$ is essentially surjective. Obviously, the object of $\EE$ whose underlying partially ordered set is empty is in the image of $ZO$. Choose then $n\in\NN\setminus\set{0}$. We will prove that $x_n$, as given in \sref{construction:category_Elmendorf}, is in the image of $ZO$. Consider the order preserving injective function
\[ \func{j_n}{\ZZ}{\RR}{k}{\nicefrac{k}{n}} \]
which induces a $\ZZ$-equivariant map on the $\ZZ$-set, $U(x_n)$, underlying $x_n$
\[ j_n:U(x_n)\To\RR \]
Thus $j_n$ is an object of $\pi_0\big(\M_\ZZ(\RR)\big)$ such that
\[ ZO(j_n)=x_n \]

Let $x,y\in\FinSet_\ZZ$, and
\begin{align*}
f&:x\To\RR\\
g&:y\To\RR
\end{align*}
be injective $\ZZ$-equivariant maps. We will now demonstrate that
\[ ZO:\pi_0\big(\M_\ZZ(\RR)(f,g)\big)\To\EE\big(ZO(f),ZO(g)\big) \]
is a bijection.

Given $a:ZO(f)\to ZO(g)$ in $\EE$, consider the underlying map of $\ZZ$-sets
\[ Ua:x\To y \]
and define the Moore path $(\lambda,1)\in H\big(\Map^\ZZ(x,\RR);f,g\circ Ua\big)$ by
\[ \lambda(\mathtt{t})=\left\{
\begin{array}{ll}
\!\mathtt{t}(g\circ Ua)+(1-\mathtt{t})f&\ \text{if }\mathtt{t}\in\clcl{0}{1}\\
\!g\circ Ua&\ \text{if }\mathtt{t}\in\clop{1}{+\infty}
\end{array}
\right. \]
for $\mathtt{t}\in\clop{0}{+\infty}$.
As a consequence of
\begin{align*}
uf&:u\big(ZO(f)\big)\To\RR\\
u(g\circ Ua)&:u\big(ZO(g)\big)\To\RR
\end{align*}
being order preserving, the Moore path $\lambda$ is actually $\op{\FinSet_\ZZ}$-sticky for $\Map^\ZZ(-,\RR)$ at $x$. Therefore, we obtain a point
\[ (\lambda,1)\in\M_\ZZ(\RR)(f,g)=\ \,\coprod_{\mathclap{h\in\FinSet_\ZZ(x,y)}}\ \,\disccat{\big(\stcat_\ZZ(\RR)(x)\big)}(f,g\circ h) \]
(see proposition \sref{proposition:description_Groth_discrete}) determined by $Ua$ and $(\lambda,1)$:
{\small
\[ (\lambda,1)\in\disccat{\big(\stcat_\ZZ(\RR)(x)\big)}(f,g\circ U\!a)\xhookrightarrow{\,\inclusion_{U\!a}\,}\ \coprod_{\mathclap{h\in\FinSet_\ZZ(x,y)}}\ \,\disccat{\big(\stcat_\ZZ(\RR)(x)\big)}(f,g\circ h) \]
}This morphism $(\lambda,1)\in\M_\ZZ(\RR)(f,g)$ verifies
\[ ZO\circ\proj(\lambda,1)=a \]
where
\[ \proj:\M_\ZZ(\RR)\To\pi_0\big(\M_\ZZ(\RR)\big) \]
In conclusion, the function
\[ ZO:\pi_0\big(\M_\ZZ(\RR)(f,g)\big)\To\EE\big(ZO(f),ZO(g)\big) \]
is surjective.

Now assume we are given two morphisms $a,b:f\to g$ in $\M_\ZZ(\RR)$ such that
\[ ZO\circ\proj(a)=ZO\circ\proj(b) \]
Proposition \sref{proposition:comm_diagram_U_ZO_pi} entails that
\[ \pi(a)=U\circ ZO\circ\proj(a)=U\circ ZO\circ\proj(b)=\pi(b) \]
where $\pi:\M_\ZZ(\RR)\to\FinSet_\ZZ$. Taking into account that (proposition \sref{proposition:description_Groth_discrete})
\[ \M_\ZZ(\RR)(f,g)=\ \,\coprod_{\mathclap{h\in\FinSet_\ZZ(x,y)}}\ \,\disccat{\big(\stcat_\ZZ(\RR)(x)\big)}(f,g\circ h) \]
we conclude
\[ a,b\in\disccat{\big(\stcat_\ZZ(\RR)(x)\big)}\big(f,g\circ\overbrace{\pi(a)}^{\smash{=\mathrlap{\pi(b)}}}\big) \xhookrightarrow{\ \inclusion_{\pi(a)}\ }\M_\ZZ(\RR)(f,g) \]
Since the space $\disccat{\big(\stcat_\ZZ(\RR)(x)\big)}(f,g\circ\pi a)$ is contractible by proposition \sref{proposition:homotopical_discrete_stcat_ZZ(RR)}, there exists a path in $\M_\ZZ(\RR)(f,g)$ from $a$ to $b$, i.e.
\[ \proj(a)=\proj(b)\in\pi_0\big(\M_\ZZ(\RR)(f,g)\big) \]
In summary
\[ ZO:\pi_0\big(\M_\ZZ(\RR)(f,g)\big)\To\EE\big(ZO(f),ZO(g)\big) \]
is injective.
\end{proof}

\section{Relation to associative PROP}\label{section:M(S^1)_relation_associative_PROP}

The category $\EE$ has the advantage of being concrete and easy to manipulate. That will be useful in constructing a functor to the category $\Ord\Sigma$ (consult construction \sref{construction:OrdSigma}) underlying the associative PROP.

\begin{construction}[functor $\psi:\EE\to\Ord\Sigma$]\label{construction:psi_EE->OrdSigma}
We will construct a functor
\[ \psi:\EE\To\Ord\Sigma \]
For each object $x\in\ob\EE$, $\psi(x)$ is the finite set
\[ \psi(x)\defeq\quot{\ZZ}{(Ux)} \]
where $U:\EE\to\FinSet_\ZZ$ is the forgetful functor.\\
Given a morphism $f:x\to y$ in $\EE$, we define the morphism \[ \psi(f)\in\Ord\Sigma\big(\quot{\ZZ}{(Ux)},\quot{\ZZ}{(Uy)}\big) \]
\begin{enum}
\item the map of sets underlying $\psi(f)$ is
\[ \quot{\ZZ}{(Uf)}:\quot{\ZZ}{(Ux)}\To\quot{\ZZ}{(Uy)} \]
\item given $a\in\quot{\ZZ}{(Uy)}$, the total order on $\big(\quot{\ZZ}{(Uf)}\big)^{-1}(\set{a})$ is the order induced from the bijection
\[ \proj:(uf)^{-1}(\set{\overline{a}})\xTo{\simeq}\big(\quot{\ZZ}{(Uf)}\big)^{-1}(\set{a}) \]
where $(uf)^{-1}(\set{\overline{a}})\subset ux$ has the order induced from $ux$. Here, $\overline{a}\in Uy$ is any representative of $a\in\quot{\ZZ}{(Uy)}$.
\end{enum}
\end{construction}

\begin{proposition}
The square diagram
\begin{diagram}[midshaft,h=2.1em]
\EE&\rTo{\psi}&\Ord\Sigma\\
\dTo_{U}&&\dTo_{\proj}\\
\FinSet_\ZZ&\rTo{\ \quot{\ZZ}{-}\ }&\FinSet
\end{diagram}
commutes.
\end{proposition}

\section{Relation to $\op{\Delta}$}\label{section:M(S^1)_relation_opDelta}

In this section we construct a homotopy cofinal functor $\op{\Delta}\To\EE$.

\begin{construction}[functor $\iota:\op{\Delta}\to\EE$]\label{construction:iota_opDelta->EE}
Let $n\in\NN\setminus\set{0}$. The object $n\in\op{\Delta}$ maps to
\[ \iota(n)\defeq x_n \]
where $x_n$ is as defined in construction \sref{construction:category_Elmendorf}.\\
For $i\in\set{0,\ldots,n}$, the $i$-th face map $d_i\in\op{\Delta}(n+1,n)$ gets mapped to $\iota(d_i)$ which is determined uniquely (thanks to $\ZZ$-equivariance) by
\[ \big(u\circ\iota(d_i)\big)(a)=\left\{
\begin{array}{ll}
\! a &\ \text{if }0\leq a\leq i\\
\! a-1 &\ \text{if }i<a\leq n+1
\end{array}
\right. \]
for $a\in\ZZ$. On the other hand, for $i\in\set{1,\ldots n}$, the $i$-th degeneracy map $s_i\in\op{\Delta}(n,n+1)$ gets mapped to $\iota(s_i)$ which is determined uniquely by
\[ \big(u\circ\iota(s_i)\big)(a)=\left\{
\begin{array}{ll}
\! a &\ \text{if }0\leq a<i\\
\! a+1 &\ \text{if }i\leq a\leq n
\end{array}
\right. \]
for $a\in\ZZ$.
\end{construction}

\begin{remark}
The functor $\iota:\op{\Delta}\To\EE$ gives an isomorphism between $\op{\Delta}$ and the subcategory of $\EE$ whose set of objects is $\set{x_n\suchthat n\in\NN\setminus\set{0}}$, and whose morphisms are the $f:x_i\to x_j$ such that $uf(0)=0$ (for $i,j\in\NN\setminus\set{0}$).\\
We can conclude that the functor $\iota$ lifts to an equivalence
\[ \op{\Delta}\xTo{\;\sim\;}(\iota 1)/\EE \]
\end{remark}

We state the following proposition without proof. The proof is essentially a copy of the proof of proposition 2.7 in \cite{Dwyer-Hopkins-Kan}. Equivalently, it is an instance of lemma 1.6 in \cite{Bokstedt-Hsiang-Madsen} (for the case of $\Lambda_\infty$, in their notation). We encourage the reader to analyze the relevant results in those articles and give a detailed proof of this proposition.

\begin{proposition}[homotopy cofinality]
The functor \[ \iota:\op{\Delta}\To\EE \] is homotopy cofinal.
\end{proposition}

\section{Cyclic bar construction}\label{section:cyclic_bar_const}

\begin{construction}[cyclic bar construction]
Let $(C,\tensor,I)$ be a symmetric monoidal category, and assume $A$ is an associative monoid in $C$. Then the bar construction
\[ \BarConst(A,A,A):\op{\Delta}\To\bimod{A}{A}  \]
takes values in $A$-bimodules or, equivalently, left $A\tensor\op{A}$-modules. We can then consider the objectwise tensor product (which always exists)
\[ \BarConst^{\text{cyc}}(A):\op{\Delta}\To C \]
\[ \BarConst^{\text{cyc}}(A)\defeq A\tensor_{A\tensor\op{A}}\BarConst(A,A,A) \]
This is the usual {\em cyclic bar construction} of $A$.
\end{construction}

\begin{remark}
The cyclic bar construction of an associative monoid $A$ verifies
\[ \BarConst^{\text{cyc}}(A)(n)\simeq A^{\tensor n} \]
for any object $n$ of $\op{\Delta}$.
\end{remark}

\begin{remark}
The above construction gives a functor from the category of associative monoids in $(C,\tensor,I)$ to the category of simplicial objects in $C$.
\end{remark}

Recall the functors $\psi:\EE\to\Ord\Sigma$ and $\iota:\op{\Delta}\to\EE$ from constructions \sref{construction:psi_EE->OrdSigma} and \sref{construction:iota_opDelta->EE}. We state the following proposition without proof.

\begin{proposition}\label{proposition:cyclic_bar_OrdSigma}
Let $C$ be a symmetric monoidal category, $\underline{A}$ a $\Ass$-algebra in $C$, and $A$ the underlying associative monoid of $\underline{A}$ (see example \sref{example:associative_PROP}).\\
Then there is an isomorphism
\[ \underline{A}\circ\psi\circ\iota\simeq\BarConst^{\text{cyc}}(A) \]
which is natural in the $\Ass$-algebra $\underline{A}$.
\end{proposition}

We will now apply these results to topological Hochschild homology. We assume that $(\spectra,\wedge,S)$ is a symmetric monoidal simplicial model category in which the unit $S$ is cofibrant: in particular, $\wedge$ verifies the pushout-product axiom. This holds for the category of symmetric spectra.

\begin{definition}[topological Hochschild homology]
Let $A$ be an associative monoid in the symmetric monoidal category of spectra, $(\spectra,\wedge,S)$.\\
The {\em topological Hochschild homology} of $A$ is the geometric realization of the cyclic bar construction of $A$:
\[ THH(A)\defeq\realization{\BarConst^{\mathrm{cyc}}(A)} \]
\end{definition}

\begin{proposition}\label{proposition:THH_hocolim_EE}
Let $\underline{A}$ be a $\Ass$-algebra in the symmetric monoidal category of spectra, $(\spectra,\wedge,S)$. Let $A$ denote the underlying associative monoid of $\underline{A}$.\\
There exists a canonical zig-zag
\[ THH(A)\longleftarrow\hocolim_{\op{\Delta}}\big(\BarConst^{\mathrm{cyc}}(A)\big)\To\hocolim_\EE\,(\underline{A}\circ\psi) \]
which is natural in the $\Ass$-algebra $\underline{A}$.\\
Both maps in the zig-zag are weak equivalences if the unit map of $A$, $S\to A$, is a cofibration in $\spectra$.
\end{proposition}

The zig-zag in the statement is standard: the left arrow is the Bousfield-Kan map (see definition 18.7.3 in \cite{Hirschhorn}); the right arrow is just the natural map between the homotopy colimits, together with the isomorphism from proposition \sref{proposition:cyclic_bar_OrdSigma}.

The proof of the statement is also quite standard. Assuming the unit of $A$ is a cofibration, the cyclic bar construction of $A$ is a Reedy cofibrant simplicial object in $\spectra$. Therefore the left arrow
\[ THH(A)\longleftarrow\hocolim_{\op{\Delta}}\big(\BarConst^{\mathrm{cyc}}(A)\big) \]
is a weak equivalence. On the other hand, the right arrow
\[ \hocolim_{\op{\Delta}}\big(\BarConst^{\mathrm{cyc}}(A)\big)\To\hocolim_\EE\,(\underline{A}\circ\psi) \]
is a weak equivalence since $\iota:\op{\Delta}\to\EE$ is homotopy cofinal and $\underline{A}$ is ob\-ject\-wise cofibrant.

It is possible to prove a similar statement without assuming that $S$ is cofibrant (e.g.\ for the category of spectra of \cite{EKMM}), but the proof becomes more involved.
\comment{Need something like a locally finite generalized Reedy category with a minimum object; need to demand Reedy cofibrancy except at the minimum object; need a left proper model category. The argument should then proceed by induction using pushout diagrams. Maybe a generalization of the section 17.9 of \cite{Hirschhorn} (on sequential colimits) could help in abstracting/axiomatizing this situation.}


%% file: embeddings.tex



\chapter{Spaces of embeddings of manifolds}\label{chapter:embeddings_manifolds}

\section*{Introduction}

This chapter is devoted primarily to discussing smooth manifolds and their embeddings, and constructing useful operads from those. 

We have given the construction $\M(X)$ of sticky configurations in chapter \ref{chapter:sticky_conf}, and have seen how one particular example, $\M(S^1)$, can be used to recover topological Hochschild homology of associative ring spectra (chapter \ref{chapter:sticky_conf_S^1}).

In order to generalize this picture, we need to first replace associative monoids by other algebraic structures. Therefore, one of the goals of this chapter is to introduce the relevant PROPs $\E^G_n$, which will be closely related to spaces of embeddings between manifolds. These PROPs will turn out to be similar to little discs operads.

We also need to define an invariant for $\E^G_n$-algebras which generalizes $THH$ of associative monoids. The necessary objects for such a definition are constructed in this chapter: to each manifold (with some geometric structure) we associate a right module over the aforementioned PROPs. These right modules will also play a role in connecting back to the construction $\M(X)$.

\section*{Summary}

Section \sref{section:spaces_embeddings} deals with defining spaces of embeddings between manifolds --- which will be the basic objects for this chapter --- and topologically enriched categories of $n$-manifolds and embeddings.

Sections \sref{section:PROPs_embeddings} and \sref{section:modules_PROPs_embeddings} are meant as simplified illustrations of the constructions which will appear later in the chapter. In section \sref{section:PROPs_embeddings}, the categories of embeddings from section \sref{section:spaces_embeddings} are used to build a PROP made up of spaces of embeddings between $n$-manifolds. Section \sref{section:modules_PROPs_embeddings} associates to each manifold a natural right module over those PROPs, whose homotopy type is determined in section \sref{section:homotopy_type_E_n}.

At this point, we genuinely begin the trek towards building the desired PROPs $\E^G_n$. Section \sref{section:Gstructures} defines the necessary geometric structures on $n$-manifolds, which are associated to each topological group $G$ over $GL(n,\RR)$, and are called $G$-structures. Sections \sref{section:constructions_Gstructures} and \sref{section:examples_Gstructures} dwell into some properties and examples of such geometric structures.

Section \sref{section:modified_embedding_spaces} uses the $G$-structures of section \sref{section:Gstructures} to define a homotopical modification of embedding spaces between manifolds: these are called ``$G$-augmented embeddings''. Section \sref{section:interregnum_hopullback} is an aside to talk about homotopy pullbacks over a fixed space, since such a concept is necessary to define the space of $G$-augmented embeddings.

In section \sref{section:categories_augmented_embeddings}, we assemble the spaces of $G$-augmented embeddings into topologically enriched categories, whose objects are $n$-manifolds with a $G$-structure. Out of these categories we extract the sought after PROPs $\E^G_n$ in section \sref{section:PROPs_augmented_embeddings}. We also compare these PROPs to known ones, and in particular prove that $\E^1_n$ is equivalent to the little n-discs PROP, $\D_n$.

Finally, sections \sref{section:right_modules_E^f_n} and \sref{section:internal_presheaves_E^f_n} describe a right module over $\E^G_n$ for each $n$-manifold with a $G$-structure. Section \sref{section:homotopy_modules_E^G_n} provides a simple analysis of the homotopy type of these right modules.

\section{Spaces of embeddings and categories of manifolds}\label{section:spaces_embeddings}

By ``manifold'' we will always mean a smooth manifold (possibly with boundary) whose underlying topological space is paracompact and Hausdorff. We will need to consider the space of embeddings between two manifolds.

\begin{definition}[space of embeddings]
Given two manifolds $M$, $N$, the {\em space of embeddings} $\Emb(M,N)$ is the topological space given by:
\begin{enum}
\item the elements of $\Emb(M,N)$ are smooth embeddings of $M$ into $N$, i.e.\ smooth maps $M\to N$ which are a homeomorphism onto the image, and whose derivative at each point of $M$ is injective;
\item the topology on $\Emb(M,N)$ is the compact-open $C^1$-topology, also called the weak $C^1$-topology (see \cite{Hirsch}).
\end{enum}
\end{definition}

Throughout the rest of this section we fix a (dimension) $n\in\NN$. We continue by observing that $n$-manifolds and spaces of embeddings form a $\Top$-enriched category.

\begin{definition}[$\Top$-category of $n$-manifolds and embeddings]
The $\Top$-category $\Emb_n$ of {\em \(n\)-dimensional manifolds and embeddings} is defined by:
\begin{enum}
\item the objects of $\Emb_n$ are $n$-dimensional manifolds without boundary;
\item given $n$-manifolds $M$, $N$ without boundary, $\Emb_n(M,N)\defeq\Emb(M,N)$;
\item composition is given by composition of embeddings.
\end{enum}
\end{definition}

$\Emb_n$ is actually a symmetric monoidal $\Top$-category: the symmetric monoidal structure is given by disjoint union of manifolds (and embeddings)
\[ \func{\amalg}{\Emb_n\times\Emb_n}{\Emb_n}{(M,N)}{M\amalg N} \]

A slight modification of $\Emb_n$, which will be a useful example later, is given by restricting to oriented manifolds and orientation preserving embeddings.

\begin{definition}[$\Top$-category of oriented $n$-manifolds and embeddings]
The $\Top$-category $\Emb^\oriented_n$ of {\em \(n\)-dimensional oriented manifolds and orientation preserving embeddings} is defined by:
\begin{enum}
\item the objects of $\Emb^\oriented_n$ are $n$-dimensional oriented manifolds without boundary;
\item given oriented $n$-manifolds $M$, $N$ without boundary, the morphism space $\Emb^\oriented_n(M,N)$ is the subspace of $\Emb(M,N)$ constituted by the orientation preserving embeddings;
\item composition is given by composition of embeddings.
\end{enum}
\end{definition}

$\Emb^\oriented_n$ is also a symmetric monoidal $\Top$-category via disjoint union of oriented manifolds. Moreover, the obvious map $\Emb^\oriented_n\to\Emb_n$ is a symmetric monoidal $\Top$-functor.

\section{Simple PROPs of embeddings}\label{section:PROPs_embeddings}

Some of the most important PROPs in this document will be given by considering variations on the full subcategory of $\Emb_n$ generated by disjoint unions of copies of $\RR^n$, where we introduce modifications to the spaces of embeddings. As an expository prelude, we now examine the example of two simpler PROPs derived directly from $\Emb_n$ and $\Emb^\oriented_n$. We again fix $n\in\NN$.

\begin{definition}[PROP of embeddings]\label{definition:PROP_embeddings}
The $\Top$-PROP $\E_n$ is the full symmetric monoidal $\Top$-subcategory of $\Emb_n$ generated by $\RR^n$ (in particular, $\ob(\E_n)=\set{(\RR^n)^{\amalg k}:k\in\NN}$, and the symmetric monoidal structure is given by disjoint union). The generator of $\E_n$ is the object $\RR^n$.
\end{definition}

\begin{definition}[PROP of orientation preserving embeddings]\label{definition:PROP_oriented_embeddings}
The $\Top$-PROP $\E^\oriented_n$ is defined as the full symmetric monoidal $\Top$-sub\-cat\-e\-go\-ry of $\Emb^\oriented_n$ generated by $\RR^n$. The generator of $\E^\oriented_n$ is also the object $\RR^n$.
\end{definition}

Observe that the obvious map $\inclusion:\E^\oriented_n\to\E_n$ is a map of $\Top$-PROPs. It is also easy to see that both $\E_n$ and $\E^\oriented_n$ are categories of operators (recall \sref{notation:category_operators}), and therefore can be recovered from their underlying operads.

We will now identify $\RR^n$ and the interior of the $n$-disc, $\interior D^n$, via some orientation preserving smooth diffeomorphism $\phi:\interior D^n\to\RR^n$. This induces maps of PROPs (the PROPs $\D_n^\bullet$ are introduced in section \sref{section:terminology_examples_PROPs})
\begin{align*}
F_\phi&:\D^{O(n)}_n\To\E_n\\
F_\phi&:\D^{SO(n)}_n\To\E^\oriented_n
\end{align*}
which are given on morphisms by conjugation by $\phi$: \[ F_\phi(f)=\phi^{\amalg l}\circ f\circ\big(\phi^{\amalg k}\big)^{-1} \] for any morphism $f:(D^n)^{\amalg k}\to(D^n)^{\amalg l}$ in $\D^{SO(n)}_n$ or $\D^{O(n)}_n$. Note that the following diagram commutes
\begin{diagram}[midshaft,height=2.1em]
\D^{SO(n)}_n&\rTo{F_\phi}&\E^\oriented_n\\
\dTo{\inclusion}&&\dTo_{\inclusion}\\
\D^{O(n)}_n&\rTo{F_\phi}&\E_n
\end{diagram}
It is easy to see that both maps $F_\phi$ are essentially surjective local homotopy equivalences of $\Top$-categories: this is essentially a consequence of the fact that the inclusions
\begin{align*}
SO(n)&\Into\Emb^\oriented(\RR^n,\RR^n)\\
O(n)&\Into\Emb(\RR^n,\RR^n)
\end{align*}
are homotopy equivalences (which follows, for example, from proposition \sref{proposition:D_0_equivalence}).

In conclusion, we have constructed weak equivalences of PROPs between $\E_n$, $\E^\oriented_n$ and certain PROPs of framed little $n$-discs, thus lending interest to these new PROPs. However, the PROP of little $n$-discs, $\D_n$, is not available through these simple constructions of PROPs of embeddings. We will need to modify the spaces of embeddings slightly in order to recover $\D_n$ from these methods: that will be the goal of section \sref{section:modified_embedding_spaces}.

\section{Right modules over PROPs of embeddings}\label{section:modules_PROPs_embeddings}

One advantage of the new PROPs $\E_n$ and $\E^\oriented_n$ (over the PROPs $\D^\bullet_n$) is that each (oriented) $n$-manifold naturally determines a right module (in spaces) over them. However, we first need to move to a cartesian closed category of spaces, since $\Top$ is not itself a $\Top$-category. Recall that $\kappa:\Top\to k\Top$ denotes the (product preserving) functor from $\Top$ into the cartesian closed category of compactly generated spaces (see \sref{notation:compactly_generated_spaces}).

\begin{definition}[right modules over $\kappa\E_n$]
Let $M$ be a $n$-manifold.\\
The restriction of the $k\Top$-functor
\[ \Yoneda_{\kappa\Emb_n}(M)=\kappa\Emb_n(-,M):\op{(\kappa\Emb_n)}\To k\Top \]
to the category $\op{(\kappa\E_n)}$ is called
\[ \kappa\E_n[M]:\op{(\kappa\E_n)}\To k\Top \]
\end{definition}

\begin{definition}[right modules over $\kappa\E^\oriented_n$]
Let $M$ be an oriented $n$-manifold.\\
The restriction of the $k\Top$-functor
\[ \Yoneda_{\kappa\Emb^\oriented_n}(M)=\kappa\Emb^\oriented_n(-,M):\op{(\kappa\Emb^\oriented_n)}\To k\Top \]
to the category $\op{(\kappa\E^\oriented_n)}$ is (also) denoted 
\[ \kappa\E^\oriented_n[M]:\op{(\kappa\E^\oriented_n)}\To k\Top \]
\end{definition}

Note that these constructions aggregate into $k\Top$-functors
\begin{align*}
\kappa\E_n[-]&:\Emb_n\To\HomCat[k\Top]{\op{(\E_n)}}{k\Top}\\
\kappa\E^{\oriented}_n[-]&:\Emb^\oriented_n\To\HomCatLarge[k\Top]{\op{\big(\E^\oriented_n\big)}}{k\Top}
\end{align*}

\section{Homotopy type of the right modules over $\E_n$}\label{section:homotopy_type_E_n}

In this section we study the homotopy type of
\[ \Emb_n((\RR^n)^{\smash{\amalg k}},M)=\kappa\E_n[M]((\RR^n)^{\smash{\amalg k}}) \]
For simplicity of notation, we will make the identification
\[ k\times\RR^n=(\RR^n)^{\smash{\amalg k}} \]
for $k\in\NN$.

We need a few preliminary definitions, which will also be useful throughout the remainder of the present chapter.

\begin{notation}[frame bundle]\label{notation:frame_bundle}
Let $E$ be a $n$-dimensional vector bundle.\\
$\Frame(E)$ denotes the frame bundle of $E$, which is the (locally trivial) principal $GL(n,\RR)$-bundle associated with the vector bundle $E$. Note that both $E$ and $\Frame(E)$ have the same base space.
\end{notation}

\begin{remark}[maps of vector bundles and principal bundles]
Assuming $E$, $E'$ are $n$-dimensional vector bundles, let us denote the space of vector bundle maps $E\to E'$ by $\Map^{\text{vec}}(E,E')$.\\
Within $\Map^{\smash{\text{vec}}}(E,E')$ we identify the subspace $\Map^{\smash{\text{vec}}}_{\smash{\text{iso}}}(E,E')$ constituted by the maps which are {\em fibrewise (linear) isomorphisms}.\\
Lastly, observe that the canonical map \[ \Map^{\smash{GL(n,\RR)}}\!\big(\Frame(E),\Frame(E')\big)\To\Map^{\smash{\text{vec}}}(E,E') \] induces a homeomorphism
\[ v:\Map^{\smash{GL(n,\RR)}}\!\big(\Frame(E),\Frame(E')\big)\xTo{\smash{\ \cong\ }}\Map^{\smash{\text{vec}}}_{\smash{\text{iso}}}(E,E') \]
\end{remark}

\begin{definition}[derivative of an embedding]\label{definition:derivative}
Let $M$, $N$ be $n$-dimensional manifolds.\\
We define the {\em derivative map}
\[ D:\Emb(M,N)\To\Map^{\smash{GL(n,\RR)}}\!\big(\Frame(TM),\Frame(TN)\big) \]
as the composition
\begin{align*}
\Emb(M,N)&\xTo{\smash{\ d\ }}\Map^{\text{vec}}_{\text{iso}}(TM,TN)\\
&\xTo{\;v^{-1}\,}\Map^{GL(n,\RR)}\!\big(\Frame(TM),\Frame(TN)\big)
\end{align*}
where $d$ takes an embedding $h:M\to N$ to its differential $dh:TM\to TN$, which is a fibrewise linear isomorphism.
\end{definition}

We will now use the derivative map to construct an approximation to the desired space $\Emb(R,M)$.
Consider for that purpose the {\em canonical inclusion at the origins}
\begin{equation}\label{equation:i_k}
\func[\rInto]{i_k}{k}{k\times\RR^n}{i}{(i,0)}
\end{equation}
\comment{
Observe also that $GL(n,\RR)^{\times k}$ acts on the left on $k\times\RR^n$, where each copy of $GL(n,\RR)$ acts on the corresponding copy of $\RR^n$
\[ \sfunc{GL(n,\RR)^{\times k}\times k\times\RR^n}{k\times\RR^n}{(g,i,x)}{(i,g_i\cdot x)} \]
Therefore, $GL(n,\RR)^{\times k}$ acts on the right on $\Emb(k\times\RR^n,M)$.
}
Additionally, $k\times\RR^n$ is canonically parallelized, and so $\Frame(T(k\times\RR^n))$ acquires a corresponding trivialization
\[ \Frame\big(T(k\times\RR^n)\big)=GL(n,\RR)\times k\times\RR^n \]

\begin{definition}[derivative at the origins]\label{definition:derivative_origins}
Let $M$ be a $n$-manifold without boundary, and $k\in\NN$.\\
Consider the composition
\begin{align*}
\Emb(k\times\RR^n,M)&\xTo[\sref{definition:derivative}]{D}\Map^{GL(n,\RR)}\!\big(GL(n,\RR)\times k\times\RR^n,\Frame(TM)\big)\\
&\xTo{\restrict{(-)}{k}}\Map^{GL(n,\RR)}\!\big(GL(n,\RR)\times k,\Frame(TM)\big)\\
&\xTo{\simeq}\big(\Frame(TM)\big)^{\times k}\\
&\xTo{\simeq}\Frame\big(T(M^{\times k})\big)
\end{align*}
where the second map is restriction along the inclusion $i_k:k\hookrightarrow k\times\RR^n$ (equation \seqref{equation:i_k}) of the base spaces. That composition induces a map
\[ D_0:\Emb\big((\RR^n)^{\amalg k},M\big)\To\Frame\big(T\Conf(M,k)\big) \]
(note that $\Conf(M,k)$ is an open submanifold of $M^{\smash{\times k}}$) which we call the {\em derivative at the origins}.
\end{definition}

The following is now a standard result.

\begin{proposition}\label{proposition:D_0_equivalence}
Let $M$ be a $n$-manifold without boundary, and $k\in\NN$.\\
The map
\[ D_0:\Emb\big((\RR^n)^{\amalg k},M\big)\To\Frame\big(T\Conf(M,k)\big) \]
is a Hurewicz fibration and a homotopy equivalence.
\end{proposition}
\begin{proof}[Ingredients for proof]
The map $D_0$ is equivariant with respect to the action of the topological group $\Diff(M)\times GL(n,\RR)^{\times k}$ on the source and target. We can therefore use it to prove local triviality for $D_0$. Given $x\in\Frame\big(T\Conf(M,k)\big)$ and a sufficiently small neighborhood $U$ of $x$, choose a map
\[ \varphi:U\To\Diff(M)\times GL(n,\RR)^{\times k} \]
such that
\begin{align*}
\varphi(y)\cdot x&=y\qquad\text{for }y\in U\\
\varphi(x)&=\text{unit}
\end{align*}
Use $\varphi$ to translate between the fibres of $D_0$ over the points of $U$.

The proof that $D_0$ is a homotopy equivalence proceeds in three steps. The first is to give a section $\sigma$ of $D_0$. The second is to give a homotopy over $\Frame\big(T\Conf(M,k)\big)$
\[ O:\Emb(k\times\RR^n,M)\times I\To\Emb(k\times\RR^n,M) \]
such that $O(-,0)=\id$, and for $f\in\Emb(k\times\RR^n,M)$
\[ \im\big(O(f,1)\big)\subset\im\big(\sigma(D_0 f)\big) \]
Intuitively, $O$ is shrinking the image of $f$ so that it fits within the image of $\sigma(D_0 f)$. This can be done by a simple manipulation of the domain of $f$.

The last step is to define the homotopy
\[ O':\Emb(k\times\RR^n,M)\times I\To\Emb(k\times\RR^n,k\times\RR^n) \]
by the formula
\begin{align*}
O'(f,1)&=\id\vphantom{\frac{1}{1}}\\
\quad O'(f,\tau)(x)&=\smash{\frac{1}{1-\tau}}\big(\sigma(D_0 f)^{-1}\circ O(f,1)\big)\big((1-\tau)x\big)&&\text{for }\tau\in\clop{0}{1},\\
&&&\hphantom{\text{for }}x\in k\times\RR^n
\end{align*}
for $f\in\Emb(k\times\RR^n,M)$. Concatenating $O$ with the homotopy given by the composition
\begin{align*}
\Emb(k\times\RR^n,M)\times I&\xTo{(O',\proj)}\Emb(k\times\RR^n,k\times\RR^n)\times\Emb(k\times\RR^n,M)\\
&\xTo{\id\times(\sigma\circ D_0)}\Emb(k\times\RR^n,k\times\RR^n)\times\Emb(k\times\RR^n,M)\\
&\xTo{\composition}\Emb(k\times\RR^n,M)
\end{align*}
(where ``$\composition$'' indicates composition of embeddings), gives a homotopy between the identity on $\Emb(k\times\RR^n,M)$ and $\sigma\circ D_0$.
\end{proof}

One can easily formulate an analog of this result for the case of orientation preserving embeddings.

\section{$G$-structures on manifolds}\label{section:Gstructures}

Soon, we will construct modifications of the embedding spaces of manifolds which will allow us to build new PROPs similar to $\E_n$.
First, and arguably most important, we need to add geometric structures to our manifolds, as was already hinted by our use of {\em oriented} $n$-manifolds to build the PROP $\E^\oriented_n$: these geometric structures will come in the form of reductions of the structure group of the tangent bundle. Second, we need to augment the embedding spaces of manifolds with a ``homotopical component'' relating to the aforementioned geometric structures.

We now define a convenient category of groups over $GL(n,RR)$.

\begin{definition}[topological groups over $GL(n,\RR)$]\label{definition:groups/GL(n,R)}
The {\em category of topological groups over \(GL(n,\RR)\)}, $\overGL{n}$, is defined by:
\begin{enum}
\item the objects of $\overGL{n}$ are maps of topological groups \[ G\To GL(n,\RR) \] from a topological group $G$ to $GL(n,\RR)$
\item given two maps of topological groups
\begin{align*}
f&:G\To GL(n,\RR)\\
g&:H\To GL(n,\RR)
\end{align*}
a morphism in $\overGL{n}(f,g)$ is a pair $(h,A)$ where \[ h:G\To H \] is a map of topological groups, and $A\in GL(n,\RR)$ conjugates $g\circ h$ to $f$
\[ A\cdot(g\circ h)=f\cdot A \]
\item the composition (of composable morphisms) $(h,A)$ and $(h',A')$ in $\overGL{n}$ is given by
\[ (h,A)\circ (h',A')=(h\circ h',A\cdot A') \]
\end{enum}
We will call the objects of $\overGL{n}$ {\em topological groups over \(GL(n,\RR)\)}.
\end{definition}

\begin{remark}
There is an obvious functor from $\overGL{n}$ to the category of topological groups, which associates to a map $f:G\to GL(n,\RR)$ the topological group $G$, called the {\em underlying topological group} of $f$.\\
We will call a morphism in $\overGL{n}$ a {\em weak equivalence} if the corresponding map of underlying topological groups is a weak equivalence.
\end{remark}

\begin{notation}\label{notation:overGL_f_G}
We will often denote an object of $\overGL{n}$ either by $f,g,\ldots$ or by $G,H,\ldots$ depending on the emphasis: if the map to $GL(n,\RR)$ is essential, we denote the object of $\overGL{n}$ by $f,g,\ldots$~; otherwise, we tend to use the letters $G,H,\ldots$\\
If $G,H,\ldots$ designates an object of $\overGL{n}$, we will often denote its underlying topological group by $G,H,\ldots$ as well.\\
Furthermore, if $G$ is a given topological group for which one can infer from context an obvious map $f:G\to GL(n,\RR)$, we will often denote the corresponding object of $\overGL{n}$ (namely $f$) simply by $G$, without further note.
\end{notation}

Let us now introduce the relevant geometric structures on manifolds. Since these structures only involve the tangent bundle, let us start with the corresponding definition on vector bundles. Recall (notation \sref{notation:frame_bundle}) $\Frame(E)$ denotes the frame bundle of a vector bundle E.

\begin{definition}[$G$-structure on a vector bundle]\label{definition:Gstructure}
Let $f:G\to GL(n,\RR)$ be a topological group over $GL(n,\RR)$ (i.e.\ an object of $\overGL{n}$).\\
Given a $n$-dimensional vector bundle $E$ (over a space $X$), a {\em \(f\)-structure on \(E\)} is a reduction of the structure group of $E$ across the map $f$. More precisely, we require
\begin{enum}
\item a locally trivial principal $G$-bundle $\Princ_G E$ over $X$, and
\item an isomorphism of principal $GL(n,\RR)$-bundles
\[ \lambda_f(E):f_\ast\big(\Princ_G E\big)\xTo{\;\smash{\simeq}\;}\Frame(E) \]
over the identity map on $X$.
\end{enum}
\end{definition}

\begin{remark}\label{remark:f-structure_f_ast}
For later use, we remark here that, given two vector bundles $E$, $F$ with $f$-structure, there is a natural map
\[ \Map^{\smash{G}}\!\big(\Princ_G E,\Princ_G F\big)\xTo{\;f_\ast\;}\Map^{\smash{GL(n,\RR)}}\!\big(\Frame(E),\Frame(F)\big) \]
defined as the composition
\begin{align*}
\Map^{\smash{G}}\!\big(\Princ_G E,\Princ_G F\big)&\xTo{\;f_\ast\;}\Map^{\smash{GL(n,\RR)}}\!\big(f_\ast(\Princ_G E),f_\ast(\Princ_G F)\big)\\
&\xTo{\;\cong\;}\Map^{GL(n,\RR)}\!\big(\Frame(E),\Frame(F)\big)
\end{align*}
where the bottom homeomorphism is induced by the isomorphisms $\lambda_f(E)$ and $\lambda_f(F)$ which are part of the $f$-structures on $E$ and $F$.
\end{remark}

\begin{definition}[$G$-structure on a manifold]\label{definition:manifold_Gstructure}
Let $G$ be a topological group over $GL(n,\RR)$ (i.e.\ an object of $\overGL{n}$).\\
Given a $n$-manifold without boundary $M$, a {\em \(G\)-structure on \(M\)} is defined as a $G$-structure on the tangent vector bundle $TM$.
\end{definition}

\begin{remark}
The preceding definitions of $G$-structure --- where $G$ is a topological group over $GL(n,\RR)$ --- could be easily generalized, with only mild modifications, to the case of spaces over $BGL(n,\RR)$. However, we will not do so in the interest of simplicity.
\end{remark}

These geometric structures are functorial on maps of groups: if $h:G\to H$ is a morphism in $\overGL{n}$, then a $G$-structure on a $n$-manifold $M$ can be pushed forward along $h$ to a $H$-structure on $M$. The same statement holds for $G$-structures on vector bundles. In addition, if $h$ is a weak equivalence (of underlying topological groups) then any $H$-structure on a $n$-manifold $M$ can be lifted --- essentially uniquely --- to a $G$-structure, whose push-forward along $h$ is the original $H$-structure.

\comment{Here we use that: Milnor's construction classifies numerable principal $G$-bundles over paracompact spaces; Milnor's construction gives a locally trivial fibration with fibre $G$, and in particular a Serre fibration; the long exact sequence in homotopy now proves that a weak equivalence $G\to H$ induces a weak equivalence on the base spaces of the corresponding Milnor constructions.}

\section{Constructions on $G$-structures}\label{section:constructions_Gstructures}

In this section, we give some simple constructions involving the geometric structures defined in the previous section.

\begin{definition}[$G$-structure on disjoint union]\label{definition:Gstructure_disjoint_union}
Let $f:G\to GL(n,\RR)$ be an object of $\overGL{n}$ (where $n\in\NN$).\\
Assume $M$, $N$ are $n$-manifolds with a $f$-structure.\\
The {\em induced \(f\)-structure on the disjoint union} $M\amalg N$ is given by
\begin{enum}
\item the principal $G$-bundle $\Princ_G\big(T(M\amalg N)\big)\defeq\Princ_G(TM)\amalg\Princ_G(TN)$ over $M\amalg N$;
\item the required isomorphism $\lambda_f(T(M\amalg N))$ of principal $GL(n,\RR)$-bundles is
\begin{align*}
f_\ast\big(\Princ_G\big(T(M\amalg N)\big)\big)&\xTo{\simeq}f_\ast\big(\Princ_G(TM)\big)\amalg f_\ast\big(\Princ_G(TN)\big)\\
&\xTo{\lambda_f(TM)\amalg\lambda_f(TN)}\Frame(TM)\amalg\Frame(TN)\\
&\xTo{\simeq}\Frame\big(T(M\amalg N)\big)
\end{align*}
\end{enum}
\end{definition}

\begin{definition}[induced $G$-structure on open submanifold]\label{definition:Gstructure_submanifold}
Let $f:G\to GL(n,\RR)$ be a topological group over $GL(n,\RR)$ (where $n\in\NN$).\\
Let $M$ be a $n$-manifold equipped with a $f$-structure, and $N$ an open submanifold of $M$.\\
The {\em induced \(f\)-structure on the open submanifold} $N$ of $M$ is defined by:
\begin{enum}
\item $\Princ_G(TN)$ is the restriction to $N$ of the principal $G$-bundle $\Princ_G(TM)$;
\item the isomorphism of principal $GL(n,\RR)$-bundles $\lambda_f(TN)$ is the restriction to $N$ of the isomorphism
\[ \lambda_f(TM):f_\ast(\Princ_G(TM))\xto{\smash{\,\simeq\,}}\Frame(TM) \]
coming from the $f$-structure on $M$.
\end{enum}
\end{definition}

\begin{definition}\label{definition:f_boxtimes_g}
Let $m,n\in\NN$, and assume
\vspace{-.1em}\begin{align*}
f&:G\To GL(m,\RR)\\
g&:H\To GL(n,\RR)
\end{align*}
are maps of topological groups.\\
We define the topological group over $GL(m+n,\RR)$ \[ f\boxtimes g: G\times H\To GL(m+n,\RR) \] to be the composition
\[ G\times H\xTo{f\times g}GL(m,\RR)\times GL(n,\RR)\xhookrightarrow{\,\ \oplus\ \,}GL(m+n,\RR) \]
where $\oplus$ is the canonical inclusion.
\end{definition}

\begin{notation}
In particular, we denote the underlying group of $G\boxtimes H$ by $G\times H$.
\end{notation}

\begin{definition}[geometric structure on product]\label{definition:product_GHstructure}
Let $m,n\in\NN$, and assume
\vspace{-.1em}\begin{align*}
f&:G\To GL(m,\RR)\\
g&:H\To GL(n,\RR)
\end{align*}
are maps of topological groups.\\
Additionally, let $M$ be a $m$-manifold with a $f$-structure and $N$ a $n$-manifold with a $g$-structure.\\
The {\em product \(f\boxtimes g\)-structure} on $M\times N$ is given by:
\begin{enum}
\item the $G\times H$-principal bundle $\Princ_{G\times H}\big(T(M\times N)\big)$ over $M\times N$ is the product
\[ \Princ_{G\times H}\big(T(M\times N)\big)\defeq\Princ_G(TM)\times\Princ_H(TN) \]
\item the isomorphism $\lambda_{f\boxtimes g}(T(M\times N))$ of principal $GL(m+n,\RR)$-bundles is the composition
\begin{align*}
(f\boxtimes g)_\ast\big(\Princ_G(TM)\!\times\!\Princ_H(TN)\big)\hspace{-.2em}&\xTo{\simeq}\oplus_\ast\big(f_\ast\big(\Princ_G(TM)\big)\times g_\ast\big(\Princ_H(TN)\big)\big)\\
&\xTo{\!\oplus_\ast\hspace{-.05em}(\lambda_f({\scriptscriptstyle TM})\hspace{-.05em}\times\hspace{-.05em}\lambda_g({\scriptscriptstyle TN}))}\oplus_\ast\hspace{-.08em}\big(\Frame(TM)\!\times\!\Frame(TN)\big)\\
&\xTo{\simeq}\Frame(TM\times TN)\\
&\xTo{\simeq}\Frame\big(T(M\times N)\big)
\end{align*}
where the non-named isomorphisms are canonical with respect to products of principal bundles (first isomorphism), vector bundles (third isomorphism), and manifolds (last isomorphism), respectively.
\end{enum}
\end{definition}

In the next example, we apply these constructions to the space of configurations $\Conf(M,k)$ of a $n$-manifold $M$ with a $G$-structure (for $G$ in $\overGL{n}$). We denote by $G^{\boxtimes k}$ the object in $\overGL{kn}$ which is the result of iterating the construction in definition \sref{definition:f_boxtimes_g}.

\begin{example}[$G^{\smash{\boxtimes k}}$-structure on $\Conf(M,k)$]\label{example:fk-structure_conf}
Let $M$ be a $n$-manifold with a $G$-structure, and $k\in\NN$.\\
The previous definition gives a $G^{\smash{\boxtimes k}}$-structure on $M^{\smash{\times k}}$. Define the $G^{\smash{\boxtimes k}}$-struc\-ture on $\Conf(M,k)$ to be the structure induced on the open submanifold $\Conf(M,k)$ of $M^{\smash{\times k}}$.
\end{example}

\section{Examples of $G$-structures}\label{section:examples_Gstructures}

Observe that every $n$-manifold has a canonical $GL(n,\RR)$-structure. This section gives a few more examples of less trivial $G$-structures.

\begin{example}[orientations and $GL^+(n,\RR)$]
Consider the subgroup $GL^+(n,\RR)$ of $GL(n,\RR)$.\\
An orientation on a $n$-dimensional manifold $M$ determines an essentially unique $GL^+(n,\RR)$-structure on $M$ (inducing that orientation on $M$). It is essentially unique in the following sense: any two reductions of the structure group of $TM$ to $GL^+(n,\RR)$ which induce the same orientation on $M$ are uniquely isomorphic.
\end{example}

\begin{example}[Riemannian structures and $O(n)$]\label{example:O(n)structure_Riemann}
Consider the subgroup $O(n)$ of $GL(n,\RR)$.\\
A Riemannian structure on a $n$-manifold $M$ is equivalent to a $O(n)$-structure on $M$. More precisely, a reduction of the structure group of $TM$ to $O(n)$ determines a Riemannian structure on $M$, and any two reductions giving the same Riemannian structure are uniquely isomorphic.
\end{example}

Note that, in view of the two preceding examples, one concludes that a $SO(n)$-structure on a $n$-manifold $M$ is equivalent to giving a Riemannian structure and an orientation on $M$. Additionally, the previous example can be easily modified to give, for example, a similar relation between symplectic structures on a manifold and $Sp(2n,\RR)$-structures (where $Sp(2n,\RR)$ is the group of symplectic real $2n\times 2n$ matrices).

We now analyze a very important example of the geometric structures under consideration.

\begin{example}[parallelizations and the trivial group]\label{example:1structure_parallelization}
A trivialization of a $n$-dimensional vector bundle $E$ is equivalent to a reduction of the structure group of $E$ to the trivial group $1$. In particular, giving a $1$-structure on $M$ is equivalent to giving a parallelization of $M$ (i.e.\ a trivialization of $TM$).
\end{example}

\begin{example}[$G$-structure on $\RR^n$]\label{example:Gstructure_R^n}
Note that the manifold $\RR^n$ is naturally parallelized, and therefore is equipped with a canonical $1$-structure (by the previous example).\\
Consequently, for any map of topological groups $f:G\to GL(n,\RR)$, the manifold $\RR^n$ has a canonical $f$-structure obtained by pushing-forward (along $1\to G$) the canonical $1$-structure on $\RR^n$. In particular, the bundle $\Princ_G(T\RR^n)$ is canonically trivialized: \[ \Princ_G(T\RR^n)=\RR^n\times G \]
\end{example}

\section{Augmented embedding spaces}\label{section:modified_embedding_spaces}

Having defined the necessary geometric structures on manifolds, in this section we will introduce the related ``homotopical'' modifications of the spaces of embeddings of manifolds. We fix throughout this section a topological group over $GL(n,\RR)$, $f:G\to GL(n,\RR)$ (where $n\in\NN$ is also fixed). Recall from \sref{notation:overGL_f_G} that we will sometimes denote this object of $\overGL{n}$ simply by $G$, if the map $f$ is not essential at the time.

From a naive perspective, one might expect that the correct space of embeddings would be the subspace of $\Emb(M,N)$ of the embeddings which preserve the $G$-structures, in some sense. We now exemplify what this could mean.

\begin{example}[preservation of $G$-structure by an embedding]\label{example:preserve_f-structure}
Assume that $f:G\to GL(n,\RR)$ is the inclusion of a closed subgroup. Then the induced map
\[ \Princ_G(TM)\Into f_\ast\big(\Princ_G(TM)\big)\xTo[\smash{\sref{definition:Gstructure}}]{\lambda_f(TM)}\Frame(TM) \]
is the inclusion of a closed subspace, for any $n$-manifold $M$ with a $f$-structure. So if $M^n$, $N^n$ both have a $G$-structure, it makes sense to say that an embedding $e\in\Emb(M,N)$ {\em preserves the \(f\)-structure} when the derivative map (as in definition \sref{definition:derivative}) \[ De:\Frame(TM)\To\Frame(TN) \] carries the subspace $\Princ_G(TM)$ of $\Frame(TM)$ into the closed subspace $\Princ_G(TN)$ of $\Frame(TN)$. In this case, $De$ determines a $G$-map \[ \restrict{De}{\Princ_G(TM)}:\Princ_G(TM)\To\Princ_G(TN) \]
\end{example}

However, important cases of $f$-structures, such as parallelizations (example \sref{example:1structure_parallelization}), are often very rigid, not allowing for the existence of many embeddings which preserve the $f$-structures. Next we give an example of this kind of rigidity.

\begin{example}
There are no embeddings $\RR^2\to S^2$ which preserve the usual Riemannian structures on these manifolds (since $S^2$ has constant non-zero curvature). In particular, there are no embeddings $\RR^2\to S^2$ which preserve the corresponding $O(2)$-structures (see example \sref{example:O(n)structure_Riemann}).\\
Informally, one could say $S^2$ has no $O(2)$-charts.
\end{example}

To avoid this issue of rigidity, we demand that the $f$-augmented embed\-dings (defined next) preserve the $f$-structures only up to homotopy.

\begin{definition}[$G$-augmented embedding spaces]\label{definition:modified_embedding_space}
Let $M$, $N$ be $n$-manifolds equipped with a $f$-structure (see definition \sref{definition:manifold_Gstructure}).\\
The {\em space of \(f\)-augmented embeddings}, $\REmb^f_n(M,N)$, is the homotopy pullback over $\Map(M,N)$ of the following diagram over $\Map(M,N)$
{\setlength\abovedisplayskip{7pt plus 0pt minus 7pt}
\begin{diagram}[midshaft,height=2.2em]
&&\Map^G\!\big(\Princ_G(TM),\Princ_G(TN)\big)\\
&&\dTo^{f_\ast}_{\sref{remark:f-structure_f_ast}}\\
\Emb(M,N)&\rTo_{\smash{\ \ \sref{definition:derivative}\ \ }}^{D}&\Map^{\smash{GL(n,\RR)}}\!\big(\Frame(TM),\Frame(TN)\big)
\end{diagram} }
\end{definition}

Before proceeding, we need to explain the meaning of the homotopy pullback appearing in the definition. That is the purpose of the next section. Before that, let us just give a definition close in spirit to \sref{definition:derivative}.

\begin{definition}[$G$-augmented derivative]\label{definition:fderivative}
Let $M$, $N$ be $n$-manifolds equipped with a $G$-structure.\\
We call the canonical projection
\[ \REmb^G_N(M,N)\To\Map^G\!\big(\Princ_G(TM),\Princ_G(TN)\big) \]
the {\em \(G\)-augmented derivative} map, and denote it by $\RDer^G$.
\end{definition}

\section{Interlude: homotopy pullbacks over a space}\label{section:interregnum_hopullback}

The diagram in definition \sref{definition:modified_embedding_space} sits over $\Map(M,N)$ (as stated there), by which it is meant that the diagram
\begin{diagram}
\Map^{GL(n,\RR)}\!\big(\Frame(TM),\Frame(TN)\big)&\lTo{\ f_\ast\ }&\Map^G\!\big(\Princ_G(TM),\Princ_G(TN)\big)\\
\uTo^{D}&\rdTo{p_1}&\dTo_{p_2}\\
\Emb(M,N)&\rTo{\inclusion}&\Map(M,N)
\end{diagram}
commutes, where the maps $p_1$ and $p_2$ associate to a map of principal bundles the corresponding map on the base spaces. Equivalently (letting $\bullet$ stand for the appropriate space of maps between principal bundles)
\[ (\Emb(M,N),\inclusion)\xTo{\;D\;}(\bullet,p_1)\xlongleftarrow{\;f_\ast\;}(\bullet,p_2) \]
is a diagram in the over-category $\Top/\Map(M,N)$. We will now describe the homotopy pullback of this last diagram over $\Map(M,N)$, which is the object specified in definition \sref{definition:modified_embedding_space} to be $\REmb^f_n(M,N)$.

Let $W$ be a topological space, and suppose we are given a diagram $\mathcal{D}$ in $\Top/W$
\[ \mathcal{D}\,:\ \quad (X,p_X)\xTo{g}(Y,p_Y)\xlongleftarrow{h}(Z,p_Z) \]
The homotopy pullback of $\mathcal{D}$ over $W$ is defined to be the limit in $\Top$
\[ \hopb_{\nicefrac{}{W}}(\mathcal{D})\defeq\lim\!\left(
\begin{diagram}[midshaft,h=1.8em,w=1em]
X&\rTo{g}&Y&&W\\
&&\uTo_{\evaluation{0}}&&\dTo_{\text{const}}\\
&&\Map(I,Y)&\rTo{\;\Map(I,p_Y)\;}&\Map(I,W)\\
&&\dTo_{\evaluation{1}}\\
Z&\rTo{h}&Y
\end{diagram} \right) \]
More informally, $\hopb_{\nicefrac{}{W}}(\mathcal{D})$ is the subspace of the usual homotopy pullback of $X\xto{\;\smash{g}\;}Y\xleftarrow{\;\smash{h}\;}Z$ which sits over the constant paths in $\Map(I,W)$. In particular, $\hopb_{\nicefrac{}{W}}(\mathcal{D})$ naturally maps to $W$.

\begin{remark}[homotopical properties of $\hopb_{\nicefrac{}{W}}$]
The specific model above, $\hopb_{\nicefrac{}{W}}$, gives indeed a homotopy pullback (in the sense of model categories) for the category $\Top/W$, modulo fibrancy conditions in $\Top/W$: we need that
\begin{align*}
p_Y&:Y\To W\\
p_Z&:Z\To W
\end{align*}
are Serre fibrations.\\
If $p_Y$, $p_Z$ are indeed Serre fibrations (respectively, Hurewicz fibrations), the natural inclusion of $\hopb_{\nicefrac{}{W}}(\mathcal{D})$ into the usual homotopy pullback (in $\Top$) of $X\xto{\;\smash{g}\;}Y\xleftarrow{\;\smash{h}\;}Z$ is a weak equivalence (respectively, homotopy equivalence).\\
\comment{$p_Y$ being a fibration guarantees that the space $Y^I_W$ of paths in $Y$ within the fibres over $W$ is such that $Y^I_W\to Y\times_W Y$ is a fibration. The homotopy pullback over $W$ is the pullback of $X\times_W Z\to Y\times_W Y\leftarrow Y^I_W$.\\
If $p_Y$ and $p_X$ are fibrations fibration then the square
\begin{diagram}
X\times_W Z&\rTo&Y\times_W Y\\
\dTo&&\dTo\\
X\times Z&\rTo&Y\times Y
\end{diagram}
is homotopy cartesian.}
Also, if $p_Y:Y\To W$ is a Serre (respectively, Hurewicz) fibration then the projection
\[ \hopb_{\nicefrac{}{W}}(\mathcal{D})\To X\underset{W}{\times} Z \]
is also a Serre (respectively, Hurewicz) fibration.
\end{remark}

\begin{remark}[homotopical properties of $\REmb^G_n(M,N)$]\label{remark:homotopy_properties_REmb}
In the case appearing in definition \sref{definition:modified_embedding_space}, we do obtain a homotopy pullback in $\Top/\Map(M,N)$ since the necessary fibrancy conditions are verified: namely the projections
\begin{align*}
\Map^{GL(n,\RR)}\!\big(\Frame(TM),\Frame(TN)\big)&\To\Map(M,N)\\
\Map^G\!\big(\Princ_G(TM),\Princ_G(TN)\big)&\To\Map(M,N)
\end{align*}
are Hurewicz fibrations.\\
So the inclusion of $\REmb^G_n(M,N)$ into the usual homotopy pullback of the diagram in \sref{definition:modified_embedding_space} is a homotopy equivalence.\\
Furthermore, the canonical projection
\begin{equation}\label{equation:fibration_REmb}
\REmb^G_n(M,N)\To\Map^G\!\big(\Princ_G(TM),\Princ_G(TN)\big)\underset{\mathclap{\Map(M,N)}}{\times}\,\Emb(M,N)
\end{equation}
is a Hurewicz fibration. Since the map
\[ \Map^G\!\big(\Princ_G(TM),\Princ_G(TN)\big)\To\Map(M,N) \]
is a Hurewicz fibration, it follows that the canonical projection
\[ q:\REmb^G_n(M,N)\To\Emb(M,N) \]
is also a Hurewicz fibration.
\end{remark}

\begin{notation}[$G$-augmented embeddings]\label{notation:modified_embeddings}
We will denote elements of $\hopb_{\nicefrac{}{W}}(\mathcal{D})$ (defined above) by triples $(x,\gamma,z)$ where $x\in X$, $z\in Z$ and $\gamma\in\Map(I,Y)$.\\
In particular, we will denote elements of $\REmb^G_n(M,N)$ (that is, $G$-augmented embeddings) by triples $(e,\gamma,g)$ where:
\begin{enum}
\item $e\in\Emb(M,N)$;
\item $g:\Princ_G(TM)\to\Princ_G(TN)$ is a map of principal $G$-bundles over $e$;
\item $\gamma$ is a path in $\Map^{GL(n,\RR)}\!\big(\Frame(TM),\Frame(TN)\big)$: it goes from $De$ to the map induced by $g$, and sits over the constant path in $\Map(M,N)$ with value $e$.
\end{enum}
\end{notation}

As a simple illustration of this notation, assume that $f:G\to GL(n,\RR)$ is the inclusion of a closed subgroup, and that $M^n$, $N^n$ have $G$-structures. Then any embedding $e\in\Emb(M,N)$ which preserves the $G$-structure (in the sense of example \sref{example:preserve_f-structure}) determines a $G$-augmented embedding \[ \big(e,\text{const}_{De},\restrict{De}{\Princ_G(TM)}\big)\in\REmb^f_n(M,N) \] where $De$ is the derivative of $e$ (as defined in \sref{definition:derivative}), and $\text{const}_{De}$ is the constant path equal to $De$. We will denote this augmented embedding simply by $e$.

\section{Categories of augmented embeddings}\label{section:categories_augmented_embeddings}

In this section, we explain how $n$-dimensional manifolds with a $G$-struc\-ture, together with $G$-augmented embeddings define a symmetric monoidal $\Top$-category. Fix $n\in\NN$ and an object $G$ of $\overGL{n}$.

Let us begin by describing the composition of $G$-augmented embeddings. 

\begin{definition}[composition of $G$-augmented embeddings]
Assume $M$, $N$, and $P$ are $n$-manifolds equipped with a $G$-structure.\\
Let $(e,\gamma,g)\in\REmb^G_n(M,N)$, and $(e',\gamma',g')\in\REmb^G_n(N,P)$ (recall remark \sref{notation:modified_embeddings} on notation for augmented embeddings).\\
The composite of the $G$-augmented embeddings $(e,\gamma,g)$ and $(e',\gamma',g')$ is
\[ (e',\gamma',g')\circ(e,\gamma,g)\defeq (e'\circ e,\gamma'\circ\gamma,g'\circ g)\in\REmb^G_n(M,P) \]
where $\gamma'\circ\gamma$ denotes the pointwise composition
\[ (\gamma'\circ\gamma)(\tau)=\gamma'(\tau)\circ\gamma(\tau)\qquad\mathrlap{\text{for } \tau\in I=\clcl{0}{1}}\qquad \]
\end{definition}

\begin{definition}[$\Top$-category of $G$-augmented embeddings]
The $\Top$-category $\REmb^G_n$ is defined by:
\begin{enum}
\item the objects of $\REmb^G_n$ are the $n$-manifolds (without boundary) with a $G$-struc\-ture;
\item given $n$-manifolds with a $G$-structure, $M$ and $N$, the morphism space $\REmb^G_n(M,N)$ is the space of $G$-augmented embeddings already defined;
\item composition in $\REmb^G_n$ is given by composition of $G$-augmented embeddings, as described above.
\end{enum}
\end{definition}

The symmetric monoidal structure on the $\Top$-category $\REmb^G_n$ is given by disjoint union of manifolds:
\[ \func{\amalg}{\REmb^G_n\times\REmb^G_n}{\REmb^G_n}{(M,N)}{M\amalg N} \]
which is well defined since the disjoint union of manifolds with $G$-structures has an induced $G$-structure (definition \sref{definition:Gstructure_disjoint_union}).

\begin{remark}[functoriality of $\REmb^\bullet_n$]
Like the geometric structures on manifolds, these categories are functorial on maps of groups. More precisely, we have a functor
\[ \REmb^\bullet_n:\overGL{n}\To\Top\dash\SymmMonCAT \]
which associates to a topological group over $GL(n,\RR)$ the symmetric monoi\-dal $\Top$-category $\REmb^G_n$. The functor induced by a morphism $h:G\to H$ in $\overGL{n}$ is called
\[ h_\ast:\REmb^G_n\To\REmb^H_n \]
If $h:G\to H$ is a weak equivalence, then $h_\ast$ is an essentially surjective local weak equivalence of $\Top$-categories.
\end{remark}

\comment{Here we use that $\Map^G(\Princ_G(TM),\Princ_G(TN))\to\Map(M,N)$ is a locally trivial fibration, and in particular a Serre fibration. The fibre over each point is empty or $\Map(M,G)$ (note also that $\Map(M,-)$ preserves weak equivalences since $M$ is a CW-complex). Therefore for a weak equivalence $G\to H$, it follows that $\Map^G(\Princ_G(TM),\Princ_G(TN))\to\Map^H(\Princ_H(TM),\Princ_H(TN))$ is a weak equivalence (after proving that it is surjective on $\pi_0$: this follows from the classification of numerable bundles via Milnor's construction, which preserves weak equivalences of groups). It now follows that $\REmb^G_n(M,N)\to\REmb^H_n(M,N)$ is a weak equivalence (by homotopy invariance of the homotopy pullback).}

These categories can be related with $\Emb_n$ and $\Emb^\oriented_n$. Specifically, the canonical projection from the space of augmented embeddings to the (usual) space of embeddings gives a functor
\[ q:\REmb^G_n\To\Emb_n \]
This defines a cocone for the functor $\REmb^\bullet_n$. In other words, $q$ gives a natural transformation
\begin{equation}\label{equation:q_Remb_Emb}
q:\REmb^\bullet_n\To\Emb_n
\end{equation}
from $\REmb^\bullet_n$ to the constant functor equal to $\Emb_n$.

If $G$ actually sits over $GL^+(n,\RR)$, then $q$ factors through $\Emb^\oriented_n$.

We can still say more: the description (at the end of the previous section \sref{section:interregnum_hopullback}) of a $G$-augmented embedding associated to any embedding which preserves $G$-structures (example \sref{example:preserve_f-structure}) determines inclusions
\begin{equation}\label{equation:E->E^G}
\begin{split}
\Emb_n&\Into\REmb^{\smash{GL(n,\RR)}}_n\\
\Emb^\oriented_n&\Into\REmb^{GL^+(n,\RR)}_n
\end{split}
\end{equation}
given that any embedding preserves the $GL(n,\RR)$-structures, and any orientation preserving embedding between oriented $n$-manifolds preserves the associated $GL^+(n,\RR)$-structures. These inclusions are symmetric monoidal $\Top$-functors which are essentially surjective local homotopy equivalences. Furthermore, the composition
\[ \Emb_n\Into\REmb^{\smash{GL(n,\RR)}}_n\xTo{\;q\;}\Emb_n \]
is equal to the identity functor.

\section{PROPs of augmented embeddings}\label{section:PROPs_augmented_embeddings}

We will now define the PROPs $\E^G_n$ in analogy with our definition of the PROPs $\E_n$ and $\E^\oriented_n$ before (see definitions \sref{definition:PROP_embeddings} and \sref{definition:PROP_oriented_embeddings}). We fix a dimension $n\in\NN$ throughout this section.

\begin{definition}[PROPs of $G$-augmented embeddings]\label{definition:PROPs_modified_embeddings}
Let $G$ be a topological group over $GL(n,\RR)$.\\
The $\Top$-PROP $\E^G_n$ is the full symmetric monoidal $\Top$-subcategory of $\REmb^G_n$ generated by $\RR^n$ (see example \sref{example:Gstructure_R^n}). The generator of $\E^{\smash{G}}_n$ is the object $\RR^n$.
\end{definition}

In particular, $\ob(\E^{\smash{G}}_n)=\set{(\RR^n)^{\smash{\amalg k}}:k\in\NN}$, and the symmetric monoidal structure on $\E^{\smash{G}}_n$ is given by disjoint union. It is straightforward to prove that $\E^{\smash{G}}_n$ is a category of operators and, in particular, can be reconstructed from its underlying $\Top$-operad.

\begin{remark}[functoriality and homotopy invariance of $\E^G_n$]
The $\Top$-PROP $\E^{\smash{G}}_n$ is functorial in the topological group $G$ over $GL(n,\RR)$, since the same is true for $\REmb^{\smash{G}}_n$.\\
Furthermore, a weak equivalence $h:G\xto{\smash{\,\sim\,}} H$ in $\overGL{n}$ induces a weak equivalence $h_\ast:\E^{\smash{G}}_n\xto{\smash{\,\sim\,}}\E^{\smash{H}}_n$.
\end{remark}

The rest of this section is dedicated to comparing these new PROPs to prior ones. There is a canonical map of $\Top$-PROPs
\[ \E^{\smash{GL^+(1,\RR)}}_1\xTo{\sim}\Ass \]
which is a weak equivalence. Also, the inclusions in \seqref{equation:E->E^G} restrict to weak equivalences
\begin{align*}
\E_n&\xhookrightarrow{\smash{\ \sim\ }}\E^{\smash{GL(n,\RR)}}_n\\
\E^\oriented_n&\xhookrightarrow{\ \sim\ }\E^{GL^+(n,\RR)}_n
\end{align*}
In particular, $\E^{GL(n,\RR)}_n\simeq\D^{O(n)}_n$ and $\E^{GL^+(n,\RR)}_n\simeq\D^{SO(n)}_n$ (via the weak equivalences described at the end of section \sref{section:PROPs_embeddings}).

We can now also compare $\E^1_n$ with the little $n$-discs PROP, $\D_n$. Consider the closed subgroup $\RR^+$ of $GL(n,\RR)$ given by the positive multiples of the identity. Additionally, fix an isomorphism
\[ \qquad\phi:\interior D^n\xTo{\,\simeq\,}\RR^n\quad\text{in }\REmb^{\smash{\RR^+}}_n \]
Then conjugation with $\phi$ defines a map of $\Top$-PROPs
\[ F_\phi:\D_n\To\E^{\RR^+}_n \]
given by\vspace{-.2em}
\[ \hspace{3.5em} F_\phi(f)=\phi^{\amalg l}\circ f\circ\big(\phi^{\amalg k}\big)^{-1}\quad\text{for } f{\,\in\,}\D_n\hspace{-.1em}\big(\hspace{-.08em}(\hspace{-.1em}D^n\hspace{-.08em})^{\amalg k}\!,\hspace{-.15em}(\hspace{-.1em}D^n\hspace{-.08em})^{\amalg l}\big) \]
(note that $f$ restricts to an embedding of the interiors which preserves the $\RR^+$-structure, and thus determines a $\RR^+$-augmented embedding of the same name).
$F_\phi$ can be seen to be a weak equivalence, in part by observing that
\[ \REmb^{\RR^+}_n(\RR^n,\RR^n)\simeq\REmb^1_n(\RR^n,\RR^n)\simeq\bigast\simeq\D_n(D^n,D^n) \]
The weak equivalence $(1\hookrightarrow\RR^+)_\ast:\E^1_n\xto{\,\smash{\sim}\,}\E^{\smash{\RR^+}}_n$ now gives a zig-zag
\[ \D_n\xTo[\smash{F_\phi}]{\sim}\E^{\RR^+}_n\xlongleftarrow{\sim}\E^1_n \]
which shows that $\E^1_n\simeq\D_n$.

\section{Right modules over PROPs of augmented embeddings}\label{section:right_modules_E^f_n}

Fix (throughout this section) an object $G$ of $\overGL{n}$ (where $n\in\NN$).
As before (with $\E_n$ and $\E^\oriented_n$), a $n$-manifold with a $G$-structure determines a right module over $\E^G_n$, after passing to a cartesian closed category of spaces. Again, $\kappa:\Top\to k\Top$ denotes the inclusion of $\Top$ into a cartesian closed category of spaces.

\begin{definition}[right modules over $\kappa\E^G_n$]\label{definition:kappaE^f_n}
Let $M$ be a $n$-manifold with a $G$-structure.\\
The restriction of the $k\Top$-functor
\[ \Yoneda_{\kappa\REmb^G_n}(M)=\kappa\REmb^G_n(-,M):\op{\big(\kappa\REmb^G_n\big)}\To k\Top \]
to the category $\op{\big(\kappa\E^G_n\big)}$ is called
\[ \kappa\E^G_n[M]:\op{\big(\kappa\E^G_n\big)}\To k\Top \]
\end{definition}

Observe that the functoriality of the Yoneda embedding actually extends this construction to a $k\Top$-functor:
\[ \kappa\E^G_n[-]:\kappa\REmb^G_n\To\HomCatLarge[k\Top]{\op{\big(\kappa\E^G_n\big)}}{k\Top} \]
In addition, given a morphism $h:G\to H$ in $\overGL{n}$, there is an induced natural transformation (merely witnessing the fact that $h_\ast:\REmb^{\smash{G}}_n\to\REmb^{\smash{H}}_n$ is a $\Top$-functor)
\begin{diagram}[midshaft]
\op{\big(\kappa\E^G_n\big)}&\rTo{\ \op{(h_\ast)}\ }&\op{\big(\kappa\E^H_n\big)}\\
&\rdTo(1,2)_{\mathllap{\kappa\E^G_n[M]}}\ \twocell[.5em]{\kappa\varepsilon_h\quad}{.6em}{\mathclap{\smash{\ \ \xRightarrow{\smash{\hspace{1em}}}}}}{32}\ldTo(1,2)_{\mathrlap{\kappa\E^H_n[h_\ast M]}}\\
&k\Top
\end{diagram}
which is functorial in $h$ (i.e.\ these natural transformations compose appropriately) and $M$. This means, more precisely, that we have for each $h:G\to H$ as above a $k\Top$-natural transformation
\begin{diagram}
\HomCatLarge[k\Top]{\op{\big(\kappa\E^G_n\big)}}{k\Top}&\lTo{\ \HomCat[k\Top]{\op{\kappa(h_\ast)}}{k\Top}\;}&\HomCatLarge[k\Top]{\op{\big(\kappa\E^H_n\big)}}{k\Top}\\
\uTo{\kappa\E^G_n[-]}&\twocell[-.2em]{\kappa\varepsilon_h\ }{.2em}{\Longrightarrow}{0}&\uTo_{\kappa\E^H_n[-]}\\
\kappa\REmb^G_n&\rTo_{\kappa(h_\ast)}&\kappa\REmb^H_n
\end{diagram}
and these compose in the obvious manner when we stack two of these squares side by side (given two composable morphisms in $\overGL{n}$). Finally, we remark that $\kappa\varepsilon_h$ is an objectwise weak equivalence if $h:G\to H$ is a weak equivalence.

\comment{
Diagrams similar to \seqref{diagram:commuting_groups_kappaE} and \seqref{diagram:commuting_groups_kappaE[-]} can be constructed when comparing the functors $\kappa\E_n[-]$ and $\kappa\E^{\smash{GL(n,\RR)}}_n[-]$ (by replacing $h_\ast$ with the inclusions): there exists a natural transformation
\begin{diagram}[midshaft]
\HomCat[k\Top]{\op{(\kappa\E_n)}}{k\Top}&\lTo{\ \HomCat[k\Top]{\op{\kappa\,\inclusion}}{k\Top}\;}&\HomCatLarge[k\Top]{\op{\big(\kappa\E^{GL(n,\RR)}_n\big)}}{k\Top}\\
\uTo{\kappa\E_n[-]}&\twocell{}{0pt}{\quad\Longrightarrow}{0}&\uTo_{\kappa\E^{GL(n,\RR)}_n[-]}\\
\kappa\Emb_n&\rInto_{\kappa\,\inclusion}&\kappa\REmb^{GL(n,\RR)}_n
\end{diagram}
which is an objectwise weak equivalence.
}

\section{Internal presheaves on $\E^G_n$}\label{section:internal_presheaves_E^f_n}

Fix throughout this section an object $G$ of $\overGL{n}$ (where $n\in\NN$).

Having to first switch to a cartesian closed category of spaces is slightly unsatisfying (even if irrelevant from a homotopical viewpoint), since all our constructions so far are within $\Top$. We can remedy this situation --- and remain within $\Top$ --- by using internal presheaves in $\Top$. The actual intention of introducing these internal presheaves is to later establish a connection with the construction $\M(X)$ given in chapter \ref{chapter:sticky_conf}.

Recall that any $\Top$-category, $C$, gives rise to a category object, $\internal C$, in $\TOP$ (or $\Top$, if $\ob C$ happens to be small) with a discrete space of objects. Recall also that for an internal category $\mathcal{C}$, each object $x$ of $\mathcal{C}$ determines an internal (Yoneda) presheaf $\Yoneda_{\mathcal{C}}(x)=\mathcal{C}(-,x)$ on $\mathcal{C}$ (see example \sref{example:internal_Yoneda_presheaf}).

\begin{definition}[internal presheaves over $\internal\E^G_n$]\label{definition:internal_presheaves_E^f_n}
Let $M$ be a $n$-manifold with a $G$-structure.\\
The restriction of the internal $\TOP$-valued functor
\[ \Yoneda_{\internal\REmb^G_n}(M)=\internal\REmb^G_n(-,M):\op{\big(\internal\REmb^G_n\big)}\To\TOP \]
to the category $\op{\internal\E^G_n}$ (internal to $\Top$) is an internal $\Top$-valued presheaf called
\[ \internal\E^G_n[M]:\op{\big(\internal\E^G_n\big)}\To\Top \]
\end{definition}

\begin{remark}
If we push the internal presheaf $\internal\E^{\smash{G}}_n[M]$ along $\kappa:\Top\to k\Top$, we obtain an internal presheaf on $\kappa\internal\E^{\smash{G}}_n=\internal\kappa\E^{\smash{G}}_n$. Since $k\Top$ is cartesian closed, this internal presheaf induces a $k\Top$-enriched presheaf on $\kappa\E^{\smash{G}}_n$. This induced presheaf on $\kappa\E^{\smash{G}}_n$ is exactly $\kappa\E^{\smash{G}}_n[M]$, as defined in \sref{definition:kappaE^f_n}.
\end{remark}

Similar statements hold regarding the functoriality --- on $M$ and $G$ --- of $\internal\E^{\smash{G}}_n[M]$ as were made for $\kappa\E^{\smash{G}}_n[M]$ immediately following definition \sref{definition:kappaE^f_n}. More concretely, there are functors
\[ \internal\E^G_n[-]:\big(\REmb^G_n\big)_0\To\Cat(\Top)\left(\op{\big(\internal\E^G_n\big)},\Top\right) \]
where the right hand side denotes the category of internal $\Top$-valued pre\-sheaves on $\internal\E^{\smash{G}}_n$, and the subscript $0$ on the left hand side gives the underlying $\Set$-category. Also, given a morphism $h:G\to H$ in $\overGL{n}$, there is a canonical natural transformation (which is merely a vestige of the $\Top$-functor $h_\ast:\REmb^{\smash{G}}_n\to\REmb^{\smash{H}}_n$)
\begin{equation}\label{diagram:natural_transf_internal_group}
\begin{diagram}
\Cat(\Top)\big(\op{\internal\E^G_n},\Top\big)&\lTo{\scriptscriptstyle\,\Cat(\Top)(\op{\internal(h_\ast)},\Top)\,}&\Cat(\Top)\big(\op{\internal\E^H_n},\Top\big)\\
\uTo{\internal\E^G_n[-]}&\twocell[-.2em]{\internal\varepsilon_h\ }{.2em}{\Longrightarrow}{0}&\uTo_{\internal\E^H_n[-]}\\
\big(\REmb^G_n\big)_0&\rTo_{(h_\ast)_0}&\big(\REmb^H_n\big)_0
\end{diagram}
\end{equation}
and these compose adequately when we place two of these diagrams side by side (given two composable morphisms in $\overGL{n}$). As before, $\internal\varepsilon_h$ is an objectwise weak equivalence if $h:G\to H$ is a weak equivalence.

\begin{remark}\label{remark:internal_presheaf_E_n}
Note that we can define, for any $n$-dimensional manifold $M$, an analogous internal $\Top$-valued presheaf (with similar functoriality in $M\in\Emb_n$)
\[ \internal\E_n[M]:\op{(\internal\E_n)}\To\Top \]
for the case of the $\Top$-PROP $\E_n$.\\
The comparison between this and $\internal\E^{\smash{GL(n,\RR)}}_n[M]$ takes the form of an internal natural transformation
\begin{diagram}[midshaft]
\op{\big(\internal\E_n\big)}&\rTo{\ \op{\inclusion}\ }&\op{\big(\internal\E^{GL(n,\RR)}_n\big)}\\
&\rdTo(1,2)_{\mathllap{\internal\E_n[M]}}\ \twocell[.5em]{}{0pt}{\mathclap{\smash{\ \ \xRightarrow{\smash{\hspace{1em}}}}}}{32}\ldTo(1,2)_{\mathrlap{\internal\E^{GL(n,\RR)}_n[M]}}\\
&\Top
\end{diagram}
which is actually a weak equivalence of internal $\Top$-valued presheaves.\\
On the other hand, $q:\REmb^\bullet_n\to\Emb_n$ (expression \seqref{equation:q_Remb_Emb}) gives a natural transformation
\begin{equation}\label{diagram:E^G_n_E_n}
\begin{diagram}[midshaft]
\op{\big(\internal\E^G_n\big)}&\rTo{\ \op{q}\ }&\op{\big(\internal\E_n\big)}\\
&\rdTo(1,2)_{\mathllap{\internal\E^G_n[-]}}\ \twocell[.5em]{q\ \,}{.7em}{\mathclap{\smash{\ \ \xRightarrow{\smash{\hspace{1em}}}}}}{32}\ldTo(1,2)_{\mathrlap{\internal\E_n[-]}}\\
&\Top
\end{diagram}
\end{equation}
which ``absorbs'' the above $\internal\varepsilon_h$:
\[ \big(q\circ\op{(h_\ast)}\big)\cdot\internal\varepsilon_h=q \]
\end{remark}

\section{Homotopy type of the right modules over $\E^G_n$}\label{section:homotopy_modules_E^G_n}

Fix an object of $\overGL{n}$, $f:G\to GL(n,\RR)$ (where $n\in\NN$). As is now usual, it will be often denoted simply by $G$. Fix also $k\in\NN$.

We will end this chapter by describing the homotopy type of the internal presheaf $\internal\E^G_n[M]$. For that purpose, we will analyse separately each of the pieces $\REmb^{\smash{G}}_n((\RR^n)^{\smash{\amalg k}},M)$ (this is the piece over $(\RR^n)^{\smash{\amalg k}}$). For legibility, we will again make the identification
\[ (\RR^n)^{\smash{\amalg k}}=k\times\RR^n \]

We will mimic the constructions in section \sref{section:homotopy_type_E_n} for the space of augmented embeddings $\REmb^{\smash{G}}_n(R,M)$.

For the next definition, recall the $G^{\smash{\boxtimes k}}$-structure on $\Conf(M,k)$ from example \sref{example:fk-structure_conf}: it gives us, in particular, a $G^{\smash{\times k}}$-principal bundle \[ \Princ_{G^{\smash{\times k}}}\big(T\Conf(M,k)\big)\To\Conf(M,k) \] which is the restriction of $\Princ_{G^{\times k}}\big(T(M^{\times k})\big)$ to $\Conf(M,k)$. Observe also that $\Princ_G(T(k\times\RR^n))$ is canonically trivialized (since $k\times\RR^n$ actually has a $1$-structure):
\[ \Princ_G(T(k\times\RR^n))=G\times k\times\RR^n \]

\begin{definition}[$G$-augmented derivative at the origins]
Let $M$ be a $n$-manifold with a $G$-structure.\\
Recall the induced $G^{\smash{\boxtimes k}}$-structure on $\Conf(M,k)$ from example \sref{example:fk-structure_conf}.\\
Consider the composition
\begin{align*}
\REmb^G_n(k\times\RR^n,M)&\xTo[\sref{definition:fderivative}]{\RDer^G}\Map^G\!\big(G\times k\times\RR^n,\Princ_G(TM)\big)\\
&\xTo{\restrict{(-)}{k}}\Map^G\!\big(G\times k,\Princ_G(TM)\big)\\
&\xTo{\simeq}\big(\Princ_G(TM)\big)^{\times k}\\
&\xlongequal[\sref{definition:product_GHstructure}]{\quad}\Princ_{G^{\times k}}\big(T(M^{\times k})\big)
\end{align*}
This composition induces a map
\[ \RDer^G_0:\REmb^G_n\big((\RR^n)^{\amalg k},M\big)\To\Princ_{G^{\times k}}\big(T\Conf(M,k)\big) \]
which we call the {\em \(G\)-augmented derivative at the origins}.
\end{definition}

\begin{remark}
There is an action of $G^{\smash{\times k}}$ on $\REmb^{\smash{G}}_n(k\times\RR^n,M)$ for which $\RDer^G_0$ is $G^{\smash{\times k}}$-equivariant.\\
This action makes $\REmb^{\smash{G}}_n(k\times\RR^n,M)$ into a principal $G^{\smash{\times k}}$-space such that the map to the quotient by the group action is a fibration.
\end{remark}

\begin{remark}[naturality of $\RDer^G_0$]\label{remark:naturality_RDer}
The source of $\RDer^G_0$ has the functoriality with respect to $G$ and $M$ which is inherited from the functor $\REmb^\bullet_n:\overGL{n}\to\Top\dash\CAT$. The target of $\RDer^G_0$ acquires similar functoriality.\\
The map $\RDer^G_0$ is natural with respect to that functoriality of the source and target. We leave it to the reader to make this assertion precise.
\end{remark}

\begin{proposition}\label{proposition:RDer^G_0_equivalence}
Let $M$ be a $n$-manifold with a $G$-structure.\\
The map
\[ \RDer^G_0:\REmb^G_n\big((\RR^n)^{\amalg k},M\big)\To\Princ_{G^{\times k}}\big(T\Conf(M,k)\big) \]
is a homotopy equivalence and a Hurewicz fibration.
\end{proposition}
\begin{proof}[Sketch of proof]
For brevity of notation, let us define:
\[ R\defeq k\times\RR^n=(\RR^n)^{\amalg k} \]
To verify that the map is a fibration, proposition \sref{proposition:D_0_equivalence} and the remarks in \sref{remark:homotopy_properties_REmb} are useful:
using that the maps
\[ \Map^G\!\big(\Princ_G(TR),\Princ_G(TM)\big)\To\Map(R,M) \] 
\[ D_0:\Emb(R,M)\To\Frame\big(T\Conf(M,k)\big) \]
are fibrations, we can prove (through some simple manipulation of principal $G$-bundles) that the natural map
\[ \Map^G\!\big(\Princ_G(TR),\Princ_G(TM)\big)\underset{\mathclap{\Map(R,M)}}{\times}\,\Emb(R,M)\To\Princ_{G^{\times k}}\big(T\Conf(M,k)\big) \]
is a fibration. Composing with the fibration (expression \seqref{equation:fibration_REmb})
\[ \REmb^G(R,M)\To\Map^G\!\big(\Princ_G(TR),\Princ_G(TM)\big)\underset{\mathclap{\Map(R,M)}}{\times}\,\Emb(R,M) \]
gives the map $\RDer^G_0$, which is consequently a Hurewicz fibration.

In order to prove that $\RDer^G_0$ is a homotopy equivalence, first we construct a homotopy equivalence from $\REmb^{\smash{G}}_n(k\times\RR^n,M)$ to the homotopy pullback, $X$, of
\begin{equation}\label{diagram:small_diagram_Conf(M,k)}
\begin{diagram}[h=2.1em]
&&\Princ_{G^{\times k}}\big(T(M^{\times k})\big)\\
&&\dTo{f_\ast}\\
\Frame(T\Conf(M,k))&\rInto&\Frame\big(T(M^{\times k})\big)
\end{diagram}
\end{equation}
This diagram is just the result of taking the diagram in \sref{definition:modified_embedding_space} which defines $\REmb^{\smash{G}}_n(k\times\RR^n,M)$, and substituting each of the entries by (smaller) equivalent spaces that they map to (one of those substitutions is given by proposition \sref{proposition:D_0_equivalence}; see diagram below). Thus we get a natural objectwise homotopy equivalence of diagrams from the original one (from definition \sref{definition:modified_embedding_space} applied to the present case) to the one displayed above (recall that $R\defeq k\times\RR^n$):
\begin{equation}\label{diagram:simplification_hopb_REmb}
\begin{diagram}[midshaft,h=2.2em,balance,hug]
&&\Map^G\!\big(\Princ_G(TR),\Princ_G(TM)\big)\\
&&\dTo{\raisebox{-1em}{$\scriptstyle f_\ast$}}&\rdTo^{\sim}_{\restrict{(-)}{k}}\\
&&&&\Princ_{G^{\times k}}\big(T(M^{\times k})\big)\\
\Emb(R,M)&\rTo{\ \ D\ \ }&\Map^{\smash{GL(n,\RR)}}\!\!\big(\Frame(TR),\Frame(TM)\big)&&\dTo{\raisebox{-1em}{$\scriptstyle f_\ast$}}\\
&\rdTo[rightshortfall=.7em]^{\sim}_{D_0}&&\rdTo^{\sim}_{\restrict{(-)}{k}}\\
&&\Frame(T\Conf(M,k))&\rInto&\Frame\big(T(M^{\times k})\big)
\end{diagram}
\end{equation}
From this we derive the desired homotopy equivalence of homotopy pullbacks
\[ u:\REmb^{\smash{G}}_n(k\times\RR^n,M)\xTo{\;\sim\;} X \]
(recall from \sref{remark:homotopy_properties_REmb} that $\REmb^{\smash{G}}_n(k\times\RR^n,M)$ is homotopy equivalent to the usual homotopy pullback).

Let us now augment diagram \seqref{diagram:small_diagram_Conf(M,k)} with more arrows:
\begin{diagram}[h=2.1em]
\Princ_{G^{\times k}}(T\Conf(M,k))&\rInto{\qquad}&\Princ_{G^{\times k}}\big(T(M^{\times k})\big)\\
\dTo{f_\ast}&&\dTo{f_\ast}\\
\Frame(T\Conf(M,k))&\rInto&\Frame\big(T(M^{\times k})\big)\\
\dTo{\proj}&&\dTo{\proj}\\
\Conf(M,k)&\rInto&M^{\times k}
\end{diagram}
The bottom square is a pullback square. It is therefore a homotopy pullback square given that
\[ \proj:\Frame\big(T(M^{\times k})\big)\To M^{\times k} \]
is a Hurewicz fibration. Similarly, the big outer square is cartesian, and therefore homotopy cartesian since
\[ \proj\circ f_\ast:\Princ_{G^{\times k}}\big(T(M^{\times k})\big)\To M^{\times k} \]
is a Hurewicz fibration. Consequently, the top square is homotopy cartesian as well, which provides a homotopy equivalence
\[ v:\Princ_{G^{\times k}}(T\Conf(M,k))\xTo{\;\sim\;} X \]
Observe that $v\circ\RDer^{\smash{G}}_0\simeq u$ (the two maps are {\em not} equal): this essentially amounts to a chase around diagram \seqref{diagram:simplification_hopb_REmb}, using the definition of $G$-aug\-ment\-ed embedding spaces. In conclusion, $\RDer^{\smash{G}}_0$ is a homotopy equivalence.
\end{proof}


%% file: filtered.tex



\chapter{Stratified spaces}\label{chapter:stratified_spaces}

\section*{Introduction}

This chapter deals with the notion of stratified space, which is essentially a space equipped with a filtration. The goal is to develop the basis for applications to the construction $\M(X)$ from chapter \ref{chapter:sticky_conf}. More precisely, we define the category of filtered paths on a stratified space, and compare it with the construction $\M(X)$. We also use a well behaved class of stratified spaces --- the homotopically stratified spaces --- to compute the homotopy type of the morphism spaces in $\M(M)$ for $M$ a manifold.

\section*{Summary}

This chapter is mostly expository, and gives a convenient theory of stratified spaces.

The first section, \sref{section:stratified_spaces}, gives a naive definition of stratified space (also called a filtered space elsewhere) and a few important examples. Section \sref{section:filtered_paths} defines the space of filtered Moore paths on a stratified space. Using this, it then proceeds to associate to each stratified space $X$ the (internal) topological category of filtered paths in $X$, $\smash{\popathcat}(X)$. Notions similar to this exist in the literature, and are often called the ``exit-path category'' (see, for example, \cite{Woolf}, \cite{Treumann}, or the appendix A to \cite{Lurie1}).

Section \sref{section:strong_popath} puts a useful new topology on the space of filtered Moore paths, resulting in the strong space of filtered Moore paths. Section \sref{section:sticky_filtered} recovers the category $\M(X)$ from categories of filtered paths, and in particular gives a description of the morphism spaces of $\M(X)$ in terms of spaces of filtered paths.

Now starts the journey to define a convenient class of stratified spaces, and to give tools to analyze them homotopically. Essentially, the new concepts discussed in the remaining sections were originally introduced in the article \cite{Quinn} of Frank Quinn.

Section \sref{section:holink_spaces} defines the notion of homotopy links on a stratified space $X$. It also gives several properties of the space of homotopy links of (the stratified space associated to) a pair of spaces. These properties are used to analogize the space of homotopy links on a pair of spaces with the normal sphere bundle of an embedding of manifolds.

Section \sref{section:tameness} discusses the notion of tame subspace, under which the space of homotopy links of a pair is particularly well behaved. Section \sref{section:homotopically_stratified_spaces} finally defines a convenient class of stratified spaces: the homotopically stratified spaces. These allow for a vast simplification in the homotopical analysis of filtered paths on a stratified space: the theorem of David Miller from \cite{Miller} essentially says that the inclusion of a certain space of homotopy links into a corresponding space of filtered paths is a homotopy equivalence. We also give a variation on this result using the strong space of filtered paths which will be necessary in the next chapter.

Finally, section \sref{section:application_stratified_M(M)} applies the results from the preceding sections to analyze the homotopy type of the morphism spaces of $\M(M)$, for $M$ a manifold.

\section{Stratified spaces}\label{section:stratified_spaces}

In this section we introduce the notion of stratified spaces as spaces over a partially ordered set, appropriately topologized.

\begin{notation}
Given a partially ordered set $A$, we will use the following abbreviation
\[ \clop{a}{+\infty}\defeq\set{b\in A\suchthat a\leq b} \]
for $a\in A$.
\end{notation}

\begin{definition}[topology on partially ordered set]
Let $A$ be a partially ordered set.\\
We define the topological space $\potop(A)$ to be the set $A$ equipped with the smallest topology such that $\clop{a}{+\infty}$ is closed for all $a\in A$.
\end{definition}

\begin{remark}
The construction above extends in the obvious manner to a functor from the category of partially ordered sets, $\poset$, to $\Top$
\[ \potop:\poset\To\Top \]
\end{remark}

\begin{definition}[stratified space]
The {\em category of stratified topological spaces}, $\filtTop$, is the over-category
\[ \filtTop\defeq\Top/\potop \]
where
\begin{enum}
\item the objects are triples $\big(X,A,X\to\potop(A)\big)$, with $X$ a topological space and $A$ a partially ordered set;
\item the morphisms $(X,A,f)\to(Y,B,g)$ are pairs $(u:X\to Y,v:A\to B)$ such that the diagram
\begin{diagram}[midshaft,h=2.1em]
X&\rTo{f}&\potop(A)\\
\dTo{u}&&\dTo_{\potop(v)}\\
Y&\rTo{g}&\potop(B)
\end{diagram}
commutes
\end{enum}
\end{definition}

\begin{notation}\label{notation:stratified_space_underlying_space}
An element $(X,A,f)$ of $\filtTop$ will also be called a {\em stratification} on its underlying space, $X$.\\
We will often denote a stratification on $X$ simply by $X$, if the stratification is clear from context (e.g.\ if there is a canonical one).
\end{notation}

\begin{notation}[strata and filtration stages]
Given a stratified space $\big(X,A,f:X\to\potop(A)\big)$, we use the following abbreviations
\begin{align*}
X_a&\defeq f^{-1}(\set{a})\\
X_{\geq a}&\defeq f^{-1}\big(\clop{a}{+\infty}\big)
\end{align*}
for $a\in A$. We call $X_a$ a {\em stratum} of $X$, and $X_{\geq a}$ a {\em filtration stage} of $X$.
\end{notation}

\begin{remark}\label{remark:attention_stratified}
The data for a stratified space $X$ is equivalent to assigning a closed set $X_{\geq a}$ to each element $a\in A$, in such a way that
\[ X_{\geq a}\subset X_{\geq b}\qquad\text{if } b\leq a \]
for $a,b\in A$.\\
In particular, the map $a\Mapsto X_{\geq a}$ is order reversing. This is opposite the conventional definition of stratified spaces as spaces over a partially ordered set. The reason for this disparity has to do with convenience in our principal example 
\sref{example:principal_example_stratified_space} of a stratified space.
\end{remark}

\begin{example}[pair of spaces]\label{example:pair_stratified}
A simple example of a stratified space is given by a pair of topological spaces $(X,Y)$ where $Y$ is a closed subspace of $X$. The element of $\filtTop$ associated to such a pair is $\big(X,(\set{0,1},\leq),f\big)$ where
\begin{align*}
f(Y)&=\set{1}\\
f(X\setminus Y)&=\set{0}
\end{align*}
We denote this stratified space by $\langle\overrightarrow{X,Y}\rangle$ in order to avoid confusion.
\end{example}

\begin{example}[intervals in $\RR$]\label{example:intervals_stratified}
Any interval $J$ in $\RR$ --- with the subspace topology --- can be canonically upgraded to a stratified space $\big(J,(J,\leq),\id_J\!\big)$.\\
Examples of this are given by $I=\clcl{0}{1}$ and $\clop{0}{+\infty}$. According to notation \sref{notation:stratified_space_underlying_space}, we will denote the corresponding stratified spaces simply by $I$ and $\clop{0}{+\infty}$, respectively.
\end{example}

\begin{example}[mapping spaces]\label{example:principal_example_stratified_space}
Let $X$, $Y$ be topological spaces with $Y$ Hausdorff.\\
There is a canonical stratification on $\Map(X,Y)$ (with the compact-open topology) whose underlying partially ordered set is $(\text{equiv(X)},\subset)$, the set of all equivalence relations on $X$ equipped with the inclusion partial order.\\
The stratified space associated with $\Map(X,Y)$ is then \[ \big(\!\Map(X,Y),(\text{equiv}(X),\subset),p\big) \] where for any map $g:X\to Y$, $p(g)$ is the equivalence relation induced on $X$ by $g$, i.e.
\[ \qquad(x,y)\in p(g)\;\Longleftrightarrow\; g(x)=g(y)\qquad\text{for }x,y\in X  \]
Again, as per notation \sref{notation:stratified_space_underlying_space}, this stratified space will be designated simply by $\Map(X,Y)$.\\
For future reference, note that an example of this stratified space is given by the product $Y^{\times S}=\Map(S,Y)$ for any set $S$.
\end{example}

\begin{remark}
The preceding construction of a stratification on $\Map(X,Y)$ (for $X$ a space and $Y$ a Hausdorff space) extends canonically to a functor
\[ \Map:\op{\Top}\times\mathcal{H}\Top_{inj}\To\filtTop \]
where $\mathcal{H}\Top_{inj}$ is the subcategory of $\Top$ generated by the injective continuous maps between Hausdorff spaces.
\end{remark}

\begin{definition}[underlying space of a stratified space]
We define the {\em underlying space functor} to be the canonical projection
\[ u:\filtTop\To\Top \]
\end{definition}

\begin{definition}[space of stratified maps]
Let $X$, $Y$ be stratified spaces.\\
The {\em space of stratified maps}, $\smash{\poMap(X,Y)}$, is the subspace of $\Map(uX,uY)$ constituted by the elements in the image of
\[ u:\filtTop(X,Y)\To\Top(uX,uY) \]
The elements of $\smash{\poMap(X,Y)}$ will be called {\em stratified maps}.
\end{definition}

\comment{
\begin{remark}[functoriality of space of stratified maps]
The above definition of the space of stratified maps extends to a functor
\[ \poMap:\op{\filtTop}\times\\filtTop\To\Top \]
\end{remark}
}

\begin{example}[space of filtered paths]
Recall that $I$ is canonically a stratified space (example \sref{example:intervals_stratified}).\\
Given a stratified space $(X,A,f)$, we can then consider the space $\smash{\poMap(I,X)}$, which we will call the {\em space of filtered paths in \(X\)}.\\
Observe that $\smash{\poMap(I,X)}$ is the subspace of $\Map(I,X)$ constituted by the paths $\gamma:I\to X$ such that
\[ f\circ\gamma:I\To A \]
is order preserving.
\end{example}

With this example in mind, it is sensible to look for an analogous space of filtered Moore paths --- recall the concept of Moore paths from section \sref{section:terminology_Moore_path_space}. Such a space will be a focus of the next section.

\section{Spaces and categories of filtered paths}\label{section:filtered_paths}

Being in possession of the space of filtered paths $\smash{\poMap(I,X)}$, we now turn to define the analogous space of filtered Moore paths.

\begin{definition}[space of filtered Moore paths]
Let $X$ be a stratified space.\\
The {\em space of filtered Moore paths in \(X\)}, $\smash{\popath(X)}$, is defined to be the pullback of ($\clop{0}{+\infty}$ is canonically stratified as in example \sref{example:intervals_stratified})
\begin{diagram}[h=2.3em]
&&\poMap(\clop{0}{+\infty},X)\times\clop{0}{+\infty}\\
&&\dInto_{\inclusion}\\
H(X)&\rInto{\quad\inclusion\quad}&\Map(\clop{0}{+\infty},X)\times\clop{0}{+\infty}
\end{diagram}
that is, the subspace of $H(X)$ constituted by the elements $(\gamma,\tau)\in H(X)$ such that $\gamma:\clop{0}{+\infty}\to X$ is a stratified map.
\end{definition}

\begin{remark}[functoriality of space of filtered Moore paths]
The above definition of the space of filtered Moore paths extends to a functor
\[ \popath:\filtTop\To\Top \]
\end{remark}

\begin{remark}[maps on space of filtered Moore paths]\label{remark:maps_popath}
We have maps
\begin{align*}
s&:\popath(X)\To X\\
t&:\popath(X)\To X\\
l&:\popath(X)\To\clop{0}{+\infty}\\
i&:X\To\popath(X)\\
cc&:\popath(X)\,\underset{X}{\smash{_t\times_s}}\,\popath(X)\To\popath(X)
\end{align*}
obtained by restricting the maps of the same name on $H(X)$ (consult \sref{definition:maps_Moore_path_space} and \sref{definition:concat_Moore_paths}).
\end{remark}

\begin{definition}[subspaces of filtered Moore paths]
Let $X$ be a stratified space. Let $A$, $B$ be subspaces of $X$.\\
We define the space $\smash{\popath(X;A,B)}$ to be the pullback in $\Top$ of
\begin{diagram}[h=2.1em]
&&\popath(X)\\
&&\dTo{(s,t)}\\
A\times B&\rInto{\qquad}&X\times X
\end{diagram}
We call $\smash{\popath(X;A,B)}$ the {\em space of filtered paths in \(X\) starting in \(A\) and ending in \(B\)}.
\end{definition}

\begin{remark}
The space $\smash{\popath(X;A,B)}$ is the subspace of $x\in\smash{\popath(X)}$ such that $s(x)\in A$ and $t(x)\in B$.\\
This subspace has natural source and target maps
\begin{align*}
s&:\popath(X;A,B)\To A\\
t&:\popath(X;A,B)\To B
\end{align*}
Additionally, concatenation defines a map
\[ cc:\popath(X;A,B)\,\underset{B}{\smash{_t\times_s}}\,\popath(X;B,C)\To\popath(X;A,C) \]
\end{remark}

Unlike the case of $H(X)$, the source map on $\smash{\popath}(X)$ is not a fibration in general. However, it becomes a fibration by restricting to paths which start at a fixed stratum of $X$.

\begin{proposition}
Let $(X,A,f)$ be a stratified space, and $a,b\in A$.\\
Let $Y$ be a subspace of $X$.\\
The maps
\begin{align*}
s&:\popath(X;X_a,Y)\To X_a\\
t&:\popath(X;Y,X_b)\To X_b\\
(s,t)&:\popath(X;X_a,X_b)\To X_a\times X_b
\end{align*}
are Hurewicz fibrations.
\end{proposition}

With the space of filtered Moore paths in hand, we can now define a filtered path category of a stratified space, in analogy with the path category of a space defined in example \sref{example:path_category}.

\begin{definition}[filtered path category of stratified space]
Let $X$ be a stratified space.\\
We define the {\em filtered path category of \(X\)} to be the internal category in $\Top$, $\smash{\popathcat(X)}$, given by (recall observation \sref{remark:maps_popath})
\begin{enum}
\item the {\em object} space is \[ \ob\big(\popathcat(X)\big)\defeq X \]
\item the {\em morphism} space is \[ \mor\big(\popathcat(X)\big)\defeq\popath(X) \]
\item the {\em source} map of $\smash{\popathcat(X)}$ is \[ s:\smash{\popath(X)}\To X \]
\item the {\em target} map of $\smash{\popathcat(X)}$ is \[ t:\smash{\popath(X)}\To X \]
\item the {\em identity} map is \[ i:X\To\smash{\popath(X)} \]
\item the {\em composition} map is\[ cc:\popath(X)\,\underset{X}{\smash{_t\times_s}}\,\popath(X)\To\popath(X) \]
\end{enum}
\end{definition}

\begin{remark}[functoriality of filtered path category]
The construction of the filtered path category extends to a functor
\[ \popathcat:\filtTop\To\Cat(\Top) \]
\end{remark}

\begin{remark}\label{remark:morphism_space_popathcat}
Given $x,y\in X$, the corresponding morphism space of the discretization $\disccat{\smash{\popathcat}}(X)$ is
\[ \disccat{\popathcat}(X)(x,y)=\popath(X;\set{x},\set{y}) \]
\end{remark}

\section{Strong spaces of filtered paths}\label{section:strong_popath}

In this section, we will define a different topology on $\smash{\popath}(X)$ which will be used later. First we introduce functions on $\smash{\popath}(X)$ which are very useful in practice.

\begin{definition}[time of entrance into filtration stage]
Let $(X,A,f)$ be a stratified space, and $a\in A$.\\
We define the {\em time of entrance into \(X_{\geq a}\)}, $e_a$, to be the function
\[ \func{e_a}{\popath(X)}{\clop{0}{+\infty}}{(\gamma,\tau)}{\min\set{\tau,\inf\big(\gamma^{-1}(X_{\geq a})\big)}} \]
\end{definition}

Unfortunately, the time of entrance function $e_a$ is {\em not} continuous except in rather trivial cases. We can, however, change the topology to make it continuous.

\begin{definition}[strong space of filtered paths]
Let $(X,A,f)$ be a stratified space.\\
We define the {\em strong space of filtered Moore paths in \(X\)}, $\smash{\popathstrong(X)}$, to be the underlying set of $\smash{\popath(X)}$ equipped with the smallest topology for which
\begin{enum}
\item the identity function \[ \id:\smash{\popathstrong(X)}\To\smash{\popath(X)} \] is continuous, and
\item the function \[ e_a:\smash{\popathstrong(X)}\To\clop{0}{+\infty} \] is continuous, for each $a\in A$.
\end{enum}
\end{definition}

\begin{remark}[maps on strong space of filtered Moore paths]\label{remark:maps_popathstrong}
The maps in observation \sref{remark:maps_popath} induce continuous functions
\begin{align*}
s&:\popathstrong(X)\To X\\
t&:\popathstrong(X)\To X\\
l&:\popathstrong(X)\To\clop{0}{+\infty}\\
cc&:\popathstrong(X)\,\underset{X}{\smash{_t\times_s}}\,\popathstrong(X)\To\popathstrong(X)
\end{align*}
\end{remark}

\begin{notation}[subspaces $\smash{\popathstrong(X;A,B)}$]
Let $X$ be a stratified space. Let $A$, $B$ be subspaces of $X$.\\
The subspace of $\smash{\popathstrong(X)}$ given by the pullback of
\begin{diagram}[midshaft,h-2.1em]
&&\popath(X;A,B)\\
&&\dInto{\inclusion}\\
\popathstrong(X)&\rTo{\ \id\ }&\popath(X)
\end{diagram}
is denoted $\smash{\popathstrong(X;A,B)}$. It is, equivalently, the subspace of $x\in\smash{\popathstrong(X)}$ such that $s(x)\in A$ and $t(x)\in B$.
\end{notation}

\begin{remark}[strong filtered path category]
One could define, for each stratified space $X$, a {\em strong filtered path category} --- analogous to $\smash{\popathcat}(X)$ --- whose space of morphisms would be $\smash{\popathstrong}(X)$.
\end{remark}

\section{Application: sticky homotopies from filtered paths}\label{section:sticky_filtered}

The next proposition says that filtered paths can recover the notion of sticky homotopies in many cases of interest. Recall for that purpose the functor
\[\scriptsize \stcat:\Top_{inj}\xTo{\overline{\Map}}\CAT^{(2)}_\cartesian(\op{\FinSet},\Top)\xTo{\stickypathcat_{\op{\FinSet}}}\HomCat{\op{\FinSet}}{\Cat(\Top)} \]
from construction \sref{definition:stickypathcat_Map}.

\begin{definition}
Define the functor
\[ \overline{\Map}_{strat}:\mathcal{H}\Top_{inj}\To\HomCatlarge{\op{\FinSet}}{\filtTop} \]
as the composition
\[ \mathcal{H}\Top_{inj}\xTo[\sref{example:principal_example_stratified_space}]{\,\Map\,}\HomCatlarge{\op{\Top}}{\filtTop}\xTo{\HomCat{\op{\inclusion}}{\filtTop}}\HomCatlarge{\op{\FinSet}}{\filtTop} \]
\end{definition}

The following proposition is an exercise with the definition of categories of sticky homotopies and the definition of filtered path categories.

\begin{proposition}
There is a unique natural isomorphism
\[ \alpha:\restrict{\stcat}{\mathcal{H}\Top_{inj}}\xTo{\ \simeq\ }\HomCatLARGE{\op{\FinSet}}{\popathcat}\circ\overline{\Map}_{strat} \]
such that for any Hausdorff space $X$ the equation
\[ \ob\circ\alpha_X=\id_{\overline{\Map}(X)_1} \]
holds and the diagram
\begin{diagram}[midshaft]
\mor\circ\stcat(X)&\rTo{\ \mor\circ\alpha_X\ }&\mor\circ\popathcat\circ\overline{\Map}_{strat}(X)\\
\dTo{\inclusion}&&\dEqual\\
H\circ\overline{\Map}(X)_1&\lTo{\inclusion}&\popath\circ\overline{\Map}_{strat}(X)
\end{diagram}
commutes.
\end{proposition}

The following construction uses this isomorphism to calculate the morphism spaces of $\M(X)$ in terms of spaces of filtered paths.

\begin{construction}[morphisms of $\M(X)$ as filtered paths]\label{construction:mor_M(X)_popath}
Let $X$ be a Hausdorff topological space.\\
From the isomorphism $\alpha$ above we immediately get a canonical isomorphism
\[ \Groth\big(\disccat{\alpha}_X\big):\Groth\big(\disccat{\stcat}(X)\big)\xTo{\simeq}\Groth\Big(\disccat{\popathcat}\circ\overline{\Map}_{strat}(X)\Big) \]
Recall from definition \sref{definition:sticky_configurations} that $\M(X)$ is a full subcategory of the source of this isomorphism. In particular, we get an isomorphism
\[ \Groth\big(\disccat{\alpha}_X\big):\M(X)(a,b)\xTo{\simeq}\Groth\Big(\disccat{\popathcat}\circ\overline{\Map}_{strat}(X)\Big)(a,b) \]
for any injections $a:S\to X$ and $b:S'\to X$, where $S$, $S'$ are finite sets.\\
By observation \sref{remark:morphism_space_popathcat}, and proposition \sref{proposition:description_Groth_discrete}, the morphism space on the right is
\begin{align*}
\qquad\coprod_{\mathclap{f\in\FinSet(S,S')}}\ \ \popath\big(X^{\times k};\set{a},\set{b\circ f}\big)&\cong\popath\Big(X^{\times k};\set{a},\set{b\circ f:f\in\Set(S,S')}\Big)\\
&=\popath\Big(X^{\times k};\set{a},b\circ\Set(S,S')\Big)
\end{align*}
where the first (canonical) homeomorphism is a consequence of $b$ being injective, and the bottom equality results from an abbreviation of the notation.\\
In conclusion, we have constructed a canonical homeomorphism
\begin{equation}\label{equation:mor_M(X)_popath}
\M(X)(a,b)\cong\popath\Big(X^{\times k};\set{a},b\circ\Set(S,S')\Big)
\end{equation}
for any $a:S\to X$ and $b:S'\to X$ injective ($S$ and $S'$ being finite sets).\\
For simplicity, we will use this homeomorphism to identify its source with its target. We will therefore occasionally write the two spaces as equal.
\end{construction}

\section{Homotopy link spaces}\label{section:holink_spaces}

In order to define a class of ``good'' stratified spaces, we will study spaces of ``homotopy links'' in this section.

\begin{definition}[space of homotopy links]
Let $(X,A,f)$ be a stratified space.\\
The {\em space of homotopy links in \(X\)}, $\holink(X)$, is the subspace of $\smash{\popath}(X)$ given by
\[ \holink(X)\defeq\set{(\gamma,\tau)\in\popath(X)\suchthat f\circ\restrict{\gamma}{\clop{0}{\tau}}\text{\, is constant}} \]
\end{definition}

\begin{remark}
Intuitively, a homotopy link is a filtered (Moore) path which remains in the same stratum until the last possible moment.
\end{remark}

\begin{remark}
Observe that the space of homotopy links $\holink(X)$ is also a subspace of $\smash{\popathstrong}(X)$. In fact, the inclusion
\[ \holink(X)\Into\popathstrong(X) \]
is a closed subspace (unlike the inclusion into $\smash{\popath}(X)$).
\end{remark}

\begin{remark}[functoriality of space of homotopy links]
The space of homotopy links of a stratified space extends to a functor
\[ \holink:\filtTop\To\Top \]
\end{remark}

\begin{definition}[subspaces of homotopy links]
Let $X$ be a stratified space. Let $A$, $B$ be subspaces of $X$.\\
We define the {\em space of homotopy links in \(X\) starting in \(A\) and ending in \(B\)}, $\holink(X;A,B)$, to be the subspace
\[ \holink(X)\cap\popath(X;A,B) \]
of $\smash{\popath}(X)$.\\
Equivalently, $\holink(X;A,B)$, is the subspace of $\holink(X)$ given by
\[ \holink(X;A,B)\defeq\set{x\in\holink(X)\suchthat s(x)\in A,\, t(x)\in B} \]
\end{definition}

\begin{proposition}\label{proposition:holink_s_fibration}
Let $(X,A,f)$ be a stratified space, and $a\in A$.\\
Let $Y$ be a subspace of $X$.\\
The map
\[ s:\holink(X;X_a,Y)\To X_a \]
is a Hurewicz fibration.
\end{proposition}

We will illustrate the role of the space of homotopy links with the case of a pair of spaces.

\begin{remark}[homotopy link space for a pair of spaces]
Suppose we are given a pair of spaces $(X,Y)$ where $Y$ is a closed subspace of $X$. Via example \sref{example:pair_stratified}, we get a stratified space $\langle\overrightarrow{X,Y}\rangle$ with two strata: $Y$ and $X\setminus Y$.\\
Thus we can consider, for example, the homotopy link space
\[ \holink\big(\langle\overrightarrow{X,Y}\rangle;X\setminus Y,Y\big) \]
which we will call the {\em homotopy link space of the pair \((X,Y)\)}.
\end{remark}

What does the homotopy link space
\[ \holink\big(\langle\overrightarrow{X,Y}\rangle;X\setminus Y,Y\big) \]
represent? Let us put forward that it aims to give a generalization of the notion of normal sphere bundle of a closed embedding of manifolds. A few propositions (stated without proof) will partly justify this statement.

\begin{proposition}[shrinking homotopy links]\label{proposition:holink_shrink_neigh}
Let $X$ be a metrizable topological space, $Y$ a closed subspace of $X$, and $U$ a neighborhood of $Y$ in $X$.\\
Then the inclusion
\[ \holink\big(\langle\overrightarrow{U,Y}\rangle;U\setminus Y,Y\big)\Into\holink\big(\langle\overrightarrow{X,Y}\rangle;X\setminus Y,Y\big) \]
is a homotopy equivalence over $Y$ (both spaces map to $Y$ via the target map, $t$).\\
Furthermore, the map
\[ t:\holink\big(\langle\overrightarrow{X,Y}\rangle;X\setminus Y,Y\big)\To Y \]
is a Hurewicz\comment{ (or Serre, if $X$ is not necessarily metrizable)} fibration if and only if
\[ t:\holink\big(\langle\overrightarrow{U,Y}\rangle;U\setminus Y,Y\big)\To Y \]
is a Hurewicz\comment{ (or Serre, if $X$ is not necessarily metrizable)} fibration.
\end{proposition}

\begin{remark}
The condition that $X$ be metrizable is only necessary to guarantee that the space of homotopy links is metrizable, and therefore paracompact.\\
The proof of the above proposition reduces essentially to finding a map
\[ \phi:\holink\big(\langle\overrightarrow{X,Y}\rangle;X\setminus Y,Y\big)\To\clop{0}{+\infty} \]
such that
\[ \gamma\big(\phi(\gamma,\tau)\big)\in U\setminus Y \]
for all $(\gamma,\tau)\in\holink\big(\langle\overrightarrow{X,Y}\rangle;X\setminus Y,Y\big)$. Such a map exists by a standard partition of unity argument.\\
The metrizability of $X$ will be assumed repeatedly with a similar purpose.
\end{remark}

\begin{proposition}[factorization of homotopy links]\label{proposition:fact_holink}
Let $X$ be a metrizable topological space, $Y$ a closed subspace of $X$, and $U$ a neighborhood of $Y$ in $X$.\\
Let $P$ denote the pullback of
\begin{diagram}[h=2.1em,midshaft,balance]
&&\mathclap{\holink\big(\langle\overrightarrow{U,Y}\rangle;U\setminus Y,Y\big)}\\
&&\dTo{s}\\
\qquad\holink\big(\langle\overrightarrow{X,Y}\rangle;X\setminus Y,U\setminus Y\big)&\rTo{\quad t\quad}&U\setminus Y
\end{diagram}
Then concatenation of filtered paths gives a natural map
\[ \concat:P\To\holink\big(\langle\overrightarrow{X,Y}\rangle;X\setminus Y,Y\big) \]
which is a homotopy equivalence over $(X\setminus Y)\times Y$.
\end{proposition}

\begin{remark}
Propositions \sref{proposition:holink_shrink_neigh} and \sref{proposition:fact_holink} have obvious analogues which hold for spaces of filtered paths.
\end{remark}

\begin{definition}[fibrewise open cone]
Let $f:X\to Y$ be a map of topological spaces.
We define the {\em fibrewise open cone} of $f$, $C_Y f$, to be the pushout of
\begin{diagram}[h=2.1em]
X\,&\cong X\times\set{0}&\rInto{\ \ \quad}&X\times\clop{0}{+\infty}\\
\dTo{f}\\
Y
\end{diagram}
\end{definition}

\begin{remark}
Note that $Y$ includes as a closed subspace of $C_Y f$, and $C_Y f$ naturally projects to $Y$. Furthermore, there is a natural inclusion
\[ X\cong X\times\set{1}\Into C_Y f\setminus Y \]
\end{remark}

\begin{proposition}\label{proposition:holink_cone}
Let $f:X\to Y$ be a Hurewicz fibration.\\
Then the map
\[ t:\holink\big(\langle\overrightarrow{C_Y f,Y}\rangle;C_Y f\setminus Y,Y\big)\To Y \]
is a Hurewicz fibration. Moreover, the map
\[ s:\holink\big(\langle\overrightarrow{C_Y f,Y}\rangle;C_Y f\setminus Y,Y\big)\To C_Y f\setminus Y \]
is a homotopy equivalence which is homotopic to a homotopy equivalence over $Y$. Finally, there are two natural maps
\[ X\Into C_Y f\setminus Y\To\holink\big(\langle\overrightarrow{C_Y f,Y}\rangle;C_Y f\setminus Y,Y\big) \]
which are both homotopy equivalences over $Y$.
\end{proposition}

The previous results allow a preliminary justification of our motto that homotopy link spaces generalize normal sphere bundles (of closed embeddings of manifolds).

\begin{proposition}
Assume $i:Y\to X$ is a closed embedding of manifolds without boundary.\\
Let $\nu_i$ be the normal bundle of $i$ (over $Y$), and $S(\nu_i)$ the unit sphere bundle of $\nu_i$.\\
Then there are homotopy equivalences over $Y$
\[ \holink(\langle\overrightarrow{X,Y}\rangle;X\setminus Y,Y)\underset{Y}{\simeq}\nu_i\setminus Y\underset{Y}{\simeq}S(\nu_i) \]
\end{proposition}
\begin{proof}
Let $U$ be a tubular neighborhood for $i$. Identifying $Y$ with $i(Y)$, and using proposition \sref{proposition:holink_shrink_neigh} we conclude that the inclusion
\[ \holink\big(\langle\overrightarrow{U,Y}\rangle;U\setminus Y,Y\big)\Into\holink\big(\langle\overrightarrow{X,Y}\rangle;X\setminus Y,Y\big) \]
is a homotopy equivalence over $Y$. On the other hand\[ C_Y\big(S(\nu_i)\big)\underset{Y}{\cong}\nu_i\cong U \] and so proposition \sref{proposition:holink_cone} gives homotopy equivalences over $Y$
\[ \holink(\langle\overrightarrow{U,Y}\rangle;U\setminus Y,Y)\underset{Y}{\simeq}\nu_i\setminus Y\underset{Y}{\simeq}S(\nu_i) \]
\end{proof}

\begin{notation}
The normal bundle of $i:Y\to X$ is
\[ \nu_i\defeq\mbox{\Large$\quot{TY}{(i^\ast{TX})}$} \]
The unit sphere bundle of $\nu_i$ is
\[ S(\nu_i)\defeq\mbox{\Large$\quot{\RR^+}{(\nu_i\hspace{.06em}\setminus\hspace{.08em}Y)}$} \]
\end{notation}

We have shown that $\holink\big(\langle\overrightarrow{X,Y}\rangle;X\setminus Y,Y\big)$ recovers (the homotopy type over $Y$ of) the normal sphere bundle of a closed embedding $Y\hookrightarrow X$ of manifolds without boundary. In more general settings, the space of homotopy links of a pair attempts to give a homotopical version of the normal sphere bundle (which is not available). The fibres need not be spheres.

\section{Tameness and homotopy links}\label{section:tameness}

This section will describe conditions on pairs of spaces under which the homotopy link space of the pair behaves homotopically like a normal sphere bundle of an embedding of manifolds.

\begin{definition}[neighborhood of tameness]\label{definition:neigh_tameness}
Let $X$ be a topological space, and $Y$ a closed subspace of $X$.\\
A neighborhood $U$ of $Y$ in $X$ is called a {\em neighborhood of tameness} of $Y$ in $X$ if there exists a map
\[ G:U\times I\To X \]
such that
\begin{align*}
&G(-,0)=\inclusion\\
&\restrict{G(-,\tau)}{Y}=\inclusion\ ,\ \ \text{for } \tau\in I\\
&G(U\times\set{1})\subset Y\\
&G\big((U\setminus Y)\times\clop{0}{1}\big)\subset X\setminus Y
\end{align*}
If the map $G$ factors through $U$, we call $U$ a {\em strong neighborhood of tameness} of $Y$ in $X$.
\end{definition}

\begin{remark}[restatement in terms of stratified maps]
We could rephrase the conditions appearing in the definition above by stating that $G$ is a stratified map
\[ G:\langle\overrightarrow{U{\times} I,Y{\times} I\cup U{\times}\set{1}}\rangle\To\langle\overrightarrow{X,Y}\rangle \]
which gives a strong deformation retraction of $U$ onto $Y$ within $X$.
\end{remark}

\begin{definition}[tame subspace]\label{definition:tame_subspace}
Let $X$, $Y$ be topological spaces, with $Y$ a closed subspace of $X$.\\
$Y$ is said to be a {\em tame subspace} of $X$ if there exists a neighborhood of tameness $U$ of $Y$ in $X$, and a map
\[ \phi:X\to I \]
such that
\begin{align*}
\phi^{-1}(\set{1})&=Y\\
\overline{\phi^{-1}(\opcl{0}{1})}&\subset U
\end{align*}
\end{definition}

\begin{remark}[simplification of definition for $X$ metrizable]\label{remark:tame_simplification_metric}
If $X$ is metrizable (or more generally, Hausdorff and perfectly normal) then the existence of $\phi$ in the above definition \sref{definition:tame_subspace} is automatic: $Y$ is a tame subspace of $X$ if and only if $Y$ has a neighborhood of tameness in $X$.\\
The definition of tame subspace explicitly mentions $\phi$ only because of the important role it plays in the proof of proposition \sref{proposition:holink_pair_pushout}.
\end{remark}

\begin{definition}[local tameness]
Let $Y$ be a closed subspace of a metrizable space $X$.\\
We say $Y$ is {\em locally tame} in $X$ if each point of $Y$ has a neighborhood $U$ in $X$ for which there exists a stratified map
\[ G:\langle\overrightarrow{U{\times} I,(U\cap Y){\times}I\cup U{\times}\set{1}}\rangle\To\langle\overrightarrow{X,Y}\rangle \]
which gives a deformation retraction, rel $U\cap Y$, of $U$ into $Y$ within $X$ (i.e.\ $G$ verifies the conditions in definition \sref{definition:neigh_tameness}, as stated).
\end{definition}

The following local characterization of tameness is essentially lemma 2.5 in \cite{Quinn}. Observe only that in \cite{Quinn}, a ``tame subspace'' of a metrizable space need not be closed.

\begin{proposition}[local tameness implies tameness]\label{proposition:locally_tame_tame}
Let $X$ be a metrizable topological space, and $Y$ a closed subspace of $X$.\\
$Y$ is a tame subspace of $X$ if and only if $Y$ is locally tame in $X$.
\end{proposition}

Under the condition of tameness, the homotopy link space of a pair shares the following two properties (which we state without proof) with the normal sphere bundle of a closed embedding of manifolds.

\begin{proposition}\label{proposition:strong_neigh_tameness_holink}
Let $X$ be a topological space, and $Y$ a closed subspace of $X$.\\
If $U$ is a strong neighborhood of tameness of $Y$ in $X$, then the map
\[ s:\holink\big(\langle\overrightarrow{U,Y}\rangle;U\setminus Y,Y\big)\To U\setminus Y \]
is a homotopy equivalence. In particular (by proposition \sref{proposition:holink_shrink_neigh}), if $X$ is metrizable there is a homotopy equivalence
\[ \holink\big(\langle\overrightarrow{X,Y}\rangle;X\setminus Y,Y\big)\To U\setminus Y \]
\end{proposition}

\begin{remark}
Under the assumption that $Y$ is an exceptionally tame subspace of metrizable $X$ (as defined later in \sref{definition:exceptionally_tame}), we can weaken the conditions on the neighborhood $U$: it need only deformation retract to $Y$ (not necessarily strongly) through filtered paths in $\langle\overrightarrow{U,Y}\rangle$. The proof in this case uses proposition \sref{proposition:Miller}.
\end{remark}

\begin{proposition}\label{proposition:holink_pair_pushout}
Let $X$ be a topological space, and $Y$ a tame subspace of $X$.\\
Let $hoP$ be the homotopy pushout of
\begin{diagram}[midshaft,h=2.1em]
\holink\big(\langle\overrightarrow{X,Y}\rangle;X\setminus Y,Y\big)&\rTo{t}&Y\\
\dTo{s}\\
X\setminus Y
\end{diagram}
There is a natural map
\[ hoP\To X \]
which is a homotopy equivalence under $Y$.
\end{proposition}

One very useful property of the normal sphere bundle of a closed embedding $Y\to X$ (of manifolds without boundary) is that it fibres over $Y$. The following definition axiomatizes that for the case of homotopy link space of a pair.

\begin{definition}[exceptionally tame subspace]\label{definition:exceptionally_tame}
Let $X$ be a topological space, and $Y$ a tame subspace of $X$.\\
We say $Y$ is an {\em exceptionally tame subspace} of $X$ if the map
\[ t:\holink\big(\langle\overrightarrow{X,Y}\rangle;X\setminus Y,Y\big)\To Y \]
is a Hurewicz fibration.
\end{definition}

\begin{proposition}
Let $X$ be a metrizable space, $Y$ a closed subspace of $X$, and $U$ a neighborhood of $Y$ in $X$.\\
Then $Y$ is an exceptionally tame subspace of $U$ if and only if $Y$ is an exceptionally tame subspace of $X$.
\end{proposition}
\begin{proof}[Sketch of proof]
This result follows from remark \sref{remark:tame_simplification_metric} and proposition \sref{proposition:holink_shrink_neigh}.
\end{proof}

We now present a very simple local condition for a subspace to be exceptionally tame.

\begin{proposition}\label{proposition:locally_trivial_holink_fibration}
Let $X$ be a metrizable space, and $Y$ a closed subspace of $X$.\\
Assume that for each point $y\in Y$, there is a neighborhood $U$ of $y$ in $X$, a pointed space $(Z,\set{z})$, and a homeomorphism
\[ f:(Z,\set{z})\times(U\cap Y)\xTo{\ \cong\ }(U,U\cap Y) \]
(of pairs of spaces) such that the induced map
\[ f:\set{z}\times(U\cap Y)\To U\cap Y \]
is the canonical projection.\\
Then
\[ t:\holink\big(\langle\overrightarrow{X,Y}\rangle;X\setminus Y,Y\big)\To Y \]
is a Hurewicz fibration.
\end{proposition}

The proof of \sref{proposition:locally_trivial_holink_fibration} (which is omitted) uses a Hurewicz uniformization result to conclude that a local fibration is a fibration.
The following result is an immediate corollary of propositions \sref{proposition:locally_tame_tame} and \sref{proposition:locally_trivial_holink_fibration}.

\begin{corollary}[local triviality and tameness implies exceptional tameness]\label{corollary:local_condition_exceptional_tameness}
Let $X$ be a metrizable space, and $Y$ a closed subspace of $X$.\\
Assume that for each point $y\in Y$, there is a neighborhood $U$ of $y$ in $X$, a pointed space $(Z,\set{z})$ such that $\set{z}$ is a tame subspace of $Z$, and a homeomorphism
\[ f:(Z,\set{z})\times(U\cap Y)\xTo{\ \cong\ }(U,U\cap Y) \]
(of pairs of spaces) such that the induced map
\[ f:\set{z}\times(U\cap Y)\To U\cap Y \]
is the canonical projection.\\
Then $Y$ is an exceptionally tame subspace of $X$.
\end{corollary}

\section{Homotopically stratified spaces}\label{section:homotopically_stratified_spaces}

Now we apply the notions of tameness in the previous section to define a rather well behaved class of stratified spaces.

\begin{definition}[homotopically stratified space]
Let $(X,A,f)$ be a stratified space with $X$ metrizable and $A$ finite.\\
We say $(X,A,f)$ is {\em homotopically stratified} if, for any $a,b\in A$ with $a\leq b$, $X_b=f^{-1}(\set{b})$ is an exceptionally tame subspace of $f^{-1}(\set{a,b})$.
\end{definition}

\begin{remark}
Note that the fibration condition in \sref{definition:exceptionally_tame} translates in this case to the condition that
\[ t:\holink(X,X_a,X_b)\To X_b \]
is a Hurewicz fibration.
\end{remark}

\begin{remark}
We assume $X$ is metrizable and $A$ is finite in the definition above for a matter of convenience: the main general results in this section make use of those hypotheses.\\
It is not difficult, however, to relax the finiteness conditions on $A$. We will not pursue this.
\end{remark}

\begin{proposition}\label{proposition:exceptionally_tame_homotopically_stratified}
Assume $Y$ is a closed subspace of a metrizable space $X$.\\
The filtered space $\langle\overrightarrow{X,Y}\rangle$ is homotopically stratified if and only if $Y$ is an exceptionally tame subspace of $X$.
\end{proposition}

We now give our principal example of a homotopically stratified space, which is an elaboration of example \sref{example:principal_example_stratified_space}.

\begin{example}\label{example:M^k_homotopically_stratified}
Let $M$ be a manifold, and $S$ a finite set.\\
Then the stratified space $M^{\times S}=\Map(S,M)$ (example \sref{example:principal_example_stratified_space}) is actually homotopically stratified. This can be easily proved by using corollary \sref{corollary:local_condition_exceptional_tameness}.
\end{example}

\begin{example}
We can generalize the previous example.\\
Let $M$ be a $n$-dimensional manifold. Call a finite set $A$ of closed submanifolds of $M$ {\em locally flat} if any point of $M$ has a chart around it
\[ \varphi:\RR^n\To M \]
such that for any $N\in A$, $\varphi^{-1}(N)$ is a linear subspace of $\RR^n$.\\
Assume now $A$ is a finite set of closed submanifolds of $M$ such that
\begin{enum}
\item $A$ is locally flat,
\item $A$ is closed under intersections, and
\item $M\in A$.
\end{enum}
Then we get a homotopically stratified space $(M,A,f)$, where $A$ is ordered by reverse inclusion, and
\[ \func{f}{M}{\potop(A)\qquad}{x}{\min\set{N\in A:x\in N}} \]
As in the previous example (which is a particular case of the present one), the proof amounts to a simple application of corollary \sref{corollary:local_condition_exceptional_tameness}.
\end{example}

Having introduced the notion of homotopically stratified space, we will now enunciate the main general theorems which we will use in our applications.

The next result is theorem 4.9 in the article \cite{Miller} by Miller. A few cautionary remarks are in order regarding notation in that article: the order of the underlying partially ordered set of a stratified space is reversed in \cite{Miller}, in comparison with the definition here (as cautioned earlier in remark \sref{remark:attention_stratified}). Consequently, the direction of the filtered paths is also reversed. In addition \[ \holink\big(\langle\overrightarrow{X,Y}\rangle;X\setminus Y,Y\big) \] is denoted simply by $\holink(X,Y)$ in \cite{Miller}.

\begin{proposition}[Miller's theorem]\label{proposition:Miller}
Let $(X,A,f)$ be a homotopically stratified space.\\
Given any $a,b\in A$, there is a homotopy
\[ G:\popath(X;X_a,X_b)\times I\To\popath(X;X_a,X_b) \]
such that:
\begin{align*}
&G(-,0)=\id_{\popath(X;X_a,X_b)}\\
&(s,t)\circ G(-,\tau)=(s,t)\ ,\ \ \text{for }\tau\in I\\
&G\Big(\popath(X;X_a,X_b)\times\opcl{0}{1}\Big)\subset\holink(X;X_a,X_b)
\end{align*}
In particular, the inclusion
\[ \holink(X;X_a,X_b)\Into\popath(X;X_a,X_b) \]
is a homotopy equivalence over $X_a\times X_b$.
\end{proposition}

\begin{remark}
This result introduces a remarkable simplification to the analysis of the homotopy type of the space of filtered paths. Replacing a large stratified space, now we need only deal with the much simpler case of a pair of spaces:
\[ \holink(X;X_a,X_b)=\holink\Big(\langle\overrightarrow{f^{-1}(\set{a,b}),X_b}\rangle;X_a,X_b\Big) \]
which can be analyzed with the tools of the previous two sections.
\end{remark}

Miller's result has an immediate corollary about deforming homotopies of filtered paths which will be used (in its full strength) in the next chapter. It says roughly that we can deform a path $\gamma$ in $\popath(X;X_a,X_b)$ in a way that:
\begin{enum}
\item the deformation keeps $\gamma(0)$ and $\gamma(1)$ fixed;
\item for any $\tau\in I$, the source and target of $\gamma(\tau)$ are kept fixed through the deformation;
\item for any $\tau\in\opop{0}{1}$, $\gamma(\tau)$ gets immediately deformed to a homotopy link.
\end{enum}

\begin{corollary}\label{corollary:Miller}
Let $(X,A,f)$ be a homotopically stratified space.\\
Given any $a,b\in A$, there is a homotopy
\[ G:\Map\big(I,\popath(X;X_a,X_b)\big)\times I\To\Map\big(I,\popath(X;X_a,X_b)\big) \]
such that
\begin{align*}
&G(-,0)=\id\\
&(\evaluation{0},\evaluation{1})\circ G(-,\tau)=(\evaluation{0},\evaluation{1})&&,\ \ \text{for }\tau\in I\\
&\im\big(\evaluation{\sigma}\circ G(-,\tau)\big)\subset\holink(X;X_a,X_b)\hspace{-1.4em}&&,\ \ \text{for }(\sigma,\tau)\in\opop{0}{1}\times\opcl{0}{1}
\end{align*}
and the diagram
\begin{diagram}[midshaft]
\Map\big(I,\popath(X;X_a,X_b)\big)\times I&\rTo{G}&\Map\big(I,\popath(X;X_a,X_b)\big)\\
\dTo{\proj}&&\dTo{\Map(I,(s,t))}\\
\Map\big(I,\popath(X;X_a,X_b)\big)&\rTo{\Map(I,(s,t))}&\Map(I,X_a\times X_b)
\end{diagram}
commutes.
\end{corollary}

Unfortunately, Miller's result does not meet our needs completely. We conjecture a strengthening of \sref{proposition:Miller} which would be sufficient. It is entirely analogous to \sref{proposition:Miller}, except that it does not fix the stratum at which filtered paths must end.

\begin{conjecture}[strengthening of Miller's result]\label{conjecture:stronger_Miller}
Let $(X,A,f)$ be a homotopically stratified space.\\
For any $a\in A$, there is a homotopy
\[ G:\popath(X;X_a,X)\times I\To\popath(X;X_a,X) \]
such that:
\begin{align*}
&G(-,0)=\id_{\popath(X;X_a,X)}\\
&(s,t)\circ G(-,\tau)=(s,t)\ ,\ \ \text{for }\tau\in I\\
&G\Big(\popath(X;X_a,X)\times\opcl{0}{1}\Big)\subset\holink(X;X_a,X)
\end{align*}
In particular, the inclusion
\[ \holink(X;X_a,X)\Into\popath(X;X_a,X) \]
is a homotopy equivalence over $X_a\times X$.
\end{conjecture}

\comment{
An idea for a possible proof is given now. Let
\[ B\defeq\popath\Big(\langle\overrightarrow{f^{-1}\set{a,b},X_b}\rangle;X_a,X_b\Big) \]
for some $a,b\in A$ with $a\leq b$. Then we would first verify that the homotopy \[ G:B\times I\To B \] constructed in lemma 3.5 of \cite{Miller} verifies a strong form of continuity, which we will call {\em closure continuity}: given a map
\[ \gamma:\clop{0}{1}\To B \subset H(X) \]
such that for any $tau\in\RR^+$
\[ \restrict{\gamma(-)}{\clop{\tau}{l}} \]
is constant on $\clop{\tau}{1}$
[[???]]

talk about Miller's results here, say how they are not sufficient for me, and say how they can plausibly be strengthened to be sufficient for our purposes (by using the apparent ``closure continuity'' of the homotopies to obtain a special path inclusion from $holink(X_b,X)$ into $popath(X_b,X,X)$, in his notation)
State precisely and give a good name to this stronger result.
}

To compensate for this unproven conjecture, we present the following result which will serve our goals. It has the side effect of bringing strong spaces of filtered paths into play. The somewhat involved proof, which is omitted, uses a partition of unity argument similar to the usual proof of Dold's uniformization theorem.

\begin{proposition}\label{proposition:holinks_equiv_strongpopaths}
Let $X$ be a homotopically stratified space.\\
Then there exists a strong deformation retraction of $\popathstrong(X)$ onto $\holink(X)$ over $X\times X$.\\
More concretely, there is a homotopy
\[ G:\popathstrong(X)\times I\To\popathstrong(X) \]
such that
\begin{align*}
&G(-,0)=\id_{\popathstrong(X)}\\
&(s,t)\circ G(-,\tau)=(s,t)\ ,\ \ \text{for }\tau\in I\\
&\restrict{G(-,\tau)}{\holink(X)}=\inclusion\ ,\ \ \text{for }\tau\in I\\
&G\big(\popathstrong(X)\times\set{1}\big)\subset\holink(X)
\end{align*}
\end{proposition}

\section{Application: spaces related to $\M(M)$}\label{section:application_stratified_M(M)}

We have established the connection between categories of sticky configurations, $\M(X)$, and categories of filtered paths, $\smash{\popathcat(X)}$, in section \sref{section:sticky_filtered}. In the preceding three sections, we have given tools to analyze the homotopy type of the filtered path spaces which are the morphisms in $\smash{\popathcat(X)}$, most importantly, propositions \sref{proposition:Miller} and \sref{proposition:strong_neigh_tameness_holink}. In this section, we will use those tools to describe the homotopy type of spaces related to $\M(M)$ for $M$ a manifold.

\begin{proposition}\label{proposition:holink_many_discs_hodiscrete}
Let $k,l,n\in\NN$.\\
The map
\[ s:\holink\big(\!\Map(k,l{\times}\RR^n);\Conf(l{\times}\RR^n,k),i_l\circ\Set(k,l)\big)\To\Conf(l{\times}\RR^n,k) \]
is a homotopy equivalence (where $i_l:l\hookrightarrow l\times\RR^n$ is the canonical inclusion at the origins).
\end{proposition}
\begin{proof}
Set $Y\defeq i_l\circ\Set(k,l)=\set{i_l\circ g\suchthat g\in\Set(k,l)}$ for brevity. Observe first that the source of the map in the proposition is equal to
\[ \holink\big(\langle\overrightarrow{X,Y}\rangle;\Conf(l\times\RR^n,k),Y\big)=\holink\big(\langle\overrightarrow{X,Y}\rangle;X\setminus Y,Y\big) \]
where X is the subspace of $\Map(k,l\times\RR^n)$ given by
\[ X\defeq\Conf(l{\times}\RR^n,k)\cup Y \]
(note that $\Conf(l\times\RR^n,k)=X\setminus Y$).

We will now prove that $X$ is a strong neighborhood of tameness of $Y$ in $X$. Define the continuous map
\[ \func[\rTo{\smash{\qquad}}]{G}{\Map(k,l\times\RR^n)\times I}{\Map(k,l\times\RR^n)}{\quad(f,\tau)}{(1-\tau)f} \]
where multiplication by a scalar is done in each component of $l\times\RR^n$ separately. Note that $G$ gives $\Map(k,l\times\RR^n)$ as a neighborhood of tameness of $Y$ in $\Map(k,l\times\RR^n)$ (see definition \sref{definition:neigh_tameness}). Furthermore, $G$ restricts to a map
\[ G:X\times I\To X \]
which therefore gives $X$ as a strong neighborhood of tameness of $Y$ in $X$. An application of proposition \sref{proposition:strong_neigh_tameness_holink} now finishes the proof.
\end{proof}

While the relation of the space in the previous proposition to categories of sticky configurations is slightly indirect, the next results deal directly with the morphism spaces of $\M(M)$.

\begin{lemma}\label{lemma:homotopy_type_morphism_M}
Assume $k,l,n\in\NN$, and $M$ is a $n$-dimensional manifold without boundary.\\
Let $f:l\times\RR^n\to M$ be an embedding of manifolds.\\
Let $P$ be the pullback of
\begin{diagram}[midshaft,h=3.3em]
&&\holink\big(\parbox[t]{7em}{$\Map(k,\im f);$\\$\Conf(\im f,k),$\\$f\circ i_l\circ\Set(k,l)\big)$}\\
&&\dTo^{\raisebox{-1em}{\scriptsize$s$}}\\
\popath\big(M^{\times k};\Conf(M,k),\Conf(\im f,k)\big)&\rTo{t}&\Conf(\im f,k)
\end{diagram}
Then concatenation of filtered paths induces a natural map
\[ \concat:P\To\holink\big(M^{\times k};\Conf(M,k),f\circ i_l\circ\Set(k,l)\big) \]
which is a homotopy equivalence over $\Conf(M,k)$.
\end{lemma}
\begin{proof}
Let $Y\defeq f\circ i_l\circ\Set(k,l)$ and consider the subspace $X$ of $M^{\times k}$ given by
\[ X\defeq\Conf(M,k)\cup Y \]
Also, let $U$ be the neighborhood of $Y$ in $X$ given by
\[ X\defeq\Conf(\im f,k)\cup Y \]
The result now follows from applying proposition \sref{proposition:fact_holink} to $X$, $Y$, $U$ as defined here.
\end{proof}

\begin{proposition}[homotopy type of morphism space of $\M(M)$]\label{proposition:homotopy_type_morphism_M}
Assume $k,l,n\in\NN$, and $M$ is a $n$-dimensional manifold without boundary.\\
Let $f:l\times\RR^n\to M$ be an embedding of manifolds, and $a:k\to M$ an injective function.\\
Let $Q$ be the pullback of
\begin{diagram}[midshaft,h=3.3em]
&&\holink\big(\parbox[t]{7em}{$\Map(k,l\times\RR^n);$\\$\Conf(l\times\RR^n,k),$\\$i_l\circ\Set(k,l)\big)$}\\
&&\dTo^{\raisebox{-1em}{\scriptsize$s$}}\\
\popath\big(M^{\times k};\set{a},\Conf(\im f,k)\big)&\rTo{f^{-1}\circ\, t}&\Conf(l\times\RR^n,k)
\end{diagram}
Then concatenation of filtered paths induces a natural map
\[ \widetilde{\concat}:Q\To\M(M)(a,f\circ i_l) \]
which is a homotopy equivalence.
Additionally, the canonical projection
\[ \proj:Q\To\popath\big(M^{\times k};\set{a},\Conf(\im f,k)\big) \]
is a homotopy equivalence.
\end{proposition}
\begin{proof}
Given that the map from the preceding lemma \sref{lemma:homotopy_type_morphism_M}
\[ \concat:P\To\holink\big(M^{\times k};\Conf(M,k),f\circ i_l\circ\Set(k,l)\big) \]
is a homotopy equivalence over $\Conf(M,k)$, we obtain a homotopy equivalence
\[ \widetilde{\concat}:Q\To\holink\big(M^{\times k};\set{a},f\circ i_l\circ\Set(k,l)\big) \]
by pulling back along $\set{a}\hookrightarrow\Conf(M,k)$. The last statement of \sref{proposition:Miller} (Miller's theorem) now implies that composing with the inclusion of homotopy links into filtered paths
\[ \widetilde{\concat}:Q\To\popath\big(M^{\times k};\set{a},f\circ i_l\circ\Set(k,l)\big) \]
gives a homotopy equivalence (recall that $M^{\times k}$ is homotopically stratified by example \sref{example:M^k_homotopically_stratified}). From construction \sref{construction:mor_M(X)_popath} (more precisely, equation \seqref{equation:mor_M(X)_popath}), we conclude that
\[ \widetilde{\concat}:Q\To\M(M)(a,f\circ i_l) \]
is a homotopy equivalence. We leave it to the reader to verify that the map coincides with the map described in the statement of the proposition being proved.

Finally, propositions \sref{proposition:holink_many_discs_hodiscrete} and \sref{proposition:holink_s_fibration} guarantee that
\[ s:\holink\big(\!\Map(k,l{\times}\RR^n);\Conf(l{\times}\RR^n,k),i_l\circ\Set(k,l)\big)\To\Conf(l{\times}\RR^n,k) \]
is both a fibration and a homotopy equivalence. Consequently, the pullback
\[ \proj:Q\To\popath\big(M^{\times k};\set{a},\Conf(\im f,k)\big) \]
is also a homotopy equivalence.
\end{proof}

This result has an immediate corollary which comes from noticing that there is a homotopy equivalence (induced by reparametrization of Moore paths) between
\[ \popath\big(M^{\times k};\set{a},\Conf(\im f,k)\big) \]
and the homotopy fibre over $a$ of 
\[ f\circ -:\Conf(l\times\RR^n,k)\To\Conf(M,k) \]

\begin{corollary}
Assume $k,l,n\in\NN$, and $M$ is a $n$-dimensional manifold without boundary.\\
Let $f:l\times\RR^n\to M$ be an embedding of manifolds, and $a:k\to M$ an injective function.\\
There is a homotopy equivalence
\[ \M(M)(a,f\circ i_l)\xTo{\ \simeq\ }\hofibre_{\hspace{.1em}a}\hspace{-.12em}\Big(\Conf(l\times\RR^n,k)\xTo{f\circ -}\Conf(M,k)\Big) \]
\end{corollary}


%% file: conf_emb.tex



\chapter{Sticky configurations and embedding spaces}\label{chapter:sticky<->embeddings}

\section*{Introduction}

In this chapter, we will return to the construction $\M(X)$, and show how it relates to the constructions in the chapter \ref{chapter:embeddings_manifolds} on spaces of embeddings of manifolds. We will essentially show that for any $n$-manifold $M$ (with the appropriate geometric structure), $\M(M)$ is equivalent to the Grothendieck construction of the right module over $\E^G_n$ associated to $M$.

\section*{Summary}

The first section, \sref{section:Groth_embeddings}, defines the appropriate Grothendieck-like construction, $\Total^G_n[M]$, of the right module over $\E^G_n$ associated with any manifold $M$ with a $G$-structure. It also defines the analogous Grothendieck-like construction for $\E_n$, namely $\Total_n[M]$.

Section \sref{section:invariance_total_category} shows that all the $\Top$-categories $\disccat{\Total^G_n[M]}$ (for any $G$ over $GL(n,\RR)$) and $\disccat{\Total_n[M]}$ are equivalent if the underlying manifold $M$ is the same.

Section \sref{section:morphisms_Total_nM} provides a useful analysis of the homotopy type of the morphism spaces in $\disccat{\Total_n[M]}$.

The remaining sections of this chapter are quite long, due to the necessarily convoluted nature of the comparison between $\Total_n[M]$ and $\M(M)$, for $M$ a $n$-dimensional manifold.

In section \sref{section:T_nM<->M(M)} we partly construct a category $\mathcal{Z}_M$, and define functors --- $F_\Total$ and $F_\M$ --- from it to $\disccat{\Total_n[M]}$ and $\M(M)$. Section \sref{section:composition_ZM} finishes the construction of the category $\mathcal{Z}_M$, by defining the composition. Finally, section \sref{section:equiv_TnM_M(M)} shows that the functors 
\begin{align*}
F_\Total&:\mathcal{Z}_M\To\disccat{\Total_n[M]}\\
F_\M&:\mathcal{Z}_M\To\M(M)
\end{align*}
are weak equivalences of Top-categories.

\section{The Grothendieck construction of embeddings}\label{section:Groth_embeddings}

Let us fix $n\in\NN$. This section will describe Grothendieck constructions involving the functors $\internal\E^G_n[M]$ of section \sref{section:internal_presheaves_E^f_n}.

\begin{definition}[total category of manifold with $G$-structure]
Let $G$ be topological group over $GL(n,\RR)$ (see definition \sref{definition:groups/GL(n,R)}).\\
Let $M$ be a $n$-manifold equipped with a $G$-structure (definition \sref{definition:manifold_Gstructure}).\\
We define the {\em total category of \(M\)}, $\Total^G_n[M]$, as the Grothendieck construction of the path category (see definition \sref{definition:path_category_presheaf}) of $\internal\E^{\smash{G}}_n[M]$ (definition \sref{definition:internal_presheaves_E^f_n}):
\[ \Total^G_n[M]\defeq\Groth\big(\pathcat\circ\internal\E^G_n[M]\big) \]
which is an internal category in $\Top$.
\end{definition}

\begin{remark}[functoriality of total category]
The above definition extends to a functor
\[ \Total^G_n[-]:\big(\REmb^G_n\big)_0\To\Cat(\Top) \]
in a straightforward manner.
\end{remark}

\begin{remark}\label{remark:fibrant_T^G_n}
The internal category (in $\Top$) $\Total^G_n[M]$ is fibrant (in the sense of definition \sref{definition:fibrant_internal_category}) by corollary \sref{corollary:fibrant_Groth(path)}.
\end{remark}

Given a morphism $h:G\to H$ in $\overGL{n}$, the two-cell $\internal\varepsilon_h$ in diagram \seqref{diagram:natural_transf_internal_group} gives a natural transformation
\[ \internal\varepsilon_h:\internal\E^G_n[-]\To\internal\E^H_n[h_\ast -]\circ\op{\internal(h_\ast)} \]
This induces (via construction \sref{construction:Groth(f,alpha)} and proposition \sref{proposition:Groth(functorA->B)}; see also proposition \sref{proposition:naturality_internal_path_category_presheaf}) a natural transformation in $\HomCatlarge{(\REmb^{\smash{G}}_n)_0}{\Cat(\Top)}$
\[ \Total^h_n[-]:\Total^G_n[-]\To\Total^H_n[h_\ast -] \]
Moreover, given another morphism $h':H\to I$ in $\overGL{n}$, the diagram
\begin{equation}\label{diagram:comm_diagram_T^G_n}
\begin{diagram}[midshaft,h=2.4em]
\Total^G_n[-]&\rTo{\;\Total^h_n[-]\;}&\Total^H_n[h_\ast -]\\
&\rdTo_{\mathllap{\Total^{h' h}_n[-]}}&\dTo_{\mathrlap{\Total^{h'}_n[h_\ast -]}}\\
&&\Total^I_n[(h'\circ h)_\ast -]
\end{diagram}
\end{equation}
commutes in $\HomCatlarge{(\REmb^{\smash{G}}_n)_0}{\Cat(\Top)}$.

\begin{definition}[total category of manifold]
Let $M$ be a $n$-dimensional manifold without boundary.\\
The {\em total category of \(M\)}, $\Total_n[M]$, is defined to be the Grothendieck construction of the path category (definition \sref{definition:path_category_presheaf}) of $\internal\E_n$ (see observation \sref{remark:internal_presheaf_E_n}):
\[ \Total_n[M]\defeq\Groth\big(\pathcat\circ\internal\E_n[M]\big) \]
\end{definition}

\begin{remark}
Similarly to observation \sref{remark:fibrant_T^G_n}, we conclude that the internal category (in $\Top$) $\Total_n[M]$ is fibrant.
\end{remark}

As before, this extends to a functor
\[ \Total_n[-]:(\Emb_n)_0\To\Cat(\Top) \]

The two-cell in diagram \seqref{diagram:E^G_n_E_n} translates to a natural transformation
\[ q:\internal\E^G_n[-]\To\internal\E_n[-]\circ\op{q} \]
for any $G$ in $\overGL{n}$. This induces (via \sref{proposition:Groth(functorA->B)} and \sref{proposition:naturality_internal_path_category_presheaf}) a natural transformation in $\HomCat{(\Emb^G_n)_0}{\Cat(\Top)}$
\[ q:\Total^G_n[-]\To\Total_n[-] \]
(where we identify $qM$ with $M$). Moreover, for any morphism $h:G\to H$ in $\overGL{n}$, we obtain a commutative diagram
\begin{equation}\label{diagram:q_absorbs_T}
\begin{diagram}
\Total^G_n[-]&\rTo{\;\Total^h_n[-]\;}&\Total^H_n[h_\ast -]\\
&\rdTo(1,2)_{\mathllap{q}}\ldTo(1,2)_{\mathrlap{q}}\\
&\Total_n[-]
\end{diagram}
\end{equation}

\comment{
The two-cell in diagram \seqref{diagram:E_n_E^GL_n} translates to a natural transformation
\[ \internal\E_n[-]\To\internal\E^{GL(n,\RR)}_n[-]\circ\op{\inclusion} \]
(where $\inclusion:\internal\E_n\hookrightarrow\internal\E^{\smash{GL(n,\RR)}}_n$) which is an objectwise weak equivalence. This, in turn, induces (via construction \sref{construction:Groth(f,alpha)}) a natural transformation in $\HomCat{(\Emb_n)_0}{\Cat(\Top)}$
\[ \Total^\inclusion_n[-]:\Total_n[-]\To\Total^{GL(n,\RR)}_n[-] \]
which, according to proposition \sref{} [[missing reference]], is an objectwise weak equivalence (\sref{} [[missing reference: notion of weak equivalence of internal categories in $\Top$]]). Consequently, for any $n$-manifold without boundary $M$, we obtain a weak equivalence of $\Top$-categories
\[ \disccat{\Total^\inclusion_n[M]}:\disccat{\Total_n[M]}\xTo{\;\sim\;}\disccat{\Total^{GL(n,\RR)}_n[M]} \]
(by proposition \sref{} [[missing reference: weak equivalence between fibrant internal categories gives weak equivalence of corresponding $\Top$-categories]] since both $\Total_n[M]$ and $\Total^{\smash{GL(n,\RR)}}_n[M]$ are fibrant internal categories in $\Top$).
\end{remark}
}

\section{Homotopy invariance of total category}\label{section:invariance_total_category}

Let us fix $n\in\NN$. This section is devoted to proving the following result.

\begin{proposition}\label{proposition:equivalence_T^G_n}
Let $G$ be an object of $\overGL{n}$, $M$ a $n$-manifold equipped with a $G$-structure.\\
The $\Top$-functor
\[ \disccat{q}:\disccat{\Total^G_n[M]}\To\disccat{\Total_n[M]} \]
is a weak equivalence of $\Top$-categories (see definition \sref{definition:weak_equivalence_Top-cat}).
\end{proposition}

We prove this result in two parts.

\begin{lemma}\label{lemma:local_equivalence_T^G_n}
Let $G$ be an object of $\overGL{n}$, and $M$ a $n$-manifold with a $G$-structure.\\
The $\Top$-functor
\[ \disccat{q}:\disccat{\Total^G_n[M]}\To\disccat{\Total_n[M]} \]
is a local homotopy equivalence of $\Top$-categories.
\end{lemma}
\begin{proof}[Sketch of proof]
According to corollary \sref{corollary:local_homotopy_equiv_id_equiv}, it suffices to show that for each $k,l\in\NN$ the commutative square
\begin{diagram}[midshaft,objectstyle=\scriptstyle,labelstyle=\scriptscriptstyle]
\REmb^G_n(l\times\RR^n,k\times\RR^n)\times\REmb^G_n(k\times\RR^n,M)&\rTo{\ q\times q\ }&\Emb(l\times\RR^n,k\times\RR^n)\times\Emb(k\times\RR^n,M)\\
\dTo{(\composition,\proj)}&&\dTo{(\composition,\proj)}\\
\REmb^G_n(l\times\RR^n,M)\times\REmb^G_n(k\times\RR^n,M)&\rTo{q\times q}&\Emb(l\times\RR^n,M)\times\Emb(k\times\RR^n,M)
\end{diagram}
(where ``$\composition$'' denotes the composition of embeddings) is homotopy cartesian. The bottom map is a Hurewicz fibration by the last remark in \sref{remark:homotopy_properties_REmb}. It is straightforward to show that the diagram is also a pullback square: it amounts to a tedious verification (on the level of sets and topologies) directly from the definition of $G$-augmented embedding spaces, noticing that any map of principal $G$-bundles is a fibrewise isomorphism. In conclusion, the square above is homotopy cartesian, as was required.
\end{proof}

\begin{lemma}\label{lemma:essentially_surjective_T^G_n}
Let $G$ be an object of $\overGL{n}$, and $M$ a $n$-manifold with a $G$-structure.\\
The functor
\[ \pi_0\big(\disccat{q}\big):\pi_0\Big(\disccat{\Total^G_n[M]}\Big)\To\pi_0\Big(\disccat{\Total_n[M]}\Big) \]
is essentially surjective.
\end{lemma}
\begin{proof}
Assume $k\in\NN$, and $f\in\Emb(k\times\RR^n,M)$ is an object of $\disccat{\Total_n[M]}$. Consider the maps
\[ \Emb(k\times\RR^n,M)\xTo{D_0}\Frame\big(T\Conf(M,k)\big)\xTo{\proj}\Conf(M,k) \]
and
\[ \REmb^G_n(k\times\RR^n,M)\xTo{\RDer^G_0}\Princ_{G^{\times k}}\big(T\Conf(M,k)\big)\xTo{\proj}\Conf(M,k) \]
and let $g\in\REmb^G_n(k\times\RR^n,M)$ be such that
\[ \proj\circ D_0(f)=\proj\circ\RDer^G_0(g) \]
which exists because $\RDer^G_0$ is a trivial fibration (proposition \sref{proposition:RDer^G_0_equivalence}), and therefore surjective. The object $\disccat{q}(g)$ is just $q(g)$ where $q$ is as usual the projection
\[ q:\REmb^G_n(k\times\RR^n)\To\Emb(k\times\RR^n,M) \]
The following lemma implies that $q(g)$ is isomorphic to $f$ in $\pi_0\big(\disccat{\Total_n[M]}\big)$, since
\[ q(g)\circ i_k=\proj\circ D_0\circ q(g)=\proj\circ\RDer^G_0(g)=\proj\circ D_0(f)=f\circ i_k \]
which ends this proof.
\end{proof}

\begin{lemma}
Let $M$ be a $n$-manifold, and $k\in\NN$.\\
If $f,g\in\Emb(k\times\RR^n,M)$ verify ($i_k:k\hookrightarrow k\times\RR^n$ is the inclusion at the origins from \seqref{equation:i_k})
\[ f\circ i_k=g\circ i_k \]
then $f$, $g$ are isomorphic objects of $\pi_0\big(\disccat{\Total_n[M]}\big)$.
\end{lemma}
\begin{proof}[Sketch of proof]
Let us first prove the following special case of the lemma: if $f,g\in\Emb(k\times\RR^n,M)$ verify
\begin{align*}
f\circ i_k&=g\circ i_k\\
\im f&\subset\im g
\end{align*}
then $f$, $g$ are isomorphic objects of $\pi_0\big(\disccat{\Total^G_n[M]}\big)$.

Under those conditions, there exists an embedding
\[ \phi:k\times\RR^n\To k\times\RR^n \]
such that $g\circ\phi=f$. This determines a morphism $\phi:f\to g$ in $\disccat{\Total_n[M]}$.
Choose now an embedding
\[ \phi':k\times\RR^n\To k\times\RR^n \]
such that $\phi'\circ i_k=i_k$ and its differential at a point $(i,0)$ of $k\times\RR^n$ is
\[ d\phi'(i,0)=\big(d\phi(i,0)\big)^{-1} \]
Then we conclude that
\[ D_0(g)=D_0(f\circ\phi') \]
Since $D_0$ is a trivial fibration (by proposition \sref{proposition:D_0_equivalence}) and therefore has contractible fibres, there exists a Moore path, $x=(\gamma,1)$, in $\Emb(k\times\RR^n,M)$ which
\begin{enum}
\item starts at $g$,
\item ends at $f\circ\phi'$, and
\item such that $D_0\circ\gamma$ is a constant path in $\Emb(k\times\RR^n,M)$.
\end{enum}
We thus get an induced morphism $(x,\phi'):g\to f$ in $\disccat{\Total_n[M]}$.

The composition
\[ f\xTo{\phi}g\xTo{(x,\phi')}f \]
is such that the corresponding path $D_0\big(\gamma(-)\circ\phi\big)$ is constant. Since $D_0$ has contractible fibres, $\gamma(-)\circ\phi$ can be deformed --- keeping the endpoints fixed --- to a Moore path of length 0 in $\Emb(k\times\RR^n,M)$. This supplies a path $\lambda_f$ in $\disccat{\Total_n[M]}(f,f)$ from $\phi\circ(x,\phi')$ to $\id_f$.

Analogously, the composition
\[ g\xTo{(x,\phi')}f\xTo{\phi}g \]
is such that the corresponding path $D_0\circ\gamma$ is constant. Since $D_0$ is a trivial fibration, $\gamma$ can be deformed --- keeping the endpoints fixed --- to a Moore path of length 0 in $\Emb(k\times\RR^n,M)$. This gives a path $\lambda_g$ in $\disccat{\Total_n[M]}(g,g)$ from $(x,\phi')\circ\phi$ to $\id_g$.

The paths $\lambda_f$ and $\lambda_g$ show that the the morphisms $(x,\phi')$ and $\phi$ induce inverse isomorphisms in $\pi_0\big(\disccat{\Total^G_n[M]}\big)$. Hence $f$ is isomorphic to $g$ in $\pi_0\big(\disccat{\Total^G_n[M]}\big)$. We have thus proved the special case of the lemma.

Assuming now the special case of the lemma, the general case follows easily. Let $f,g\in\Emb(k\times\RR^n,M)$ be such that $f\circ i_k=g\circ i_k$. Choose an embedding
\[ \phi:k\times\RR^n\To k\times\RR^n \]
such that
\begin{align*}
\phi\circ i_k&=i_k\\
\im(f\circ\phi)&\subset\im g
\end{align*}
Then the embeddings $f\circ\phi$ and $f$ verify the hypothesis of the special case of the lemma, as do the embeddings $f\circ\phi$ and $g$. Therefore, there are isomorphisms
\[ f\simeq f\circ\phi\simeq g \]
in $\pi_0\big(\disccat{\Total^G_n[M]}\big)$.
\end{proof}

Proposition \sref{proposition:equivalence_T^G_n} follows from lemmas \sref{lemma:local_equivalence_T^G_n} and \sref{lemma:essentially_surjective_T^G_n}. It has the following corollary.

\begin{corollary}
Let $h:G\to H$ be a morphism in $\overGL{n}$, and $M$ a $n$-manifold with a $G$-structure.\\
The $\Top$-functor
\[ \disccat{\Total^h_n[M]}:\disccat{\Total^G_n[M]}\To\disccat{\Total^H_n[h_\ast M]} \]
is a weak equivalence of $\Top$-categories.
\end{corollary}
\begin{proof}
The functor above is one of the arrows in diagram \seqref{diagram:q_absorbs_T}. Since the other two arrows are weak equivalences --- by proposition \sref{proposition:equivalence_T^G_n} --- then $\disccat{\Total^h_n[M]}$ is one as well.
\end{proof}

\section{Analysis of morphisms of $\Total_n[M]$}\label{section:morphisms_Total_nM}

This section is devoted to a simple analysis of the homotopy type of the morphism spaces in $\disccat{\Total_n[M]}$, which will be of use later.

\begin{construction}
Let be a $n$-manifold without boundary ($n\in\NN$), and
\begin{align*}
e&\in\Emb_n(k\times\RR^n,M)\\
f&\in\Emb_n(l\times\RR^n,M)
\end{align*}
Recall the natural homotopy equivalence
\begin{equation}\label{equation:map_T_n->hofib_Emb}
\disccat{\Total_n[M]}(e,f)\To\hofibre_e\!\Big(\Emb(k\times\RR^n,l\times\RR^n)\xTo{f\circ -}\Emb(k\times\RR^n,M)\Big)
\end{equation}
from proposition \sref{proposition:morphism_space_Groth(path)}.\\
Observing that the square
\begin{equation}\label{diagram:square_Emb->Conf}
\begin{diagram}
\Emb_n(k\times\RR^n,l\times\RR^n)&\rTo{\ -\circ i_k\ }&\Conf(l\times\RR^n,k)\\
\dTo{f\circ -}&&\dTo{f\circ -}\\
\Emb_n(k\times\RR^n,M)&\rTo{-\circ i_k}&\Conf(M,k)
\end{diagram}
\end{equation}
commutes, we obtain an induced map between the homotopy fibres of the vertical maps. Composing it with \seqref{equation:map_T_n->hofib_Emb} gives a map
\[ \varsigma:\disccat{\Total_n[M]}(e,f)\To\hofibre_{e\circ i_k}\!\Big(\Conf(l\times\RR^n,k)\xTo{f\circ -}\Conf(M,k)\Big) \]
\end{construction}

\begin{proposition}[homotopy type of morphisms in $\disccat{\Total_n{[}M{]}}$]\label{proposition:analysis_Total_nM(e,f)}
Let $M$ be a $n$-manifold without boundary (where $n\in\NN$).\\
Let $k,l\in\NN$, $e\in\Emb_n(k\times\RR^n,M)$, and $f\in\Emb_n(l\times\RR^n,M)$.\\
The map
\[ \varsigma:\disccat{\Total_n[M]}(e,f)\To\hofibre_{e\circ i_k}\!\Big(\Conf(l\times\RR^n,k)\xTo{f\circ -}\Conf(M,k)\Big) \]
is a homotopy equivalence.
\end{proposition}
\begin{proof}
From the construction of the map $\varsigma$, it suffices to show that the commutative square \seqref{diagram:square_Emb->Conf} is homotopy cartesian. Consider then the commutative diagram
\begin{diagram}
\Emb_n(k\times\RR^n,l\times\RR^n)&\rTo_{\sim}^{\;D_0\;}&\Frame\big(T\Conf(l\times\RR^n,k)\big)&\rTo{\;\proj\;}&\Conf(l\times\RR^n,k)\\
\dTo{f\circ -}&&\dTo{D f\circ -}&&\dTo{f\circ -}\\
\Emb_n(k\times\RR^n,M)&\rTo_{\sim}^{D_0}&\Frame\big(T\Conf(M,k)\big)&\rTo{\proj}&\Conf(M,k)
\end{diagram}
where $D_0$ is the derivative at the origins defined in \sref{definition:derivative_origins}, and $D f$ is the derivative of $f$ (definition \sref{definition:derivative}). The two arrows marked $D_0$ are homotopy equivalences, by proposition \sref{proposition:D_0_equivalence}. Consequently, the inner left square is homotopy cartesian. On the other hand, the maps marked $\proj$ are Hurewicz fibrations; also, the inner right square is cartesian, since that square is a map of principal $GL(n,\RR)$-bundles. Therefore, the inner right square is homotopy cartesian. In conclusion, the outer square is a homotopy pullback square. The proof is completed by identifying the outer square in the above diagram with the square \seqref{diagram:square_Emb->Conf}.
\end{proof}

\begin{remark}
This proposition and proposition \sref{proposition:homotopy_type_morphism_M} (which says that a morphism space of $\M(M)$ is equivalent to the target of $\varsigma$) are the motivation and the foundation for the proof given later that $\M(M)$ and $\disccat{\Total_n[M]}$ are weakly equivalent.
\end{remark}

We could similarly construct a homotopy equivalence (natural in $G\in\overGL{n}$ and $M\in\REmb^{\smash{G}}_n$\footnote{We leave it to the reader (again, as in observation \sref{remark:naturality_RDer}) to precisely say what this means.})
\begin{equation}\label{equation:map_Total^G_n->Conf}
\disccat{\Total^G_n[M]}(e,f)\xTo{\ \sim\ }\hofibre_{e\circ i_k}\!\Big(\Conf(l\times\RR^n,k)\xTo{f\circ -}\Conf(M,k)\Big)
\end{equation}
for a $n$-manifold $M$ with a $G$-structure, and $G$-augmented embeddings $e$, $f$. The proof that it is a homotopy equivalence would parallel the proof above (using proposition \sref{proposition:RDer^G_0_equivalence} instead of \sref{proposition:D_0_equivalence}).

The natural homotopy equivalence \seqref{equation:map_Total^G_n->Conf} gives an immediate direct proof of the fact that
\[ \disccat{\Total^h_n[M]}:\disccat{\Total^G_n[M]}\To\disccat{\Total^H_n[h_\ast M]} \]
is a local homotopy equivalence, for any morphism $h:G\to H$ in $\overGL{n}$. Furthermore, the obvious commutative diagram (where we denote the underlying embeddings of $e$ and $f$ by the same letters)
\begin{diagram}
\disccat{\Total^G_n[M]}(e,f)&\rTo{\disccat{q}}&\disccat{\Total_n[M]}(e,f)\\
&\rdTo(1,2){\mathrlap{\seqref{equation:map_Total^G_n->Conf}}}\ldTo(1,2){\mathllap{\varsigma}}\\
&\hspace{-6em}\hofibre_{e\circ i_k}\!\Big(\Conf(l\times\RR^n,k)\xTo{f\circ -}\Conf(M,k)\Big)\hspace{-6em}
\end{diagram}
gives a quick reproof of lemma \sref{lemma:local_equivalence_T^G_n}.

\section{Connecting $\Total_n[M]$ and $\M(M)$}\label{section:T_nM<->M(M)}

Throughout this section we fix $n\in\NN$, and $M$ a $n$-manifold without boundary.

Having established in the preceding section that any $\disccat{\Total^G_n[M]}$ is weakly equivalent (as a $\Top$-category) to $\disccat{\Total_n[M]}$, we will now proceed to show that $\disccat{\Total_n[M]}$ is weakly equivalent to $\M(M)$ (from section \sref{section:sticky_configurations}). We will do this by constructing a natural (two arrow) zig-zag of $\Top$-categories between $\disccat{\Total_n[M]}$ and $\M(M)$ in the present section and the next, and proving later that the maps in the zig-zag are weak equivalences.

We could similarly construct direct two arrow zig-zags (of weak equivalences) between $\disccat{\Total^G_n[M]}$ and $\M(M)$. For simplicity, we will omit this.

We proceed to construct the zig-zag in $\Top$-categories, which we will denote by
\[ \disccat{\Total_n[M]}\xlongleftarrow{\ F_\Total\ }\mathcal{Z}_M\rTo{\ F_\M\ }\M(M) \]

\begin{definition}[objects of $\mathcal{Z}_M$]
Define $\ob\mathcal{Z}_M$ to be the set
\[ \ob\mathcal{Z}_M\defeq\ob\big(\disccat{\Total_n[M]}\big)=\coprod_{k\in\NN}\Emb(k\times\RR^n,M) \]
\end{definition}

\begin{remark}
We ignore the obvious topology $\ob\mathcal{Z}_M$, since it will not be necessary.\\
However it is possible to define a category $C$ internal to $\Top$ whose discretization $\disccat{C}$ is $\mathcal{Z}_M$ and whose space of objects is the topological space
\[ \ob C=\coprod_{k\in\NN}\Emb(k\times\RR^n,M) \]
with the coproduct topology.
\end{remark}

\begin{definition}[functor $F_\Total:\mathcal{Z}_M\to\Total_n{[}M{]}$: map on objects]
The map of sets $\ob F_\Total$ is defined to be the identity function
\[ \id:\ob\mathcal{Z}_M\To\ob\Big(\disccat{\Total_n[M]}\Big) \]
\end{definition}

\begin{definition}[functor $F_\M:\mathcal{Z}_M\to\M(M)$: map on objects]\label{definition:F_M_objects}
The map of sets
\[ \ob F_\M:\ob\mathcal{Z}_M\To\ob\M(M) \]
associates to an embedding $e\in\Emb(k\times\RR^n,M)$ the following element of $\Conf(M,k)$
\[ \ob F_\M(e)=e\circ i_k \]
where $i_k:k\to k\times\RR^n$ is the canonical inclusion at the origins (equation \seqref{equation:i_k}).
\end{definition}

A few pictorially inclined definitions will be useful in order to define the morphisms in $\mathcal{Z}_M$, and give some intuition into their structure.

\begin{definition}[squares with assigned verticals]
Let $X$ be a topological space, and $Y\to H(X)$ a map of topological spaces ($H$ is the space of Moore paths from \sref{definition:Moore_path_space}).\\
The {\em space of squares in \(X\) with verticals in \(Y\)}, $\squares{X}{Y}$, is defined to be
\[ \squares{X}{Y}\defeq H(Y) \]
\end{definition}

The above are only squares in the loosest sense of the word. Nevertheless, the pictorial intuition coming from this designation is useful. It comes from the fact that we have a map
\begin{equation}\label{equation:squares->mapsonsquare}
\squares{X}{Y}\To H(H(X))\simeq\Map(\clcl{0}{1}\times\clcl{0}{1},X)
\end{equation}
The purpose of the above definition is two-fold: first, it constrains the resultant maps $\clcl{0}{1}\times\clcl{0}{1}\to X$ to having some specified type (given by the inclusion $Y\to H(X)$)) when the first coordinate is fixed. Second, the use of Moore paths, instead of usual paths, allows for strictly associative concatenation or gluing: in the case of squares, we can glue along common edges.

\begin{definition}[edges of squares]
Let $X$ be a topological space, and $f:Y\to H(X)$ a map of topological spaces.\\
We define the maps
\begin{align*}
T&:\squares{X}{Y}\To H(X)&&\textit{top edge}\\
B&:\squares{X}{Y}\To H(X)&&\textit{bottom edge}\\
L&:\squares{X}{Y}\To Y&&\textit{left edge}\\
R&:\squares{X}{Y}\To Y&&\textit{right edge}
\end{align*}
to be (recall the maps defined on Moore paths in \sref{definition:maps_Moore_path_space})
\begin{align*}
T&:H(Y)\xTo{H(f)}H(H(X))\xTo{H(t)}H(x)\\
B&:H(Y)\xTo{H(f)}H(H(X))\xTo{H(s)}H(x)
\end{align*}
\begin{align*}
L&:H(Y)\xTo{s}Y\\
R&:H(Y)\xTo{t}Y
\end{align*}
\end{definition}

\begin{definition}[triangles with assigned verticals]
Let $X$ be a topological space, and $f:Y\to H(X)$ a map of topological spaces.\\
The {\em space of triangles in \(X\) with verticals in \(Y\)}, $\rtriangles{X}{Y}$, is defined to be the limit of
\vspace{-.3em}
\begin{diagram}[midshaft,h=2em]
\set{0}\\
\dInto\\
\clop{0}{+\infty}&\lTo{l}&H(X)&\lTo{f}&Y&\lTo{L}&\squares{X}{Y}
\end{diagram}
\ie the subspace of $\squares{X}{Y}$ of squares whose left edge has zero length.
\end{definition}

The motivation for this nomenclature is the existence of a map
\begin{equation}\label{equation:triangles->mapsontriangles}
\rtriangles{X}{Y}\To\Map(tri,X)
\end{equation}
where \[ tri\defeq\set{(x,y)\in\clcl{0}{1}\times\clcl{0}{1}\suchthat y\leq x} \] is the subspace of $\clcl{0}{1}\times\clcl{0}{1}$ below the diagonal. Again, the purpose of this definition is to constrain the type of paths obtained when fixing the first coordinate, and also allowing for strictly associative gluing of maps.

\begin{definition}[edges of triangles]
Let $X$ be a topological space, and $f:Y\to H(X)$ a map of topological spaces.\\
We define the maps
\begin{align*}
T&:\rtriangles{X}{Y}\To H(X)&&\textit{top (or diagonal) edge}\\
B&:\rtriangles{X}{Y}\To H(X)&&\textit{bottom edge}\\
R&:\rtriangles{X}{Y}\To Y&&\textit{right edge}
\end{align*}
to be the restriction of the corresponding maps on $\square_X(Y)$ to $\rtriangle_X(Y)$:
\begin{align*}
T&:\rtriangles{X}{Y}\Into\squares{X}{Y}\xTo{T}H(x)\\
B&:\rtriangles{X}{Y}\Into\squares{X}{Y}\xTo{B}H(x)\\
R&:\rtriangles{X}{Y}\Into\squares{X}{Y}\xTo{R}Y
\end{align*}
\end{definition}

We can now define the morphisms of the category $\mathcal{Z}_M$. Recall the several spaces of filtered paths in stratified spaces from chapter \ref{chapter:stratified_spaces}.

\begin{remark}\label{remark:maps_mor_Total_n}
Given embeddings $e\in\Emb_n(k\times\RR^n,M)$, $f\in\Emb_n(l\times\RR^n,M)$, there are natural projections
\begin{align*}
\mathbf{m}&:\disccat{\Total_n[M]}(e,f)\To\Emb(k\times\RR^n,l\times\RR^n)\\
\mathbf{h}&:\disccat{\Total_n[M]}(e,f)\To H\big(\Emb(k\times\RR^n,M)\big)
\end{align*}
\end{remark}

\begin{definition}[morphisms of $\mathcal{Z}_M$]\label{definition:mor_Z_M}
Let $k,l\in\NN$.\\
For any $e\in\Emb(k\times\RR^n,M)$, and $f\in\Emb(l\times\RR^n,M)$, define the topological space $\mathcal{Z}_M(e,f)$ to be the subspace of the product
\[ \big(\disccat{\Total_n[M]}(e,f)\big)\times\popathstrong\big(\Map(k,l\times\RR^n)\big)\times\rtrianglesBig{M^{\times k}}{\popathstrong\big(M^{\times k}\big)}\times\popath\big(M^{\times k}\big) \]
constituted by tuples $(a,b,c,d)$ such that
\begin{align*}
T(c)&=d\\
B(c)&=\mathbf{h}(a)\circ i_k\\
f\circ b&=R(c)\\
s(b)&=\mathbf{m}(a)\circ i_k\\
t(b)&=i_k\circ\pi_0\big(\mathbf{m}(a)\big)
\end{align*}
where
\begin{enum}
\item we make the identifications (in the obvious way)
\[ \pi_0(k\times\RR^n)=k\qquad\quad\pi_0(l\times\RR^n)=l \]
\item $i_k:k\hookrightarrow k\times\RR^n$ is the canonical inclusion (at the origins)
\item $s(b)=b(0)$ denotes the source (beginning point) of $b$
\item $t(b)=b(l(b))$ denotes the target (end point) of $b$
\end{enum}
\end{definition}

\begin{remark}[variation: $\mathcal{Z}'_M$]\label{remark:Z'M}
We could replace both occurrences of strong spaces of filtered paths ($\smash{\popathstrong}$) in the definition above by usual filtered path spaces ($\smash{\popath}$). Let us call the $\Top$-category resulting from such a replacement by $\mathcal{Z}'_M$.\\
The composition in $\mathcal{Z}'_M$ is defined by the same formula as the composition of $\mathcal{Z}_M$ (which will be described in the next section).
\end{remark}

\begin{remark}\label{remark:mor_Z_M_bigdiagram}
The above is a fairly complicated definition. It defines $\mathcal{Z}_M(e,f)$ as the limit of the following diagram of topological spaces
\begin{diagram}[balance,midshaft,objectstyle=\scriptstyle,labelstyle=\scriptscriptstyle]
H(\Emb(k\times\RR^n,M))&\lTo{\mathbf{h}}&\fbox{\vphantom{{\scriptsize p}{\small d}}\smash{$\scriptstyle\hspace{-.1em}\disccat{\Total_n[M]}(e,f)\hspace{-.1em}$}}&\rTo{\mathbf{m}}&\Emb(k\times\RR^n,l\times\RR^n)&\rTo{\pi_0}&\Set(k,l)\\
\dTo_{-\circ i_k}&&&&\dTo{-\circ i_k}&&\dTo{i_k\circ -}\\
H(M^{\times k})&\lTo{B}&\fbox{\vphantom{{\scriptsize p}d}\smash{$\scriptstyle\hspace{-.12em}\rtriangles{M^{\times k}}{\popathstrong(M^{\times k})}\hspace{-.12em}$}}&&\Map(k,l\times\RR^n)&&\Map(k,l\times\RR^n)\\
&&\dTo{T}&\rdTo{\mathrlap{R}}&&\luTo{s}&\uTo{t}\\
\fbox{\vphantom{{\tiny p}{\small d}}\smash{$\scriptstyle\hspace{-.12em}\popath(M^{\times k})\hspace{-.12em}$}}&\rTo{\inclusion}&H(M^{\times k})&&\popathstrong(M^{\times k})&\lTo{f\circ -}&\fbox{\vphantom{{\scriptsize p}{\small d}}\smash{$\scriptstyle\hspace{-.12em}\popathstrong(\Map(k,l\times\RR^n))\hspace{-.12em}$}}
\end{diagram}
where boxes are drawn around each factor appearing in definition \sref{definition:mor_Z_M}.\\
Again, we could replace all three occurrences of strong spaces of filtered paths ($\smash{\popathstrong}$) in the diagram above by usual filtered path spaces ($\smash{\popath}$). The resulting diagram would have $\mathcal{Z}'_M(e,f)$ as its limit.
\end{remark}

\begin{notation}
We will denote elements of $\mathcal{Z}_M(e,f)$ by 4-tuples like we did in definition \sref{definition:mor_Z_M}.
\end{notation}

We will finish this section by defining the functors $F_\M$ and $F_\Total$ on the level of morphisms. In the next section, we will occupy ourselves with describing the composition in $\mathcal{Z}_M$.

\begin{definition}[functor $F_\Total:\mathcal{Z}_M\to\Total_n{[}M{]}$: map on morphisms]
Let $k,l\in\NN$, $e\in\Emb(k\times\RR^n,M)$, and $f\in\Emb(l\times\RR^n,M)$.\\
The map
\[ \mor F_\Total:\mathcal{Z}_M(e,f)\To\disccat{\Total_n[M]}(e,f) \]
is the canonical projection (see definition \sref{definition:mor_Z_M}).
\end{definition}

\begin{definition}[functor $F_\M:\mathcal{Z}_M\to\M(M)$: map on morphisms]
Let $k,l\in\NN$, $e\in\Emb(k\times\RR^n,M)$, and $f\in\Emb(l\times\RR^n,M)$.\\
Recall definition \sref{definition:F_M_objects}.\\
The map
\[ \mor F_\M:\mathcal{Z}_M(e,f)\To\M(M)(e\circ i_k,f\circ i_l) \]
is given by (denoting elements of $\mathcal{Z}_M(e,f)$ by 4-tuples like we did in definition \sref{definition:mor_Z_M})
\[ (\mor F_\M)(a,b,c,d)=d \]
Equivalently, $\mor F_\M$ is the unique map which makes the diagram
\begin{diagram}[midshaft]
\mathcal{Z}_M(e,f)&\rTo{\mor F_\M}&\M(M)(e\circ i_k,f\circ i_l)\\
&\rdTo(1,2)_{\mathllap{\proj}}\ldTo(1,2)_{\mathrlap{\inclusion}}\\
&\popath\big(M^{\times k}\big)
\end{diagram}
commute.
\end{definition}

\section{Composition in $\mathcal{Z}_M$}\label{section:composition_ZM}

We again fix (in this section) $n\in\NN$, and a $n$-manifold without boundary $M$.

We now proceed to construct the composition in $\mathcal{Z}_M$, which is a bit involved. The most non-trivial part of defining the composition resides in gluing the triangles in the morphisms of $\mathcal{Z}_M$ to obtain new triangles.

\begin{construction}[horizontal gluing of squares ($g_H$)]
Let $X$ be a topological space, and $f:Y\to H(X)$ a map of topological spaces.\\
Let $P_H$ be the pullback in the pullback square
\begin{diagram}[midshaft,h=2.1em]
P_H&\rTo{p_1}&\squares{X}{Y}\\
\dTo{p_2}&&\dTo_{R}\\
\squares{X}{Y}&\rTo{L}&Y
\end{diagram}
The {\em horizontal gluing} map
\[ g_H:P_H\To\squares{X}{Y} \]
is exactly the concatenation map for $H(Y)$ (see definition \sref{definition:concat_Moore_paths}).
\end{construction}

\begin{remark}[intuition for horizontal gluing]
This horizontal gluing operation corresponds approximately (under the map \seqref{equation:squares->mapsonsquare}) to having continuous functions
\begin{align*}
\clcl{0}{1}\times\clcl{0}{1}&\To X\\
\clcl{1}{2}\times\clcl{0}{1}&\To X
\end{align*}
which coincide on $\set{1}\times\clcl{0}{1}$, and gluing them to obtain a continuous function
\[ \clcl{0}{2}\times\clcl{0}{1}\To X \]
\end{remark}

\begin{construction}[vertical gluing of squares ($g_V$)]
Let $X$ be a topological space, and $f:Y\to H(X)$ a map of topological spaces which is injective.\\
Assume that there is a (necessarily unique) commutative diagram in $\Top$
\begin{equation}\label{diagram:aux_vert_concat_squares}
\begin{diagram}[midshaft,h=2.3em]
Y\underset{X}{\times}Y&\rTo{\concat}& Y\\
\dTo{f\times f}&&\dTo{f}\\
H(X)\underset{X}{\smash{_t\times_s}}H(X)&\rTo{\ \concat\ }&H(X)
\end{diagram}
\end{equation}
where the bottom map is concatenation of Moore paths in $X$ (definition \sref{definition:concat_Moore_paths}), and the pullbacks are the obvious ones for which this concatenation map makes sense (we indicate the maps in the pullback).\\
Let $P_V$ be the pullback in the pullback square
\begin{diagram}[midshaft,h=2.1em]
P_V&\rTo{p_1}&\squares{X}{Y}\\
\dTo{p_2}&&\dTo{T}\\
\squares{X}{Y}&\rTo{B}&H(X)
\end{diagram}
Then $P_V$ is naturally identified with $H\big(Y\times_X Y\big)$ (the pullback $Y\times_X Y$ being the same as in diagram \seqref{diagram:aux_vert_concat_squares})
The {\em vertical gluing} map
\[ g_V:P_V\To\squares{X}{Y} \]
is defined to be the map
\[ P_V=H\Big(Y\underset{X}{\times}Y\Big)\xTo{\ H(\concat)\ }H(Y) \]
where $\concat$ is the map assumed to exist in diagram \seqref{diagram:aux_vert_concat_squares}.
\end{construction}

\begin{remark}[intuition for vertical gluing]
This vertical gluing operation corresponds approximately (under the map \seqref{equation:squares->mapsonsquare}) to having continuous functions
\begin{align*}
\clcl{0}{1}\times\clcl{0}{1}&\To X\\
\clcl{0}{1}\times\clcl{1}{2}&\To X
\end{align*}
which coincide on $\clcl{0}{1}\times\set{1}$, and gluing them to obtain a continuous function
\[ \clcl{0}{1}\times\clcl{0}{2}\To X \]
\end{remark}

Note that the assumptions of the previous construction are verified for the inclusions
\begin{align*}
\popath(X)\Into H(X)\\
\popathstrong(X)\Into H(X)
\end{align*}
for any stratified space $X$.

\begin{construction}[triangle gluing operation ($g_T$)]
Let $X$ be a topological space, and $f:Y\to H(X)$ a map of topological spaces which is injective.\\
Assume that there is a (necessarily unique) commutative diagram in $\Top$ given by \seqref{diagram:aux_vert_concat_squares}.\\
Let $P_T$ be the subspace of the product $\rtriangles{X}{Y}\times\squares{X}{Y}\times\rtriangles{X}{Y}$ which is the limit of the diagram
\begin{diagram}[midshaft,h=2em]
&&&&\rtriangles{X}{Y}\\
&&&&\dTo{B}\\
&&&&H(X)\\
&&&&\uTo{T}\\
\rtriangles{X}{Y}&\rTo{R}&Y&\lTo{L}&\squares{X}{Y}
\end{diagram}
where the first copy of $\rtriangles{X}{Y}$ in the product corresponds to the bottom left entry in the diagram.\\
Then there is a gluing map
\[ g_T:P_T\To\rtriangles{X}{Y} \]
which is given equivalently by either of the two following procedures:
\begin{enum}
\item gluing horizontally a triangle in the bottom left (entry of the diagram above) with the result of gluing vertically a square in the bottom right with a triangle in the top right:
\[ \qquad g_T(a,b,c)=g_H(a,g_V(b,c))\qquad\text{for }(a,b,c)\in P_T \]
\item gluing horizontally a triangle in the bottom left (entry of the diagram above) with a square in the bottom right and then glue vertically the result with a triangle in the top right:
\[ \qquad g_T(a,b,c)=g_V(g_H(a,b),c)\qquad\text{for }(a,b,c)\in P_T \]
\end{enum}
\end{construction}

\begin{remark}[intuition for triangle gluing]\label{remark:g_T->gluingontriangles}
So now we know that we can glue two triangles and one square (with compatible edges) into one triangle. This corresponds approximately (under the maps \seqref{equation:triangles->mapsontriangles} and \seqref{equation:squares->mapsonsquare}) to having continuous functions
\begin{align*}
\set{(x,y)\in\clcl{0}{1}\times\clcl{0}{1}:y\leq x}&\To X\\
\clcl{1}{2}\times\clcl{0}{1}&\To X\\
\set{(x,y)\in\clcl{1}{2}\times\clcl{1}{2}:y\leq x}&\To X
\end{align*}
such that the first two coincide on $\set{1}\times\clcl{0}{1}$ and the last two coincide on $\clcl{1}{2}\times\set{1}$, and gluing them to obtain a continuous function
\[ \set{(x,y)\in\clcl{0}{2}\times\clcl{0}{2}:y\leq x}\To X \]
\end{remark}

Now we can define the composition in $\mathcal{Z}_M$. Recall that we denote the elements of $\mathcal{Z}_M(e,f)$ by 4-tuples in the product
\[ \big(\disccat{\Total_n[M]}(e,f)\big)\times\popathstrong\big(\Map(k,l\times\RR^n)\big)\times\rtrianglesBig{M^{\times k}}{\popathstrong\big(M^{\times k}\big)}\times\popath\big(M^{\times k}\big) \]
for $e\in\Emb(k\times\RR^n,M)$ and $f\in\Emb(l\times\RR^n,M)$. Also, for the following definition, it may be useful to refer to the diagram within observation \sref{remark:mor_Z_M_bigdiagram}.

\begin{definition}[composition in $\mathcal{Z}_M$]
Let $k,l,m\in\NN$, $e\in\Emb(k\times\RR^n,M)$, $e'\in\Emb(l\times\RR^n,M)$, and $e''\in\Emb(m\times\RR^n,M)$.\\
Given $(a,b,c,d)\in\mathcal{Z}_M(e,e')$ and $(a',b',c',d')\in\mathcal{Z}_M(e',e'')$, their composition is defined to be the element
\begin{align*}
(a',b',c',d')\circ(a,b,c,d)\defeq\bigg(&a\circ a',\\
&\concat\Big(\mathbf{m}(a')\circ b,b'\circ\pi_0\big(\mathbf{m}(a)\big)\Big),\\
&g_T\Big(c,\mathbf{h}(a')\circ b,c'\circ\pi_0\big(\mathbf{m}(a)\big)\Big),\\
&\concat\Big(d,d'\circ\pi_0\big(\mathbf{m}(a)\big)\Big)\bigg)
\end{align*}
of $\mathcal{Z}_M(e,e'')$. Here $\concat$ designates concatenation of filtered paths.
\end{definition}

\begin{remark}
There is not much to say about this composition: it is essentially the only thing that can be done, if one takes the mental picture from observation \sref{remark:g_T->gluingontriangles} seriously.\\
An element of $(a,b,c,d)\mathcal{Z}_M(e,e')$ and an element of $(a',b',c',d')\mathcal{Z}_M(e',e'')$ each determine a triangle. In order to obtain a new triangle, we just need to use an appropriate square and apply the triangle gluing operation. This square is obtained by taking the filtered path $c$ and tracing  it along the homotopy of embeddings $\mathbf{h}(a')$.
\end{remark}

After the laborious definitions and constructions given in this section and the previous one, we leave it to the dedicated reader to check that all relevant maps are well-defined and continuous, that the data for $\mathcal{Z}_M$ indeed defines a $\Top$-category, and that $F_\Total$, $F_\M$ give $\Top$-functors.

\section{Equivalence between $\Total_n[M]$ and $\M(M)$}\label{section:equiv_TnM_M(M)}

Let us fix $n\in\NN$ in this section. We collect in the following proposition the results described in the previous two sections.

\begin{proposition}
$\mathcal{Z}_\bullet$ defines a functor
\[ \mathcal{Z}_\bullet:\big(\Emb_n\big)_0\To\Top\dash\Cat \]
Furthermore, $F_\Total$ and $F_\M$ give natural transformations
\begin{align*}
F_\Total&:\mathcal{Z}_\bullet\To\disccat{\Total_n[-]}\\
F_\M&:\mathcal{Z}_\bullet\To\M
\end{align*}
\end{proposition}

\begin{remark}
Recall from observation \sref{remark:functoriality_M(X)} that $\M(-)$ is functorial with respect to injective maps of topological spaces. Therefore, one easily extracts a functor
\[ \M:\big(\Emb_n\big)_0\To\Top\dash\CAT \]
which is used in the previous definition.
\end{remark}

\begin{remark}[case of $\mathcal{Z}'_M$]
Recall the category $\mathcal{Z}'_M$ from \sref{remark:Z'M}. It admits completely analogous functors
\begin{align*}
\mathcal{Z}'_M&\To\disccat{\Total_n[M]}\\
\mathcal{Z}'_M&\To\M(M)
\end{align*}
\end{remark}

In this section, we will show that $F_\Total$ and $F_\M$ give essentially surjective local homotopy equivalences of $\Top$-categories. We state the result now.

\begin{proposition}[$F_\M$, $F_\Total$ are weak equivalences]\label{proposition:M(M)_zig-zag_equiv_TotalM}
For each $n$-manifold without boundary, $M$, the $\Top$-functors
\begin{align*}
F_\Total&:\mathcal{Z}_M\To\disccat{\Total_n[M]}\\
F_\M&:\mathcal{Z}_M\To\M(M)
\end{align*}
are essentially surjective local homotopy equivalences of $\Top$-categories.
\end{proposition}

It is quite easy to see that $F_\M$ and $F_\Total$ are essentially surjective. We are left with proving that the functors $F_\Total$ and $F_\M$ are local homotopy equivalences.
This will be a consequence of a sequence of constructions and lemmas, which will occupy us for the remainder of this section.

Let us start by fixing a $n$-dimensional manifold $M$ without boundary, $k,l\in\NN$, $e\in\Emb(k\times\RR^n,M)$, and $f\in\Emb(l\times\RR^n,M)$.

\begin{describe}[strategy for the proof]\label{describe:initial_strategy_FM_FT}
In our proof that $F_\M$ and $F_\Total$ are local equivalences, we would like to compare $\mathcal{Z}_M(e,f)$ with the (appropriate) fibre product of $\disccat{\Total_n[M]}(e,f)$ and the space of homotopies of filtered paths.\\
This comparison would ideally take the form of a map that
\begin{enum}
\item projects $\mathcal{Z}_M(e,f)$ onto $\disccat{\Total_n[M]}(e,f)$;
\item associates to a triangle $a$ with vertical filtered paths and whose edges are filtered paths (such a triangle is part of the data for an element of $\mathcal{Z}_M(e,f)$), a homotopy through filtered paths from the edge $T(a)$ to the concatenation of the other edges, $\concat(B(a),R(a))$.
\end{enum}
The purpose of doing this is that this fibre product is obviously homotopy equivalent to $\mathcal{Z}_M(e,f)$ (since we can deform homotopies to a constant one).\\
On the other hand, the fibre product has a map to $\M(M)(e\circ i_k,f\circ i_l)$ (by taking the appropriate endpoint of the homotopy of filtered paths), which can be shown to be an equivalence (essentially by using propositions \sref{proposition:analysis_Total_nM(e,f)} and \sref{proposition:homotopy_type_morphism_M}).\\
If done compatibly with the functors $F_\M$ and $F_\Total$, this comparison would prove that both functors are local homotopy equivalences.
\end{describe}

\begin{describe}[correction to the strategy for the proof]\label{describe:correction_strategy_FM_FT}
There is one obvious problem with the above strategy: there is no easy or meaningful way to associate to a triangle $a$ (as in the preceding description) a homotopy through filtered paths from the edge $T(a)$ to the concatenation of the other edges, $\concat(B(a),R(a))$. Any naive systematic attempt to do so will result in general in homotopies through non-filtered paths.\\
To fix this problem, we consider the subspace, $V^\holink_M(e,f)$, of $\mathcal{Z}_M(e,f)$ of elements whose corresponding triangles have vertical homotopy links (instead of just vertical filtered paths).\\
Applying a naive procedure to such a triangle does indeed give a homotopy through filtered paths. The comparison mentioned in the uncorrected strategy above will thus take the form of a zig-zag through $V^\holink_M(e,f)$.
\end{describe}

\begin{definition}[subspace of $\mathcal{Z}_M(e,f)$ of vertical homotopy links]\label{definition:Z_vertical_holinks}
Let $k,l\in\NN$, $e\in\Emb(k\times\RR^n,M)$, and $f\in\Emb(l\times\RR^n,M)$.\\
The {\em subspace of vertical homotopy links}, $V^\holink_M(e,f)$, of $\mathcal{Z}_M(e,f)$ is given by 
\[ V^\holink_M(e,f)\defeq\set{(a,b,c,d)\in\mathcal{Z}_M(e,f)\suchthat c\in\rtrianglesbig{M^{\times k}}{\holink(M^{\times k})}} \]
Equivalently, $V^\holink_M(e,f)$ can be defined as the limit of the diagram obtained from the one in \sref{remark:mor_Z_M_bigdiagram} by replacing all three occurrences of strong spaces of filtered paths ($\smash{\popathstrong}$) with the corresponding homotopy link spaces ($\holink$).
\end{definition}

\begin{lemma}[central lemma 1]\label{lemma:V->Z_equiv}
The inclusion map
\[ V^\holink_M(e,f)\Into\mathcal{Z}_M(e,f) \]
is the inclusion of a strong deformation retract. In particular, it is a homotopy equivalence.
\end{lemma}
\begin{proof}
This is an immediate consequence of proposition \sref{proposition:holinks_equiv_strongpopaths}.
\end{proof}

\begin{remark}[case of $\mathcal{Z}'_M$]\label{remark:conj_stronger_Miller_Z'}
Recall the category $\mathcal{Z}'_M$ from \sref{remark:Z'M}. It would be an immediate consequence of conjecture \sref{conjecture:stronger_Miller} that the inclusion
\[ \inclusion:V^\holink_M(e,f)\Into\mathcal{Z}'_M(e,f) \]
is a homotopy equivalence.
\end{remark}

\begin{describe}[current position within the strategy for the proof]
Having defined the subspace of $\mathcal{Z}_M(e,f)$ whose triangles have vertical homotopy links, we will describe precisely in the next two constructions a specific (naive) procedure to convert such triangles into homotopies of filtered paths (see also description \sref{describe:correction_strategy_FM_FT}).
\end{describe}

\begin{construction}
We will construct a precise map from squares in $X$ to $\Map(I\times I,X)$.
Let $f:Y\to H(X)$ be any map. Then we define
\[ sq:\squares{X}{Y}\To\Map(I\times I,X) \]
to be the composition
\begin{align*}
\squares{X}{Y}&\xlongequal{\smash{\quad}}H(Y)\\
&\Into H(Y)\\
&\xTo{\reparam}\Map\big(I,Y)\\
&\xTo{\Map(I,f)}\Map\big(I,H(X)\big)\\
&\xTo{\reparam}\Map\big(I,\Map(I,X)\big)\\
&\xlongequal{\smash{\quad}}\Map(I\times I,X)
\end{align*}
(where ``$\reparam$'' is the canonical reparametrization map of Moore paths from \sref{definition:reparametrization_Moore_paths}). Under the map $sq$, the edges correspond in the obvious way:
\begin{enum}
\item the top edge of a square in $\squares{X}{Y}$ corresponds to restricting to $I\times\set{1}$;
\item the bottom edge of a square corresponds to restricting to $I\times\set{0}$;
\item the right edge of a square corresponds to restricting to $\set{1}\times I$;
\item the left edge of a square corresponds to restricting to $\set{0}\times I$.
\end{enum}
\end{construction}

\begin{construction}[from triangles of links to homotopies of filtered paths]
Let
\[ p:I\times I\To I\times I \]
be any map which is injective in the interior and such that for $x\in\clcl{0}{1}$
\begin{align*}
p(x,1)&=(1,1)\\
p(0,x)&=(x,1)\\
p(x,0)&=(0,1-x)\\
p\left(1,\frac{x}{2}\right)&=(x,0)\\
p\left(1,\frac{1+x}{2}\right)&=(1,x)
\end{align*}
Let $X$ be a stratified space, and consider the subspace of $\rtrianglesbig{X}{\holink(X)}$ given by
\begin{align*}
BT_{Y,Z}\defeq\big\{x\in\rtrianglesbig{X}{\holink(X)}\suchthat&B(x)\in\popath(X),\\
&T(x)\in\smash{\popath}(X;Y,Z)\big\}
\end{align*}
for $Y$, $Z$ subspaces of $X$.
Then the composition
\begin{align*}
BT_{Y,Z}&\Into\rtrianglesbig{X}{\holink(X)}\\
&\Into\squaresbig{X}{\holink(X)}\\
&\xTo{sq}\Map(I\times I,X)\\
&\xTo{\Map(p,X)}\Map(I\times I,X)\\
&\xlongequal{\smash{\quad}}\Map\big(I,\Map(I,X)\big)
\end{align*}
(where the last map preserves the orders, as shown, of the two copies of $I$) factors through the subspace
\[ \Map\Big(I,\poMap_{Y,Z}(I,X)\Big) \]
of $\Map(I,\Map(I,X))$, where
\[ \poMap_{Y,Z}(I,X)\defeq\set{\gamma\in\poMap(I,X)\suchthat\gamma(0)\in Y\,,\,\gamma(1)\in Z} \]
is a subspace of $\smash{\poMap}(I,X)$.
Thus we constructed a map
\[ tr:BT_{Y,Z}\To\Map\Big(I,\poMap_{Y,Z}(I,X)\Big) \]

Moreover, for any $x$ in the image of $tr$, $x(\tau)$ is a homotopy link in $X$ for any $\tau\in\opop{0}{1}$. This observation leads us to the next lemma, which is a straightforward application of the corollary \sref{corollary:Miller} to Miller's result.
\end{construction}

\begin{lemma}\label{lemma:aux_triangles_holinks->homotopies_popaths}
Assume $X$ is a homotopically stratified space, and $X_a$, $X_b$ are strata of $X$.\\
Then the map
\[ tr:BT_{X_a,X_b}\To\Map\Big(I,\poMap_{X_a,X_b}(I,X)\Big) \]
is part of a homotopy equivalence where
\begin{enum}
\item the homotopy
\[ \Map\Big(I,\poMap_{X_a,X_b}(I,X)\Big)\times I\To\Map\Big(I,\poMap_{X_a,X_b}(I,X)\Big) \]
fixes the boundary of $I\times I$, pointwise;
\item the homotopy
\[ BT_{X_a,X_b}\times I\To BT_{X_a,X_b} \]
fixes the canonical reparametrizations of all edges $T$, $R$, $B$ of triangles.
\end{enum}
\end{lemma}

\begin{remark}
We demand that the homotopies in the previous lemma have such specific properties so that the homotopy equivalence is preserved by taking fibre products over the edge maps. This will be used in lemma \sref{lemma:V->Aux_equiv}.
\end{remark}

\begin{describe}[current position within the strategy for the proof]
The following construction details the fibre product, mentioned in the initial strategy \sref{describe:initial_strategy_FM_FT}, which we will compare with $\mathcal{Z}_M(e,f)$.\\
This comparison will be made via a zig-zag through $V^\holink_M(e,f)$, which will be defined in \sref{construction:V->Aux}.
\end{describe}

\begin{construction}[auxiliary spaces]\label{construction:auxiliary_spaces}
Let $\mathtt{Aux}_1$ denote the limit of
\begin{diagram}[midshaft,balance,h=2.2em]
&&&&\mathclap{\shortstack[l]{$\mathllap{\holink\big(}\Map(k,l\times\RR^n);$\\$\Conf(l\times\RR^n,k),$\\$i_l\circ\Set(k,l)\big)$}}\\
&&&&\dTo{s}\\
\disccat{\Total_n[M]}(e,f)&\rTo_{\sref{remark:maps_mor_Total_n}}^{\ \;\mathbf{m}\;\ }&\Emb(k\times\RR^n,l\times\RR^n)&\rTo{\ -\circ i_k\ }&\Conf(l\times\RR^n,k)&
\end{diagram}
viewed as a subspace of the product of the bottom left and the top right entries of the diagram.\\
Define $\mathtt{Aux}_2$ to be the subspace of (recall equation \seqref{equation:mor_M(X)_popath} from construction \sref{construction:mor_M(X)_popath})
\[ \M(M)(e\circ i_k,f\circ i_l)=\popath\big(M^{\times k};\set{e\circ i_k},f\circ i_l\circ\Set(k,l)\big) \]
constituted by the paths of length $1$. Note that $\mathtt{Aux}_2$ is naturally a subspace of $\smash{\poMap\big(I,M^{\times k}\big)}$. As a useful aside, observe that the inclusion \[ \mathtt{Aux}_2\xhookrightarrow{\smash{\ \ }}\M(M)(e\circ i_k,f\circ i_l) \] is a homotopy equivalence.\\
With this, the map
\[ \mathtt{ccr}:\mathtt{Aux}_1\To\mathtt{Aux}_2 \]
is defined by ($\concat$ designates concatenation of paths, and $\mathbf{h}$ is given in \sref{remark:maps_mor_Total_n})
\[ \mathtt{ccr}(a,b)\defeq\reparam\Big(\concat\big(\reparam(\mathbf{h}(a)\circ i_k),\reparam(f\circ b)\big)\Big) \]
for $(a,b)\in\mathtt{Aux}_1$. Observe that the effect of the multiple reparametrizations ($\reparam$) is to obtain a path of length 1, where each of the concatenated paths occupies half of that length.\\
We finally construct the pullback of
\begin{diagram}[midshaft,h=2em]
&&\mathtt{Aux}_1\\
&&\dTo{\mathtt{ccr}}\\
\Map(I,\mathtt{Aux}_2)&\rTo{\ \evaluation{1}\ }&\mathtt{Aux}_2
\end{diagram}
which we call $\mathtt{Aux}$.
\end{construction}

\begin{lemma}[central lemma 2]\label{lemma:central_1_Aux1->Tn}
The canonical projection
\[ \proj:\mathtt{Aux}_1\To\disccat{\Total_n[M]}(e,f) \]
is a homotopy equivalence.
\end{lemma}
\begin{proof}
$\mathtt{Aux}_1$ is the limit of
\begin{diagram}[midshaft,balance,h=2.2em]
&&&&\quad\mathclap{\shortstack[l]{$\mathllap{\holink\big(}\!\Map(k,l\times\RR^n);$\\$\Conf(l\times\RR^n,k),$\\$i_l\circ\Set(k,l)\big)$}}\\
&&&&\dTo{s}\\
\disccat{\Total_n[M]}(e,f)&\rTo{\ \;\mathbf{m}\;\ }&\Emb(k\times\RR^n,l\times\RR^n)&\rTo{\ -\circ i_k\ }&\Conf(l\times\RR^n,k)&
\end{diagram}
The vertical map is known to be a Hurewicz fibration (by proposition \sref{proposition:holink_s_fibration}) and a homotopy equivalence (proposition \sref{proposition:holink_many_discs_hodiscrete}), and we conclude that
\[ \proj:\mathtt{Aux}_1\To\disccat{\Total_n[M]}(e,f) \]
is a homotopy equivalence.
\end{proof}

\begin{lemma}[central lemma 3]\label{lemma:central_2_ccr}
The map
\[ \mathtt{ccr}:\mathtt{Aux}_1\To\mathtt{Aux}_2 \]
is a homotopy equivalence.
\end{lemma}
\begin{proof}[Sketch of proof]
Consider the following homotopy commutative diagram
\begin{diagram}[midshaft,hug]
\mathtt{Aux}_1&\rTo{\mathtt{ccr}}&\mathtt{Aux}_2\\
&\rdTo_{\mathllap{\widetilde{\mathtt{cc}}}}&\dInto_{\inclusion}\\
&&\M(M)(e\circ i_k,f\circ i_l)
\end{diagram}
where $\widetilde{\mathtt{cc}}$ is defined similarly to $\mathtt{ccr}$ but without reparametrizing (see construction \sref{construction:auxiliary_spaces}):
\[ \qquad\widetilde{\mathtt{cc}}(a,b)\defeq\concat\big(\mathbf{h}(a)\circ i_k,f\circ b\big)\qquad\text{for }(a,b)\in\mathtt{Aux}_1 \]
Since the triangle above commutes up to homotopy, and
\[ \inclusion:\mathtt{Aux}_2\Into\M(M)(e\circ i_k,f\circ i_l) \]
is a homotopy equivalence, we need only show that $\widetilde{\mathtt{cc}}$ is a homotopy equivalence.

Now consider the commutative diagram
\begin{diagram}[midshaft,objectstyle=\scriptstyle,labelstyle=\scriptscriptstyle,h=2.9em]
\mathtt{Aux}_1&\rTo{h1}&Q&\rTo{h2}&\holink(\parbox[t]{6em}{\scriptsize$\Map(k,l\times\RR^n);$\\$\Conf(l\times\RR^n,k),$\\$i_l\circ\Set(k,l))$}\\
\dTo{v1}&&\dTo{v2}&&\dTo^{\raisebox{-.8em}{$\scriptscriptstyle s$}}\\
\disccat{\Total_n[M]}(e,f)&\rTo_{\sref{remark:maps_mor_Total_n}}^{\ \mathbf{h}(-)\circ i_k\ }&\popath(M^{\times k};\set{e\circ i_k},\Conf(\im f,k))&\rTo{\ \ f^{-1}\circ\, t\ \ }&\Conf(l\times\RR^n,k)
\end{diagram}
in which both small squares are cartesian. The outer square is just the pullback square that defines $\mathtt{Aux}_1$.

We know that the arrow $s$ is a Hurewicz fibration (by \sref{proposition:holink_s_fibration}) and a homotopy equivalence (by \sref{proposition:holink_many_discs_hodiscrete}). Therefore both $v1$ and $v2$ are homotopy equivalences.

On the other hand, the map $\mathbf{h}(-)\circ i_k$ is a homotopy equivalence: it fits in a commutative triangle
\begin{diagram}[midshaft,hug,h=2.7em]
\disccat{\Total_n[M]}(e,f)&\rTo{\ \mathbf{h}(-)\circ i_k\ }&\popath\big(M^{\times k};\set{e\circ i_k},\Conf(\im f,k)\big)\\
&\rdTo_{\mathllap{\varsigma}}^{\mathrlap{\sim}}&\dTo^{\rotc{90}{$\scriptstyle\sim$}}_{\reparam}\\
&&\hofibre_{e\circ i_k}\!\Big(\Conf(l\times\RR^n,k)\xTo{f\circ -}\Conf(M,k)\Big)
\end{diagram}
where $\varsigma$ is the homotopy equivalence from proposition \sref{proposition:analysis_Total_nM(e,f)}. We conclude that $h1$ is also a homotopy equivalence (given that $v1$, $v2$, and $\mathbf{h}(-)\circ i_k$ are).

Observe now that the space $Q$ is exactly the one appearing in proposition \sref{proposition:homotopy_type_morphism_M} under the same name, and that we actually have a commuting diagram
\begin{diagram}[midshaft]
\mathtt{Aux}_1&\rTo^{h1}_{\sim}&Q\\
&\rdTo(1,2)_{\mathllap{\widetilde{\mathtt{cc}}}}\ldTo(1,2)_{\mathrlap{\widetilde{\concat}}}\\
&\M(M)(e\circ i_k,f\circ i_l)
\end{diagram}
(where we have used the notation of \sref{proposition:homotopy_type_morphism_M}). Since proposition \sref{proposition:homotopy_type_morphism_M} tells us that the map on the right, $\widetilde{\concat}$, is a homotopy equivalence, we conclude that the map on the left, $\widetilde{\mathtt{cc}}$, is a homotopy equivalence, which ends this proof.
\end{proof}

\begin{remark}
Lemma \sref{lemma:central_2_ccr} does most of the work comparing $\Total_n[M]$ and $\M(M)$, being the only one to make use of propositions \sref{proposition:homotopy_type_morphism_M} and \sref{proposition:analysis_Total_nM(e,f)}.
\end{remark}

\begin{construction}[from $V^\holink_M(e,f)$ to the auxiliary spaces]\label{construction:V->Aux}
Observe that for any $(a,b,c,d)\in V^\holink_M(e,f)$
\[ c\in(T,B)^{-1}\big(\popath\big)\subset\rtrianglesbig{M^{\times k}}{\holink(M^{\times k})} \]
From this we define the map
\[ \func{\tau'}{V^\holink_M(e,f)}{\Map(I,\mathtt{Aux}_2)}{(a,b,c,d)}{tr(c)} \]
where we have identified $\mathtt{Aux}_2$ with a subspace of $\smash{\poMap(I,M^{\times k})}$.\\
There is also a canonical projection (recall the definition of $\mathtt{Aux}_1$ from \sref{construction:auxiliary_spaces})
\[ \func{\tau''}{V^\holink_M(e,f)}{\mathtt{Aux}_1}{(a,b,c,d)}{(a,c)} \]
The maps $\tau'$ and $\tau''$ give a commutative diagram
\begin{diagram}[midshaft,h=2em]
V^\holink_M(e,f)&\rTo{\tau''}&\mathtt{Aux}_1\\
\dTo{\tau'}&&\dTo_{\mathtt{ccr}}\\
\Map(I,\mathtt{Aux}_2)&\rTo{\ \evaluation{1}\ }&\mathtt{Aux}_2
\end{diagram}
and so assemble into a map
\[ \tau:V^\holink_M(e,f)\To\mathtt{Aux} \]
\end{construction}

The following lemma is a straightforward consequence of \sref{lemma:aux_triangles_holinks->homotopies_popaths} and the definition of $V^\holink_M(e,f)$ in \sref{definition:Z_vertical_holinks}.

\begin{lemma}[central lemma 4]\label{lemma:V->Aux_equiv}
The map
\[ \tau:V^\holink_M(e,f)\To\mathtt{Aux} \]
is a homotopy equivalence.
\end{lemma}

The four central lemmas --- \sref{lemma:V->Z_equiv}, \sref{lemma:central_1_Aux1->Tn}, \sref{lemma:central_2_ccr}, and \sref{lemma:V->Aux_equiv} --- are so called because of their instrumentality in our proof that $F_\M$ and $F_\Total$ are local equivalences. More precisely, these lemmas encapsulate all the homotopical properties of
\begin{enum}
\item spaces of filtered paths and homotopy links,
\item morphisms spaces in $\M(M)$,
\item morphism spaces in $\Total_n[M]$
\end{enum}
which we use in showing $F_\M$ and $F_\Total$ are local equivalences.

\begin{describe}[current position within the strategy for the proof]
We have now defined all pertinent objects and maps.\\
All that remains to do is using the zig-zag (via $V^\holink_M(e,f)$) between the space $\mathcal{Z}_M(e,f)$ and the auxiliary fibre product $\mathtt{Aux}$ to prove that $F_\M$ and $F_\Total$ are local equivalences.
\end{describe}

We now relate the map $\tau$ with the functor $F_\Total$ and use this relation to prove $F_\Total$ is a local homotopy equivalence.

\begin{lemma}\label{lemma:Aux_V_Z_Total_commutes}
The diagram
\begin{equation}\label{diagram:Aux_V_Z_Total}
\begin{diagram}[midshaft,h=2.3em]
V^\holink_M(e,f)&\rTo{\tau}&\mathtt{Aux}&\rTo{\proj}&\mathtt{Aux}_1\\
\dInto{\inclusion}&&&&\dTo{\proj}\\
\mathcal{Z}_M(e,f)&&\rTo{\mor F_\Total}&&\disccat{\Total_n[M]}(e,f)
\end{diagram}
\end{equation}
commutes.
\end{lemma}

\begin{proposition}[$F_\Total$ is a local homotopy equivalence]
The map
\[ \mor F_\Total:\mathcal{Z}_M(e,f)\To\disccat{\Total_n[M]}(e,f) \]
is a homotopy equivalence.
\end{proposition}
\begin{proof}
The central lemmas \sref{lemma:V->Z_equiv}, \sref{lemma:V->Aux_equiv}, and \sref{lemma:central_1_Aux1->Tn} state that the maps
\begin{align*}
\inclusion&:V^\holink_M(e,f)\Into\mathcal{Z}_M(e,f)\\
\tau&:V^\holink_M(e,f)\To\mathtt{Aux}\\
\proj&:\mathtt{Aux}_1\xTo{\;\sim\;}\disccat{\Total_n[M]}(e,f)
\end{align*}
are homotopy equivalences. Since $\mathtt{Aux}$ is the pullback of
\begin{diagram}[midshaft,h=2em]
&&\mathtt{Aux}_1\\
&&\dTo_{\mathtt{ccr}}\\
\Map(I,\mathtt{Aux}_2)&\rTo{\ \evaluation{1}\ }&\mathtt{Aux}_2
\end{diagram}
and $\evaluation{1}$ is a Hurewicz fibration and a homotopy equivalence, it follows that the projection
\[ \proj:\mathtt{Aux}\xTo{\;\sim\;}\mathtt{Aux}_1 \]
is a homotopy equivalence.

We have proved that all the arrows in diagram \seqref{diagram:Aux_V_Z_Total} are homotopy equivalences, except the bottom one. We conclude that the bottom arrow
\[ \mor F_\Total:\mathcal{Z}_M(e,f)\To\disccat{\Total_n[M]}(e,f) \]
is also a homotopy equivalence.
\end{proof}

The next two lemmas relate $\tau$ with the functor $F_\M$, and prove that $F_\M$ is a local homotopy equivalence.

\begin{lemma}
The diagram (where $\reparam$ gives the canonical reparametrization to a filtered path of length 1)
\begin{equation}\label{diagram:Aux_V_Z_M_commutes}
\begin{diagram}[midshaft]
V^\holink_M(e,f)&\rTo{\tau}&\mathtt{Aux}&\rTo{\proj}&\Map(I,\mathtt{Aux}_2)\\
\dInto{\inclusion}&&&&\dTo{\evaluation{0}}\\
\mathcal{Z}_M(e,f)&\rTo{\mor F_\M}&\M(e\circ i_k,f\circ i_l)&\rTo{\reparam}&\mathtt{Aux}_2
\end{diagram}
\end{equation}
commutes.
\end{lemma}

\begin{proposition}[$F_\M$ is a local homotopy equivalence]
The map
\[ \mor F_\M:\mathcal{Z}_M(e,f)\To\M(M)(e\circ i_k,f\circ i_l) \]
is a homotopy equivalence.
\end{proposition}
\begin{proof}
The central lemmas \sref{lemma:V->Z_equiv} and \sref{lemma:V->Aux_equiv} state that the maps
\begin{align*}
\inclusion&:V^\holink_M(e,f)\Into\mathcal{Z}_M(e,f)\\
\tau&:V^\holink_M(e,f)\To\mathtt{Aux}
\end{align*}
are homotopy equivalences. Furthermore, it is easy to see that
\begin{align*}
\evaluation{0}&:\Map(I,\mathtt{Aux}_2)\To\mathtt{Aux}_2\\
\reparam&:\M(M)(e\circ i_k,f\circ i_l)\To\mathtt{Aux}_2
\end{align*}
are both homotopy equivalences.

The canonical projection
\[ \proj:\mathtt{Aux}\To\Map(I,\mathtt{Aux}_2) \]
is a homotopy equivalence because the commutative square
\begin{diagram}[midshaft,h=2em]
\mathtt{Aux}&\rTo{\proj}&\mathtt{Aux}_1\\
\dTo{\proj}&&\dTo_{\mathtt{ccr}}\\
\Map(I,\mathtt{Aux}_2)&\rTo{\ \evaluation{1}\ }&\mathtt{Aux}_2
\end{diagram}
is cartesian (by definition of $\mathtt{Aux}$), $\evaluation{1}$ is a Hurewicz fibration, and
\[ \mathtt{ccr}:\mathtt{Aux}_1\To\mathtt{Aux}_2 \]
is a homotopy equivalence (by central lemma \sref{lemma:central_2_ccr}).

We have proved that all arrows in diagram \seqref{diagram:Aux_V_Z_M_commutes}, except for $\mor F_\M$, are homotopy equivalences. In conclusion, $\mor F_\M$ is a homotopy equivalence as well.
\end{proof}

\begin{remark}[case of $\mathcal{Z}'_M$: removing strong spaces of filtered paths]
Recall the category $\mathcal{Z}'_M$ from remark \sref{remark:Z'M}. As stated before it admits functors to $\M(M)$ and $\disccat{\Total_n[M]}$. The proof that $F_\M$ and $F_\Total$ are weak equivalences holds for those functors as well, with the exception of central lemma \sref{lemma:V->Z_equiv}. As explained in observation \sref{remark:conj_stronger_Miller_Z'}, this could be remedied by assuming the conjecture \sref{conjecture:stronger_Miller} strengthening Miller's result.\\
In conclusion, if we assume conjecture \sref{conjecture:stronger_Miller}, then we can construct weak equivalences between $\mathcal{Z}'_M$ and each of the $\Top$-categories $\M(M)$ and $\disccat{\Total_n[M]}$. This would obviate our use of strong spaces of filtered paths, since they only enter in the definition of the category $\mathcal{Z}_M$.
\end{remark}


%% file: homotopycat.tex



\chapter{Homotopical properties of enriched categories}\label{chapter:homotopical_properties_enriched_categories}

\section*{Introduction}

This chapter aims to present some elementary aspects of an ad hoc theory of homotopy colimits in enriched model categories. In order to meet our needs in the final chapter, we will present a definition of {\em derived enriched colimit} for an enriched indexing category. To the author's knowledge, this notion appears somewhat rarely in the literature in such a generality. It has nevertheless been considered, for example, in \cite{Shulman} in a more systematic manner. The reader can also look at section A.3.3 of \cite{Lurie3}.

The final section describes, without proof, how these derived enriched colimits behave with respect to Grothendieck constructions.

\section*{Summary}

Section \sref{section:categories_intervals} analyzes some categories of {\em intervals} defined as subcategories of the category $\Ord$ of finite ordinals. Section \sref{section:monads_intervals} explains how these categories naturally index functors associated to algebras over a monad.

Section \sref{section:bar_construction_enriched} constructs monads whose algebras are $V$-functors, for $V$ a closed symmetric monoidal category. Together with the functors defined in section \sref{section:monads_intervals}, this gives rise to several bar-type constructions for enriched functors. In particular, given $V$-categories $A$ and $C$, and functors $F:A\to C$ and $G:\op{A}\to V$, we obtain a two-sided bar construction $\BarConst(G,A,F)$.

Section \sref{section:derived_enriched_colimits} uses the two-sided bar construction of the previous section to define the derived enriched colimit $G\tensor^\derived_A F$ when $V$ is an appropriate simplicial model category, and $C$ is a $V$-model category. This concept is applied in section \sref{section:homotopy_colimits_enriched} to construct the homotopy colimit $\hocolim F$, when the monoidal structure on $V$ is cartesian. Furthermore, a notion of homotopy cofinality is explored in this context.

Section \sref{section:weak_equivalence_enriched_categories} defines when a functor between two $V$-categories is a weak equivalence, and proves that such a functor is necessarily homotopy cofinal.

The final section \sref{section:Grothendieck_construction_homotopy_colimit} states without proof a result which informally says that
\[ \hocolim_{\Groth(G)}(F\circ\pi)\simeq\realization{\!\vphantom{\big(}\Nerve{\disccat{G}}}\tensor^\derived_A F \]
for $G:\op{\internal A}\to\Cat(V)$ an internal $\Cat(V)$-valued functor (where the functor $\pi:\Groth(F)\to A$ is the projection).

\section{Categories of intervals}\label{section:categories_intervals}

Recall the category $\Ord$ of finite ordinals (see \sref{notation:Ord_Delta}). Given a non-empty ordinal $a$ in $\Ord$, we will denote its minimum (respectively, maximum) by $\min a$ (respectively, $\max a$).

We will now define several subcategories of $\Ord$ which correspond to demanding the preservation of minima and or maxima of the ordinals. These will be very useful for defining bar-like constructions.

\begin{definition}[category of left intervals]
We define the {\em category of left intervals}, $\LeftInt$, to be the subcategory of $\Ord$ consisting of the morphisms $f:a\to b$ in $\Ord$ such that $a$ is non-empty and
\[ f(\min a)=\min b \]
\end{definition}

\begin{definition}[category of right intervals]
We define the {\em category of right intervals}, $\RightInt$, to be the subcategory of $\Ord$ consisting of the morphisms $f:a\to b$ in $\Ord$ such that $a$ is non-empty and
\[ f(\max a)=\max b \]
\end{definition}

\begin{definition}[categories of intervals]
We define the {\em category of intervals}, $\Int$, to be the subcategory of $\Ord$ consisting of the morphisms $f:a\to b$ in $\Ord$ such that $a$ is non-empty and
\begin{align*}
f(\min a)&=\min b\\
f(\max a)&=\max b
\end{align*}
Define the {\em category of strict intervals}, $\StrictInt$, to be the full subcategory of $\Int$ generated by the ordinals $a$ such that
\[ \min a\neq\max a \]
\end{definition}

\begin{remark}
Equivalently, $\StrictInt$ is the full subcategory of $\Int$ generated by the ordinals with at least two elements.
\end{remark}

\begin{construction}[functors which reverse order]
There is a functor
\[ \reverse:\Ord\To\Ord \]
which takes an ordinal $a$ to the ordinal $\reverse(a)$ which has the same underlying set as $a$, and whose order is the reverse of that of $a$. Informally, $\reverse$  reverses the order of an ordinal.\\
The functor $\reverse$ is an isomorphism of categories such that
\[ \reverse\circ\reverse=\id_\Ord \]
Additionally, $\reverse$ restricts to the following isomorphisms of categories:
\begin{align*}
\reverse&:\LeftInt\To\RightInt\\
\reverse&:\RightInt\To\LeftInt\\
\reverse&:\Int\To\Int\\
\reverse&:\StrictInt\To\StrictInt
\end{align*}
which always verify
\[ \reverse\circ\reverse=\id \]
for appropriate compositions.
\end{construction}

Recall that the category $\Ord$ has a monoidal structure
\[ +:\Ord\times\Ord\To\Ord \]
whose unit is the empty ordinal.

\begin{construction}[functors which add minimum or maximum]
The functor
\[ 1+-:\Ord\To\Ord \]
lands within the category of left intervals. In particular, we get an induced functor
\[ 1+-:\Ord\To\LeftInt \]
Similarly, there are functors
\begin{align*}
1+-&:\RightInt\To\StrictInt\\
-+1&:\Ord\To\RightInt\\
-+1&:\LeftInt\To\StrictInt\\
1+-+1&:\Ord\To\StrictInt
\end{align*}
These functors verify the formulas
\begin{align*}
\reverse(1+-)&=\reverse(-)+1\\
\reverse(-+1)&=1+\reverse(-)
\end{align*}
\end{construction}

\begin{remark}[simplicial object with extra degeneracy]
The concept of an augmented simplicial object with extra degeneracy can be stated easily using these functors.\\
Giving an extra degeneracy to an augmented simplicial object in a category $C$
\[ F:\op{\Ord}\To C \]
is equivalent to finding an extension
\[ \widetilde{F}:\op{\LeftInt}\To C \]
of $F$ making the diagram
\begin{diagram}[midshaft,h=2.3em]
\op{\Ord}&\rTo{F}&C\\
\dTo{\op{(1+-)}}&\ruTo_{\mathrlap{\widetilde{F}}}\\
\op{\LeftInt}
\end{diagram}
commute.
\end{remark}

Given an appropriate simplicial space $F:\op{\Delta}\to\Top$ with an augmentation $F\to X$, and an extra degeneracy, it is a standard result that the geometric realization of $F$ is equivalent to $X$. The following proposition puts this in perspective.

\begin{proposition}\label{proposition:1+-:Delta->LeftInt_homotopy_final}
The functor
\[ F:\Delta\Into\Ord\xTo{1+-}\LeftInt \]
is a homotopy final functor.
\end{proposition}
\begin{proof}[Sketch of proof]
We are required to show that for any left interval $x$, the category $F/x$ has weakly contractible nerve. We know that $F/x$ is the (usual) Grothendieck construction of the functor
\[ \LeftInt(F-,x):\op{\Delta}\To\Set \]
By Thomason's theorem (e.g.\ theorem 18.9.3 of \cite{Hirschhorn}), we know that there is a weak equivalence of simplicial sets
\[ \Nerve\big(\Groth\big(\LeftInt(F-,x)\big)\big)\xTo{\;\sim\;}\LeftInt(F-,x) \]
Hence we obtain a weak equivalence in $\sSet$
\[ \Nerve(F/x)\xTo{\;\sim\;}\LeftInt(F-,x) \]
Using the previous observation, we notice that the augmented simplicial set $\LeftInt(F-,x)\to 1$ has an extra degeneracy, and therefore the augmentation
\[ \LeftInt(F-,x)\To 1 \]
gives a weak equivalence of simplicial sets. In conclusion, the nerve of the category $F/x$ is weakly contractible.
\end{proof}

\begin{corollary}
The functor
\[ \Delta\Into\Ord\xTo{-+1}\RightInt \]
is homotopy final.
\end{corollary}
\begin{proof}
There is a unique isomorphism of categories
\[ \reverse:\Delta\To\Delta \]
such that
\begin{diagram}[midshaft,h=2.1em]
\Delta&\rInto&\Ord\\
\dTo{\reverse}&&\dTo{\reverse}\\
\Delta&\rInto&\Ord
\end{diagram}
commutes up to natural isomorphism. Since the diagram
\begin{diagram}[midshaft,h=2.1em]
\Ord&\rTo{-+1}&\RightInt\\
\dTo{\reverse}&&\dTo{\reverse}\\
\Ord&\rTo{1+-}&\LeftInt
\end{diagram}
commutes, we thus get a square
\begin{diagram}[midshaft,h=2.1em]
\Delta&\rTo{-+1}&\RightInt\\
\dTo{\reverse}&&\dTo{\reverse}\\
\Delta&\rTo{1+-}&\LeftInt
\end{diagram}
which commutes up to a natural isomorphism. The result now follows from the previous proposition coupled with the fact that the two vertical arrows are isomorphisms of categories.
\end{proof}

There is one last useful property of the categories of intervals which we must address. It concerns a duality $\op{\Ord}\simeq\Int$.

\begin{construction}[Stone duality for intervals]
There is a natural functor
\[ \dual:\op{\Ord}\To\Int \]
such that for an ordinal $a$, $\dual(a)$ is the set $\Ord(a,2)$ with the partial order induced from $2$: it is actually a total order.\\
Reciprocally, there is a functor
\[ \dual:\Int\To\op{\Ord} \]
which associates with each interval $a$, the set $\Int(a,2)$ with the partial order induced from $2$: again, this is a total order.\\
Then the compositions
\begin{align*}
&\Int\xTo{\dual}\op{\Ord}\xTo{\dual}\Int\\
&\op{\Ord}\xTo{\dual}\Int\xTo{\dual}\op{\Ord}
\end{align*}
are naturally isomorphic to the identity functors. In particular, the duality functors are inverse equivalences of categories.
\end{construction}

\begin{remark}
The duality functors above give rise to a duality isomorphism
\[ \op{\LeftInt}\simeq\RightInt \]
where, for example, the diagram
\begin{diagram}[midshaft,h=2.3em]
\op{\LeftInt}&\rTo_{\sim}^{\ \dual\ }&\RightInt\\
\dTo{\op{\inclusion}}&&\dTo{1+-}\\
\op{\Ord}&\rTo{\ \dual\ }&\Int
\end{diagram}
commutes.
\end{remark}

\begin{corollary}\label{corollary:opDelta_equiv_StrictInt}
The equivalence
\[ \dual:\op{\Ord}\To\Int \]
restricts to an equivalence
\[ \dual:\op{\Delta}\To\StrictInt \]
\end{corollary}

\begin{remark}
We will consider $\op{\Delta}$ naturally as an equivalent subcategory of $\StrictInt$ via this duality functor.
\end{remark}

\begin{proposition}\label{proposition:StrictInt->RightInt_homotopy_cofinal}
The inclusion functors
\begin{align*}
\StrictInt&\Into\RightInt\\
\StrictInt&\Into\LeftInt
\end{align*}
are homotopy cofinal.
\end{proposition}
\begin{proof}
It is enough to prove that the first functor is homotopy cofinal: the second one follows by using the functor $\reverse$ appropriately.

The inclusion
\[ \StrictInt\Into\RightInt \]
fits into a commutative diagram
\begin{diagram}[midshaft,h=2.3em]
\op{\Delta}&\rTo{\op{(1+-)}}&\op{\LeftInt}\\
\dTo{\dual}&&\dTo{\dual}\\
\StrictInt&\rInto{\inclusion}&\RightInt
\end{diagram}
Since the two vertical arrows are equivalences of categories, and the top arrow is homotopy cofinal (proposition \sref{proposition:1+-:Delta->LeftInt_homotopy_final}), we conclude that
\[ \StrictInt\Into\RightInt \]
is homotopy cofinal.
\end{proof}

\section{Monads and categories of intervals}\label{section:monads_intervals}

The relevance of the categories of intervals defined in the previous proposition is related to how they naturally index algebras over a monad. We collect in this section the relevant results. Let us fix a category $C$ and a monad $T$ on $C$.

We begin with the well-known fact that monoidal functors from $\Ord$ into a monoidal category $D$ are the same as monoids in $D$. We apply it to the monoidal category $\HomCat{C}{C}$ of endo-functors of $C$.

\begin{construction}
Any monad $T$ on a category $C$ gives rise to an essentially unique monoidal functor
\[ T^\bullet:\Ord\To\HomCat{C}{C} \]
such that $T$ is the monoid $T^\bullet(1)$ in the monoidal category $\HomCat{C}{C}$. The monoidal structure on $\HomCat{C}{C}$ is given by composition.\\
Composing with the evaluation at an object $x$ of $C$ gives a functor
\[ T^\bullet x:\Ord\To C \]
This construction obviously extends to a functor
\[ T^\bullet -:C\To\HomCat{\Ord}{C} \]
\end{construction}

We will now describe how the functor $T^\bullet x$ changes when one allows $x$ to become an algebra for $T$.

\begin{proposition}\label{proposition:functors_intervals_T-algebra}
Assume $x$ is a $T$-algebra in $C$. Then there exist functors
\begin{align*}
T^{\bullet-1} x&:\RightInt\To C\\ 
T^{\bullet-1} x&:\Int\To T\dash\alg
\end{align*}
such that the diagram
\begin{diagram}[midshaft]
\Int&\rInto&\RightInt&\lTo{\ -+1\ }&\Ord\\
\dTo{T^{\bullet-1}x}&&\dTo{T^{\bullet-1}x}&\ldTo_{\mathrlap{T^{\bullet}x}}\\
T\dash\alg&\rTo{\proj}&C
\end{diagram}
commutes.
\end{proposition}

\begin{notation}
In our notation $\bullet$ is meant to represent the number of points in an ordinal. So the notation $T^{\bullet-1}x$ is meant to inform about its value at an ordinal with $n+1$ points ($n\in\NN$):
\[ (T^{\bullet-1}x)(n+1)\simeq T^{\circ n}x \]
\end{notation}

To finish this section, we analyze the case of a free $T$-algebra.

\begin{proposition}\label{proposition:functor_intervals_free_T-algebra}
Let $x$ be an object of $C$, and consider the free $T$-algebra $Tx$ in $C$.\\
There exists a functor
\[ T^\bullet x:\LeftInt\To T\dash\alg \]
such that the diagram
\begin{diagram}[midshaft,h=2.3em]
\Int&\rInto&\LeftInt&\rInto&\Ord\\
&\rdTo_{\mathllap{T^{\bullet-1}(Tx)}}&\dTo_{T^\bullet x}&&\dTo_{T^\bullet x}\\
&&T\dash\alg&\rTo{\proj}&C
\end{diagram}
commutes.
\end{proposition}

\section{Bar constructions for enriched categories}\label{section:bar_construction_enriched}

In this section, we apply the constructions of the preceding section to obtain bar constructions for enriched functors. Throughout this section, let $V$ denote a bicomplete symmetric monoidal closed category, and let $C$ be a $V$-category which is cocomplete (as a $V$-category). Recall that $C_0$ denotes the underlying category of $C$.

\begin{construction}[monad whose algebras are enriched functors]\label{construction:T_AC_monad}
Let $A$ be a small $V$-category. We will now construct a monad on the category $\HomCat{\ob A}{C_0}$ whose category of algebras is equivalent to $V\dash\CAT(A,C)$.\\
Define the functor
\[ T_{A,C}:\HomCat{\ob A}{C_0}\To\HomCat{\ob A}{C_0} \]
on objects by
\[ T_{A,C}(F)\defeq\coprod_{a\in\ob A}F(a)\tensor A(a,-) \]
Here, ``$\tensor$'' denotes tensoring of an object of $V$ with an object of $C$, which is always possible since $C$ is cocomplete. We expect the functoriality of $T_{A,C}$ to be clear to the reader.\\
The functor $T_{A,C}$ becomes a monad on $\HomCat{\ob A}{C_0}$ with the unit
\[ \eta:\id_{\HomCat{\ob A}{C_0}}\To T_{A,C} \]
being determined by the identity morphisms of the $V$-category $A$, and the multiplication
\[ \mu:T_{A,C}\circ T_{A,C}\To T_{A,C} \]
arising from the composition in $A$. We leave the straightforward details of defining the unit and multiplication for $T_{A,C}$ to the reader.
\end{construction}

\begin{proposition}\label{proposition:T_AC-alg=V-CAT(A,C)}
Let $A$ be a small $V$-category.\\
There is a canonical isomorphism of categories
\[ T_{A,C}\dash\alg\xTo{\ \cong\ }V\dash\CAT(A,C) \]
\end{proposition}

\begin{remark}
Let us consider the specific case $C=V$.\\
In order to define the monad $T_{A,V}$, it is only necessary that $V$ be a monoidal category whose monoidal product preserves coproducts in each variable. In this case, the algebras for $T_{A,V}$ are somewhat akin to internal functors on $\internal A$. In fact, if $V$ is cartesian monoidal with totally disjoint small coproducts, then the two notions coincide.\\
Moreover, the possibility of defining $T_{A,V}$ with fewer assumptions on $V$ could be used to extend the concepts in the next sections, for instance, to the case of $\Top$. We will, however, not pursue this.
\end{remark}

In view of this proposition, we will identify $V$-functors $A\to C$ with $T_{A,C}$-algebras. The upshot of this perspective is that we can immediately obtain bar constructions for functors.

\begin{construction}[bar construction for enriched functor]\label{construction:bar_const_enriched_functor}
Let $A$ be a small $V$-category, and $F:A\to C$ a $V$-functor.\\
According to the last proposition, $F$ gives an algebra for $T_{A,C}$, and therefore we get functors (proposition \sref{proposition:functors_intervals_T-algebra})
\begin{align*}
(T_{A,C})^{\bullet-1} F&:\RightInt\To\HomCat{\ob A}{C_0}\\ 
(T_{A,C})^{\bullet-1} F&:\Int\To V\dash\CAT(A,C)
\end{align*}
We will rename them
\begin{align*}
\BarConst(A,F)&:\RightInt\To\HomCat{\ob A}{C_0}\\ 
\BarConst(A,F)&:\Int\To V\dash\CAT(A,C)
\end{align*}
These functors verify a commutative diagram
\begin{diagram}[midshaft]
\Int&\rInto&\RightInt\\
\dTo{\BarConst(A,F)}&&\dTo{\BarConst(A,F)}\\
V\dash\CAT(A,C)&\rTo{\ \proj\ }&\HomCat{\ob A}{C_0}
\end{diagram}
Both bar constructions $\BarConst(A,F)$ are functorial on $F\in V\dash\CAT(A,C)$ and the $V$-category $A$.
\end{construction}

\begin{notation}
By default, we will usually mean the functor
\[ \BarConst(A,F):\Int\To V\dash\CAT(A,C) \]
when we refer to $\BarConst(A,F)$.
\end{notation}

\begin{remark}
The value of $\BarConst(A,F)$ at an ordinal with $n+1$ points ($n\in\NN$) is
\begin{align*}
\BarConst(A,F)(n+1)&\simeq (T_{A,C})^{\circ n}F\\
&=\ \ \coprod_{\mathclap{a_1,\ldots,a_n\in\ob A}}\ \ {\mbox{\small $F(a_1)\tensor A(a_1,a_2)\tensor\cdots\tensor A(a_{n-1},a_n)\tensor A(a_n,-)$}}
\end{align*}
Moreover, $\BarConst(A,F)(1)=F$.
\end{remark}

\begin{notation}[bar construction for contravariant functor]\label{notation:bar_construction_contravariant_functor}
Let $A$ be a small $V$-category, and $G:\op{A}\to C$ a $V$-functor.\\
We will denote the bar construction $\BarConst(\op{A},G)$ by $\BarConst(G,A)$
\[ \BarConst(G,A)\defeq\BarConst(\op{A},G) \]
\end{notation}

\begin{construction}[two-sided bar construction]
Let $A$ be a small $V$-category.\\
Let $F:A\to C$ and $G:\op{A}\to V$ be $V$-functors.\\
We define the two-sided bar construction
\[ \BarConst(G,A,F):\Int\To C \]
to be the composition
\[ \Int\xTo{\BarConst(A,F)}V\dash\CAT(A,C)\xTo{G\tensor_A -} C \]
which exists since $C$ is cocomplete by assumption. In equation form, we get
\[ \BarConst(G,A,F)=G\tensor_A\BarConst(A,F) \]
The bar construction $\BarConst(G,A,F)$ is functorial in $A$, $F$, and $G$.
\end{construction}

\begin{remark}
Note that the restriction of $\BarConst(G,A,F)$ to $\op{\Delta}$ (recall corollary \sref{corollary:opDelta_equiv_StrictInt})
\[ \op{\Delta}\xTo[\sim]{\,\dual\,}\StrictInt\Into\Int\xTo{\BarConst(G,A,F)}C \]
is naturally isomorphic to the usual two-sided bar construction.\\
From this observation, it becomes natural to consider the restriction of $\BarConst(G,A,F)$ to the category of strict intervals.
\end{remark}

\begin{remark}
The value of the two-sided bar construction on a strict interval with $n+2$ elements ($n\in\NN$) is
{\small \[ \BarConst(G,A,F)(1+n+1)\simeq\ \ \coprod_{\mathclap{a_0,\ldots,a_n\in\ob A}}\ \ F(a_0)\tensor A(a_0,a_1)\tensor\cdots\tensor A(a_{n-1},a_n)\tensor G(a_n) \]}
\end{remark}

\begin{proposition}
Let $A$ be a small $V$-category. Let $F:A\to C$ and $G:\op{A}\to V$ be $V$-functors.\\
The two-sided bar construction $\BarConst(G,A,F)$ is naturally isomorphic to (recall notation \sref{notation:bar_construction_contravariant_functor})
\[ \BarConst(G,A,F)=\BarConst(G,A)\tensor_A F \]
\end{proposition}

\begin{remark}[symmetry of bar construction]
From this proposition, we immediately get the symmetry of the two sided-bar construction
\[ \BarConst(F,\op{A},G)=\BarConst(G,A,F) \]
for any functors $F:A\to V$ and $G:\op{A}\to V$.
\end{remark}

\section{Derived enriched colimits}\label{section:derived_enriched_colimits}

In this section, we assume that $V$ is a symmetric monoidal closed simplicial model category with cofibrant unit (see section \sref{section:model_categories}). To be more precise, we demand that $V$ is a bicomplete symmetric monoidal closed $\sSet$-category, and a simplicial model category, for which the monoidal product in $V$ verifies the pushout-product axiom (definition \sref{definition:left_Quillen_bifunctor}). Moreover, the unit $I$ of the monoidal structure of $V$ is required to be cofibrant.

\comment{We could replace (without adding any significant changes to the proofs below) the assumption of (1) $V$ being simplicial with (2) a strong symmetric monoidal functor $-\tensor I:\sSet\To V$ preserving colimits, and preserving cofibrations and trivial cofibrations. This is almost the same as the above, except that $V$ does not need to have a simplicial enrichment. This refinement encompasses the case of $\Top$.\\
One may be able to assume that $V$ is just framed. However the proofs below would not always work because they assume quite a few properties of enriched (simplicial) colimits, and thus might not transfer immediately to the framed case.}

\begin{definition}[locally cofibrant enriched category]
Let $A$ be a $V$-category.\\
We say that $A$ is {\em locally cofibrant} if for any $a,b\in\ob A$, the morphism object $A(a,b)$ is cofibrant in $V$.
\end{definition}

\begin{definition}[identity-cofibrant enriched category]\label{definition:identity-cofibrant_V-cat}
Let $A$ be a $V$-category.\\
We say that $A$ is {\em identity-cofibrant} if $A$ is locally cofibrant, and the identity morphism of $A$ at $a$ \[ \id_a:I\To A(a,a) \] is a cofibration for any $a\in\ob A$.
\end{definition}

\begin{proposition}
The functor
\[ V(I,-):V\To\sSet \]
is a lax symmetric monoidal $\sSet$-functor.
\end{proposition}

\begin{remark}\label{remark:enriched_colimit_V_sSet}
In particular, any $V$-category $C$ gives rise to a $\sSet$-category $[V(-,I)]C$.\\
A $\sSet$-enriched colimit in $[V(-,I)]C$, $G\tensor_A F$, for any $\sSet$-functors
\begin{align*}
F&:A\To [V(-,I)]C\\
G&:\op{A}\To\sSet
\end{align*}
can be computed as a $V$-enriched colimit:
\[ G\tensor_A F=(G\tensor I)\ \ \tensor_{\mathclap{(-{\tensor}I)A}}\ \ F \]
where $-\tensor I:\sSet\To V$ is a strong symmetric monoidal $\sSet$-functor (due to $V$ being closed).
\end{remark}

\begin{proposition}
For any $V$-model category $C$, the $\sSet$-category $[V(-,I)]C$ is canonically a simplicial model category.
\end{proposition}

\begin{definition}[derived enriched colimit]
Let $A$ be a small locally cofibrant $V$-category, and $C$ a cocomplete $V$-model category (definition \sref{definition:V-model_category}).\\
Let $F:A\to C$, $G:\op{A}\to V$ be $V$-functors which are objectwise cofibrant.\\
We define the {\em derived enriched colimit} $G\tensor^\derived_A F$ to be
\[ G\tensor^\derived_A F\defeq\hocolim_{\StrictInt}\BarConst(G,A,F) \]
\end{definition}

\begin{remark}[clarification of definition]
In the previous definition, the simplicial model category $[V(I,-)]C$ underlying $C$ is required to make sense of the homotopy colimit.\\
Alternatively, one can just observe (thanks to remark \sref{remark:enriched_colimit_V_sSet}) that the homotopy colimit is
\[ \hocolim_{\StrictInt}\big(\BarConst(G,A,F)\big)=\big[\Nerve\big(\op{\StrictInt}/-\big)\tensor I\big]\ \tensor_{\mathclap{\StrictInt}}\ \BarConst(G,A,F) \]
The enriched colimit on the right hand side is then just a $V$-enriched colimit in $C$, with no reference to the underlying simplicial category of $C$.
\end{remark}

\begin{remark}[cofibrancy conditions]
For simplicity, we assume that $F$ and $G$ are objectwise cofibrant, so as not to introduce cofibrant replacements into the definition.\\
Note that the restriction of the bar construction in the definition above to $\StrictInt$ is objectwise cofibrant, as needed to have homotopy invariance of the homotopy colimit.\\
If $A$ is identity-cofibrant, we can replace the above homotopy colimit along $\StrictInt$ with the geometric realization
\[ \hocolim_{\StrictInt}\BarConst(G,A,F)\xTo{\;\sim\;}\realization{\BarConst(G,A,F)} \]
which is defined after restricting the bar construction to $\op{\Delta}$. The map is a weak equivalence since the bar construction then gives a Reedy cofibrant simplicial object in $C$.
\end{remark}

\begin{remark}[homotopy invariance]
The construction $G\tensor^\derived_A F$ is homotopy invariant: if $G\to G'$ and $F\to F'$ are natural transformations which are objectwise weak equivalences of objectwise cofibrant functors, then the natural map
\[ G\tensor^\derived_A F\To G'\tensor^\derived_A F' \]
is a weak equivalence in $C$. This homotopy invariance is a straightforward consequence of the homotopy invariance of the homotopy colimit (along $\StrictInt$).
\end{remark}

\begin{lemma}
Let $A$ be a small locally cofibrant $V$-category, and $C$ a cocomplete $V$-model category.\\
If $F:A\to C$ is an objectwise cofibrant functor, then the canonical augmentation (note that $\BarConst(A,F)(1)=F$)
\[ \BarConst(A,F)\To F \]
induces a $V$-natural weak equivalence
\[ \hocolim_{\StrictInt}\BarConst(A,F)\xTo{\,\sim\,} F \]
of objectwise cofibrant functors $A\to C$.
\end{lemma}
\begin{proof}[Sketch of proof]
It is enough to verify that for each $a\in\ob A$, the map
\begin{equation}\label{equation:aux_hocolim_Bar(F,A)_equiv_F}
\hocolim_{\StrictInt}\big(\BarConst(A,F)(a)\big)\To F(a)
\end{equation}
is a weak equivalence of cofibrant objects in $C$. $F(a)$ is cofibrant by hypothesis. The left hand side is also cofibrant since it is the homotopy colimit of an objectwise cofibrant diagram in $C$.

There exists a commutative diagram
\begin{diagram}[midshaft,h=2.3em]
\Int&\rTo{\BarConst(A,F)}& V\dash\CAT(A,C)&\rTo{\evaluation{a}}&C\\
\dInto&&\dTo{\proj}&&\dEqual\\
\RightInt&\rTo{\BarConst(A,F)}&\HomCat{\ob A}{C_0}&\rTo{\evaluation{a}}&C\\
\end{diagram}
which gives a factorization
\[ \hocolim_{\StrictInt}\big(\BarConst(A,F)(a)\big)\xTo{f}\hocolim_{\RightInt}\big(\BarConst(A,F)(a)\big)\xTo{g}F(a) \]
of the map \seqref{equation:aux_hocolim_Bar(F,A)_equiv_F}. The second arrow, $g$, is a weak equivalence in $C$ since
\begin{enum}
\item $\BarConst(A,F)(a)$ is objectwise cofibrant;
\item $1$ is terminal in $\RightInt$;
\item $\big[\BarConst(A,F)(a)\big](1)=F(a)$, and $g$ is induced by the unique morphism from each object of $\RightInt$ to $1$.
\end{enum}
On the other hand, the first arrow, $f$, is a weak equivalence because
\[ \inclusion:\StrictInt\Into\RightInt \]
is homotopy cofinal (proposition \sref{proposition:StrictInt->RightInt_homotopy_cofinal}) and $\BarConst(A,F)(a)$ is objectwise cofibrant.

In conclusion, the map \seqref{equation:aux_hocolim_Bar(F,A)_equiv_F} is a weak equivalence, as required.
\end{proof}

\begin{definition}[homotopy left Kan extension]
Let $A$, $B$ be small locally cofibrant $V$-categories, and $C$ a cocomplete $V$-model category (definition \sref{definition:V-model_category}).\\
Let $F:\op{A}\to C$ be an objectwise cofibrant $V$-functor. Let $f:A\to B$ be a $V$-functor.\\
We define the {\em homotopy left Kan extension} of $F$ along $\op{f}$
\[ \hoLKE_{\op{f}}F:\op{B}\To C \]
by the expression
\[ \hoLKE_{\op{f}}F\defeq\hocolim_{\StrictInt}\big(\LKE_{\op{f}}\BarConst(F,A)\big) \]
\end{definition}

\begin{proposition}[change of categories]\label{proposition:derived_enriched_colimit_change_categories}
Let $A$, $B$ be small locally cofibrant $V$-categories, and $C$ a cocomplete $V$-model category.\\
Let $F:B\to C$, $G:\op{A}\to V$ be $V$-functors which are objectwise cofibrant. Let $f:A\to B$ be a $V$-functor.\\
Then there is a $V$-natural weak equivalence in $C$
\[ \big(\hoLKE_{\op{f}}G\big)\tensor^\derived_B F\xTo{\,\sim\,} G\tensor^\derived_A(F\circ f) \]
\end{proposition}
\begin{proof}[Sketch of proof]
Let $X:B\to C$ be given by
\[ X\defeq\hocolim_{\StrictInt}\BarConst(B,F) \]
By the previous lemma, the canonical map $X\to F$ is a natural weak equivalence of objectwise cofibrant functors $B\to C$. Therefore, the induced map
\[ G\tensor^\derived_A(X\circ f)\xTo{\,\sim\,} G\tensor^\derived_A(F\circ f) \]
is a weak equivalence. We will now manipulate the left hand side:
\begin{align*}
G\tensor^\derived_A(X\circ f)&=\hocolim_{\StrictInt}\BarConst(G,A,X\circ f)\\
&=\hocolim_{\StrictInt}\Big(\BarConst(G,A)\tensor_A(X\circ f)\Big)\\
&=\Big(\hocolim_{\StrictInt}\BarConst(G,A)\Big)\tensor_A (X\circ f)\\
&=\LKE_{\op{f}}\!\Big(\hocolim_{\StrictInt}\BarConst(G,A)\Big)\tensor_B X \\
&=\big(\hoLKE_{\op{f}}G\big)\tensor_B X\\
&=\big(\hoLKE_{\op{f}}G\big)\tensor_B\Big(\hocolim_{\StrictInt}\BarConst(B,F)\Big)\\
&=\hocolim_{\StrictInt}\Big(\big(\hoLKE_{\op{f}}G\big)\tensor_B\BarConst(B,F)\Big)\\
&=\hocolim_{\StrictInt}\BarConst\!\big(\hoLKE_{\op{f}}G,B,F\big)\\
&=\big(\hoLKE_{\op{f}}G\big)\tensor^\derived_B F
\end{align*}
\end{proof}

\section{Homotopy colimits of enriched functors}\label{section:homotopy_colimits_enriched}

Assume now that $V$ is a cartesian closed simplicial model category whose unit is cofibrant (see section \sref{section:model_categories}). To be more precise, we demand that $V$ is a bicomplete cartesian closed $\sSet$-category, and a simplicial model category, for which the product in $V$ verifies the pushout-product axiom. Furthermore, the terminal object $1$ of $V$ is cofibrant.

\begin{definition}[homotopy colimit of enriched functor]
Let $A$ be a small locally cofibrant $V$-category, and $C$ a cocomplete $V$-model category (see definition \sref{definition:V-model_category}).\\
Let $F:A\to C$ be a $V$-functor which is objectwise cofibrant.\\
We define the {\em homotopy colimit} of $F$ to be
\[ \hocolim F\defeq 1\tensor^\derived_A F \]
\end{definition}

\begin{remark}[homotopy invariance]
The homotopy invariance for the derived enriched colimit entails that when $F\to F'$ is a natural transformation which is an objectwise weak equivalence of objectwise cofibrant functors, the induced morphism
\[ \hocolim F\To\hocolim F' \]
is a weak equivalence in $C$.
\end{remark}

\begin{remark}[relation to usual homotopy colimit]
If $A$ happens to be an ordinary small category (viewed as a $V$-category in the usual manner), then the homotopy colimit above is canonically equivalent to the usual homotopy colimit:
\begin{align*}
1\tensor^\derived_A F&=\hocolim_\StrictInt\BarConst(1,A,F)\\
&=\hocolim_\StrictInt\Big(\BarConst(1,A)\tensor_A F\Big)\\
&=\Big(\hocolim_\StrictInt\BarConst(1,A)\Big)\tensor_A F\\
&\xTo{\sim}\realization{\BarConst(1,A)}\tensor_A F\\
&=\Nerve(\op{A}/-)\tensor_A F
\end{align*}
The first entry is the enriched homotopy colimit as defined above. The last entry is the usual homotopy colimit of $F$ along the ordinary category $A$.\\
The non-identity weak equivalence in the calculation above follows from the fact that $A$ is identity-cofibrant.
\end{remark}

\begin{definition}[homotopy cofinal enriched functor]
Let $A$, $B$ be locally cofibrant $V$-categories, with $A$ small.\\
Let $F:A\to B$ be a $V$-functor.\\
We say the functor $F$ is {\em homotopy cofinal} with respect to $V$ if for any object $x$ of $B$, the unique map
\[ 1\tensor^\derived_A B(x,F-)\To 1 \]
is a weak equivalence in $V$.
\end{definition}

\begin{remark}
It is easy to relax the conditions of the definition to allow any $V$-category $B$: just substitute $B(x,F-)$ with a cofibrant replacement of it.
\end{remark}

\begin{proposition}\label{proposition:homotopy_cofinal_colimits_equiv}
Let $A$, $B$ be small locally cofibrant $V$-categories, and $C$ a cocomplete $V$-model category.\\
Let $F:B\to C$ be a $V$-functor which is objectwise cofibrant. Let $f:A\to B$ be a $V$-functor.\\
There is a natural morphism in $C$
\[ \hocolim(F\circ f)\To\hocolim F \]
which is a weak equivalence if $f$ is homotopy cofinal.
\end{proposition}
\begin{proof}[Sketch of proof]
The natural map arises from the functoriality of derived enriched colimits, which is a consequence of the functoriality of the two-sided bar construction. We leave the details to be worked out by the reader.

The map in the statement fits as the bottom map in the commutative triangle
\begin{diagram}[h=2.3em,balance]
&\big(\hoLKE_{\op{f}}1\big)\tensor^\derived_B F\\
\ldTo(1,2)[uppershortfall=-.2em,lowershortfall=0em]{\mathllap{g}}&&\rdTo(1,2)[uppershortfall=-.2em,lowershortfall=0em]{\mathrlap{h}}\\
1\tensor_A^\derived F&\rTo&1\tensor_B^\derived F
\end{diagram}
The map $g$ in the triangle is the map from proposition \sref{proposition:derived_enriched_colimit_change_categories}, and therefore is a weak equivalence. The map $h$ in the triangle is induced by the unique map
\[ \hoLKE_{\op{f}} 1\To 1 \]
Given $x\in\ob B$, there are natural isomorphisms
\begin{align*}
\big(\hoLKE_{\op{f}}1\big)(x)&=\hocolim_{\StrictInt}\big(\LKE_{\op{f}}\BarConst(1,A)\big)(x)\\
&=\hocolim_{\StrictInt}\Big(\BarConst(1,A)\tensor_A B(x,f-)\Big)\\
&=\hocolim_{\StrictInt}\BarConst\big(1,A,B(x,f-)\big)\\
&=1\tensor^\derived_A B(x,f-)
\end{align*}
It follows that if $f$ is homotopy cofinal then
\[ \hoLKE_{\op{f}} 1\To 1 \]
is a natural weak equivalence (of objectwise cofibrant functors). Therefore the map $h$ in the commutative triangle above is a weak equivalence.

In conclusion $g$ and $h$ in the triangle are weak equivalences, and so the the bottom map is a weak equivalence, as we intended to prove.
\end{proof}

\section{Weak equivalence of enriched categories}\label{section:weak_equivalence_enriched_categories}

Let $V$ denote a complete model category whose terminal object, $1$, is cofibrant. We consider $V$ with the cartesian symmetric monoidal structure.

\begin{construction}
Let us define the functor
\[ \pi_0:V\To\Set \]
to be given on an object $x$ of $V$ by
\[ \pi_0(x)=\pi_l(1,x) \]
where $\pi_l(1,x)$ denotes the quotient of $V(1,x)$ by the equivalence relation of left homotopy (see definition 7.3.2 of \cite{Hirschhorn}): left homotopy gives an equivalence relation because $1$ is cofibrant (see proposition 7.4.5 of \cite{Hirschhorn}).\\
The functoriality of $\pi_0$ is induced from that of $V(1,-)$.
\end{construction}

\begin{remark}
Note that in the cases of simplicial sets or topological spaces, the functor $\pi_0$ defined above is canonically isomorphic to the usual functor $\pi_0$.
\end{remark}

We leave the proof of the following statement to the reader.

\begin{proposition}
The functor
\[ \pi_0:V\To\Set \]
preserves all finite products.
\end{proposition}

\begin{definition}[weak equivalence of $V$-categories]\label{definition:weak_equivalence_V-categories}
Let $F:A\to B$ be a $V$-functor.\\
We say $F$ is a {\em weak equivalence} with respect to $V$ if $F$ is locally a weak equivalence in $V$ (recall terminology \sref{notation:local_nameproperty}), and the functor
\[ \pi_0 F:\pi_0 A\To\pi_0 B \]
is essentially surjective.
\end{definition}

\begin{example}
In the case of $\Top$ with the Str{\o}m model structure, this recovers the notion of weak equivalence between topologically enriched categories (see definition \sref{definition:weak_equivalence_Top-cat}).
\end{example}

\begin{lemma}
Assume $V$ is a cartesian closed model category with cofibrant unit.\\
Let $A$ be a locally cofibrant $V$-category, and $a,b\in\ob A$.\\
If $a$, $b$ are isomorphic in $\pi_0 A$ then there exists a morphism $f:a\to b$ in $A$ such that the natural transformation
\[ A(b,-)\xTo{-\circ f} A(a,-) \]
is an objectwise weak equivalence in $V$.
\end{lemma}
\begin{proof}[Sketch of proof]
If $a$, $b$ are isomorphic in $\pi_0 A$, there exist morphisms $f:a\to b$, $g:b\to a$ in $A$:
\begin{align*}
f&:1\To A(a,b)\\
g&:1\To A(b,a)
\end{align*}
together with a left homotopy
\[ lH_1:f\circ g\underset{l}{\simeq}\id_b \]
of maps $1\to A(b,b)$ in $V$, and a left homotopy
\[ lH_2:g\circ f\underset{l}{\simeq}\id_a \]
of maps $1\to A(a,a)$ in $V$.

The morphisms $f$, $g$ induce natural transformations
\begin{align*}
f^\ast:A(b,-)&\xTo{-\circ f} A(a,-)\\
g^\ast:A(a,-)&\xTo{-\circ g} A(b,-)
\end{align*}
and we can construct a left homotopy
\[ (f^\ast\circ g^\ast)_x\underset{l}{\simeq}\id_{A(a,x)} \]
for each $x\in\ob A$ by taking the product of the left homotopy $lH_2$ with $A(a,x)$ and composing in $A$. More precisely, let the left homotopy $lH_2$ be given by the cylinder
\[ 1\amalg 1\xTo{h} X\xTo{\;\sim\;}1 \]
(where $h$ is a cofibration), and the map
\[ lH_2:X\To A(a,a) \]
such that
\[ lH_2\circ h=(f\circ g)\amalg\id_a \]
Then we construct the cylinder
\[ A(a,x)\amalg A(a,x)\xTo{h\times A(a,x)} X\times A(a,x)\xTo{\;\sim\;}A(a,x) \]
where the first arrow is a cofibration because $A(a,x)$ is cofibrant. The left homotopy between $(f^\ast\circ g^\ast)_x$ and $\id_{A(a,x)}$ is defined by the composition
\[ X\times A(a,x)\xTo{lH_2\times\id}A(a,a)\times A(a,x)\xTo{\composition} A(a,x) \]
(where the second arrow is composition in $A$).

Similarly, the left homotopy $lH_1$ can be used to construct a left homotopy
\[ (g^\ast\circ f^\ast)_x\underset{l}{\simeq}\id_{A(b,x)} \]
for each $x\in\ob A$. In conclusion, for each $x\in\ob A$, $(f^\ast\circ g^\ast)_x$ and $(g^\ast\circ f^\ast)_x$ are left homotopic to the respective identity maps, and in particular are weak equivalences. Given that $(f^\ast)_x\circ (g^\ast)_x$ and $(g^\ast)_x\circ(f^\ast)_x$ are weak equivalences, the two-out-of-six property of model categories implies that both $(f^\ast)_x$ and $(g^\ast)_x$ are weak equivalences. This finishes the proof.
\end{proof}

We finish this section with a predictable result.

\begin{proposition}[weak equivalence implies homotopy cofinal]\label{proposition:weak_equivalence_implies_homotopy_cofinal}
Assume $V$ is a cartesian closed simplicial model category whose unit is cofibrant.\\
Let $A$, $B$ be locally cofibrant $V$-categories, with $A$ small.\\
Let $F:A\to B$ be a $V$-functor.\\
If $F$ is a weak equivalence with respect to $V_0$ then $F$ is homotopy cofinal with respect to $V$.
\end{proposition}
\begin{proof}[Proof]
Consider an object $b$ of $B$. We want to prove that
\[ 1\tensor^\derived_A B(b,F-)\xTo{\;\sim\;} 1 \]
is a weak equivalence in $V$.

Let $a\in\ob A$ be such that there is an isomorphism
\[ (\pi_0 F)(a)\simeq b \]
in $\pi_0 B$: such an object of $A$ exists since $\pi_0 F$ is essentially surjective. By the preceding lemma, there exists a morphism $f:b\to Fa$ in $B$ such that the natural transformation
\[ B(Fa,-)\xTo{-\circ f} B(b,-) \]
is an objectwise weak equivalence in $V$. On the other hand, the functor $F$ induces a natural transformation
\[ A(a,-)\To B(Fa,F-) \]
which is an objectwise weak equivalence. Since all three functors are objectwise cofibrant, we obtain weak equivalences in $V$
\[ 1\tensor^\derived_A A(a,-)\xTo{\;\sim\;}1\tensor^\derived_A B(Fa,F-)\xTo{\;\sim\;}1\tensor^\derived_A B(b,F-) \]
The result is now a consequence of
\[ 1\tensor^\derived_A A(a,-)\xTo{\;\sim\;} 1 \]
being a weak equivalence, by the lemma \sref{lemma:A(x,-)tensor_1_contractible} presented next.
\end{proof}

\begin{lemma}\label{lemma:A(x,-)tensor_1_contractible}
Assume $V$ is a cartesian closed simplicial model category whose unit is cofibrant.\\
Let $A$ be a small locally cofibrant $V$-category.\\
For any object $x$ of $A$, the unique map
\[  1\tensor^\derived_A A(x,-)\To 1 \]
is a weak equivalence.
\end{lemma}
\begin{proof}
By definition (see construction \sref{construction:bar_const_enriched_functor} for last line)
\begin{align*}
1\tensor^\derived_A A(x,-)&=\hocolim_\StrictInt\BarConst\big(1,A,A(x,-)\big)\\
&=\hocolim_\StrictInt\Big(1\tensor_A\BarConst\big(A,A(x,-)\big)\Big)\\
&=\hocolim_\StrictInt\Big(1\tensor_A\big[(T_{A,V})^{\bullet-1}\big(A(x,-)\big)\big]\Big)
\end{align*}
On the other hand, $A(x,-)$ is a free $T_{A,V}$-algebra (recall construction \sref{construction:T_AC_monad}):
\[ A(x,-)=T_{A,V}(\delta_x) \]
where
\[ \delta_x:\ob A\To C_0 \]
is defined by
\[ \delta_x(a)=\left\{
\begin{array}{ll}
\!1&\text{if } a=x\\
\!\emptyset&\text{if } a\neq x
\end{array}
\right. \]
Here, $\emptyset$ denotes an initial object of $V$. Given that $A(x,-)$ is a free $T_{A,V}$-algebra, there is a commutative diagram
\begin{diagram}[midshaft,h=2.3em]
\Int&\rTo{\;(T_{A,V})^{\bullet-1}(A(x,-))\;}&V\dash\CAT(A,V)\\
&\rdInto(1,2)\ruTo(1,2)_{\mathrlap{(T_{A,V})^\bullet(\delta_x)}}\\
&\LeftInt\\
\end{diagram}
by proposition \sref{proposition:functor_intervals_free_T-algebra} (recall also proposition \sref{proposition:T_AC-alg=V-CAT(A,C)}). This induces maps
\begin{align*}
1\tensor^\derived_A A(x,-)&\xlongequal{\smash{\quad}}\hocolim_\StrictInt\Big(1\tensor_A\big[(T_{A,V})^{\bullet-1}\big(A(x,-)\big)\big]\Big)\\
&\xTo{f}\hocolim_\LeftInt\Big(1\tensor_A\big((T_{A,V})^\bullet(\delta_x)\big)\Big)\\
&\To 1
\end{align*}
The middle arrow, $f$, is a weak equivalence because the inclusion
\[ \inclusion:\StrictInt\Into\LeftInt \]
is homotopy cofinal (proposition \sref{proposition:StrictInt->RightInt_homotopy_cofinal}). The last arrow
\[ \hocolim_\LeftInt\Big(1\tensor_A\big((T_{A,V})^\bullet(\delta_x)\big)\Big)\To 1 \]
is a weak equivalence because $\LeftInt$ has a terminal object, $1$, and
\begin{align*}
\Big(1\tensor_A\big((T_{A,V})^\bullet(\delta_x)\big)\Big)(1)&=1\tensor_A T_{A,V}(\delta_x)\\
&=1\tensor_A A(x,-)\\
&=1
\end{align*}
In summary
\[ 1\tensor^\derived_A A(x,-)\xTo{\;\sim\;}1 \]
is a weak equivalence.
\end{proof}

\section{Grothendieck constructions}\label{section:Grothendieck_construction_homotopy_colimit}

Grothendieck constructions play an important role in this text. Subsequently, it is useful to know how they relate to derived enriched colimits. In this section we give a calculation of the homotopy left Kan extension of a functor along the projection of a Grothendieck construction.

The motivation for such a calculation is a simple categorical fact: if $A$ is a small category, $F:\op{A}\to\Cat$ is a functor, and
\[ \pi:\Groth(F)\To A \]
is the canonical projection, then
\[ \LKE_{\op{\pi}}G=\colim_{\op{F(-)}}G \]
for each $G:\op{\Groth(F)}\to C$.
This result is a simple consequence of the adjunction
\begin{diagram}[midshaft]
F(x)&\pile{\rInto\\ \ \smash{\scriptscriptstyle\perp}\\ \lTo}&x/\pi
\end{diagram}
which shows that the inclusion $F(x)\hookrightarrow x/\pi$ is a final functor, for each $x\in\ob A$.

The main result in this section is a homotopical generalization of that calculation. We state it without proof after a few preliminary definitions.

\begin{definition}[value of internal $\Cat(V)$-valued functor]
Let $V$ be a category with pullbacks.\\
Let $A$ be a category object in $V$, and
\[ F:A\To\Cat(V) \]
an internal $\Cat(V)$-valued functor.\\
Given an internal functor $x:1\to A$, the {\em value of \(F\) at \(x\)}, $F(x)$, is the internal category in $V$ corresponding to the internal functor
\[ F\circ x:1\To\Cat(V) \]
\end{definition}

\begin{remark}
Note that an internal $\Cat(V)$-valued functor
\[ (P,p_0,p_1):1\To\Cat(V) \]
is the same as an internal category $P$ in $V$.
\end{remark}

\begin{remark}
An internal functor $x:1\to A$ is uniquely determined by $\ob x:1\to\ob A$.\\
In particular, if $A=\internal B$ for $B$ a $V$-category, an object $x\in\ob B$ is equivalent to giving an internal functor $x:1\to A$.
\end{remark}

\begin{definition}[pointwise locally cofibrant internal functor]
Let $V$ be a category with pullbacks and a model category. Let $A$ a be an internal category in $V$.\\
We say that an internal $\Cat(V)$-valued functor
\[ F:A\To\Cat(V) \]
is {\em pointwise locally cofibrant} if for any internal functor $x:1\to A$, the $V$-category $\disccat{F(x)}$ is locally cofibrant.
\end{definition}

\begin{proposition}
Let $V$ be a cartesian closed simplicial model category with cofibrant unit, and $C$ a cocomplete $V$-model category.\\
Assume that $V_0$ has totally disjoint small coproducts (see definition \sref{definition:totally_disjoint_coproducts} and terminology \sref{notation:small_totally_disjoint_coproducts}), and that the object $1$ of $V_0$ is connected over $\Set$ (definition \sref{definition:connected_object}).\\
Assume furthermore that $A$ is a small locally cofibrant $V$-category, and
\[ F:\op{\internal A}\To\Cat(V_0) \]
is a pointwise locally cofibrant internal $\Cat(V_0)$-valued functor.\\
For any objectwise cofibrant $V$-functor
\[ G:\op{\big(\disccat{\Groth(F)}\big)}\To C \]
there are canonical $V$-natural weak equivalences
\[ \hocolim_{\op{\left(\disccat{F(-)}\vphantom{M^d}\right)}}G\xTo{\;\sim\;}\hoLKE_{\op{\pi}}G\xTo{\;\sim\;}\hocolim_{\op{\left(\disccat{F(-)}\vphantom{M^d}\right)}}G \]
of objectwise cofibrant $V$-functors $\op{A}\to C$, whose composition (as displayed) is the identity.\\
Here, $\pi:\disccat{\Groth(F)}\to A$ denotes the canonical projection (see propositions \sref{proposition:projection_Groth_construction} and \sref{proposition:disccat_internal=id}).
\end{proposition}

\begin{remark}[Clarification]
In the preceding statement, we use, for each $x\in\ob A$, the natural inclusion \[ \disccat{F(x)}\Into\disccat{\Groth(F)} \] to restrict the functor $G$ to $\op{\big(\disccat{F(x)}\big)}$.\\
Note also that, while $\disccat{F(-)}$ does not define a functor on $\op{A}$ in any naive sense, the construction
\[ \hocolim_{\op{\left(\disccat{F(-)}\vphantom{M^d}\right)}}G \]
does define a $V$-functor on $\op{A}$. We leave the details to the reader.
\end{remark}

Having calculated the homotopy left Kan extension along the (opposite of the) projection
\[ \pi:\disccat{\Groth(F)}\To A \]
the following result is now an application of proposition \sref{proposition:derived_enriched_colimit_change_categories}.

\begin{corollary}\label{corollary:Groth_hocolim}
Let $V$ be a cartesian closed simplicial model category with cofibrant unit, and $C$ a cocomplete $V$-model category.\\
Assume that $V_0$ has totally disjoint small coproducts (see \sref{definition:totally_disjoint_coproducts} and \sref{notation:small_totally_disjoint_coproducts}), and that the object $1$ of $V_0$ is connected over $\Set$ (see \sref{definition:connected_object}).\\
Assume furthermore that $A$ is a small locally cofibrant $V$-category, and
\[ G:\op{\internal A}\To\Cat(V_0) \]
is a pointwise locally cofibrant internal $\Cat(V_0)$-valued functor.\\
For any objectwise cofibrant $V$-functor
\[ F:A\To C \]
there is a natural weak equivalence in $C$
\[ \Big(\hocolim_{\op{\left(\disccat{G(-)}\vphantom{M^d}\right)}}1\Big)\tensor^\derived_A F\xTo{\;\sim\;}\hocolim_{\Groth(G)}(F\circ\pi) \]
\end{corollary}

\begin{remark}[relation to nerve]
The left hand side of the weak equivalence in the corollary is related to the nerve of $\disccat{G(-)}$.\\
Observe that for each small $V$-category $B$ there is a natural projection
\[ \hocolim_{\op{B}} 1\To\realization{\BarConst(1,\op{B},1)}=\realization{\BarConst(1,B,1)} \]
which is a weak equivalence under good conditions (e.g.\ if $B$ is identity-cofibrant). The object on the right is the {\em realization of the nerve of \(B\)}.
\end{remark}


%% file: invariant.tex



\chapter{Invariants of $\E^G_n$-algebras}\label{chapter:invariants_En-algebras}

\section*{Introduction}

In the last chapter of this thesis, we present a definition of a homotopical invariant, $\T^G(A;M)$, of an algebra $A$ over the PROP $\E^G_n$, for each $n$-manifold with a $G$-structure. In particular, we obtain an invariant of $E_n$-algebras in the case $G=1$. We will also prove that $\T^1(A,S^1)$ is the topological Hochschild homology of $A$, when $A$ is an associative ring spectrum.

\section*{Summary}

Section \sref{section:the_invariants} defines the simplicial PROPs $S\E^G_n$, and gives a definition of the invariant $\T^G(A;M)$ of an $S\E^G_n$-algebra $A$. It is defined for each $n$-manifold $M$ with a $G$-structure.

Section \sref{section:classifying_spaces_path_categories} calculates the homotopy colimit of the constant functor $1$ along the category $\kappa\big(\disccat{\pathcat}X\big)$ to be $X$. This calculation is used in section \sref{section:relation_invariant_M(M)} to describe the invariant $\T^G(A;M)$ as a homotopy colimit along the category $\disccat{\Total^G_n[M]}$, which is weakly equivalent to $\M(M)$ (as proved in chapter \ref{chapter:sticky<->embeddings}).

Section \sref{section:invariant_relation_THH} uses the results of chapters \ref{chapter:sticky<->embeddings} and \ref{chapter:sticky_conf_S^1} to show that when $A$ is an associative ring spectrum, $\T^1(A,S^1)$ is weakly equivalent to $THH(A)$.

\section{The invariants}\label{section:the_invariants}

We will now define the desired invariants of algebras over the PROPs $\E^G_n$. This will require taking algebras in a $V$-model category for some appropriate symmetric monoidal model category $V$. The topological nature of the $k\Top$-PROPs $\kappa\E^G_n$, and the right modules $\kappa\E^G_n[M]$, would make $k\Top$ the natural choice for the enriching category for our invariants. Unfortunately, our right modules are not valued in CW-complexes. Therefore, in order to easily obtain homotopy invariance of our construction, the model structure on $k\Top$ would have to be the Str{\o}m model structure or a mixed model structure (consult \cite{Cole} or chapter 4 of \cite{May-Sigurdsson} regarding mixed model structures). There are very few instances in the literature (known to the author) of model categories enriched over those model structures in $k\Top$. We will thus define the desired invariants for the case of simplicial model categories. This has the advantage that simplicial model categories are very common, possibly even the norm.

\begin{remark}[simplicial PROP $S\E^G_n$]
Let $n\in\NN$, and $G$ a topological group over $GL(n,\RR)$.\\
Recall the product preserving functors
\begin{align*}
S&:\Top\to\sSet\\
S&:k\Top\to\sSet
\end{align*}
which associate to each space its singular simplicial set.\\
We will be working with the $\sSet$-PROP $S\E^G_n=S\kappa\E^G_n$.
\end{remark}

\begin{definition}[simplicial right modules over $S\E^G_n$]
Let $n\in\NN$, and $G$ a topological group over $GL(n,\RR)$. Let $M$ be a $n$-manifold with a $G$-structure.\\
We define the right module over $S\E^G_n$
\[ S\E^G_n[M]:\op{\big(S\E^G_n\big)}\To\sSet \]
to be the composition of the $s\Set$-functors
\[ \op{\big(S\E^G_n\big)}=\op{\big(S\kappa\E^G_n\big)}\xTo{S(\kappa\E^G_n[M])}S(k\Top)\xTo{S}\sSet\]
\end{definition}

\begin{remark}[clarification]
In the above definition, the right module $\kappa\E^G_n[M]$ over $\kappa\E^G_n$ is defined in \sref{definition:kappaE^f_n}.\\
The category $S(k\Top)$ is the $\sSet$-category associated with the $k\Top$-cate\-go\-ry $k\Top$, and 
\[ S:S(k\Top)\To\sSet \]
is the $\sSet$ functor induced by $S$.
\end{remark}

\begin{definition}[invariants of $S\E^G_n$-algebras]
Let $n\in\NN$, and $G$ a topological group over $GL(n,\RR)$. Let $M$ be a $n$-manifold with a $G$-structure.\\
Let $C$ be a symmetric monoidal simplicial model category (definition \sref{definition:symmetric_monoidal_V-model_category}) with cofibrant unit.\\
For any objectwise cofibrant $S\E^G_n$-algebra $A$ in $C$, we define the {\em \(M\)-indexed invariant of \(A\)} to be
\[ \T^G(A;M)\defeq S\E^G_n[M]\,\tensor^\derived_{\mathclap{S\E^G_n}}\,A \]
\end{definition}

\begin{remark}[cofibrancy conditions]
Given that $C$ is a symmetric monoidal simplicial model category with cofibrant unit, the condition that $A$ be objectwise cofibrant is equivalent to requiring that $A(\RR^n)$ is cofibrant in $C$.\\
Under these conditions, the canonical map
\[ \T^G(A;M)=S\E^G_n[M]\,\tensor^\derived_{\mathclap{S\E^G_n}}\,A\To\realization{\BarConst\!\big(S\E^G_n[M],S\E^G_n,A\big)} \]
is a weak equivalence.
\end{remark}

\begin{remark}[functoriality of invariant]
The above construction is easily seen to extend to a functor
\[ \T:S\E^G_n\dash\alg(C)\times S\REmb^G_n\To C \]
\end{remark}

\begin{proposition}[homotopy invariance]
Let $n\in\NN$, and $G$ a topological group over $GL(n,\RR)$. Let $M$ be a $n$-manifold with a $G$-structure.\\
Let $C$ be a symmetric monoidal simplicial model category with cofibrant unit.\\
Given a weak equivalence
\[ F:A\To B \]
of objectwise cofibrant $S\E^G_n$-algebras $C$, the induced map
\[ \T^G(F;M):\T^G(A;M)\To\T^G(B;M) \]
is a weak equivalence in $C$.
\end{proposition}

\begin{remark}[modifying the definition if unit of $C$ is not cofibrant]
It may be useful to remove the condition that the unit $I$ is a cofibrant object of $C$ from the definition of $\T^G(-;M)$. This would allow us to apply it to the category of spectra from \cite{EKMM}, for example.\\
If $C$ is a symmetric monoidal simplicial model category in which the unit $I$ is not cofibrant in $C$, it is necessary to require that the unit map of the $S\E^G_n$-algebra $A$
\[ I\To A(\RR^n) \]
(coming from the unique morphism $\emptyset\to\RR^n$ in $\E^G_n$) is a cofibration in $C$. This guarantees that the tensor powers of $A(\RR^n)$, appearing as values of $A$, have the correct homotopy type.\\
Moreover, to obtain the ``correct'' answer, and maintain homotopy invariance of $\T^G(A;M)$, the definition would have to be modified to
\[ \widetilde{\T}(A;M)\defeq S\E^G_n[M]\,\tensor^\derived_{\mathclap{S\E^G_n}}\,A^{\text{cof}} \]
where $A^{\text{cof}}$ indicates an objectwise cofibrant replacement of $A$.\\
The remainder of the text would hold true with this modification in place.
\end{remark}

\section{Classifying spaces of path categories}\label{section:classifying_spaces_path_categories}

In the next section we will apply corollary \sref{corollary:Groth_hocolim} to the invariant $\T^G(A;M)$ to obtain it as a homotopy colimit along the (simplicial category associated to the) Grothendieck construction $\Total^G_n[M]$ of $\pathcat\circ\internal\E^G_n[M]$. With that in mind, we will show that
\[ \hocolim_{\kappa(\disccat{\pathcat X})}1\simeq X \]
for most topological spaces $X$.

\begin{construction}\label{construction:ev_real(pathX)->X}
Recall the nerve of an internal category from definition \sref{definition:nerve_functor}.\\
Given a topological space $X$, there is a canonical map
\[ ev:\realization{\Nerve\big(\pathcat(X)\big)}\To X \]
given on $k$-simplices ($k\in\NN$) by the formula
\[ ev:\Nerve\big(\pathcat(X)\big)(k+1)\times\Delta^k\To X \]
\[ ev\big((\gamma_i,\tau_i)_{i=1}^k,(\mathtt{t}_i)_{i=0}^k\big)\defeq (\gamma_1\ast\cdots\ast\gamma_k)\left(\sum_{i=1}^k\tau_i(\mathtt{t}_i+\cdots+\mathtt{t}_k)\right) \]
for
\[ (\gamma_i,\tau_i)_{i=1}^k\in\Nerve(\pathcat X)(k+1)=\overbrace{H(X)\underset{X}{_t\times_s}H(X)\underset{X}{_t\times_s}\cdots\underset{X}{_t\times_s}H(X)}^{k} \]
\[ (\mathtt{t}_i)_{i=0}^k\in\Delta^k=\set{x\in\RR^{k+1}\suchthat\mbox{$\sum_{i=0}^{k}$}x_i=1} \]
Also, $\gamma_1\ast\cdots\ast\gamma_k$ denotes the path obtained by concatenating all the Moore paths $(\gamma_i,\tau_i)$.
\end{construction}

\begin{remark}
The map
\[ ev:\realization{\Nerve\big(\pathcat(X)\big)}\To X \]
can easily be seen to be a homotopy equivalence.
\end{remark}

\begin{construction}
Let $X$ be a topological space in $k\Top$.\\
There is a canonical map
\[ \restrict{\BarConst\big(1,\kappa(\disccat{\pathcat}X),1\big)}{\op{\Delta}}\Into\Nerve\big(\pathcat(X)\big) \]
given by the canonical inclusion objectwise. We obtain an induced map on the geometric realizations
\[ \inclusion:\realization{\BarConst\big(1,\kappa(\disccat{\pathcat}X),1\big)}\To\realization{\Nerve\big(\pathcat(X)\big)} \]
(where the left realization is computed in $k\Top$, and the one on the right is computed in $\Top$).
We define the map
\[ \mathtt{ev}:\hocolim_{\kappa(\disccat{\pathcat}X)}1\To X \]
as the composition
\[ \hocolim_{\kappa(\disccat{\pathcat}X)}1\xTo{\proj}\realization{\BarConst\big(1,\kappa(\disccat{\pathcat}X),1\big)}\xTo{\inclusion}\realization{\Nerve\big(\pathcat(X)\big)}\xTo{ev}X \]
\end{construction}

\begin{remark}[metrizable spaces]\label{remark:metrizable_spaces_kTop}
Any metrizable topological space is in $k\Top$.\\
Any finite product of metrizable topological spaces is metrizable, and thus is in $k\Top$. Therefore the finite product in $k\Top$ of metrizable spaces is computed in $\Top$.\\
Moreover, given a second countable locally compact Hausdorff space $K$, the space $\Map(K,X)$ is metrizable, and therefore in $k\Top$.\\
Putting all these remarks together, we conclude that for any metrizable space $X$, the path category $\pathcat(X)$ {\em coincides} with the $k\Top$-category $\kappa\big(\pathcat(X)\big)$.
\end{remark}

We leave the following lemma to be proved by the reader. It uses the characterization of Str{\o}m cofibrations as strong neighborhood deformation retracts.

\begin{lemma}\label{lemma:X_well_pointed_implies_loopsX_well_pointed}
If $X$ is a topological space, and $x\in X$ is such that
\[ \set{x}\Into X \]
is a Str{\o}m cofibration, then
\[ \set{e}\Into H(X;x,x) \]
is a Str{\o}m cofibration. Here, $H(X;x,x)$ is the space of Moore loops on $X$ based at $x$, and $e$ is the zero-length loop in $H(X;x,x)$.
\end{lemma}

\begin{lemma}\label{lemma:ev_hocolim_equiv_X_kTop}
For any metrizable space $X$, the map
\[ \mathtt{ev}:\hocolim_{\kappa(\disccat{\pathcat}X)}1\To X \]
is a weak equivalence. If $X$ is homotopy equivalent to a CW-complex, then this map is a homotopy equivalence.
\end{lemma}
\begin{proof}
For any space $Y$, and any subset $T$ of $Y$, define $\disccat{\pathcat}(Y;T)$ to be the full $\Top$-subcategory of $\disccat{\pathcat}(Y)$ generated by $T$.

Let $P$ be a subset of $X$ such that the canonical map of sets
\[ P\Into X\xTo{\proj}\pi_0 X \]
is a bijection $P\to\pi_0 X$. Factor $P\Into X$ as
\[ P\Into\overline{X}\xTo{\sim}X \]
where $\overline{X}$ is metrizable, the first map is a cofibration in $k\Top$, and the second one is a trivial fibration (recall that we use the Str{\o}m model structure on $k\Top$; the factorization constructed in \cite{Strom} verifies these conditions).
Consider the following commutative diagram
\begin{diagram}[midshaft]
\hocolim_{\kappa(\disccat{\pathcat}(X;P))}1&\rInto{\ \ \;\sim\ \ }&\hocolim_{\kappa(\disccat{\pathcat}X)}1&\rTo{\mathtt{ev}}&X\\
\uTo[uppershortfall=0.2em]{\rotc{90}{$\scriptstyle\sim$}}&&\uTo[uppershortfall=0.2em]{\rotc{90}{$\scriptstyle\sim$}}&&\uTo{\rotc{90}{$\scriptstyle\sim$}}\\
\hocolim_{\kappa(\disccat{\pathcat}(\overline{X};P))}1&\rInto{\ \;\sim\ }&\hocolim_{\kappa(\disccat{\pathcat}\overline{X})}1&\rTo{\mathtt{ev}}&\overline{X}\\
\end{diagram}
where all the arrows marked $\xto{\,\sim\,}$ are homotopy equivalences. The vertical arrows are homotopy equivalences because $\overline{X}\to X$ is a homotopy equivalence, and so the $k\Top$-functors
\begin{align*}
\kappa\big(\disccat{\pathcat}\overline{X}\big)&\To\kappa\big(\disccat{\pathcat}X\big)\\
\kappa\big(\disccat{\pathcat}(\overline{X};P)\big)&\To\kappa\big(\disccat{\pathcat}(X;P)\big)
\end{align*}
are weak equivalences with respect to $k\Top$ (see definition \sref{definition:weak_equivalence_V-categories}) with the Str{\o}m model structure, and thus homotopy cofinal (see propositions \sref{proposition:weak_equivalence_implies_homotopy_cofinal} and \sref{proposition:homotopy_cofinal_colimits_equiv}). The horizontal inclusions are homotopy equivalences because $P\to\pi_0 X=\pi_0\overline{X}$ is a bijection and therefore
\begin{align*}
\kappa\big(\disccat{\pathcat}(\overline{X};P)\big)&\To\kappa\big(\disccat{\pathcat}\overline{X}\big)\\
\kappa\big(\disccat{\pathcat}(X;P)\big)&\To\kappa\big(\disccat{\pathcat}X\big)
\end{align*}
are also weak equivalences of $k\Top$-categories.

We will show that the composition of the bottom row
\[ \mathtt{ev}:\hocolim_{\kappa(\disccat{\pathcat}(\overline{X};P))}1\To\overline{X} \]
is a weak equivalence. First observe that this map factors as
\[ \hocolim_{\kappa(\disccat{\pathcat}(\overline{X};P))}1\xTo{\proj}\realization{\BarConst\big(1,\kappa\big(\disccat{\pathcat}(\overline{X};P)\big),1\big)}\rTo{ev}\overline{X} \]
by construction of the map $\mathtt{ev}$. The left arrow is a homotopy equivalence because $\kappa\big(\disccat{\pathcat}(\overline{X};P)\big)$ is identity-cofibrant (definition \sref{definition:identity-cofibrant_V-cat}): this follows from $P\hookrightarrow\overline{X}$ being a cofibration, together with lemma \sref{lemma:X_well_pointed_implies_loopsX_well_pointed} and remark \sref{remark:metrizable_spaces_kTop}. Thus we are left with proving that 
\begin{equation}\label{equation:aux_ev_Bar(1,path(X,P))->X}
ev:\realization{\BarConst\big(1,\kappa\big(\disccat{\pathcat}(\overline{X};P)\big),1\big)}\rTo\overline{X}
\end{equation}
is a weak equivalence. This is an immediate consequence of the natural isomorphism
\begin{align*}
\realization{\BarConst\big(1,\kappa\big(\disccat{\pathcat}(\overline{X};P)\big),1\big)}&=\coprod_{p\in P}\realization{\BarConst\big(1,\kappa\big(\disccat{\pathcat}(X;\set{p})\big),1\big)}\\
&=\coprod_{p\in P}B\big(\kappa\,H(\overline{X},p,p)\big)
\end{align*}
where $B\big(\kappa\,H(\overline{X},p,p)\big)$ denotes the classifying space (computed in $k\Top$) of the topological group $\kappa\,H(\overline{X},p,p)$ of Moore loops based at $p$. The map \seqref{equation:aux_ev_Bar(1,path(X,P))->X} is seen to be a weak equivalence as an immediate consequence of lemma 15.4 of \cite{May}, which shows that $B\big(\kappa\,H(\overline{X},p,p)\big)$ maps by a weak equivalence to the path component of $p$ in $\overline{X}$. Note only that the map in lemma 15.4 of \cite{May} differs from \seqref{equation:aux_ev_Bar(1,path(X,P))->X} by a reversal of the simplices, i.e.\ a homeomorphism of the source (compare the formula there with the formula in construction \sref{construction:ev_real(pathX)->X}).

Assume now that $X$ is homotopy equivalent to a CW-complex. In order to prove that
\[ \mathtt{ev}:\hocolim_{\kappa(\disccat{\pathcat}X)}1\To X \]
is a homotopy equivalence, it is enough to show that \seqref{equation:aux_ev_Bar(1,path(X,P))->X} is a homotopy equivalence. From what we have already proved, it suffices to show that the source of \seqref{equation:aux_ev_Bar(1,path(X,P))->X} is homotopy equivalent to a CW-complex. This follows from the homotopy equivalence
\[ \realization{\BarConst\big(1,\realization{S\big(\disccat{\pathcat}(\overline{X};P)\big)},1\big)}\xTo{\ \sim\ }\realization{\BarConst\big(1,\kappa\big(\disccat{\pathcat}(\overline{X};P)\big),1\big)} \]
where the $k\Top$-category $\realization{S\big(\disccat{\pathcat}(\overline{X};P)\big)}$ is obtained by applying $S$ and then geometric realization to $\kappa\big(\disccat{\pathcat}(\overline{X};P)\big)$. That map is a homotopy equivalence because the canonical $k\Top$-functor from which it arises
\[ F:\realization{S\big(\disccat{\pathcat}(\overline{X};P)\big)}\To\kappa\big(\disccat{\pathcat}(\overline{X};P)\big) \]
is an essentially surjective local homotopy equivalence, and therefore a weak equivalence, of identity-cofibrant $k\Top$-categories. We just need to check all these conditions for $F$. The functor $F$ is obviously essentially surjective, and a local weak equivalence. We have proved that the target of $F$ is identity-cofibrant earlier in this proof, and the source is clearly identity-cofibrant. Since the morphism spaces of the source of $F$ are CW-complexes, it remains to show that the morphism spaces of the target of $F$ are homotopy equivalent to CW-complexes.

We are thus left with proving that $H(\overline{X},p,p)$ is homotopy equivalent to a CW-complex for each $p\in P$ (recall observation \sref{remark:metrizable_spaces_kTop}). Since $H(\overline{X},p,p)$ is homotopy equivalent to $\Omega_p\overline{X}$, we can use the results from \cite{Milnor} (namely, corollary 3) to finish our proof. All we need to check is that the pair $(\overline{X},\set{p})$ is homotopy equivalent (as a pair) to a CW-pair. This follows easily from our current assumption that $X$, and hence $\overline{X}$, is homotopy equivalent to a CW-complex, together with the fact that the inclusion $\set{p}\hookrightarrow\overline{X}$ is a Str{\o}m cofibration.
\end{proof}

\begin{remark}
Our assumption that $X$ be metrizable serves only to nullify the effects of applying the functor $\kappa$ (by observation \sref{remark:metrizable_spaces_kTop}), thus minimizing any complications from switching to $k\Top$.
\end{remark}

We will now translate these results to the simplicial world.

\begin{construction}
Let $X$ be a topological space.\\
Define the map of simplicial sets
\[ \mathtt{ev}:\hocolim_{S(\disccat{\pathcat}X)}1\To SX \]
to be the adjoint to the map in $k\Top$
\[ \realization{\hocolim_{S(\disccat{\pathcat}X)}1}= \hocolim_{\realization{S(\disccat{\pathcat}X)}}1\xTo{\,\sim\,}\hocolim_{\kappa(\disccat{\pathcat}X)}1\xTo{\mathtt{ev}}X \]
where the middle arrow is a weak equivalence induced by the canonical $k\Top$-functor
\[ \realization{S\big(\disccat{\pathcat}X\big)}\To\kappa\big(\disccat{\pathcat}X\big) \]
which is an essentially surjective local weak equivalence.\\
This map is natural in the topological space $X$.
\end{construction}

The following proposition is an easy consequence of \sref{lemma:ev_hocolim_equiv_X_kTop}.

\begin{proposition}
If $X$ is a metrizable topological space, the map
\[ \mathtt{ev}:\hocolim_{S(\disccat{\pathcat}X)}1\To SX \]
is a weak equivalence of simplicial sets.
\end{proposition}

\begin{remark}
The condition that $X$ be metrizable is not essential.
\end{remark}

\begin{remark}\label{remark:hocolim_op_path_equiv_X}
Noticing that
\[ \hocolim_{\op{A}}1=1\tensor^\derived_{\op{A}}1=1\tensor^\derived_A 1=\hocolim_A 1 \]
for any $\sSet$-category $A$, we obtain a $\sSet$-natural map
\[ \mathtt{ev}:\hocolim_{\op{S(\disccat{\pathcat}X)}}1\To SX \]
for each topological space $X$, which is a weak equivalence when $X$ is metrizable.
\end{remark}

\section{Relation to $\M(M)$}\label{section:relation_invariant_M(M)}

In this section we will show that the invariant $\T^G(-;M)$ is a homotopy colimit along (the simplicial category associated with) $\disccat{\Total^G_n[M]}$. Since $\disccat{\Total^G_n[M]}$ is weakly equivalent (by a zig-zag) to $\M(M)$ (propositions \sref{proposition:equivalence_T^G_n} and \sref{proposition:M(M)_zig-zag_equiv_TotalM}), this connects the invariant $\T^G(-;M)$ to the category $\M(M)$. Let
\[ S\pi:S\disccat{\Total^G_n[M]}\To S\E^G_n \]
be the canonical projection.

\begin{proposition}\label{proposition:invariant_hocolim_Total^G_n}
Let $n\in\NN$, and $G$ a topological group over $GL(n,\RR)$. Let $M$ be a $n$-manifold with a $G$-structure.\\
Let $C$ be a symmetric monoidal simplicial model category with cofibrant unit.\\
There is a natural zig-zag of weak equivalences in $C$
\[ \T^G(A;M)\xlongleftarrow{\;\sim\;}\bullet\xTo{\;\sim\;}\hocolim_{S(\disccat{\Total^G_n[M]})}(A\circ S\pi) \]
for each objectwise cofibrant $S\E^G_n$-algebra $A$ in $C$.
\end{proposition}
\begin{proof}
The object $\bullet$ is given by
\[ \bullet\defeq\Big(\hocolim_{\op{(\disccat{G(-)})}}1\Big)\tensor^\derived_{\E^G_n} A \]
where the internal $\Cat(\sSet)$-valued functor
\[ G:\op{\internal\big(S\E^G_n\big)}\To\Cat(\sSet) \]
is simply (recall definition \sref{definition:path_category_presheaf} for the meaning of ``$\pathcat$'' in this case)
\[ G\defeq\Cat(S)\big(\pathcat\circ\internal\E^G_n[M]\big) \]

By proposition \sref{proposition:Groth_transfer_categories}
\[ \Groth(G)=\Cat(S)\left(\Groth\big(\pathcat\circ\internal\E^G_n[M]\big)\right)=\Cat(S)\big(\Total^G_n[M]\big)\]
where the last equality comes from the definition of $\Total^G_n[M]$. Proposition \sref{proposition:transfer_internal_enriched_categories} now ensures
\[ \disccat{\Groth(G)}=S\big(\disccat{\Total^G_n[M]}\big) \]
The weak equivalence
\[ \bullet\xTo{\;\sim\;}\hocolim_{S(\disccat{\Total^G_n[M]})}(A\circ S\pi) \]
is therefore a consequence of corollary \sref{corollary:Groth_hocolim}. The weak equivalence
\[ \T^G(A;M)=S\E^G_n[M]\,\tensor^\derived_{\mathclap{S\E^G_n}}\,A\xlongleftarrow{\;\sim\;}\bullet \]
is induced by the natural weak equivalence (see remark \sref{remark:hocolim_op_path_equiv_X})
\[ \mathtt{ev}:\hocolim_{\op{S(\disccat{\pathcat}X)}}1\To SX \]
applied to the values (definition \sref{definition:value_internal_Top_presheaf}) of $\internal\E^G_n[M]$:
\[ X=\internal\E^G_n[M](k\times\RR^n)=\REmb^G_n(k\times\RR^n,M) \]
for $k\in\NN$. Note that $\REmb^G_n(k\times\RR^n,M)$ is a metrizable space for each $k\in\NN$.
\end{proof}

\section{Relation to topological Hochschild homology}\label{section:invariant_relation_THH}

In this section, we will apply the results of chapter \ref{chapter:sticky_conf_S^1} to conclude that $\T^G(-;M)$ recovers topological Hochschild homology of associative ring spectra, when $G=1$ and $M$ is the parallelized manifold $S^1$. We assume that $(\spectra,\wedge,S)$ is a symmetric monoidal simplicial model category in which the unit $S$ is cofibrant. This holds for the category of symmetric spectra.

Recall from section \sref{section:PROPs_augmented_embeddings} that there are weak equivalences of $\Top$-PROPs
\[ \omega:\E^1_1\xTo{\sim}\E^{\smash{GL^+(1,\RR)}}_1\xTo{\sim}\Ass \]
We will denote the corresponding weak equivalence of $\sSet$-PROPs by
\[ S\omega:S\E^1_1\xTo{\sim}\Ass \]

\begin{proposition}
Let $\underline{A}$ be an objectwise cofibrant $\Ass$-algebra in the symmetric monoidal category of spectra, $(\spectra,\wedge,S)$. Let $A$ denote the underlying associative monoid of $\underline{A}$ (see example \sref{example:associative_PROP}).\\
There exists a zig-zag of weak equivalences in $\spectra$ connecting $THH(A)$ and $\T^1(\underline{A}\circ S\omega,S^1)$, where $S^1$ is viewed as a parallelized manifold. The zig-zag is natural in the $\Ass$-algebra $\underline{A}$.
\end{proposition}
\begin{proof}
According to proposition \sref{proposition:invariant_hocolim_Total^G_n} there is a natural zig-zag of weak equivalences
\begin{equation}\label{equation:aux_T(A,S^1)_equiv_THH(A)_1}
\T^1(\underline{A}\circ S\omega,S^1)\xlongleftarrow{\sim}\bullet\xTo{\sim}\hocolim_{S(\disccat{\Total^1_1[S^1]})}(\underline{A}\circ S\omega\circ S\pi)
\end{equation}

Consider now the diagram
\begin{diagram}[midshaft,hug]
\pi_0\big(\mathcal{Z}_{S^1}\big)&\rTo^{\ \sref{proposition:M(M)_zig-zag_equiv_TotalM}\ }_{\ \sim\ }&\pi_0\big(\M(S^1)\big)&\lTo^{\ \sref{proposition:rho_RR_MZ(RR)->M(S^1)}\ }_{\ \sim\ }&\pi_0\big(\M_\ZZ(\RR)\big)\\
\dTo^{\sref{proposition:M(M)_zig-zag_equiv_TotalM}}_{\rotc{90}{$\scriptstyle\sim$}}&&&\rdDashto_{\sim}&\dTo^{\rotc{-90}{$\scriptstyle\sim$}}_{\sref{proposition:ZO_equiv}}\\
\pi_0\big(\disccat{\Total_1[S^1]}\big)&&\rDashto^{\sim}_{f}&&\EE
\end{diagram}
where the full arrows are equivalences of categories determined by the references next to the arrows. Thus we can construct the dashed arrows in an essentially unique way so that the diagram commutes up to natural isomorphisms. In conclusion, we obtain a weak equivalence
\[ F:\disccat{\Total_1[S^1]}\xTo[\proj]{\sim}\pi_0\big(\disccat{\Total_1[S^1]}\big)\xTo[f]{\sim}\EE \]
of $\Top$-categories, since $\Total_1[S^1]$ is homotopically discrete by propositions \sref{proposition:M(M)_zig-zag_equiv_TotalM} and \sref{corollary:M(S^1)_locally_homotopically_discrete}.

Consider now the diagram
\begin{diagram}[midshaft]
\disccat{\Total^1_1[S^1]}&\rTo^{\ \;\disccat{q}\;\ }_{\sim}&\disccat{\Total_1[S^1]}&\rTo^{\ F\ }_{\sim}&\EE\\
\dTo{\pi}&&&&\dTo^{\psi}_{\sref{construction:psi_EE->OrdSigma}}\\
\E^1_1&&\rTo{\omega}&&\Ord\Sigma
\end{diagram}
where $\disccat{q}$ is the weak equivalence from lemma \sref{lemma:local_equivalence_T^G_n}. We leave to the reader the straightforward check that this diagram commutes up to natural isomorphism. This is true assuming we chose the correct orientation on $S^1$: this orientation depends on the choice of $\omega$. Applying the singular simplicial set functor, $S$, to the diagram preserves the weak equivalences. Thus we get a weak equivalence
\begin{equation}\label{equation:aux_T(A,S^1)_equiv_THH(A)_2}
\begin{split}
\hocolim_{S(\disccat{\Total^1_1[S^1]})}(\underline{A}\circ S\omega\circ S\pi)&\simeq\hocolim_{S(\disccat{\Total^1_1[S^1]})}\big(\underline{A}\circ\psi\circ SF\circ S\disccat{q}\big)\\
&\xTo{\sim}\hocolim_{\EE}(\underline{A}\circ\psi)
\end{split}
\end{equation}
given that both $SF$ and $S\disccat{q}$ are weak equivalences, and therefore homotopy cofinal. Finally, proposition \sref{proposition:THH_hocolim_EE} gives us a natural zig-zag of weak equivalences
\begin{equation}\label{equation:aux_T(A,S^1)_equiv_THH(A)_3}
\hocolim_{\EE}(\underline{A}\circ\psi)\xlongleftarrow{\sim}\bullet\xTo{\sim}THH(A)
\end{equation}

Putting together \seqref{equation:aux_T(A,S^1)_equiv_THH(A)_1}, \seqref{equation:aux_T(A,S^1)_equiv_THH(A)_2}, and \seqref{equation:aux_T(A,S^1)_equiv_THH(A)_3} gives the required natural zig-zag of weak equivalences.
\end{proof}

\comment{
\section{Comparison with other constructions}
Salvatore: map to Fulton-MacPherson operads and right modules\\
Lurie: proposition 4.2.4.1 of HTT \cite{Lurie3}
}


%% file: bibliography.tex
\nocite{*}
\bibliographystyle{amsalpha}
\bibliography{thesis}